\documentclass[english]{article}

\usepackage{amsmath}
\usepackage{longtable}
\usepackage{amsthm}
\usepackage{libertineRoman}
\usepackage[libertine]{newtxmath}
\DeclareMathAlphabet{\mathbb}{U}{msb}{m}{n}
\usepackage[varqu,varl]{inconsolata}
\usepackage[T1]{fontenc}
\usepackage[utf8]{inputenc}
\usepackage{geometry}
\usepackage{color}
\usepackage[english]{babel}
\usepackage[nottoc,notlot,notlof]{tocbibind}
\usepackage{booktabs}
\usepackage{textcomp}
\usepackage{mathrsfs}
\usepackage{mathtools}
\usepackage{url}
\usepackage{enumitem}
\usepackage[backgroundcolor=green!40, bordercolor = white, colorinlistoftodos, textsize=footnotesize]{todonotes}
\usepackage{tikz}
\usepackage[raggedright]{titlesec}
\usepackage{mleftright}
\mleftright

\usepackage{zref-clever}
\zcsetup{
    abbrev=true, 
    cap=true,
    nameinlink=false,
	rangetopair=false,
	endrange=stripprefix,
	sort=false,
	reffont=\upshape,
	lastsep={, and },
	comp=false
    }

\zcRefTypeSetup{equation}{
Name-sg = ,
name-sg = ,
Name-pl = ,
name-pl = ,
refbounds = {\textup{(},,,\textup{)}},
+refbounds-rb = {\textup{(},,,},
+refbounds-re = {,,,\textup{)}},
rangesep = {--}
}
\zcRefTypeSetup{thmstepnvi}{
  Name-sg = Step,
  name-sg = step,
  Name-pl = Steps,
  name-pl = steps
}
\zcRefTypeSetup{thmstepnvii}{
  Name-sg = Step,
  name-sg = step,
  Name-pl = Steps,
  name-pl = steps
}

\usepackage[nocfg,refpage,intoc]{nomencl}

\makenomenclature

\usepackage{multicol}
\renewcommand\nompreamble{\begin{multicols}{2}\small\raggedright}
\renewcommand\nompostamble{\end{multicols}}

\usepackage[unicode=true,pdfusetitle,
 bookmarks=true,bookmarksnumbered=false,bookmarksopen=false,
 breaklinks=true,pdfborder={0 0 0},pdfborderstyle={},backref=false,colorlinks=true]
 {hyperref}
\hypersetup{allcolors=blue}
\hypersetup{final}
 
\newcommand{\pagetarget}[1]{%
  \phantomsection%
  \label{#1}%
  \hypertarget{#1}{}
}

\theoremstyle{plain}
\newtheorem{thm}{\protect\theoremname}[section]
\theoremstyle{definition}
\newtheorem{defn}[thm]{\protect\definitionname}
\theoremstyle{remark}
\newtheorem{rem}[thm]{\protect\remarkname}
\theoremstyle{plain}
\newtheorem{prop}[thm]{\protect\propositionname}
\theoremstyle{plain}

\theoremstyle{plain}
\newtheorem{lem}[thm]{\protect\lemmaname}

\zcsetup{countertype={thm=theorem}}
\AddToHook{env/prop/begin}{\zcsetup{countertype={thm=proposition}}}
\AddToHook{env/lem/begin}{\zcsetup{countertype={thm=lemma}}}
\AddToHook{env/rem/begin}{\zcsetup{countertype={thm=remark}}}
\AddToHook{env/cor/begin}{\zcsetup{countertype={thm=corollary}}}
\AddToHook{env/defn/begin}{\zcsetup{countertype={thm=definition}}}
\newcommand{\cref}[1]{\text{\zcref{#1}}} 

\usepackage{verbatim}

\numberwithin{equation}{section}
\numberwithin{figure}{section}
\numberwithin{table}{section}

\newlist{thmstepnv}{enumerate}{4}
\setlist[thmstepnv]{leftmargin=*,align=left,wide,labelwidth=!,itemindent=!,labelindent=0pt}
\setlist[thmstepnv,1]{label={\itshape {\thmstepname} \arabic*.},ref=\arabic*}
\setlist[thmstepnv,2]{label={\itshape {\thmstepname} {\thethmstepnvi\alph*}.},ref=\thethmstepnvi\alph*}
\setlist[thmstepnv,3]{label={\itshape {\thmstepname\ \alph*.}},ref=\alph*}
\setlist[thmstepnv,4]{label={\itshape {\thmstepname} \arabic*.},ref=\arabic*}
\theoremstyle{plain}
\newtheorem{cor}[thm]{\protect\corollaryname}
\newlist{casenv}{enumerate}{4}
\setlist[casenv]{leftmargin=*,align=left,widest={iiii}}
\setlist[casenv,1]{label={{\itshape\ \casename} \arabic*.},ref=\arabic*}
\setlist[casenv,2]{label={{\itshape\ \casename} \roman*.},ref=\roman*}
\setlist[casenv,3]{label={{\itshape\ \casename\ \alph*.}},ref=\alph*}
\setlist[casenv,4]{label={{\itshape\ \casename} \arabic*.},ref=\arabic*}

\makeatletter
\let\save@mathaccent\mathaccent
\newcommand*\myif@single[3]{%
  \setbox0\hbox{${\mathaccent"0362{#1}}^H$}%
  \setbox2\hbox{${\mathaccent"0362{\kern0pt#1}}^H$}%
  \ifdim\ht0=\ht2 #3\else #2\fi
  }
\newcommand*\myrel@kern[1]{\kern#1\dimexpr\macc@kerna}
\newcommand*\wideaccent[2]{\@ifnextchar^{{\wide@accent{#1}{#2}{0}}}{\wide@accent{#1}{#2}{1}}}
\newcommand*\wide@accent[3]{\myif@single{#2}{\wide@accent@{#1}{#2}{#3}{1}}{\wide@accent@{#1}{#2}{#3}{2}}}
\newcommand*\wide@accent@[4]{%
  \begingroup
  \def\mathaccent##1##2{%
    \let\mathaccent\save@mathaccent
    \if#42 \let\macc@nucleus\first@char \fi
    \setbox\z@\hbox{$\macc@style{\macc@nucleus}_{}$}%
    \setbox\tw@\hbox{$\macc@style{\macc@nucleus}{}_{}$}%
    \dimen@\wd\tw@
    \advance\dimen@-\wd\z@
    \divide\dimen@ 3
    \@tempdima\wd\tw@
    \advance\@tempdima-\scriptspace
    \divide\@tempdima 10
    \advance\dimen@-\@tempdima
    \ifdim\dimen@>\z@ \dimen@0pt\fi
    \myrel@kern{0.6}\kern-\dimen@
    \if#41
      #1{\myrel@kern{-0.6}\kern\dimen@\macc@nucleus\myrel@kern{0.4}\kern\dimen@}%
      \advance\dimen@0.4\dimexpr\macc@kerna
      \let\final@kern#3%
      \ifdim\dimen@<\z@ \let\final@kern1\fi
      \if\final@kern1 \kern-\dimen@\fi
    \else
      #1{\myrel@kern{-0.6}\kern\dimen@#2}%
    \fi
  }%
  \macc@depth\@ne
  \let\math@bgroup\@empty \let\math@egroup\macc@set@skewchar
  \mathsurround\z@ \frozen@everymath{\mathgroup\macc@group\relax}%
  \macc@set@skewchar\relax
  \let\mathaccentV\macc@nested@a
  \if#41
    \macc@nested@a\relax111{#2}%
  \else
    \def\gobble@till@marker##1\endmarker{}%
    \futurelet\first@char\gobble@till@marker#2\endmarker
    \ifcat\noexpand\first@char A\else
      \def\first@char{}%
    \fi
    \macc@nested@a\relax111{\first@char}%
  \fi
  \endgroup
}
\makeatother
\newcommand\mybartilde[1]{\overline{\widetilde{#1}}}
\newcommand\widebartilde{\wideaccent\mybartilde}

\usepackage{anyfontsize}
\DeclareFontFamily{U}{matha}{\hyphenchar\font45}
\DeclareFontShape{U}{matha}{m}{n}{
      <5> <6> <7> <8> <9> <10> gen * matha
      <10.95> matha10 <12> <14.4> <17.28> <20.74> <24.88> matha12
      }{}
\DeclareSymbolFont{matha}{U}{matha}{m}{n}
\DeclareFontSubstitution{U}{matha}{m}{n}

\DeclareFontFamily{U}{mathx}{\hyphenchar\font45}
\DeclareFontShape{U}{mathx}{m}{n}{
      <5> <6> <7> <8> <9> <10>
      <10.95> <12> <14.4> <17.28> <20.74> <24.88>
      mathx10
      }{}
\DeclareSymbolFont{mathx}{U}{mathx}{m}{n}
\DeclareFontSubstitution{U}{mathx}{m}{n}

\DeclareMathDelimiter{\vvvert}{0}{matha}{"7E}{mathx}{"17}

\usepackage{aligned-overset}
\usepackage[section]{placeins}

\providecommand{\casename}{Case}
\providecommand{\conjecturename}{Conjecture}
\providecommand{\corollaryname}{Corollary}
\providecommand{\definitionname}{Definition}
\providecommand{\lemmaname}{Lemma}
\providecommand{\propositionname}{Proposition}
\providecommand{\remarkname}{Remark}
\providecommand{\theoremname}{Theorem}
\providecommand{\thmstepname}{Step}

\newcommand*{\R}{\mathbb{R}}

\newcommand*{\N}{\mathbb{N}}
\newcommand*{\Z}{\mathbb{Z}}

\newcommand*{\E}{\mathbb{E}}

\renewcommand*{\P}{\mathbb{P}}
\newcommand*{\al}{\alpha}

\newcommand*{\h}[1]{\hat{#1}}
\newcommand*{\s}[1]{\mathscr{#1}}
\newcommand*{\m}[1]{\mathcal{#1}}

\newcommand*{\dn}{\mathrm{d}}

\newcommand*{\ds}{\,\mathrm{d}}
\renewcommand*{\d}{\;\mathrm{d}}

\newcommand*{\der}[2]{\frac{\dn #1}{\dn #2}}

\newcommand*{\abs}[1]{\left |#1 \right|}
\newcommand*{\abss}[1]{|#1|}
\newcommand*{\absb}[1]{\bigl|#1 \bigr|}
\newcommand*{\norm}[1]{\left \|#1 \right\|}
\newcommand*{\op}[1]{\operatorname{#1}}
\newcommand*{\ti}[1]{\tilde{#1}}

\newcommand*{\tbf}[1]{\textup{\textbf{#1}}}

\newcommand*{\anon}{\,\cdot\,}

\renewcommand*{\And}{\quad \text{and} \quad}
\newcommand*{\For}{\quad \text{for }}
\newcommand*{\ForAll}{\quad \text{for all }}

\DeclareMathOperator*{\argmin}{arg\,min}
\DeclareMathOperator*{\argmax}{arg\,max}

\let\bar\smallbar

\begin{document}
\global\long\def\e{\mathrm{e}}%
\global\long\def\cc{\mathrm{c}}%

\global\long\def\ii{\mathrm{i}}%

\global\long\def\RR{\mathbb{R}}%
\global\long\def\Pr{\mathbb{P}}%
\global\long\def\EE{\mathbb{E}}%
\global\long\def\NN{\mathbb{N}}%
\global\long\def\ZZ{\mathbb{Z}}%
\global\long\def\PP{\mathbb{P}}%

\global\long\def\sfL{\mathsf{L}}%
\global\long\def\sfS{\mathsf{S}}%
\global\long\def\sfT{\mathsf{T}}%
\global\long\def\sfU{\mathsf{U}}%
\global\long\def\sfR{\mathsf{R}}%

\global\long\def\clG{\mathcal{G}}%
\global\long\def\clH{\mathcal{H}}%
\global\long\def\clU{\mathcal{U}}%
\global\long\def\clV{\mathcal{V}}%
\global\long\def\clW{\mathcal{W}}%
\global\long\def\clC{\mathcal{C}}%
\global\long\def\clT{\mathcal{T}}%
\global\long\def\clM{\mathcal{M}}%
\global\long\def\clX{\mathcal{X}}%
\global\long\def\clZ{\mathcal{Z}}%
\global\long\def\clQ{\mathcal{Q}}%
\global\long\def\clK{\mathcal{K}}%
\global\long\def\clS{\mathcal{S}}%
\global\long\def\clY{\mathcal{Y}}%
\global\long\def\clA{\mathcal{A}}%
\global\long\def\clL{\mathcal{L}}%
\global\long\def\clI{\mathcal{I}}%
\global\long\def\clJ{\mathcal{J}}%
\global\long\def\clR{\mathcal{R}}%
\global\long\def\scrF{\mathscr{F}}%
\global\long\def\scrG{\mathscr{G}}%
\global\long\def\scrX{\mathscr{X}}%

\global\long\def\ttA{\mathtt{A}}
\global\long\def\ttB{\mathtt{B}}
\global\long\def\ttE{\mathtt{E}}
\global\long\def\ttH{\mathtt{H}}
\global\long\def\ttR{\mathtt{R}}
\global\long\def\ttS{\mathtt{S}}
\global\long\def\ttXi{\mathtt{\Xi}}

\global\long\def\suchthat{:}

\global\long\def\FBSDE{\mathrm{FBSDE}}%
\global\long\def\SPDE{\mathrm{SPDE}}%
\global\long\def\stab{\mathrm{stab}}%
\global\long\def\Frob{\mathrm{F}}%
\global\long\def\op{\mathrm{op}}%
\global\long\def\Id{\mathrm{Id}}%

\global\long\def\Qbar{\overline{Q}}%

\global\long\def\setto{\mapsfrom}%

\global\long\def\anon{\cdot}%

\global\long\def\Law{\operatorname{Law}}%

\global\long\def\dist{\operatorname{dist}}%

\global\long\def\tr{\operatorname{tr}}%

\global\long\def\sgn{\operatorname{sgn}}%

\global\long\def\argmin{\operatorname*{argmin}}%

\global\long\def\argmax{\operatorname*{argmax}}%

\global\long\def\supp{\operatorname{supp}}%

\global\long\def\diam{\operatorname{diam}}%

\global\long\def\spn{\operatorname{span}}%

\global\long\def\Vol{\operatorname{Vol}}%

\global\long\def\Re{\operatorname{Re}}%

\global\long\def\Lip{\operatorname{Lip}}%

\global\long\def\Im{\operatorname{Im}}%

\global\long\def\esssup{\operatorname*{ess\,sup}}%

\global\long\def\dif{\mathrm{d}}%

\global\long\def\Dif{\mathrm{D}}%

\global\long\def\eps{\varepsilon}%

\global\long\def\Im{\operatorname{Im}}%

\newcommand*{\lipset}{\Lambda}
\newcommand*{\quadset}{\Sigma}

\newcommand{\CG}[1]{{\color{green!60!black}\textbf{CG:} #1}}
\newcommand{\AD}[1]{{\color{magenta}\textbf{AD:} #1}}
\newcommand{\ad}[1]{{\color{magenta}{#1}}}

\title{Renormalization flow for the 2D nonlinear stochastic\texorpdfstring{\\}{} heat equation: pointwise statistics and universality}
\author{Alexander Dunlap\thanks{Department of Mathematics, Duke University, Durham, NC 27708 USA. Email: \protect\url{alexander.dunlap@duke.edu}.}\and Cole Graham\thanks{Department of Mathematics, University of Wisconsin--Madison, Madison, WI
    53706 USA. Email: \protect\url{graham@math.wisc.edu}.}}

\maketitle
\begin{abstract}
  We consider a two-dimensional stochastic heat equation with noise correlated at scale $\rho \ll 1$ and of strength $\abs{\log \rho}^{-1/2} \sigma(v)$ depending nonlinearly on the solution $v$.
  Under certain conditions, the first author and Gu have shown that the one-point statistics of $v$ converge in law as $\rho \to 0$ to the terminal value of an associated forward-backward SDE.
  Here, we show that the 2D stochastic heat equation is stable under renormalization with a new effective nonlinearity tied to the decoupling function of the forward-backward SDE.
  This allows us to extend the pointwise results to a much broader class of nonlinearities.
  We also show that these limiting pointwise statistics are insensitive to the fine details of the noise, a form of universality.
\end{abstract}
\tableofcontents{}

\section{Introduction}

We consider the semilinear stochastic heat equation
\begin{equation}
  \label{eq:u-SPDE}
  \dif u_{t}^{\rho}=\frac{1}{2}\Delta u_{t}^{\rho}\,\dif t+\gamma_{\rho}\sigma(u_{t}^{\rho})\,\dif W_{t}^{\rho},\qquad t>0,\, x\in\R^{2}.
\end{equation}
Here $\sigma$ is a Lipschitz nonlinearity and $\dn W_t^\rho(x)$ is a Gaussian noise that is white in time and correlated in space at scale $\rho^{1/2}\ll 1$.
We are interested in the pointwise behavior of $u_t^\rho(x)$ as $\rho \to 0$, which calls for an attenuation factor $\gamma_{\rho}\sim|\log\rho|^{-1/2}$ due to critical scaling in two dimensions.

The linear case $\sigma(v)=\beta v$ was first studied in \cite{BC98} and has been the focus of significant recent activity, much of which relies on a connection with directed polymers~\cite{CSZ17,CSZ19,GQT19,CSZ21,CZ21,CZ23}.
It is also closely related to the logarithmically attenuated $2$D
KPZ equation studied in \cite{CD20,CSZ20,Gu20,Tao23}.
In this linear problem, $u_t^{\rho}(x)$ has log-normal pointwise statistics in the $\rho\to0$ limit provided $\beta$ lies below an explicit critical value.
At and above the critical value, the pointwise values converge to zero in probability, but at criticality a field-valued limit is obtained~\cite{CSZ21}.

Together with Gu \cite{DG22}, the first author initiated the study of the semilinear problem \cref{eq:u-SPDE}.
In analogy with the linear setting, $u^{\rho}$ has nontrivial limiting pointwise statistics provided $\Lip(\sigma)$ is below the same critical value.
These statistics are not log-normal in general, but rather are characterized by a forward-backward SDE (FBSDE) associated to $\sigma$.

In this work, we show that coarse-grained averages of \cref{eq:u-SPDE} nearly solve their own 2D stochastic heat equations with effective nonlinearities arising from the FBSDE.
That is, \cref{eq:u-SPDE} is stable under renormalization modulo a change of nonlinearity.
We thereby conclude that the Lipschitz threshold from \cite{DG22} need not be sharp: there is a broad class of nonlinearities with large Lipschitz constants that nonetheless produce nontrivial limiting pointwise statistics.
Moreover, these statistics are universal in the sense that they do not depend on the fine structure of the noise.
In fact, we devote most of our attention to a variation on \cref{eq:u-SPDE} in which we first multiply and then smooth the noise:
\begin{equation}
  \label{eq:SPDE}
  \dif v_{t}^{\rho} = \frac{1}{2}\Delta v_{t}^{\rho}\,\dif t + \gamma_{\rho} \m{G}_\rho[\sigma(v_{t}^{\rho})\,\dif W_{t}].
\end{equation}
This proves more convenient for multiscale analysis than \cref{eq:u-SPDE}, so we prove our results first for \cref{eq:SPDE} and then show that the limiting behavior is the same as that for \cref{eq:u-SPDE}.

\subsection{Setup}

We now describe the objects in \cref{eq:u-SPDE} and \cref{eq:SPDE} in detail.
Let $(\Omega,\mathcal{F},\mathbb{P})$ be a probability space and fix a target dimension $m\in\mathbb{N}$.
The solutions $u^\rho,v^\rho \colon \Omega \times \R \times \R^{2} \to \R^{m}$ are random vector-valued functions parametrized by the correlation parameter $\rho>0$.
We suppress the dependence of $v^p$ on $\omega\in\Omega$ in the notation.

Because the solutions are vector-valued, our nonlinearity $\sigma\colon\R^{m}\to\R^{m}\otimes\R^{m}$ is matrix-valued.
Let $\mathcal{H}_{+}^{m}$ denote the set of nonnegative-definite symmetric real $m\times m$ matrices.
We equip $\mathcal{H}_{+}^{m}$ with the metric induced by the Frobenius norm
\nomenclature[zzznormFrob]{$\lvert \cdot\rvert _\Frob$}{Frobenius norm}
\begin{equation*}
  |A|_{\Frob}^2\coloneqq \tr(AA^{\mathrm{T}}) = \tr(A^{2}).
\end{equation*}
Let $\sigma$ belong to the space $\Lip(\R^{m};\mathcal{H}_{+}^{m})$\nomenclature[Lip]{$\Lip(\cdot;\cdot)$}{space of uniformly Lipschitz functions between two spaces}
of uniformly Lipschitz functions $\R^{m}\to\mathcal{H}_{+}^{m}$.

Next, let $\dif W=(\dif W_{t}(x))_{t\in\R,x\in\R^{2}}$ be a standard $\R^{m}$-valued space-time white noise
generating a temporal filtration $\{\mathscr{F}_{t}\}_{t\in\R}$.
Writing $\dif W=(\dif W^{(1)},\ldots,\dif W^{(m)})$ in components,
we formally have
\begin{equation*}
  \EE\left[\dif W_{t}^{(i)}(x)\dif W_{t'}^{(i')}(x')\right]=\delta_{i,i'}\delta(t-t')\delta(x-x').
\end{equation*}

Given $\tau\ge0$, we define the heat operator
\nomenclature[G cltau]{$\clG_\anon$}{convolution with the heat kernel}
\begin{equation*}
  \mathcal{G}_{\tau}v=G_{\tau}*v,
\end{equation*}
where $G_{\tau}(x)=(2\pi \tau)^{-1}\e^{-|x|^{2}/(2\tau)}$
\nomenclature[G tau]{$G_\tau$}{two-dimensional heat kernel $(2\pi\tau)^{-1}\exp[-\lvert  x\rvert ^2/(2\tau)]$}%
denotes the standard heat kernel and $*$ denotes spatial convolution.
In \cref{eq:u-SPDE}, we define the spatially-smoothed noise $\dif W_{t}^{\rho}=G_{\rho}*\dif W_{t}$.
In components, we formally have
\begin{equation*}
  \EE\left[\dif W_{t}^{\rho,(i)}(x)\dif W_{t'}^{\rho,(i')}(x')\right]=\delta_{i,i'}\delta(t-t')G_{2\rho}(x-x').
\end{equation*}
  This choice of spatial smoothing is a matter of convenience.
  Thanks to the universality in \cref{thm:universal-informal} below, one could work equally well with a smooth, compactly-supported mollifier, or one that decays sufficiently rapidly at infinity.%

As $\rho\to0$, $\dif W^{\rho}$ approaches a space-time white noise.
In general, stochastic heat equations with multiplicative space-time white noise are not well-posed in dimension greater than $1$.
In spatial dimension $2$, weakening the noise by a logarithmic factor yields interesting limits as $\rho\to0$.
This reflects the self-similarity, or ``criticality,'' of the equation in two dimensions.
We define
\begin{equation}
\sfL(\tau)=\log(1+\tau) \For \tau \geq 0\label{eq:Ldef}
\end{equation}
\nomenclature[L  ]{$\sfL(\tau)$}{$\log(1+\tau)$, \cref{eq:Ldef}}
and set
\nomenclature[zzzgreek γ]{$\gamma_\rho$}{renormalization $\sqrt{4\pi/\sfL(1/\rho)}$, \cref{eq:gammarhodef}}
\begin{equation}
\gamma_{\rho} = \sqrt{\frac{4\pi}{\mathsf{L}(1/\rho)}}\,.\label{eq:gammarhodef}
\end{equation}
Thus, $\gamma_{\rho}^{2}$ vanishes logarithmically as $\rho\to0$.

We have now fully described the SPDEs \cref{eq:u-SPDE} and \cref{eq:SPDE}.
We emphasize that we first smooth the noise $\dn W$ and then multiply by $\sigma(u)$ in \cref{eq:u-SPDE}, while we do the opposite in \cref{eq:SPDE}.
We informally refer to these as the ``pre-'' and ``post-smoothed'' problems, respectively.
For most of the sequel, we focus on \cref{eq:SPDE}, but we prove a universality result (\zcref{thm:universal-informal} below) that allows our main results to apply equally well to \zcref{eq:u-SPDE} and \zcref{eq:SPDE}.
We interpret the latter in a mild sense: a predictable random field $v^{\rho}$ solves \cref{eq:SPDE} if for all $s<t$,
\begin{equation}
  v_{t}^{\rho}(x) = \mathcal{G}_{t-s}v_{s}^{\rho}(x) + \gamma_{\rho} \int_{s}^{t} \mathcal{G}_{t + \rho - r}[\sigma(v_{r}^{\rho})\,\dif W_{r}](x).\label{eq:mildsoln}
\end{equation}
We interpret this stochastic integral in the Itô--Walsh sense \cite{Wal86,Dal99,Kho14}.
This formulation displays the advantage of post-smoothing: the noise regularizer $\m{G}_\rho$ neatly joins the linear heat evolution $\m{G}_{t-r}$.

We look for solutions $v_t^\rho$ in the spaces $\mathscr{X}_{t}^{\ell}$\nomenclature[Xscr]{$\scrX_t^\ell$}{random fields on $\R^2$ that are $\scrF_t$-measurable and have uniformly bounded $L^\ell(\Omega)$ norm}
of $\R^{m}$-valued random fields $z$ on $\R^{2}$ that are $\mathscr{F}_{t}$-measurable and satisfy
\nomenclature[zzznormtriple]{$\vvvert \cdot\vvvert_\ell$}{maximum $L^\ell(\Omega)$ norm of a random field}
\begin{equation}
\vvvert z\vvvert_{\ell}\coloneqq\sup_{x\in\R^{2}}\left(\EE|z(x)|^{\ell}\right)^{1/\ell}<\infty.\label{eq:triplenormdef}
\end{equation}

Standard fixed-point arguments (similar to those discussed in, e.g.,\ \cite{Kho14}) show that for any $\ell\ge2$, there is a
family of random operators $(\clV_{s,t}^{\sigma,\rho})_{-\infty<s<t<\infty}$ such that if $v_{s}\in\mathscr{X}_{s}^{\ell}$, then $v_t^\rho = \clV_{s,t}^{\sigma,\rho}v_{s}$ is a mild solution of \cref{eq:u-SPDE} for $t \geq s$.
\nomenclature[V cal st]{$\clV_{s,t}^{\sigma,\rho}$}{propagator for the SPDE}
We often write $\clV_{t}^{\sigma,\rho}\coloneqq\clV_{0,t}^{\sigma,\rho}$.
\medskip

In~\cite{DG22}, the first author and Gu related the limiting statistics of $v^\rho$ to the following system of stochastic differential equations:
\nomenclature[Jsigma]{$J_\sigma$}{root decoupling function for the FBSDE \zcref[range]{eq:FBSDE,eq:Jeqn}}
\begin{align}
  \dif\Gamma_{a,Q}^{\sigma}(q) & =J_{\sigma}\bigl(Q-q,\Gamma_{a,Q}^{\sigma}(q)\bigr)\,\dif B(q), &  & a\in\RR^{m},\,0<q<Q;\label{eq:FBSDE}   \\
  \Gamma_{a,Q}^{\sigma}(0)     & =a,                                                 &  & a\in\RR^{m},\,Q\ge0;\label{eq:FBSDEic} \\
  J_{\sigma}(q,b)              & =\bigl[\EE \, \sigma^{2}\bigl(\Gamma_{b,q}^{\sigma}(q)\bigr)\bigr]^{1/2},  &  & q\ge0,\,b\in\RR^{m}.\label{eq:Jeqn}
\end{align}
Here $B$ is a standard $\RR^{m}$-valued Brownian motion and we let $A^{1/2}$ denote the unique positive-definite matrix square root of $A \in \m{H}_+^m$. We refer to \zcref{subsec:MGexptimescale} below, as well as the introduction of~\cite{DG22}, for some intuition on the origin of this FBSDE. Briefly, the parameter $q$ should be thought of as a ``scale'' parameter, and the FBSDE should be thought of as expressing the dynamics at the larger scales (smaller $q$) in terms of an average (expectation) of $\sigma^2$ of the solution on the smallest scale ($q=Q$).

The problem \zcref[range]{eq:FBSDE,eq:Jeqn} is a forward-backward SDE (FBSDE).
Given $J_{\sigma}$, \zcref[range]{eq:FBSDE,eq:FBSDEic} is simply a vector-valued stochastic differential equation.
So the crux of the complete problem \zcref[range]{eq:FBSDE,eq:Jeqn} is to find a sufficiently regular $J_{\sigma}$ such that \zcref[range]{eq:FBSDE,eq:FBSDEic} and \cref{eq:Jeqn} hold simultaneously.
Alternatively, $J_{\sigma}$ can be characterized via a deterministic nonlinear PDE as described in \cref{subsec:decouplingfunctionPDE} below.
We call $J_\sigma$ the \emph{root decoupling function}, as $J_\sigma^2$ is termed the \emph{decoupling function} in the FBSDE literature.

The SDE \zcref[range]{eq:FBSDE,eq:FBSDEic} has a good solution theory provided $J_{\sigma}$ is uniformly Lipschitz in $b$, and this condition influences much of our analysis.
We therefore make the following definition.
\begin{defn}
  \nomenclature[QbarFBSDE]{$\Qbar_\FBSDE(\sigma)$}{maximal existence time for the FBSDE \zcref[range]{eq:FBSDE,eq:Jeqn}, \cref{def:QbarFBSDE}}
  \label{def:QbarFBSDE}
  Given $\sigma\in\Lip(\RR^{m};\clH_{+}^{m})$,\vspace{2pt} let $\Qbar_{\FBSDE}(\sigma)\in[0,\infty]$ be the supremum of all $Q\ge0$ such that there is a continuous function $J_{\sigma}\colon[0,Q]\times\RR^{m}\to\clH_{+}^{m}$ satisfying \zcref[range]{eq:FBSDE,eq:Jeqn} and
  \begin{equation}
    \sup_{q\in[0,Q]}\Lip\bigl(J_{\sigma}(q,\anon)\bigr)<\infty.
    \label{eq:Juniflipschitz}
  \end{equation}
\end{defn}
In Section~\ref{sec:FBSDE}, we show that $\Qbar_{\FBSDE}(\sigma) > 0$ and that $J_\sigma\colon \big[0,\Qbar_\FBSDE(\sigma)\big)\times\RR^m\to\clH^m_+$ is unique.
In fact, we generalize Theorem~1.1 from \cite{DG22} to show that
\begin{equation}
  \label{eq:FBSDEwellposednessgeneralizes}
  \Qbar_{\FBSDE}(\sigma)\ge\Lip(\sigma)^{-2}.
\end{equation}
\cref{thm:PDEthm} below provides stronger lower bounds on $\Qbar_\FBSDE(\sigma)$ in the scalar case $m = 1$.
We now introduce a space of functions with a specified linear growth bound.
Given $M,\beta \in (0, \infty)$, let
\nomenclature[zzzgreek σplus]{$\quadset(M,\beta)$}{space of Lipschitz functions satisfying a quantitative growth bound, \cref{eq:Sigmaplusdef}}
\begin{equation}
 \quadset(M,\beta) \coloneqq \bigl\{\sigma\in\Lip(\RR^{m};\clH_{+}^{m})\suchthat|\sigma(u)|_{\Frob}^{2}\le M+\beta^{2}|u|^{2}\text{ for all }u\in\RR^m\bigr\}.\label{eq:Sigmaplusdef}
\end{equation}
In rough analogy with subcritical linear equations, we focus on nonlinearities of the following type.
\begin{defn}
  \label{def:subcrit}
  A nonlinearity $\sigma$ is \emph{$L^2$-subcritical} if $\sigma \in \Sigma(M,\beta)$ for some $M\in (0, \infty)$ and $\beta \in (0, 1)$, and $\Qbar_\FBSDE(\sigma) > 1$.
\end{defn}

\subsection{Main results}

As indicated earlier, our improvement over~\cite{DG22} relies on the observation that \cref{eq:SPDE} is approximately closed under renormalization, up to a change in nonlinearity.
We state this informally here; for a precise formulation, see \cref{cor:approx-SPDE} below.
\begin{thm}[Informal]
  \label{thm:renorm}
  Let $\sigma$ be $L^2$-subcritical.
  Fix $Q \in (0, 1)$ and define $\widetilde{\sigma} \coloneqq (1 - Q)^{1/2} J_\sigma(Q, \anon)$ and $\widetilde{\rho} \coloneqq \rho^{1 - Q}$.
  Then there is a new white noise $\dn\widetilde{W}$ such that $\m{G}_{\widetilde\rho} v$ is an approximate mild solution of \cref{eq:SPDE} with $(\widetilde\rho,\widetilde\sigma,\dn \widetilde W)$ in place of $(\rho,\sigma,\dn W)$.
\end{thm}
Thus modulo a scaling factor, we can interpret the decoupling function $J_\sigma(Q,\anon)$ as an effective nonlinearity driving \cref{eq:SPDE} at scale $\rho^{1 - Q}$.
Next, we identify the limiting one-point statistics of $v_t^\rho(x)$.
Let $\clW_{2}$ denote the Wasserstein-$2$ metric\nomenclature[W2]{$\clW_2$}{Wasserstein-$2$ metric} and $\langle a \rangle \coloneqq (\abs{a}^2 + 1)^{1/2}$ the Japanese bracket.
\begin{thm}
  \label{thm:mainthm-singlepoint}
  For each $L^2$-subcritical $\sigma$ and $\bar{T} \in [1, \infty)$, there is a constant $C(\sigma,\overline{T}) \in (\e, \infty)$ such that for all $v_0\in L^\infty(\R^2;\R^m)$, $t\in\bigl[\overline{T}^{-1},\overline{T}\bigr]$, $x\in\RR^{2}$, and $\rho\in(0,C^{-1})$, the solution $v^{\rho}$ of \cref{eq:SPDE} satisfies
  \begin{equation}
    \clW_{2}\bigl(v_{t}^{\rho}(x),\Gamma_{a,1}^{\sigma}(1)\bigr)\le C \langle\norm{v_0}_{L^\infty}\rangle \sqrt{\frac{\log|\log\rho|}{|\log\rho|}}\,,\label{eq:mainthm-singlepoint-bound}
  \end{equation}
  where $\bigl(\Gamma_{a,1}^{\sigma}(q)\bigr)_{q \in [0,1]}$ solves the FBSDE \textup{\zcref[range]{eq:FBSDE,eq:Jeqn}} with $Q = 1$ and $a=\clG_{t}v_{0}(x)$.
\end{thm}
This improves on \cite[Theorem 1.2]{DG22} in a number of ways.
First, we relax the assumption $\Lip(\sigma)<1$ to $\Qbar_{\FBSDE}(\sigma)>1$, which is less stringent by \cref{eq:FBSDEwellposednessgeneralizes}.
(This Lipschitz condition is stated as $\Lip(\sigma)<\sqrt{2\pi}$ in~\cite{DG22} due to different normalization.)
This is the most conceptually significant contribution of the present paper.
It uses the renormalization from \cref{thm:renorm} to bootstrap the results of \cite[Theorem 1.2]{DG22} to larger Lipschitz constants.
In the scalar case $m=1$, we use the results of a companion paper \cite{DG25a} to provide easy-to-check conditions ensuring $\Qbar_{\FBSDE}(\sigma)>1$; see \cref{thm:PDEthm} below.

Despite this improvement, we are still in the subcritical regime in the sense that for $\sigma(b)=\beta b$, \cref{def:subcrit} restricts us to $\beta<1$.
In fact, using the explicit solution to the FBSDE given in \cite[Section 1.3]{DG22} and \cref{subsec:examples} below, one can check that $\Qbar_{\FBSDE}(\sigma)=\beta^{-2}$ for $\sigma(b)=\beta b$.
So $\Qbar_{\FBSDE}(\sigma) > 1$ if and only if $\beta < 1$.
This restriction is natural, as there is a phase transition at $\beta=1$ at which, among other phenomena, the FBSDE blows up.
Because we can handle larger Lipschitz constants subject to $\beta$-growth bounds (see \cref{cor:simple-PDE} below), our results indicate that this phase transition is driven by the \emph{growth} rather than the \emph{regularity} of the nonlinearity.
  
Second, in contrast to \cite[Theorem 1.2]{DG22}, we treat vector-valued solutions and remove the hypothesis $\sigma(0)=0$.
This requires a new moment bound (\cref{prop:momentbd}) for solutions of \cref{eq:SPDE}; see \cref{rem:momentbdvswhatwehadbefore}.
We also rely on a ``matrix-valued reverse triangle inequality'' (\cref{prop:matrixvaluedreversetriangleinequality}) that follows from a deep result of Lieb~\cite{Lie73}.

Lastly, \cref{eq:mainthm-singlepoint-bound} gives a quantitative  rate of convergence in the Wasserstein-$2$ distance.
Previous works \cite{CSZ17,DG22} only prove convergence in distribution without a rate, even in the linear case.
We expect the rate in \cref{eq:mainthm-singlepoint-bound} to be tight up to a $\log\log$ factor.
Indeed, if $v_{0} \equiv 1$, the quantity
\begin{equation*}
  |\log\rho|^{1/2} (\m{G}_1v_{t}^{\rho} - 1)
\end{equation*}
converges in distribution to a nontrivial Gaussian random field.
This was known in the linear case \cite[Theorem 2.17]{CSZ17} and for nonlinearities with sufficiently small Lipschitz constant~\cite{Tao22}.
In a companion paper~\cite{DG25b}, we extend these results to the entire $L^2$-subcritical regime.
Thus at the pointwise level, we expect additional contributions to $v^\rho$ of order $|\log\rho|^{-1/2}$.

Strictly speaking, \cite{DG22} treats the pre-smoothed equation \cref{eq:u-SPDE}.
We conflate the pre- and post-smoothed noise in the discussion above because we will show there is little difference between the two.
This is a consequence of a broad stability result that we state informally here:
\begin{thm}[Informal]
  \label{thm:universal-informal}
  Let $\sigma$ be $L^2$-subcritical.
  If $u^\rho$ is an approximate mild solution of \cref{eq:SPDE}, then $u^\rho \approx v^\rho$ on bounded time intervals.
\end{thm}
\noindent
For a precise statement, see \cref{defn:stable} and \cref{thm:stable} below.

Because the mild formulation \cref{eq:mildsoln} does not involve any derivatives, the stability in \cref{thm:universal-informal} is very flexible.
For example, we expect it can treat discrete approximations of \cref{eq:u-SPDE}.
Here, we use it to show agreement between \cref{eq:u-SPDE} and \cref{eq:SPDE}.
\begin{cor}
  \label{cor:pre-post}
  For each $L^2$-subcritical $\sigma$ and $\bar{T} \in (0, \infty)$, there is a constant $C(\sigma,\overline{T}) \in (\e, \infty)$ such that the following holds for all $\rho\in(0,C^{-1})$.
  Let $u_t^\rho$ and $v_t^\rho$ solve \cref{eq:u-SPDE} and \cref{eq:SPDE}, respectively, with the same initial data $v_0\in L^\infty(\R^2;\R^m)$.
  Then
  \begin{equation*}
    \sup_{(t, x) \in [0, \bar{T}] \times \R^2} \E \abs{u_t^\rho(x) - v_t^\rho(x)}^2 \leq C \langle\norm{v_0}_{L^\infty}^2\rangle \frac{\log|\log\rho|}{|\log\rho|}.
  \end{equation*}
  In particular, $u_t^\rho$ also satisfies \cref{eq:mainthm-singlepoint-bound} and \cref{eq:mainthm-multipoint-conclusion} below.
\end{cor}
\cref{thm:mainthm-singlepoint} concerns one-point statistics.
It follows from the more general \cref{thm:thebigtheorem} concerning multiple points and renormalizations of the field.
This result is unwieldy to state, so we present a simpler form here.
If we consider a pair of points separated by a macroscopic space-time difference, then their solution values will be asymptotically independent.
To uncover interesting multipoint statistics, we need to consider points separated by a positive power of $\rho$.
Similarly, one can study the field coarse-grained by different powers of $\rho$.
To illustrate the phenomenology, we state a qualitative version of the multipoint theorem, analogous to \cite[Theorem 1.3]{DG22}.

We define the parabolic distance
\begin{equation}
  d\bigl((t,x),(t',x')\bigr)=|t-t'|^{1/2}\vee|x-x'|
  \label{eq:dparabolicdistancedef}
\end{equation}
between two space-time points $(t,x),(t',x')\in\RR\times\RR^{2}$.
\begin{thm}
  \label{thm:mainthm-multipoint}
  Let $\sigma$ be $L^2$-subcritical and fix $N\in\NN$, $\overline{T}\in[1,\infty)$.
  Let
  $\big\{(t_{\rho}^{(i)},R_{\rho}^{(i)}, x_{\rho}^{(i)})\big\}_{i=1}^N $ %
  be a family of triples in $[\overline{T}^{-1},\overline{T}] \times [0,\overline{T}] \times \RR^{2}$
  such that for each $i,j\in\{1,\ldots,N\}$, the following limits exist:
  \begin{gather}
	  p^{(i,j)}\coloneqq\lim_{\rho\searrow0}\log_{\rho}\bigl[\rho\vee d\bigl((t_{\rho}^{(i)},x_{\rho}^{(i)}),(t_{\rho}^{(j)},x_{\rho}^{(j)})\bigr)\bigr],\qquad q^{(i)}=\lim_{\rho\searrow0}\log_{\rho}[\rho\vee R_{\rho}^{(i)}],\label{eq:dijqilimits}\\
    x^{(i)}\coloneqq\lim_{\rho\searrow0}x_{\rho}^{(i)},\qquad\text{and}\qquad T^{(i)}\coloneqq\lim_{\rho\searrow0}[t_{\rho}^{(i)}+R_{\rho}^{(i)}].\label{eq:xiandtilimits}
  \end{gather}
  Let $B^{(1)},\ldots,B^{(N)}$ be a family of $N$ standard
  $\RR^m$-valued Brownian motions with covariance
  \begin{equation}
    \dif[B^{(i)},B^{(j)}](q)=\mathbf{1}_{[0,p^{(i,j)}]}(q)\Id_m.\label{eq:Bijcorrelations}
  \end{equation}
  Let $v_0 \in L^\infty(\R^2)$ and let $\Psi^{(1)},\ldots,\Psi^{(N)}$ solve the system of SDEs
  \begin{align}
    \dif\Psi^{(i)}(q) & =J_{\sigma}\bigl(1-q,\Psi^{(i)}(q)\bigr)\,\dif B^{(i)}(q),\label{eq:upsiloniintroeqn}\\
    \Psi^{(i)}(0)     & =\clG_{T^{(i)}}v_{0}(x^{(i)}).\label{eq:upsiloniintroic}
  \end{align}
  Then if $(v^{\rho}_t)_{t\ge0}$ solves \cref{eq:SPDE} with initial condition $v_0$, we have
  \begin{equation}
    \lim_{\rho\searrow0}\clW_{2}\left(\bigl(\clG_{R_{\rho}^{(i)}}v_{t_{\rho}^{(i)}}^{\rho}(x_{\rho}^{(i)})\bigr)_{i=1}^{N},\, \bigl(\Psi^{(i)}(q^{(i)})\bigr)_{i=1}^{N}\right)=0.\label{eq:mainthm-multipoint-conclusion}
  \end{equation}
\end{thm}
The derivation of \cref{thm:mainthm-multipoint} from \cref{thm:thebigtheorem} is given on p.~\pageref{proofofmainthmmultipoint}.
The rate of convergence in \cref{eq:mainthm-multipoint-conclusion} can be quantitatively bounded if suitable rates are imposed on the limits \zcref[range]{eq:dijqilimits,eq:xiandtilimits}.
This is essentially the content of \cref{thm:thebigtheorem} below; to adjust the quantitative statement to match the form \cref{eq:mainthm-multipoint-conclusion}, one can use \cref{prop:Jsigmatimereg-1}.

By \cref{eq:Bijcorrelations}, the Brownian motions $B^{(i)}$ and $B^{(j)}$ are identical until time $p^{(i,j)}$, after which their increments are independent.
This is consistent due to the ultrametric property $p^{(i,j)}\le \max \{p^{(i,k)}, p^{(k,j)}\}$, which follows from \cref{eq:dijqilimits}.
The tree in \cref{fig:treestructure} represents this correlation structure.
Heuristically, $\dn B^{(i)}(q)$ represents the noise acting near $(t_\rho^{(i)},x_\rho^{(i)})$ at scale $\rho^{1-q}$.
If $(t_\rho^{(i)},x_\rho^{(i)})$ and $(t_\rho^{(j)},x_\rho^{(j)})$ are separated by much less than $\rho^{1 - q}$ in parabolic distance, they effectively experience the same noise.
But once their separation exceeds $\rho^{1 - q}$, the main noise contributions come from disjoint regions, and are thus independent.
\medskip

Because \cref{thm:mainthm-multipoint} also treats renormalizations of $v$, it contains new results even in the one-point case.
Indeed, for $q \in (0, 1)$, it shows that $\m{G}_{\rho^{1 - q}} v_t^\rho(x)$ converges in law to $\Gamma_{a,1}^\sigma(q)$ with $a = \m{G}_t v_0(x)$.
This extends \cref{thm:mainthm-singlepoint}, which effectively treats the case $q = 1$.
Thus the SDE \cref{eq:FBSDE} captures the behavior of the coarse-grained field $\m{G}_{\rho^{1 - q}} v_t^\rho(x)$ as we vary the scale of renormalization.

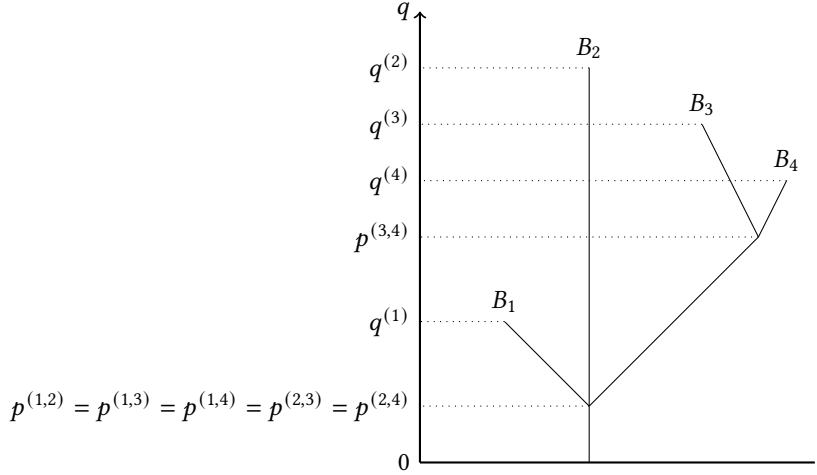
\begin{figure}[t]
	\centering
\begin{tikzpicture}[grow'=up,scale=1.5]
\begin{scope}
\clip (-3,1) -- (-3,2.25)-- (-0.5,2.25) -- (-0.5,5) -- (0.5,5) -- (0.5,4) -- (1.5,4) -- (1.5,3.5) -- (3,3.5) -- (3,1) --cycle;
\node [color=black,inner sep=0pt,outer sep=0pt] {}
child {
child {child { }}
child {child { }}
    child {
    child {}
    child {} 
    }};
\end{scope}
 \draw[->,thick] (-1.5,1)--(-1.5,5) node[left]{$q$};
 \draw[dotted,thin] (1.5,3)--(-1.5,3) node[left]{$p^{(3,4)}$};
 \draw[dotted,thin] (1.75,3.5)--(-1.5,3.5) node[left]{$q^{(4)}$};
 \draw[dotted,thin] (1,4)--(-1.5,4) node[left]{$q^{(3)}$};
\draw[dotted,thin] (0,4.5)--(-1.5,4.5) node[left]{$q^{(2)}$};
\draw[dotted,thin] (-0.75,2.25)--(-1.5,2.25) node[left]{$q^{(1)}$};
\draw[dotted,thin] (0,1.5)--(-1.5,1.5) node[left]{$p^{(1,2)} = p^{(1,3)} = p^{(1,4)} = p^{(2,3)}=p^{(2,4)}$};
\draw (-1.5,1.1) node [left] {$0$};
\draw (0,4.5) node [above] {$B_2$};
\draw (1.75,3.5) node [above] {$B_4$};
\draw (1,4) node [above] {$B_3$};
\draw (-0.75,2.25) node [above] {$B_1$};
\end{tikzpicture}
\caption{The tree structure of the Brownian motion correlations in \cref{thm:mainthm-multipoint}: $B^{(i)}$ and $B^{(j)}$ have identical increments until time $p^{(i,j)}$, after which they are independent.\label{fig:treestructure}}
\end{figure}

For simplicity, we take our initial data $v_0$ in $L^\infty(\R^2)$.
However, in principle our results should hold whenever $v_0\in L^\infty_{\mathrm{loc}}(\R^2)$ and the heat flow $\m{G}_t v_0$ is well-defined.
The stochastic effects in \cref{thm:mainthm-singlepoint} come from the noise in a small time layer preceding $t$.
The effect of the noise before this layer is negligible, which is why the initial condition only enters the statement of \cref{thm:mainthm-singlepoint} through the solution $\clG_tv_0$ to the deterministic heat equation.

One may also be interested in the SPDE \cref{eq:SPDE} posed with $\sigma$ that is not nonnegative-definite, not symmetric, or even not square.
If we replace $\sigma$ by $(\sigma\sigma^\mathrm{T})^{1/2}\in\clH^m_+$, the two SPDEs agree in law.
One can thus apply the results of this paper to a broader class of matrix-valued nonlinearities.
By Theorem VII.5.7 in \cite{Bha97}, the operation $\sigma \mapsto (\sigma\sigma^{\mathrm{T}})^{1/2}$ preserves the Lipschitz property with the bound $\Lip[(\sigma\sigma^{\mathrm{T}})^{1/2}]\le\sqrt2\Lip(\sigma)$.
The constant $\sqrt2$ cannot be improved in general (see \cite[Example~VII.5.8]{Bha97}), so some care is required in applying quantitative statements like \cref{eq:FBSDEwellposednessgeneralizes} involving $\Lip(\sigma)$.

\subsection{The decoupling flow\label{subsec:decouplingfunctionPDE}}

As noted in \cite[(1.12)\textendash (1.13)]{DG22}, the decoupling function $J_{\sigma}^{2}$ formally satisfies the (possibly degenerate) parabolic partial differential equation
\begin{align}
  \partial_{q}H(q,b) & =\frac{1}{2}[H(q,b):\nabla_{b}^{2}] H(q,b),\label{eq:HPDE-1} \\ %
  H(0,b)             & =\sigma^{2}(b).\label{eq:Hic-1} %
\end{align}
Here, we use the notation
\nomenclature[zzzz colon]{$\cdot:\cdot$}{trace of the product of two (symmetric) matrices, \cref{eq:colondef}}
\begin{equation}
  A:B = \tr[AB]\label{eq:colondef}
\end{equation}
for the trace of the product of two symmetric matrices $A$ and $B$, so
\begin{equation*}
  \left([H(q,b):\nabla_{b}^{2}]H(q,b)\right)_{ij}=\tr[H(q,b)\nabla_{b}^{2}]H_{ij}(q,b)=\sum_{k,\ell=1}^{m}H_{k\ell}(q,b)\frac{\partial^{2}H_{ij}}{\partial b_{k}\partial b_{\ell}}(q,b).
\end{equation*}

The analysis of \zcref[range]{eq:HPDE-1,eq:Hic-1} is complicated by the fact that \cref{eq:HPDE-1} is not uniformly parabolic when $H$ is not full-rank.
This is unavoidable if $\sigma$ has zeros, as $H$ preserves the zeros of $\sigma$ (\cref{prop:zerosstayzero}).
We do not expect the solution to necessarily be $\clC^{2}$ at these points of degeneracy.
We therefore introduce a notion of ``almost classical'' solution (\cref{defn:almost-classical}) to the PDE \zcref[range]{eq:HPDE-1,eq:Hic-1}.
We then show that almost classical solutions satisfying certain Lipschitz and growth bounds coincide with the decoupling function $J_\sigma^2$ (\cref{prop:satisfiesPDE}).
We therefore term \zcref[range]{eq:HPDE-1,eq:Hic-1} the ``decoupling flow.''

The relationship between the FBSDE and the decoupling flow allows us to derive lower bounds on $\Qbar_\FBSDE(\sigma)$ by studying \zcref[range]{eq:HPDE-1,eq:Hic-1}.
In the scalar case $m=1$, this is the topic of our companion paper \cite{DG25a}.
We leave the vector-valued setting to future investigation.
For simplicity, we assume the ``nondegenerate'' set $\{b\in\R\suchthat\sigma(b)>0\}$ is connected.
This entails no real loss of generality, as zeros of $\sigma$ remain zeros of $J_\sigma$ for all time (\cref{prop:zerosstayzero}), so the behaviors of the FBSDE on different components of $\{b\in\R\suchthat\sigma(b)>0\}$ do not interact with one another.
For a PDE perspective on this phenomenon, see \cite[Section~2.4]{DG25a}.

The following is a consequence of the results of \cite{DG25a}.
\begin{thm}\label{thm:PDEthm}
  Suppose $\sigma\in\Lip(\RR;\RR_{\ge 0})$ and $I\coloneqq \{b\in\R\suchthat\sigma(b)>0\}$ is an open interval.
  Suppose also that there exists $\beta\in(0,\infty)$ such that the appropriate following condition holds.
  \begin{enumerate}[label = \textup{(\roman*)}]
  \item If $I$ is a bounded interval, then there is a constant $K\in[1,\infty)$ such that for all $x\in I$, we have 
    \[
      K^{-1}\dist(x,I^{\cc})\le \sigma(x)\le \beta \dist(x,I^{\cc}).
    \]
    
  \item If $I =\R$, then there are constants $K\in [1,\infty)$ and $\gamma\in [0,2]$ such that for all $x\in\RR$, we have
    \[
      K^{-1}\langle x\rangle^{\gamma/2}\le \sigma(x)\le (\beta^2|x|^2+K)^{1/2}\wedge(K\langle x\rangle^{\gamma/2}).
    \]
  \item If $I$ is the half-line $(a,\infty)$ or $(-\infty,a)$, then there are constants $K\in [1,\infty)$ and $\gamma \in [0,2]$ such that for all $x\in I$, we have
    \[
      K^{-1}(|x-a|\wedge|x-a|^{\gamma/2})\le \sigma(x)\le (\beta|x-a|)\wedge (K|x-a|^{\gamma/2}).
    \]
  \end{enumerate}
  Then we have $\overline{Q}_\FBSDE(\sigma)\ge\beta^{-2}$.
\end{thm}
The above conditions amount to \cite[Hypothesis~\textup{H}($\beta^2$,$\gamma$)]{DG25a} for $\sigma^2$.
We derive \cref{thm:PDEthm} from the results of \cite{DG25a} in \cref{sec:PDE}.
We highlight a particularly simple case:
\begin{cor}
  \label{cor:simple-PDE}
  Let $\sigma\in\Lip(\RR;\RR_{\ge 0})$ and suppose there exist $\beta \in (0, \infty)$ and $k \in (0, \beta]$ such that ${k b_+ \leq \sigma(b) \leq \beta b_+}$ for all $b \in \R.$
  Then $\overline{Q}_\FBSDE(\sigma)\ge\beta^{-2}$.
\end{cor}

\subsection{Examples\label{subsec:examples}}

In this section we examine some nonlinearities for which the FBSDE \zcref[range]{eq:FBSDE,eq:Jeqn} admits explicit solutions.
\nomenclature[Hm]{$\clH^m$}{space of symmetric $m\times m$ matrices}
Writing $\clH^m$ for the space of symmetric $m\times m$ matrices, let $g_2\colon\clH^m\to\clH^m$ and $g_1\colon\RR^m\to\clH^m$ be linear maps and $g_0\in \clH_+^m$.
Assume $g_2(b\otimes b)+g_1(b)+g_0 \in \clH_+^m$ for all $b \in \R^m$ and define
\begin{equation*}
  \sigma(b) = [g_2(b\otimes b)+g_1(b)+g_0]^{1/2}.
\end{equation*}
We make the ansatz $H(q, b) = F_q \circ \sigma^2(b)$ for some linear map $F_q\colon\clH^m\to\clH^m$ depending on $q$.
Using this in \cref{eq:HPDE-1} and writing $\dot{F}_q = \partial F_q/\partial q$, one can compute
\begin{equation*}
  \dot F_q \circ \sigma^2 = \partial_qH = \frac12(H:\nabla^2_b)H = F_q \circ g_2 \circ H = F_q \circ g_2 \circ F_q \circ \sigma^2.
\end{equation*}
One can check that $F_q = (\Id-qg_2)^{-1}$ solves the matrix-valued ODE $\dot F_q = F_q\circ g_2\circ F_q$ with $F_0 = \Id$ on the interval $q\in [0,|g_2|^{-1}_{\mathrm{op}})$.
The corresponding PDE solution is
\begin{equation}
  \label{eq:vvlinearcasesoln}
  H(q,b)=(\Id-qg_2)^{-1}[g_2(b\otimes b)+g_1(b)+g_0]\qquad\text{for }q\in [0,|g_2|^{-1}_{\mathrm{op}}).
\end{equation}
If $\sqrt{H}$ is Lipschitz in $b$, then $H$ is a decoupling function for the FBSDE \zcref[range]{eq:FBSDE,eq:Jeqn} by \cref{prop:satisfiesPDE}.
\medskip

Let us now focus on the scalar case $m = 1$.
We first note that $g_2(X) = \beta^2 X$ and $g_1=g_0 = 0$ recovers the linear setting $\sigma(b) = \beta b$.
Then \cref{eq:vvlinearcasesoln} specializes to \cite[(1.34)]{DG22} (with different normalization).

The scalar setting allows us to relax the condition $g_2(b^2)+g_1(b)+g_0 \geq 0$ by taking absolute values, so that $\sigma(b) = |g_2(b^2)+g_1(b)+g_0|^{1/2}$.
The above calculation remains formally valid for $H = \sigma^2$ because the diffusivity $H/2$ vanishes precisely where $\partial_b^2$ hits the ``kink'' in the absolute value.
Two interesting cases correspond to previous work, though $\sigma$ is non-Lipschitz in both, so the results of the present paper do not apply.

First, if $g_2 = g_0 = 0$, $g_1(b) = b,$ then $\sigma(b) = \abs{b}^{1/2}$ and $J_\sigma(q,b)=\sigma(b)$ is independent of $q$.
With such square-root noise, the FBSDE is a Feller diffusion $\dif \Gamma = \sqrt{\Gamma}\,\dif B$ and \cref{eq:SPDE} is closely related to super-Brownian motion.
Although \cref{thm:mainthm-singlepoint} does not apply due to the non-Lipschitz noise, analogous results were obtained for super-Brownian motion in \cite{Kle97}.

Second, if $g_2(X) = -\al^2 X$, $g_1(b) = \al^2 b$, and $g_0 = 0$, we obtain $\sigma(b) = \alpha\sqrt{\abs{b(1-b)}}$.
Then
\begin{equation*}
  J_\sigma(q,b)=\sqrt{\frac{\abs{b(1-b)}}{\alpha^{-2}+q}}
\end{equation*}
and the FBSDE is a time-changed Fisher--Wright diffusion.
Notably, $J_\sigma$ converges as $\al \to \infty$.
This suggests the logarithmic attenuation in \cref{eq:SPDE} may be unnecessary in this setting, as $J_\sigma$ converges even if $\al \sim \sqrt{\sfL(1/\rho)} \to \infty$.
Results consistent with this picture were obtained for the voter model in \cite{CG86}, and \cite{MT95} shows that a long-range version of the voter model converges to the stochastic heat equation with Fisher--Wright noise.
To our knowledge, no rigorous link has been established between the voter model and an SPDE in two dimensions, but the above computations suggest some relationship.

Next, let $\alpha,\beta\ge0$ and take $g_2(X) = \beta^2X$, $g_1=0$, and $g_0 = \alpha^2\beta^2$.
In \cref{eq:u-SPDE},
\begin{equation*}
  \sigma(b)=\sigma_{\alpha,\beta}(b)\coloneqq\beta\sqrt{\alpha^{2}+b^{2}}
\end{equation*}
represents a combination of independent additive and multiplicative noise.
Using \cref{eq:vvlinearcasesoln}, we find
\begin{equation*}
  \dif\Gamma_{a,1}^{\sigma_{\alpha,\beta}}(q) =\left[\frac{\alpha^{2}+\Gamma_{a,1}^{\sigma_{\alpha,\beta}}(q)^{2}}{\beta^{-2}-(1-q)}\right]^{1/2} \dif B(q), \quad \Gamma_{a,1}^{\sigma_{\alpha,\beta}}(0) = a.
\end{equation*}
Through the time-change $X(r) \coloneqq \Gamma_{a,1}^{\sigma_{\alpha,\beta}}\bigl((\beta^{-2}-1)(\e^{r}-1)\bigr)$, this becomes the autonomous SDE
\begin{equation}
  \label{eq:time-change}
  \d X(r) = [\al^2 + X(r)^2]^{1/2} \dif\widetilde{B}(r)
\end{equation}
for a new Brownian motion $\widetilde{B}$.
This SDE has adjoint generator $\clL_{X}^{*}\mu=\frac{1}{2}\partial_{b}^{2}[(\alpha^{2}+b^{2})\mu]$, so for $\alpha>0$, \cref{eq:time-change} admits an invariant probability measure $\mu_{\alpha}$ with Cauchy density
\begin{equation*}
  \mu_\al(\dif b) = \frac{\alpha}{\pi(\alpha^{2}+b^{2})} \, \dif b.
\end{equation*}
In the degenerate case $\al = 0$, $\delta_0$ is invariant.
One can then check that $\Law X(r) \to \mu_\al$ as $r \to \infty$.
Noting that $r \to \infty$ as $\beta \nearrow 1$, we see that $\Law \Gamma_{a,1}^{\sigma_{\alpha,\beta}}(1) \to \mu_{\alpha}$ in the same limit.
After the first version of this paper was posted, the first author and Mukherjee \cite{DM24} showed that the $\rho \searrow 0$ and $\beta \nearrow 1$ limits can be exchanged: the limiting pointwise statistics of \cref{eq:u-SPDE} are Cauchy when $\beta \ge 1$.
This was known in the linear case $\al = 0$ \cite[Theorem 2.15]{CSZ17}.

\subsection{Ideas of the proof\label{subsec:MGexptimescale}}

\paragraph{The martingale.}

Our approach revolves around the following martingale, which is analogous to a discrete Markov chain used in~\cite{DG22}.
Given $0\leq T_0 < T$, let $v^\rho$ solve \cref{eq:SPDE} on $[T_{0},T]$ and define
\nomenclature[V ]{$V_t^{\rho,T}(x)$}{$\clG_{T-t}v_{t}^{\rho}(x)$, martingale in $t$, \cref{eq:Vdef}}
\begin{equation}
  V_{t}^{\rho,T}(x)=\clG_{T-t}v_{t}^{\rho}(x)\qquad\text{for }t\in[T_{0},T]\text{ and }x\in\RR^{2}.\label{eq:Vdef}
\end{equation}
Using the mild solution formula \cref{eq:mildsoln}, we can rewrite
this as
\begin{equation}
  V_{t}^{\rho,T}=\clG_{T-T_{0}}v_{T_{0}}^{\rho}+\gamma_{\rho}\int_{T_{0}}^{t}\clG_{T+\rho-r}[\sigma(v_{r}^{\rho})\,\dif W_{r}]=V_{T_{0}}^{\rho,T}+\gamma_{\rho}\int_{T_{0}}^{t}\clG_{T+\rho-r}[\sigma(v_{r}^{\rho})\,\dif W_{r}].\label{eq:Vtintegral}
\end{equation}
From \cref{eq:Vtintegral} it is clear that for fixed $T$ and $x$,
the process $\bigl(V_{t}^{\rho,T}(x)\bigr)_{t\in[T_{0},T]}$ is an $\{\scrF_{t}\}$-martingale.
Also, \cref{eq:Vdef} yields $V_{T}^{\rho,T}=v_{T}^{\rho}$, so $V^{\rho,T}$ has our quantity of interest $v_{T}^{\rho}(x)$ as its terminal value.
The quadratic variation matrix is given by
\begin{align}
  [V^{\rho,T}(x)]_{t} & = \gamma_\rho^2 \int_{T_{0}}^{t} G_{T+\rho-r}^{2} \ast (\sigma^{2} \circ v_r^\rho)(x)\,\dif r =\frac{1}{\sfL(1/\rho)}\int_{T_{0}}^{t}\frac{\clG_{(T+\rho-r)/2}[\sigma^{2}\circ v_{r}^{\rho}](x)}{T+\rho-r}\,\dif r.\label{eq:VQV}
\end{align}
We would like to approximate $\dif[V^{\rho,T}(x)]_{t}/\dif t$ by a function of $V_{t}^{\rho,T}(x)$, for then $V^{\rho,T}$ will approximately solve an SDE.

As $\rho\searrow0$, contributions to the quadratic variation become concentrated in a thin layer just before $T$.
Indeed, the attenuation $\sfL(1/\rho)^{-1}$ kills the integral over $[T_0,r]$ for any fixed $r < T$.
To account for this concentration, we make the $\rho$-dependent time-change
\begin{equation}
  q(r)=\frac{1}{\sfL(1/\rho)}\log\frac{T-T_{0}+\rho}{T-r+\rho}, \quad \text{so }\enspace \dif q=\frac{\dif r}{\sfL(1/\rho)(T+\rho-r)} \And r \approx T-\rho^q.
  \label{eq:introchgvar} 
\end{equation}
This exponential time scale is fundamental to the phenomenology of the problem.
It appears also in \cite{CSZ17,DG22}, and \cite[Section 1.1]{DG22} contains further heuristic discussion.

\paragraph{Relating the quadratic variation to the martingale.}
In the following, we use ``scale $\eps$'' to indicate time-scale $\eps$ and space-scale $\eps^{1/2}$, in keeping with parabolic scaling.

After the time-change \cref{eq:introchgvar} and some minor approximations, we obtain a martingale
\begin{equation}
  \Upsilon_q = V_{T-\rho^{q}}^{\rho,T}(x) =\clG_{\rho^{q}}v^\rho_{T-\rho^{q}}(x) \quad \text{with} \quad \dif[\Upsilon]_q \approx \clG_{\rho^{q}}[\sigma^2\circ v^{\rho}_{T-\rho^{q}}](x) \ds q.
  \label{eq:upsilonapproxqv}
\end{equation}
Because $v_t^\rho$ is spatially smooth at scale $\rho$, $\Upsilon_1 \approx v_T^\rho(x)$.
We wish to show that $\Upsilon$ approximately satisfies an SDE.
Equivalently, we want to approximate the rightmost side of \cref{eq:upsilonapproxqv} by a function of $\Upsilon_q$.
In general, there need not be any relationship between the average of a field (like $\Upsilon_q$) and the average of a nonlinear function of the field (like $\dn [\Upsilon]_q$).
However, we can use the specific structure of \cref{eq:SPDE} to establish such a link.

Following \cite{DG22}, we imagine turning off the noise on a ``quiet interval'' of duration $\rho^{q + o(1)}$ preceding time $T - 2\rho^q$.
In the absence of noise, \cref{eq:SPDE} becomes the deterministic heat equation for time $\approx \rho^q$, and the solution approximately coincides with the Gaussian average $\Upsilon_q$ at time $T - 2 \rho^q$.
We then turn the noise back on from time $T - 2\rho^q$ to $T - \rho^q$.
We can choose the duration of the quiet interval so that it is negligible on the exponential time-scale.
Then the temporary noise suppression has little effect on $v_{T - \rho^q}^\rho(x)$.
As a consequence, the field $v_{T - \rho^q}^\rho$ in \cref{eq:upsilonapproxqv} can be approximated by the solution $\ti{v}_{T - \rho^q}^\rho$ of \cref{eq:SPDE} started from data $\Upsilon_q$ at time $T - 2 \rho^q$ (\cref{prop:turnoffnoise}).
The fluctuations in $\ti{v}_{T - \rho^q}^\rho$ from the noise in \cref{eq:SPDE} decorrelate at finer scales than $\rho^q$, so the spatial average $\m{G}_{\rho^q}$ in $\dn [\Upsilon]_q$ approximates an expectation conditioned on $\Upsilon_q$.
Therefore $\dn [\Upsilon]_q$ is close to an implicit but deterministic function of the data $\Upsilon_q$ of $\ti v$ (\cref{thm:mainapproxthm}).

We next relate this function to the FBSDE \zcref[range]{eq:FBSDE,eq:Jeqn}.
We have approximated $\dn[\Upsilon]_q$ by solving \cref{eq:SPDE} from data $\Upsilon_q$ for time $\rho^{q}$, applying $\sigma^2$, and taking an expectation.
In time $\rho^q$, the noise in \cref{eq:SPDE} induces fluctuations from scales $\rho^q$ down to $\rho^1$.
Hence this operation is analogous to starting the limiting SDE \cref{eq:FBSDE} from $\Upsilon_q$ at time $q$, running to time $1$ (matching the induced scales), and then applying $\sigma^2$ and an expectation.
This is simply \cref{eq:Jeqn} with $1 - q$ in place of $q$ due to the substitution in \cref{eq:FBSDE}.
It follows that
\begin{equation}
  \label{eq:intro-key-approx}
  \clG_{\rho^{q}}[\sigma^2\circ v^{\rho}_{T - \rho^q}] \approx J_\sigma\bigl(1 - q, \Upsilon_q\bigr)^2
\end{equation}
(\cref{prop:Japprox}), and hence $V^{\rho,T}$ approximately satisfies \cref{eq:FBSDE} in the exponential time-scale.

\paragraph{Extension to longer scales.}

In \cite{DG22}, the above approximations were proved on time scales of length $\rho^{1-\Lip(\sigma)^{-2}}$, so order-$1$ time can only be treated when $\Lip(\sigma)<1$.
The most important contribution of the present paper is to extend the analysis to longer time scales. %
We accomplish this through renormalization, a form of multiscale analysis.

Given an exponent $q \in (0, 1)$, let $w_t^\rho \coloneqq \m{G}_{\rho^q} v_t^\rho$ denote the ``coarse-grained'' field at scale $\rho^q$.
Applying $\m{G}_{\rho^q}$ to \cref{eq:SPDE}, we find
\begin{equation}
  \label{eq:renorm-1}
  \dif w_{t}^{\rho} \approx \frac{1}{2}\Delta w_{t}^{\rho} \, \dif t + \gamma_{\rho} \clG_{\rho^q}[\sigma(v_{t}^{\rho})\,\dif W_{t}].
\end{equation}
By \cref{eq:intro-key-approx}, the noise term is close in law to $\gamma_\rho \m{G}_{\rho^q}[J_\sigma(1 - q, w_t^\rho) \ds W_t]$ provided $q > 1 - \Lip(\sigma)^{-2}$.
These noise fields need not be close in \emph{probability}, but we show that we can replace $\dn W$ by a new white noise $\dn\widetilde{W}$ in the second expression so that the fields themselves are close in a weak sense (\cref{prop:coupling}).
Then \cref{eq:renorm-1} becomes
\begin{equation}
  \label{eq:renorm}
  \dif w_{t}^{\rho} \approx \frac{1}{2}\Delta w_{t}^{\rho} \, \dif t + \gamma_{\rho} \m{G}_{\rho^q}[J_\sigma(1 - q, w_t^\rho) \ds \widetilde{W}_t].
\end{equation}
If we set $\ti\rho \coloneqq \rho^q$, $\ti\sigma \coloneqq q^{1/2} J_\sigma(1 - q,\anon)$, and index $w$ by $\ti \rho$ rather than $\rho$, we obtain
\begin{equation*}
  \dif w_{t}^{\ti\rho} \approx \frac{1}{2}\Delta w_{t}^{\ti\rho} \, \dif t + \gamma_{\ti\rho} \m{G}_{\ti\rho}[\ti\sigma(w_t^{\ti\rho}) \ds \widetilde{W}_t].
\end{equation*}
This is simply \cref{eq:SPDE} with larger $\rho$, new noise, and an effective nonlinearity $q^{1/2}J_\sigma(1 - q, \anon)$.
Thus \cref{eq:SPDE} is approximately closed under renormalization (\cref{cor:approx-SPDE}), and the root decoupling function $J_\sigma$ serves as an effective nonlinearity after renormalization.

Initially, our approximation \cref{eq:intro-key-approx} was only valid when $q > 1 - \Lip(\sigma)^2$.
Now, the renormalized equation \cref{eq:renorm} allows us to extend it to $q > 1 - \Lip(\sigma)^2 - \Lip[J(1 - q,\anon)]^2$ (\cref{prop:extendQSPDE-1}).
Inducting, we see that \cref{eq:intro-key-approx} holds for all $q \in [0, 1]$ provided $J_\sigma$ remains Lipschitz for $q \in [0, 1]$ (\cref{thm:QSPDEgreater}).
In turn, this ensures that $V^{\rho,T}$ approximately solves \cref{eq:FBSDE} and implies \cref{thm:mainthm-singlepoint}.
This explains our \cref{def:subcrit} of subcriticality, which generalizes the condition in~\cite{DG22} that $\Lip(\sigma) < 1$.
We also show that the stability of \cref{eq:SPDE} is preserved under renormalization, which implies a form of universality (\cref{thm:stable}/\cref{thm:universal-informal}).

\subsection{Organization}

We analyze the FBSDE \zcref[range]{eq:FBSDE,eq:Jeqn} in \cref{sec:FBSDE} and relate it to the decoupling flow \cref{eq:HPDE-1} in \cref{sec:PDE}.
We combine this with results from a companion paper \cite{DG25a} to prove \cref{thm:PDEthm}.
We establish basic properties of the SPDE in \cref{sec:basic-properties}.
In \cref{sec:local-SPDE-estimates}, we prove ``local'' estimates for the SPDE, namely results that hold on small scales depending on $\Lip(\sigma)$.
The heart of the paper is \cref{sec:extendtheestimates}, in which we carry out the renormalization.
We use its results in \cref{sec:proofofmainresults} to prove \cref{thm:renorm,thm:mainthm-singlepoint,thm:mainthm-multipoint}.
Finally, in \cref{sec:universal} we establish the universality in \cref{thm:universal-informal} and \cref{cor:pre-post}.

There are two appendices.
We collect well-known auxiliary results in \cref{appendix:auxiliary} and prove a general coupling result about space-time white noises in \cref{appendix:coupling}.
The second is crucial for our renormalization, but we place it in an appendix as it is not specific to our SPDE.

\subsection{Notational conventions\label{subsec:notationalconventions}}

Throughout the paper, $\abs{a}$ denotes the Euclidean norm\nomenclature[zzznormabs]{$\lvert \anon\rvert $}{Euclidean norm} and we write $\langle a\rangle=\sqrt{|a|^{2} + 1}$ for $a \in \R^m$.\nomenclature[zzznorm JB]{$\langle\anon\rangle$}{Japanese bracket $\sqrt{1+\lvert \anon\rvert ^2}$}
We write $\Id_m$ for the $m\times m$ identity matrix.\nomenclature[Idm]{$\Id_m$}{$m\times m$ identity matrix} In addition to the Frobenius norm $|\cdot|_{\Frob}$, we use the operator norm $|\cdot|_{\op}$\nomenclature[zzznormop]{$\lvert \anon\rvert _\op$}{operator norm} (based on the
Euclidean norm on the domain and range $\RR^{m}$) and the
nuclear norm $|\anon|_{*}$\nomenclature[zzznormnuclear]{$\lvert \anon\rvert _*$}{nuclear norm} on the space of $m\times m$ matrices.
See \cref{subsec:matrix-norms} for some facts about these norms that we use throughout the paper.
For vector-valued adapted processes $(X_q^{(i)})_q$ with $i = 1,2$, we let $[X^{(1)},X^{(2)}]_q$
denote the matrix-valued quadratic covariation process with $k_1,k_2$ entry equal to the quadratic covariation of the $k_1$ entry of $X^{(1)}$ and the $k_2$ entry of $X^{(2)}$, and also write $[X^{(i)}]_q = [X^{(i)},X^{(i)}]_q$.

When we apply a result with certain parameter choices, we use the notation $x\setto x'$ to indicate that the parameter $x$ in the original statement should be taken equal to $x'$ in the present application.\nomenclature[zzzz setto]{$\setto$}{parameter substitution}
An index of symbols is provided at the end of the paper for reference.

\subsection{Acknowledgments}

We warmly thank Yu Gu, Jean-Christophe Mourrat, and Nikos Zygouras for interesting
and helpful conversations. The authors were supported by the National Science Foundation
under grants~DMS-2002118 and~DMS-2346915
(A.D.) and DMS-2103383 (C.G.). Much of the work was completed while A.D. was at NYU Courant and C.G. was at Brown University.

\section{Well-posedness of the FBSDE\label{sec:FBSDE}}

In this section, we prove various properties of the FBSDE \zcref[range]{eq:FBSDE,eq:Jeqn}.
Recalling \cref{def:QbarFBSDE}, our main result is the following:
\begin{thm}
  \label{thm:FBSDEwellposed}
  Let $\sigma\in\Lip(\RR^{m};\clH_{+}^{m})$.
  For any $Q\in\bigl[0,\Qbar_{\FBSDE}(\sigma)\bigr)$, we have
  \begin{equation*}
    \Qbar_{\FBSDE}(\sigma)\ge Q+\Lip\bigl(J_{\sigma}(Q,\cdot)\bigr)^{-2}.
  \end{equation*}
\end{thm}
\noindent
Thus we can extend the well-posedness of \zcref[range]{eq:FBSDE,eq:Jeqn} so long as $J_\sigma$ remains Lipschitz.
In addition, we establish several growth and regularity bounds for $J_\sigma$.
\medskip

We let $\clX$ denote the Banach space of $\clH^m_+$-valued continuous functions on $\RR^{m}$ with at most linear growth at infinity.
That is, $\clX$ is the space of continuous functions $\sigma$ such that the norm\nomenclature[X]{$\clX$}{space of continuous $\clH^m_+$-valued functions with at most linear growth at infinity}
\begin{equation}
  \|\sigma\|_{\clX}\coloneqq\sup_{x\in\RR^m}\frac{|\sigma(x)|_\Frob}{\langle x\rangle}\label{eq:Xnorm}
\end{equation}
is finite. The space $\clX$ is not separable, but this will
not concern us.
As sets, $\Lip(\RR^{m};\clH_{+}^{m})\subset\clX$.

\subsection{SDE solution theory\label{subsec:SDE-solution-theory}}

In this section we recall the solution theory for stochastic differential equations with Lipschitz coefficients.
This is essentially standard; see, e.g., \cite[\S 5.1]{SV06}.
We work with a standard $\RR^{m}$-valued Brownian
motion $\bigl(B(q)\bigr)_{q\ge0}$ adapted to a filtration $\{\scrG_{q}\}_{q\ge0}$.
For an adapted process $Y$ on $[0,Q]$, a function $g\colon[0,Q]\times\RR^{m}\to\clH_{+}^{m}$,
and a constant $a\in\RR^{m}$, we define a new adapted process
$\clR_{a,Q}^{g}Y$ on $[0,Q]$ by
\nomenclature[R]{$\clR^g_{a,Q}$}{solution operator for an SDE with diffusivity $g(Q-\cdot,\cdot)$ and initial condition $a$, \cref{eq:Rdef}}
\begin{equation}
  \clR_{a,Q}^{g}Y(q)\coloneqq a+\int_{0}^{q}g\bigl(Q-p,Y(p)\bigr)\,\dif B(p),\label{eq:Rdef}
\end{equation}
whenever this stochastic integral is defined. For $Q>0$, define
\nomenclature[AQ]{$\clA_Q$}{space of continuous functions on $[0,Q]\times\RR^m$, uniformly Lipschitz in the second argument, \cref{eq:AQdef}}
\begin{equation}
    \clA_{Q}\coloneqq\Bigl\{ J\colon[0,Q]\times\RR^{m}\to\clH_{+}^{m}\text{ continuous}\suchthat \sup_{q\in[0,Q]}\Lip\bigl(J(q,\cdot)\bigr)<\infty\Bigr\} .\label{eq:AQdef}
\end{equation}
This condition matches \cref{eq:Juniflipschitz}.
\begin{prop}
  \label{prop:SDEwellposed}Fix $L<\infty$ and $Q\in(0,\infty)$ and
  suppose that $g\in\clA_{Q}$ satisfies
  \begin{equation}
    \sup_{q\in[0,Q]}\Lip\bigl(g(q,\cdot)\bigr)\le L.\label{eq:gLip}
  \end{equation}
  Then, for any $a\in\RR^m$, there is a unique strong solution $\Theta_{a,Q}^{g}$
  to the SDE
  \nomenclature[zzzgreek θ]{$\Theta_{a,Q}^g$}{solution to SDE \zcref[range]{eq:thetaSDE,eq:thetaIC}}
  \begin{align}
    \dif\Theta_{a,Q}^{g}(q) & =g\bigl(Q-q,\Theta_{a,Q}^{g}(q)\bigr)\dif B(q),\qquad q\in[0,Q];\label{eq:thetaSDE} \\
    \Theta_{a,Q}^{g}(0)     & =a.\label{eq:thetaIC}
  \end{align}
  The solution $\Theta_{a,Q}^{g}$ satisfies the moment bound
  \begin{equation}
    \sup_{q\in[0,Q]}\EE |\Theta_{a,Q}^{g}(q)|^{\ell}<\infty\qquad\text{for all }\ell\in[1,\infty).\label{eq:finite-moment}
  \end{equation}
  Moreover, there exists a constant $C=C(L,Q)<\infty$ such that for
  any $Q'\in[0,Q]$ and any adapted process $\Gamma$ on $[0,Q']$,
  we have
  \begin{equation}
    \sup_{q\in[0,Q']}\EE |\Gamma(q)-\Theta_{a,Q}^{g}(q)|^{2}\le C\cdot\sup_{q\in[0,Q']}\EE |\Gamma(q)-\clR_{a,Q}^{g}\Gamma(q)|^{2}\label{eq:SDEclose}
  \end{equation}
  and
  \begin{equation}
    \EE |\Theta_{a,Q}^{g}(q)-\Theta_{\widetilde{a},Q}^{g}(q)|^{2}\le C|a-\widetilde{a}|^{2}.\label{eq:SDEctsininitialconditions}
  \end{equation}
\end{prop}
The solution $\Theta$ is the unique fixed point of $\clR$.
This explains \cref{eq:SDEclose}: if $\Gamma$ is almost a fixed point then it is close to $\Theta$.
For the convenience of the reader, we provide a complete proof of \cref{prop:SDEwellposed} in \cref{appendix:SDE-solution-theory}.

\subsection{Local solution theory for the FBSDE\label{subsec:FBSDE-local}}

Our first main step towards proving \cref{thm:FBSDEwellposed} is essentially
a generalization of \cite[Theorem 1.1]{DG22} to the vector-valued case. (Even in the scalar case, \cite{DG22} imposes the assumption $\sigma(0)=0$, which we do not assume here.)
We think of this as a ``local'' solution
theory for the FBSDE \zcref[range]{eq:FBSDE,eq:Jeqn}, since in this
section (\cref{prop:localwellposed} below) we will construct solutions
only up to time $\Lip(\sigma)^{-2}$, thus showing that $\Qbar_{\FBSDE}(\sigma)\ge\Lip(\sigma)^{-2}$.
In \cref{subsec:extendingFBSDE}, we will extend solutions to longer
times.

Recall the definition \cref{eq:AQdef} of $\clA_{Q}$.
\begin{defn}
  \label{def:decouplingfn}Let $Q>0$ and let $\sigma\in\Lip(\RR^{m};\clH_{+}^{m})$.
  We say that $J\in\clA_{Q}$ is a \emph{root decoupling function}
  for \zcref[range]{eq:FBSDE,eq:Jeqn} on $[0,Q]$ if, for all $q\in[0,Q]$
  and all $a\in\RR^m$, we have
  \begin{equation}
    J(q,a)=\bigl[\EE \sigma^{2}\bigl(\Theta_{a,q}^{J}(q)\bigr)\bigr]^{1/2}.\label{eq:root decouplingholds}
  \end{equation}
  In this case, we also call $J^2$ the \emph{decoupling function}.
  In \cref{eq:root decouplingholds}, the process $\Theta_{a,q}^J$ is as in \zcref[range]{eq:thetaSDE,eq:thetaIC} (with $g\mapsfrom J$).
\end{defn}
In other words, $J$ is a root decoupling function for \zcref[range]{eq:FBSDE,eq:Jeqn}
if and only if taking $\Gamma_{a,Q}^{\sigma}\setto\Theta_{a,Q}^{J}$ and
$J_{\sigma}\setto J$ yields a solution to the FBSDE \zcref[range]{eq:FBSDE,eq:Jeqn}.
Moreover, it is clear that for any solution to \zcref[range]{eq:FBSDE,eq:Jeqn}
with $J_{\sigma}\in\clA_{Q}$, $J_{\sigma}$ is a root decoupling
function.

\subsubsection{An \emph{a priori} bound}

We begin by proving an \emph{a priori} bound on the Lipschitz constant
of the root decoupling function.
For $\lambda\in(0,\infty)$, we define the set of functions
\nomenclature[zzzgreek σlambda]{$\lipset(\lambda)$}{space of $\lambda$-Lipschitz functions, \cref{eq:Siglamdef}}
\begin{equation}
  \lipset(\lambda)\coloneqq\{\sigma\in\Lip(\RR^{m};\clH_{+}^{m})\suchthat\Lip(\sigma)\le\lambda\}.\label{eq:Siglamdef}
\end{equation}
\begin{prop}
  \label{prop:apriorilipschitzbd}Suppose that $\lambda\in(0,\infty)$,
  $\sigma\in\lipset(\lambda)$, $Q_{0}>0$, and that $J$ is
  a root decoupling function for \textup{\zcref[range]{eq:FBSDE,eq:Jeqn}} on $[0,Q_{0}]$.
  Then for all $Q\in[0,Q_{0}\wedge\lambda^{-2})$, we have
  \begin{equation}
    \Lip\bigl(J(Q,\anon)\bigr)\le(\lambda^{-2}-Q)^{-1/2}.\label{eq:apriorilipschitzbd}
  \end{equation}
\end{prop}

\begin{proof}
  Let $Q\in[0,Q_{0}\wedge\lambda^{-2})$.
  Using \cref{eq:root decouplingholds,prop:matrixvaluedreversetriangleinequality},
  we obtain for any $b_{1},b_{2}\in\RR^{m}$ that
  \begin{align}
    |J(Q,b_{1})-J(Q,b_{2})|_{\Frob}^{2} & =\left\lvert  \bigl[\EE \sigma^{2}\bigl(\Theta^J_{b_1,Q}(Q)\bigr)\bigr]^{1/2}-\bigl[\EE \sigma^{2}\bigl(\Theta^J_{b_2,Q}(Q)\bigr)\bigr]^{1/2}\right\rvert  _{\Frob}^{2}\nonumber                                                     \\
                                                          & \le\EE \bigl\lvert \sigma\bigl(\Theta^J_{b_1,Q}(Q)\bigr)-\sigma\bigl(\Theta^J_{b_2,Q}(Q)\bigr)\bigr\rvert _{\Frob}^{2}\le\lambda^{2}\EE |\Theta^J_{b_1,Q}(Q)-\Theta^J_{b_2,Q}(Q)|^{2}.\label{eq:Jsqbd}
  \end{align}
  We also have, for all $r\in[0,Q]$, that
  \begin{align*}
    \EE |\Theta^J_{b_1,Q}(r)-\Theta^J_{b_2,Q}(r)|^{2} & =|b_{1}-b_{2}|^{2}+\int_{0}^{r}\EE \bigl\lvert J\bigl(Q-q,\Theta^J_{b_1,Q}(q)\bigr)-J\bigl(Q-q,\Theta^J_{b_2,Q}(q)\bigr)\bigr\rvert _{\Frob}^{2}\,\dif q \\
                                                                        & \le|b_{1}-b_{2}|^{2}+\int_{0}^{r}\Lip\bigl(J(Q-q,\anon)\bigr)^{2}\EE |\Theta^J_{b_1,Q}(q)-\Theta^J_{b_2,Q}(q)|^{2}\,\dif q.
  \end{align*}
  By Gr\"onwall's inequality, this implies that
  \begin{equation}
    \EE |\Theta^J_{b_1,Q}(Q)-\Theta^J_{b_2,Q}(Q)|^{2}\le|b_{1}-b_{2}|^{2}\exp\left\{ \int_{0}^{Q}\Lip\bigl(J(q,\anon)\bigr)^{2}\,\dif q\right\} .\label{eq:Gammadiffbd}
  \end{equation}
  Now let $f(q)\coloneqq\lambda^{2}\vee\Lip\bigl(J(q,\anon)\bigr)^{2}$.
  Using \cref{eq:Gammadiffbd} in \cref{eq:Jsqbd} and taking the supremum
  over all $b_{1},b_{2}\in\RR^m$, we obtain
  \[
    f(Q)\le\lambda^{2}\exp\left\{ \int_{0}^{Q}f(q)\,\dif q\right\} .
  \]
  By \cref{lem:fQrecursive} below, we obtain \cref{eq:apriorilipschitzbd}.
\end{proof}
\begin{lem}
  \label{lem:fQrecursive}Suppose that $c\in(0,\infty)$, $\Qbar<c^{-2}$,
  and $f:[0,\Qbar]\to[c^{2},\infty)$ satisfies
  \begin{equation}
    f(Q)\le c^{2}\exp\left\{ \int_{0}^{Q}f(q)\,\dif q\right\} \label{eq:fQrecursive}
  \end{equation}
  for all $Q\in[0,\Qbar]$. Then
  \begin{equation*}
    f(Q)\le(c^{-2}-Q)^{-1}\qquad\text{for all }Q\in[0,\Qbar].
  \end{equation*}
\end{lem}

\begin{proof}
  This is a consequence of the Bihari--LaSalle inequality (see, e.g., \cite[Theorem 1.8.2]{Mao08}).
  For the reader's convenience, we give a direct proof here, following the argument on \cite[p.~46]{Mao08}.
  Define ${g(Q) = \int_0^Q f(q)\,\dif q}$. By the chain rule and the fundamental theorem of calculus, we have
\begin{equation}\label{eq:chainrule}\e^{-g(Q)} = 1 - \int_0^Q \e^{-g(q)}f(q)\,\dif q\overset{\cref{eq:fQrecursive}}{\ge} 1-c^2 Q,\end{equation}
which means that 
\[ f(Q) \overset{\cref{eq:fQrecursive}}{\le} c^2 \e^{g(Q)}\overset{\cref{eq:chainrule}}{\le} \frac{c^2}{1-c^2Q}.\qedhere\]
\end{proof}

\subsubsection{Fixed-point argument}

We now use a fixed-point argument to construct a root decoupling function.
Define
\nomenclature[Q cl]{$\clQ_\sigma$}{operator for which the root decoupling function is a fixed point, \cref{eq:Qdef}}
\begin{equation}
  \clQ_{\sigma}g(Q,a)=\bigl[\EE \sigma^{2}\bigl(\Theta_{a,Q}^{g}(Q)\bigr)\bigr]^{1/2},\label{eq:Qdef}
\end{equation}
where $\bigl(\Theta_{a,Q}^{g}(q)\bigr)_{q\in[0,Q]}$ solves \zcref[range]{eq:thetaSDE,eq:thetaIC}.
We note that a fixed point of $\clQ_{\sigma}$ is a root decoupling
function for \zcref[range]{eq:FBSDE,eq:Jeqn}. We also note that
\begin{equation}
  |\clQ_{\sigma}g(Q,a)|_{\Frob}^{2}=\bigl\lvert \bigl[\EE \sigma^{2}\bigl(\Theta_{a,Q}^{g}(Q)\bigr)\bigr]^{1/2}\bigr\rvert _{\Frob}^{2}=\EE \tr\sigma^{2}\bigl(\Theta_{a,Q}^{g}(Q)\bigr)=\EE \bigl\lvert \sigma\bigl(\Theta_{a,Q}^{g}(Q)\bigr)\bigr\rvert _{\Frob}^{2}.\label{eq:normofQ}
\end{equation}
Define the set of functions
\nomenclature[zzzgreek σMlambda]{$\lipset(M,\lambda)$}{space of $\lambda$-Lipschitz functions satisfying a quantitative growth bound, \cref{eq:SigMlamdef}}
\begin{equation}
  \lipset(M,\lambda)\coloneqq\bigl\{\sigma\in\Lip(\RR^{m};\clH_{+}^{m})\suchthat\Lip(\sigma)\le\lambda\text{ and }|\sigma(u)|_{\Frob}^{2}\le M+\lambda^{2}|u|^{2}\text{ for all }u\in\RR^{m}\bigr\}.\label{eq:SigMlamdef}
\end{equation}
For $Q_{0}<\lambda^{-2}$, define the set of functions $\clZ_{Q_{0},M,\lambda}$
by\nomenclature[Z scr]{$\clZ_{Q_0,M,\lambda}$}{function space, see \cref{eq:ZQMlambdadef}}
\begin{equation}
  \label{eq:ZQMlambdadef}
  \begin{aligned}
      \clZ_{Q_{0},M,\lambda}=\Big\{ g \colon [0,Q_{0}]& \times\RR^{m}\to \clH_{+}^{m}\text{ continuous}\suchthat\\
    &g(q,\cdot)\in\lipset\bigl((1-\lambda^{2}q)^{-2}M,(\lambda^{-2}-q)^{-1/2}\bigr)\text{ for all }q\in[0,Q_{0}]\Big\}.
  \end{aligned}
\end{equation}
The bounds $(q-\lambda^2q)^{-2}$ and $(\lambda^{-2}-q)^{-1/2}$ imposed in \zcref{eq:ZQMlambdadef} are chosen to to be compatible with some Grönwall-type arguments below; notice that they allow $g(q,\cdot)$ to blow up as $q$ approaches $\lambda^{-2}$.
We will construct the root decoupling function $J_{\sigma}$ as a fixed
point of $\clQ_{\sigma}$ in a certain $\clZ_{Q_{0},M,\lambda}$.
Recall $\Qbar_{\FBSDE}$ from \cref{def:QbarFBSDE}.
\begin{prop}
  \label{prop:localwellposed}Fix $\lambda,M\in(0,\infty)$ and $Q_{0}\in[0,\lambda^{-2})$.
  For any $\sigma\in\lipset(M,\lambda)$, there is a unique root decoupling
  function $J_{\sigma}\in\clA_{Q_{0}}$. In particular, we have
  \begin{equation*}
    \Qbar_{\FBSDE}(\sigma)\ge\Lip(\sigma)^{-2}.
  \end{equation*}
  Moreover, there is a $C=C(Q_0,M,\lambda)<\infty$ such that for
  any $g\in\clZ_{Q_{0},M,\lambda}$, we have
  \begin{equation}
    \sup_{q\in[0,Q_{0}]}\|(g-J_{\sigma})(q,\anon)\|_{\clX}\le C\cdot\sup_{q\in[0,Q_{0}]}\|(g-\clQ_{\sigma}g)(q,\anon)\|_{\clX}\label{eq:onestepclose}
  \end{equation}
  and indeed
  \begin{equation}\label{eq:nfoldlimit}
    \adjustlimits\lim_{n\to\infty} \sup_{q\in [0,Q_0]}\|(\clQ_\sigma^n g-J_\sigma)(q,\cdot)\|_\clX = 0,
\end{equation}
where $\clQ_\sigma^n$ denotes the $n$-fold iterated application of $\clQ_\sigma$.
\end{prop}

\begin{proof}
  \cref{prop:apriorilipschitzbd} tells us that any root decoupling function
  in $\clA_{Q_{0}}$ must actually lie in $\clZ_{Q_{0},M',\lambda}$
  for some $M'\in(0,\infty)$. Since the sets $\clZ_{Q_{0},M',\lambda}$
  increase as $M'$ increases, it is sufficient to show the existence
  of a unique root decoupling function $J_{\sigma}\in\clZ_{Q_{0},M',\lambda}$
  for all sufficiently large $M'$. Indeed, if there were another root decoupling
  function $\widetilde{J}_{\sigma}\in\clA_{Q_{0}}$, then we
  would have $J_{\sigma},\widetilde{J}_{\sigma}\in\clZ_{Q_{0},M'',\lambda}$
  for some sufficiently large $M''$, contradicting the uniqueness
  in $\clZ_{Q_{0},M'',\lambda}$. Since $\sigma\in\lipset(\lambda,M)$ implies that $\sigma\in\lipset(\lambda,M')$ for all $M'\ge M$, it suffices to show that that under the hypotheses of the theorem, there is a 
  unique root decoupling
  function $J_{\sigma}\in\clZ_{Q_{0},M,\lambda}$.
  We proceed in several steps.
  \begin{thmstepnv}
    \item \emph{$\clQ_{\sigma}$ maps $\clZ_{Q_{0},M,\lambda}$
      to itself.} Let $g\in\clZ_{Q_{0},M,\lambda}$. We need to check
    several properties of $\clQ_{\sigma}g$.
    \begin{thmstepnv}
      \item \emph{The upper bound on $\clQ_{\sigma}g$.} We have, using
      \zcref[range]{eq:thetaSDE,eq:thetaIC} and the fact that $g\in\clZ_{Q_{0},M,\lambda}$,
      that whenever $0\le q\le Q\le Q_{0}$,
      \begin{align*}
        \EE |\Theta_{a,Q}^{g}(q)|^{2} & =|a|^{2}+\int_{0}^{q}\EE |g\bigl(Q-r,\Theta_{a,Q}^{g}(r)\bigr)|_{\Frob}^{2}\,\dif r                                                                  \\
                                      & \le|a|^{2}+M\int_{0}^{q}\frac{\dif r}{(1-\lambda^{2}[Q-r])^{2}}+\int_{0}^{q}\frac{\EE |\Theta_{a,Q}^{g}(r)|^{2}}{\lambda^{-2}-Q+r}\,\dif r \\
                                      & =|a|^{2}+\frac{Mq}{[1-\lambda^{2}(Q-q)](1-\lambda^{2}Q)}+\int_{0}^{q}\frac{\EE |\Theta_{a,Q}^{g}(r)|^{2}}{\lambda^{-2}-Q+r}\,\dif r.
      \end{align*}
      By Grönwall's inequality, this means that
      \begin{align}
        \EE |\Theta_{a,Q}^{g}(q)|^{2} & \le\left(|a|^{2}+\frac{Mq}{[1-\lambda^{2}(Q-q)](1-\lambda^{2}Q)}\right)\exp\left\{ \int_{0}^{q}\frac{\dif r}{\lambda^{-2}-Q+r}\right\} \nonumber \\
				      & =\frac{|a|^{2}[1-\lambda^{2}(Q-q)]}{1-\lambda^{2}Q}+\frac{Mq}{(1-\lambda^{2}Q)^{2}} \le \frac{|a|^2}{1-\lambda^2Q}+\frac{Mq}{(1-\lambda^2Q)^2}.\label{eq:midexpectation}
      \end{align}
      Therefore, by using \cref{eq:normofQ} and the fact that $\sigma\in\lipset(M,\lambda)$,
      we obtain
      \begin{align}
        |\clQ_{\sigma}g(Q,a)|_{\Frob}^{2} & \le M+\lambda^{2}\EE |\Theta_{a,Q}^{g}(Q)|^{2}\le M+\frac{|a|^{2}}{\lambda^{-2}-Q}+\frac{\lambda^{2}QM}{(1-\lambda^{2}Q)^{2}}\nonumber                                                                    \\
                                          & =\frac{|a|^{2}}{\lambda^{-2}-Q}+M\left(\frac{1-\lambda^{2}Q+\lambda^{4}Q^{2}}{(1-\lambda^{2}Q)^{2}}\right)\le\frac{|a|^{2}}{\lambda^{-2}-Q}+\frac{M}{(1-\lambda^{2}Q)^{2}},\label{eq:upperboundpreserved}
      \end{align}
      where in the last inequality we used that $\lambda^{2}Q<1$ and so
      $\lambda^{2}Q>\lambda^{4}Q^{2}$.
      
      \item \emph{The Lipschitz bound on $\clQ_{\sigma}g(Q,\cdot)$.} Using
      the Itô isometry we have, for all $a,b\in\RR^m$, that whenever $0\le q\le Q\le Q_{0}$,
      \begin{align*}
        \EE |\Theta_{a,Q}^{g}(q)-\Theta_{b,Q}^{g}(q)|^{2} & =|a-b|^{2}+\int_{0}^{q}\EE \bigl\lvert g\bigl(Q-r,\Theta_{a,Q}^{g}(r)\bigr)-g\bigl(Q-r,\Theta_{b,Q}^{g}(r)\bigr)\bigr\rvert _{\Frob}^{2}\,\dif r       \\
                                                          & \le|a-b|^{2}+\int_{0}^{q}\frac{\EE |\Theta_{a,Q}^{g}(r)-\Theta_{b,Q}^{g}(r)|^{2}}{\lambda^{-2}-Q+r}\,\dif r.
      \end{align*}
      The inequality follows from the assumption $g\in\clZ_{Q_{0},M,\lambda}$, for then $\Lip\bigl(g(Q-r,\anon)\bigr)\le(\lambda^{-2}-Q+r)^{-1/2}$.
      By Grönwall's
      inequality, we therefore obtain
      \[
        \EE |\Theta_{a,Q}^{g}(Q)-\Theta_{b,Q}^{g}(Q)|^{2}\le|a-b|^{2}\exp\left\{ \int_{0}^{Q}\frac{\dif q}{\lambda^{-2}-Q+r}\right\} =\frac{|a-b|^{2}}{1-\lambda^{2}Q}.
      \]
      Therefore, we have (using \cref{prop:matrixvaluedreversetriangleinequality}
      and the fact that $\Lip(\sigma)\le\lambda$) that
      \begin{align}
        |\clQ_{\sigma}g(Q,a)-\clQ_{\sigma}g(Q,b)|_{\Frob}^{2} & \le\EE |\sigma(\Theta_{a,Q}^{g}(Q))-\sigma(\Theta_{b,Q}^{g}(Q))|_{\Frob}^{2}\nonumber                                     \\
                                                              & \le\lambda^{2}\EE |\Theta_{a,Q}^{g}(Q)-\Theta_{b,Q}^{g}(Q)|^{2}\le\frac{|a-b|^{2}}{\lambda^{-2}-Q}.\label{eq:Qglipschitz}
      \end{align}
      \item \emph{\label{step:spacetimecontinuity}The space-time continuity of
        $\clQ_{\sigma}g$.} Suppose that $0\le Q_{1}\le Q_{2}\le Q_{0}$.
      We note that
      \[
        \bigl(\Theta_{a,Q_{2}}^{g}(Q_{2}-Q_{1}+q)\bigr)_{q\in[0,Q_{1}]}
      \]
      satisfies the same SDE problem as $(\Theta_{a,Q_{1}}^{g}(q))_{q\in[0,Q_{1}]}$.
      Therefore, we have
      \begin{align*}
        [\clQ_{\sigma}g(Q_{2},a)]^{2} & =\EE \left[\EE \bigl[\sigma^{2}\bigl(\Theta_{a,Q_{2}}^{g}(Q_{2})\bigr)\;\big|\;\Theta_{a,Q_{2}}^{g}(Q_{2}-Q_{1})\bigr]\right]=\EE \left[\clQ_{\sigma}g\bigl(Q_{1},\Theta_{a,Q_{2}}^{g}(Q_{2}-Q_{1})\bigr)\right]^{2}.
      \end{align*}
      This means that
      \begin{align}
        |\clQ_{\sigma}g(Q_{2},a)-\clQ_{\sigma}g(Q_{1},a)|_{\Frob} & =\left\lvert  \bigl(\EE \bigl[\clQ_{\sigma}g\bigl(Q_{1},\Theta_{a,Q_{2}}^{g}(Q_{2}-Q_{1})\bigr)\bigr]^{2}\bigr)^{1/2}-\clQ_{\sigma}g(Q_{1},a)\right\rvert  _{\Frob}\nonumber   \\
        \overset{\cref{eq:matrixvaluedreversetriangleinequality}} & {\le}\left(\EE \bigl\lvert \clQ_{\sigma}g\bigl(Q_{1},\Theta_{a,Q_{2}}^{g}(Q_{2}-Q_{1})\bigr)-\clQ_{\sigma}g(Q_{1},a)\bigr\rvert _{\Frob}^{2}\right)^{1/2}\nonumber \\
        \overset{\cref{eq:Qglipschitz}}                           & {\le}(\lambda^{-2}-Q_{1})^{-1/2}\left(\EE |\Theta_{a,Q_{2}}^{g}(Q_{2}-Q_{1})-a|^{2}\right)^{1/2}.\label{eq:Qdiff}
      \end{align}
      We also have
      \begin{align}
        \EE & |\Theta_{a,Q_{2}}^{g}(Q_{2}-Q_{1})-a|^{2}=\int_{0}^{Q_{2}-Q_{1}}\EE \bigl\lvert g\bigl(Q_{2}-q,\Theta_{a,Q_{2}}^{g}(q)\bigr)\bigr\rvert _{\Frob}^{2}\,\dif q\nonumber                                                                                  \\
            & \le\frac{Q_{2}-Q_{1}}{1-\lambda^{2}(Q_{2}-Q_{1})}\cdot M+\int_{0}^{Q_{2}-Q_{1}}\frac{\EE |\Theta_{a,Q_{2}}^{g}(q)|^{2}}{\lambda^{-2}-q}\,\dif q\nonumber                                                               \\
            & \le\frac{Q_{2}-Q_{1}}{1-\lambda^{2}(Q_{2}-Q_{1})}\cdot M+\int_{0}^{Q_{2}-Q_{1}}\frac{a^{2}[1-\lambda^{2}(Q_{2}-q)]+(1-\lambda^{2}Q_{2})^{-1}Mq}{(1-\lambda^{2}Q_{2})(\lambda^{-2}-q)}\,\dif q,\label{eq:varboundfbsde}
      \end{align}
      with the first inequality by the fact that $g\in\clZ_{Q_{0},M,\lambda}$
      and the last inequality by \cref{eq:midexpectation}. Now the right
      side of \cref{eq:varboundfbsde} goes to zero as $Q_{2}-Q_{1}\to0$,
      locally uniformly in $a$, so together with \cref{eq:Qdiff,eq:Qglipschitz}
      this implies that $\clQ_{\sigma}g$ is continuous. This, along
      with the bounds \cref{eq:upperboundpreserved,eq:Qglipschitz},
      show that $\clQ_{\sigma}$ maps $\clZ_{Q_{0},M,\lambda}$
      to itself.
    \end{thmstepnv}
    \item \emph{$\clQ_{\sigma}$ is a contraction in an appropriate norm.}
    Suppose that $g_{1},g_{2}\in\clZ_{Q_{0},M,\lambda}$. Then,
    whenever $0\le q\le Q\le Q_{0}$, we have
    \begin{align*}
      \EE & |\Theta_{a,Q}^{g_{1}}(q)-\Theta_{a,Q}^{g_{2}}(q)|^{2}=\int_{0}^{g}\EE \bigl\lvert g_{1}\bigl(Q-r,\Theta_{a,Q}^{g_{1}}(r)\bigr)-g_{2}\bigl(Q-r,\Theta_{a,Q}^{g_{2}}(r)\bigr)\bigr\rvert _{\Frob}^{2}\,\dif r                                                                                       \\
          & \le2\int_{0}^{q}\left(\|(g_{1}-g_{2})(Q-r,\anon)\|_{\clX}^{2}(1+\EE |\Theta_{a,Q}^{g_{1}}(r)|^{2})+\frac{\EE |\Theta_{a,Q}^{g_{1}}(r)-\Theta_{a,Q}^{g_{2}}(r)|^{2}}{\lambda^{-2}-Q+r}\right)\,\dif r                                                    \\
	   \overset{\cref{eq:midexpectation}}&\le2\int_{0}^{q}\Bigg(\|(g_{1}-g_{2})(Q-r,\cdot)\|_{\clX}^{2}\left(1+\frac{|a|^{2}}{1-\lambda^{2}Q}+\frac{MQ_0}{(1-\lambda^{2}Q)^{2}}\right)                                                                     %
	   +\frac{\EE |\Theta_{a,Q}^{g_{1}}(r)-\Theta_{a,Q}^{g_{2}}(r)|^{2}}{\lambda^{-2}-Q+r}\Bigg)\,\dif r,                                                               %
    \end{align*}
    where in the first inequality we used the triangle inequality, the
    definition \cref{eq:Xnorm} of the norm on $\clX$, and the assumption
    that $g_{2}\in\clZ_{Q_{0},M,\lambda}$. %
    By Grönwall's inequality,
    we can therefore derive
    \begin{align*}
      \EE & |\Theta_{a,Q}^{g_{1}}(q)-\Theta_{a,Q}^{g_{2}}(q)|^{2}\le2\left(1+\frac{|a|^{2}}{1-\lambda^{2}Q_{0}}+\frac{MQ_{0}}{(1-\lambda^{2}Q_{0})^{2}}\right)\exp\left\{ \int_{0}^{q}\frac{2}{\lambda^{-2}-Q+r}\,\dif r\right\} \\
          & \hspace{24em}\cdot\int_{0}^{q}\|(g_{1}-g_{2})(Q-r,\anon)\|_{\clX}^{2}\,\dif r                                                                                                                                      \\
          & \le\frac{2}{(1-\lambda^{2}Q_{0})^{2}}\left(1+\frac{|a|^{2}}{1-\lambda^{2}Q_{0}}+\frac{MQ_{0}}{(1-\lambda^{2}Q_{0})^{2}}\right)\int_{0}^{q}\|(g_{1}-g_{2})(Q-r,\anon)\|_{\clX}^{2}\,\dif r.
    \end{align*}
    This means that (using \cref{prop:matrixvaluedreversetriangleinequality}
    and the fact that $\Lip(\sigma)\le\lambda$)
    \begin{align*}
      | & \clQ_{\sigma}g_{1}(Q,a)-\clQ_{\sigma}g_{2}(Q,a)|_{\Frob}^{2}\le\EE \bigl\lvert \sigma\bigl(\Theta_{a,Q}^{g_{1}}(Q)\bigr)-\sigma\bigl(\Theta_{a,Q}^{g_{2}}(Q)\bigr)\bigr\rvert _{\Frob}^{2}\le\lambda^{2}\EE \bigl\lvert \Theta_{a,Q}^{g_{1}}(Q)-\Theta_{a,Q}^{g_{2}}(Q)\bigr\rvert _{\Frob}^{2} \\
        & \le\frac{2\lambda^{2}}{(1-\lambda^{2}Q_{0})^{2}}\left(1+\frac{|a|^{2}}{1-\lambda^{2}Q_{0}}+\frac{MQ_{0}}{(1-\lambda^{2}Q_{0})^{2}}\right)\int_{0}^{Q}\|(g_{1}-g_{2})(r,\anon)\|_{\clX}^{2}\,\dif r.
    \end{align*}
    Therefore, we have
    \begin{align}
      \| & \clQ_{\sigma}g_{1}(Q,\cdot)-\clQ_{\sigma}g_{2}(Q,\cdot)\|_{\clX}^{2}\nonumber                                                                                                                                                      \\
         & \le\sup_{a\in\RR^m}\frac{2\lambda^{2}}{(1-\lambda^{2}Q_{0})^{2}(1+|a|^{2})}\left(1+\frac{|a|^{2}}{1-\lambda^{2}Q_{0}}+\frac{MQ_{0}}{(1-\lambda^{2}Q_{0})^{2}}\right)\int_{0}^{Q}\|(g_{1}-g_{2})(r,\anon)\|_{\clX}^{2}\,\dif r\nonumber \\
         & \le R(Q_{0},M,\lambda)\int_{0}^{Q}\|(g_{1}-g_{2})(r,\anon)\|_{\clX}^{2}\,\dif r,\label{eq:Qsigmag1g2}
    \end{align}
    where we have defined the constant
    \begin{align*}
      R(Q_{0},M,\lambda) & =\frac{2\lambda^{2}}{(1-\lambda^{2}Q_{0})^{2}}\left(1+\frac{1}{1-\lambda^{2}Q_{0}}+\frac{MQ_{0}}{(1-\lambda^{2}Q_{0})^{2}}\right).
    \end{align*}
    Now let $\clY_{Q_{0},M,\lambda}$ be the Banach space of continuous
    functions $g\colon[0,Q_{0}]\times\RR^m\to\clH^m_+$ such that the norm
    \[
      \|g\|_{\clY_{Q_{0},M,\lambda}}\coloneqq\sup_{Q\in[0,Q_{0}]}\left[\e^{-2R(Q_{0},M,\lambda)Q}\|g(Q,\anon)\|_{\clX}\right]
    \]
    is finite. (Recall the definition \cref{eq:Xnorm} of the norm on $\clX$.)
    We note that $\clZ_{Q_{0},M,\lambda}$ is a closed subset of
    $\clY_{Q_{0},M,\lambda}$, so the metric induced on $\clZ_{Q_{0},M,\lambda}$
    by the norm $\|\anon\|_{\clY_{Q_{0},M,\lambda}}$ makes $\clZ_{Q_{0},M,\lambda}$
    a complete metric space. Then \cref{eq:Qsigmag1g2} implies that
    \begin{align*}
      \|\clQ_{\sigma}g_{1}-\clQ_{\sigma}g_{2}\|_{\clY_{Q_{0},M,\lambda}}^{2} & \le\sup_{Q\in[0,Q_{0}]}\left[\e^{-2R(Q_{0},M,\lambda)Q}R(Q_{0},M,\lambda)\|g_{1}-g_{2}\|_{\clY_{Q_{0},M,\lambda}}^{2}\int_{0}^{Q}\e^{2R(Q_{0},M,\lambda)r}\,\dif r\right] \\
                                                                             & \le\frac{1}{2}\|g_{1}-g_{2}\|_{\clY_{Q_{0},M,\lambda}}^{2}.
    \end{align*}
    Hence $\clQ_{\sigma}$ is a contraction on $\clZ_{Q_{0},M,\lambda}$
    (with the metric induced by the norm $\|\anon\|_{\clY_{Q_{0},M,\lambda}}$)
    with contraction factor $2^{-1/2}$. The Banach contraction mapping principle thus implies that $\clQ_{\sigma}$
    has a unique fixed point $J_{\sigma}$ in $\clZ_{Q_{0},M,\lambda}$,
    and moreover (by the quantitative bound coming from the proof of the contraction mapping principle; see e.g.\ \cite[Theorem~1.1(ii)]{Bon62}) we have
    \[
      \|g-J_{\sigma}\|_{\clY_{Q_{0},M,\lambda}}\le\frac{\|g-\clQ_{\sigma}g\|_{\clY_{Q_{0},M,\lambda}}}{1-2^{-1/2}}.
    \]
    This implies \cref{eq:onestepclose} since the norms are equivalent.
    \qedhere
  \end{thmstepnv}
\end{proof}

\subsection{Extending the solution\label{subsec:extendingFBSDE}}

\cref{prop:localwellposed} showed that $\Qbar_{\FBSDE}(\sigma)\ge\Lip(\sigma)^{-2}$,
which is the $Q=0$ case of \cref{thm:FBSDEwellposed}. In this subsection,
we prove \cref{thm:FBSDEwellposed} in its full generality. We start
with a simple lemma.
\begin{lem}
  \label{lem:Jatdifferenttimes}Let $\sigma\in\Lip(\RR^m;\clH^{m}_{+})$.
  Whenever $0\le Q_{1}\le Q_{2}<\Qbar_{\FBSDE}(\sigma)$, we
  have, for any $b\in\RR^m$,
  \begin{equation}
    J_{\sigma}^{2}(Q_{2},b)=\EE \bigl[J_{\sigma}^{2}\bigl(Q_{1},\Gamma_{b,Q_{2}}^{\sigma}(Q_{2}-Q_{1})\bigr)\bigr],\label{eq:Jatdifferenttimes}
  \end{equation}
  where $\Gamma^\sigma_{b,Q_2}$ is as in \zcref[range]{eq:FBSDE,eq:Jeqn}, which is well-defined since $Q_2<\overline{Q}_\FBSDE(\sigma)$.
\end{lem}

\begin{proof}
  We have by \cref{eq:Jeqn} and the tower property of conditional expectation
  that
  \begin{equation}
    J_{\sigma}^{2}(Q_{2},b)=\EE \bigl[\sigma^{2}\bigl(\Gamma_{b,Q_{2}}^{\sigma}(Q_{2})\bigr)\bigr]=\EE \Big[\EE \bigl[\sigma^{2}\bigl(\Gamma_{b,Q_{2}}^{\sigma}(Q_{2})\bigr)\;\big|\;\scrG_{Q_{2}-Q_{1}}\bigr]\Big].\label{eq:Jsqcondition}
  \end{equation}
  We note that, for any $a\in\RR^m$,
  \[
    \EE \bigl[\sigma^{2}\bigl(\Gamma_{b,Q_{2}}^{\sigma}(Q_{2})\bigr)\;\big|\;\Gamma_{b,Q_{2}}^{\sigma}(Q_{2}-Q_{1})=a\bigr]=\EE \bigl[\sigma^{2}\bigl(\Gamma_{a,Q_{1}}^{\sigma}(Q_{1})\bigr)\bigr]=J_{\sigma}^{2}(Q_{1},a).
  \]
  Using this, along with the strong Markov property, in \cref{eq:Jsqcondition},
  we obtain \cref{eq:Jatdifferenttimes}.
\end{proof}
The main idea in this section is illustrated by the following lemma.
\begin{lem}
  \label{lem:chainJs}%
  Let $\sigma\in\Lip(\RR^m;\clH^{m}_{+})$.
  Whenever $0\le Q_{1}\le Q_{2}<\Qbar_{\FBSDE}(\sigma)$ and
  \begin{equation}
    Q_{2}-Q_{1}<\Qbar_{\FBSDE}\bigl(J_{\sigma}(Q_{1},\anon)\bigr),\label{eq:addhyp}
  \end{equation}
  we have
  \begin{equation}
    J_{\sigma}(Q_{2},b)=J_{J_{\sigma}(Q_{1},\anon)}(Q_{2}-Q_{1},b)\qquad\text{for all }b\in\RR^m.\label{eq:Jdiff}
  \end{equation}
\end{lem}
In \cref{cor:chainJs-nohyp} below, we will strengthen this result by showing that the hypothesis \cref{eq:addhyp} is automatically satisfied.
\begin{proof}
  We have by \cref{lem:Jatdifferenttimes} that, for any $Q\in\bigl[Q_{1},\Qbar_{\FBSDE}(\sigma)\bigr)$
  and any $b\in\RR^m$,
  \begin{equation}
    J_{\sigma}^{2}(Q,b)=\EE \bigl[J_{\sigma}^{2}\bigl(Q_{1},\Gamma_{b,Q}^{\sigma}(Q-Q_{1})\bigr)\bigr].\label{eq:J2CE}
  \end{equation}
  By \cref{eq:addhyp}, there is a unique solution
  to the following FBSDE problem for $Q\in[0,Q_{2}-Q_{1}]$:
  \begin{align}
    \dif\Gamma_{b,Q}^{J_{\sigma}(Q_{1},\anon)}(q) & =J_{J_{\sigma}(Q_{1},\anon)}\bigl(Q-q,\Gamma_{b,Q}^{J_{\sigma}(Q_{1},\anon)}(q)\bigr)\dif B(q),\qquad q\in(0,Q);\label{eq:FBSDEweird} \\
    \Gamma_{b,Q}^{J_{\sigma}(Q_{1},\anon)}(0)     & =b;\label{eq:FBSDEweirdic}                                                                                                  \\
    J_{J_{\sigma}(Q_{1},\anon)}(Q,b)              & = \bigl(\EE \bigl[J_{\sigma}^{2}\bigl(Q_{1},\Gamma_{b,Q}^{J_{\sigma}(Q_{1},\anon)}(Q)\bigr)\bigr]\bigr)^{1/2}.\label{eq:Jweird}
  \end{align}
  Also, given $Q\in[0,\Qbar_{\FBSDE}(\sigma)-Q_{1})$, \zcref[range]{eq:FBSDE,eq:Jeqn} and \cref{eq:J2CE} with $Q\setto Q_{1}+Q$ yield
  \begin{align*}
    \dif\Gamma_{b,Q_{1}+Q}^{\sigma}(q) & =J_{\sigma}\bigl(Q_{1}+Q-q,\Gamma_{b,Q_{1}+Q}^{\sigma}(q)\bigr)\dif B(q),\qquad q\in(0,Q);\\
    \Gamma_{b,Q_{1}+Q}^{\sigma}(0)     & =b;\\
    J_{\sigma}(Q_{1}+Q,b)              & =\bigl[\EE \sigma^{2}\bigl(\Gamma_{b,Q_{1}+Q}^{\sigma}(Q)\bigr)\bigr]^{1/2}=\bigl(\EE \bigl[J_{\sigma}^{2}\bigl(Q_{1},\Gamma_{b,Q_{1}+Q}^{\sigma}(Q)\bigr)\bigr]\bigr)^{1/2}.
  \end{align*}
  This means that
  \begin{equation*}
    \Gamma_{b,Q}^{J_{\sigma}(Q_{1},\anon)}(q)\setto\Gamma_{b,Q_{1}+Q}^{\sigma}(q)\qquad\text{and}\qquad J_{J_{\sigma}(Q_{1},\anon)}(Q,b)\setto J_{\sigma}(Q+Q_{1},b),\qquad Q\in[0,Q_{2}-Q_{1}]
  \end{equation*}
  solve \zcref[range]{eq:FBSDEweird,eq:Jweird}. By the uniqueness of
  the solutions to that problem, this implies that
  \begin{equation*}
    J_{J_{\sigma}(Q_{1},\anon)}(Q,b)=J_{\sigma}(Q+Q_{1},b)\qquad\text{for all }b\in\RR^m\text{ and }Q\in[0,Q_{2}-Q_{1}].
  \end{equation*}
  Taking $Q=Q_{2}-Q_{1}$ yields \cref{eq:Jdiff}.
\end{proof}
\cref{prop:localwellposed} was the base case in the proof of \cref{thm:FBSDEwellposed}. The next proposition
gives the inductive step. %
\begin{prop}
  \label{prop:qbarsigmatqbarjsigma}
  For any $Q'\in\bigl[0,\Qbar_{\FBSDE}(\sigma)\bigr)$, we have
  \begin{equation}
    \Qbar_{\FBSDE}(\sigma)\ge Q'+\Qbar_{\FBSDE}\bigl(J_{\sigma}(Q',\anon)\bigr).\label{eq:qbarsigmagtqbarjsigma}
  \end{equation}
\end{prop}
\begin{proof}
  Since $Q'<\Qbar_{\FBSDE}(\sigma)$, there is a unique
  root decoupling function $J_{\sigma}\in\clA_{Q'}$ for \zcref[range]{eq:FBSDE,eq:Jeqn}
  on $[0,Q']$. Let $P\in\bigl[0,\Qbar_{\FBSDE}\bigl(J_{\sigma}(Q',\anon)\bigr)\bigr)$,
  so there is a unique root decoupling function $J_{J_{\sigma}(Q',\anon)}\in\clA_{P}$
  for \zcref[range]{eq:FBSDE,eq:Jeqn} with $\sigma\setto J_{\sigma}(Q',\anon)$.

  We wish to extend the function $J_{\sigma}$ to the time interval
  $[0,Q'+P]$ by putting
  \begin{equation}
    J_{\sigma}(q,b)=J_{J_{\sigma}(Q',\anon)}(q-Q',b)\qquad\text{for all }b\in\RR^m\text{ and }q\in[Q',Q'+P],\label{eq:extendJ}
  \end{equation}
  and retaining the original definition of $J_{\sigma}(q,b)$ for $q\in[0,Q'${]}.
        We note that, with this extension, $J_{\sigma}\in\clA_{Q'+P}$.
        We need to check \cref{eq:root decouplingholds} with $J\setto J_{\sigma}$
        and $q\setto Q\in[0,Q'+P]$. This is clear for $Q\in[0,Q']$
        since $J_{\sigma}$ was assumed to be a root decoupling function for \zcref[range]{eq:FBSDE,eq:Jeqn}
        on $[0,Q']$. So suppose that $Q\in(Q',Q'+P]$. We note that, for all $b\in\RR^m$,
  \begin{equation}
    \text{conditional on }\Theta_{a,Q}^{J_{\sigma}}(Q-Q')=b,\qquad \bigl(\Theta_{a,Q}^{J_{\sigma}}(Q-Q'+q)\bigr)_{q\in[0,Q']}\overset{\mathrm{law}}{=}\bigl(\Theta_{b,Q'}^{J_{\sigma}}(q)\bigr)_{q\in[0,Q']}\label{eq:restartSDE}
  \end{equation}
  since both sides satisfy the same SDE \cref{eq:thetaSDE} (with $Q\setto Q'$).
  Applying \cref{eq:restartSDE} with $q\setto Q'$ and using %
  \cref{eq:Jeqn} with $q\setto Q'$ 
 (which we can do %
  by the assumption $Q'<\Qbar_{\FBSDE}(\sigma)$),
  we conclude that
  \begin{equation*}
    \EE \bigl[\sigma^{2}\bigl(\Theta_{a,Q}^{J_{\sigma}}(Q)\bigr)\;\big|\;\Theta_{a,Q}^{J_{\sigma}}(Q-Q')\bigr]=J_{\sigma}^{2}\bigl(Q',\Theta_{a,Q}^{J_{\sigma}}(Q-Q')\bigr).
  \end{equation*}
  Taking expectations and using the tower property of conditional expectation,
  we obtain
  \begin{equation}
      \EE \bigl[\sigma^{2}\bigl(\Theta_{a,Q}^{J_{\sigma}}(Q)\bigr)\bigr]=\EE \bigl[J_{\sigma}^{2}\bigl(Q',\Theta_{a,Q}^{J_{\sigma}}(Q-Q')\bigr)\bigr].\label{eq:sigmasquaredQthing}
  \end{equation}

  Now we note that, for $q\in[0,Q-Q']$, we have
  \begin{equation}
    \dif\Theta_{a,Q}^{J_{\sigma}}(q)\overset{\cref{eq:thetaSDE}}{=}J_{\sigma}\bigl(Q-q,\Theta_{a,Q}^{J_{\sigma}}(q)\bigr)\dif B(q)\overset{\cref{eq:extendJ}}{=}J_{J_{\sigma}(Q',\anon)}\bigl(Q-Q'-q,\Theta_{a,Q}^{J_{\sigma}}(q)\bigr)\dif B(q).\label{eq:rewriteSDEintermsofJJ}
  \end{equation}
  Comparing \cref{eq:rewriteSDEintermsofJJ} and \cref{eq:FBSDE}, we see that $\bigl(\Theta_{a,Q}^{J_{\sigma}}(q)\bigr)_{q\in[0,Q-Q']}$
  and $\bigl(\Gamma_{a,Q-Q'}^{J_{\sigma}(Q',\anon)}(q)\bigr)_{q\in[0,Q-Q']}$ satisfy
  the same SDE (with the same initial conditions), and so
  \begin{equation*}
    \bigl(\Theta_{a,Q}^{J_{\sigma}}(q)\bigr)_{q\in[0,Q-Q']}=\bigl(\Gamma_{a,Q-Q'}^{J_{\sigma}(Q',\anon)}(q)\bigr)_{q\in[0,Q-Q']}.
  \end{equation*}
  In particular, we have
  \begin{equation}
    \Theta_{a,Q}^{J_{\sigma}}(Q-Q')=\Gamma_{a,Q-Q'}^{J_{\sigma}(Q',\anon)}(Q-Q').\label{eq:gammabargammaQQprime}
  \end{equation}
  We conclude that
  \begin{align*}
    J_{\sigma}^{2}(Q,a)\overset{\cref{eq:extendJ}}{=}J_{J_{\sigma}(Q',\anon)}^{2}(Q-Q',a) \overset{\cref{eq:Jeqn}} & {=}\EE \bigl[J_{\sigma}^{2}\bigl(Q',\Gamma_{a,Q-Q'}^{J_{\sigma}(Q',\anon)}(Q-Q')\bigr)\bigr]                                                                            \\
    \overset{\cref{eq:gammabargammaQQprime}}                                                                       & {=}\EE \bigl[J_{\sigma}^{2}\bigl(Q',\Theta_{a,Q}^{J_{\sigma}}(Q-Q')\bigr)\bigr]\overset{\cref{eq:sigmasquaredQthing}}{=}\EE \bigl[\sigma^{2}\bigl(\Theta_{a,Q}^{J_{\sigma}}(Q)\bigr)\bigr],
  \end{align*}
  which is \cref{eq:root decouplingholds} with $J\setto J_{\sigma}$
  and $q\setto Q$.

  Thus we have proved that, with the definition \cref{eq:extendJ}, $J_{\sigma}\in\clA_{Q'+P}$
  is a root decoupling function for \zcref[range]{eq:FBSDE,eq:Jeqn} on the
  time interval $[0,Q'+P]$. We claim that the choice of $J_{\sigma}$
  is unique. We know that $J_{\sigma}|_{[0,Q']\times\RR^m}$ is uniquely
  determined by the assumption $Q'<\Qbar_{\FBSDE}(\sigma)$,
  so it remains to check that the extension \cref{eq:extendJ} is unique.
  But this is true by \cref{lem:chainJs}: the assumption \cref{eq:addhyp}
  is satisfied by the assumption that $P<\Qbar_{\FBSDE}\bigl(J_{\sigma}(Q',\anon)\bigr)$.
  This completes the proof.
\end{proof}
Now we can conclude the proof of \cref{thm:FBSDEwellposed}.
\begin{proof}[Proof of \cref{thm:FBSDEwellposed}.]
  We have
  \begin{equation*}
    \Qbar_{\FBSDE}(\sigma)\ge Q+\Qbar_{\FBSDE}\bigl(J_{\sigma}(Q,\anon)\bigr)\ge Q+\Lip\bigl(J_{\sigma}(Q,\anon)\bigr)^{-2},
  \end{equation*}
  with the first inequality by \cref{prop:qbarsigmatqbarjsigma} and the
  second by \cref{prop:localwellposed}.
\end{proof}
The approach we have taken to \cref{thm:FBSDEwellposed} is representative of most of our coming arguments.
We first establish ``local'' results that hold up to $\Lip(\sigma)^{-2}$.
We then show that these can be extended through a shift of scale provided $J_\sigma$ remains Lipschitz.

With \cref{thm:FBSDEwellposed} proved, it follows immediately that, for any $Q\in [0,\overline Q_\FBSDE(\sigma))$ and any $a\in\RR^m$, we have \begin{equation}\label{eq:GammaisTheta}\Gamma^\sigma_{a,Q} = \Theta^{J_\sigma}_{a,Q}.\end{equation}
We also prove that in fact equality holds in \cref{eq:qbarsigmagtqbarjsigma}.
\begin{cor}
  \label{cor:Qbarsagree}%
  Let $\sigma\in\Lip(\RR^m;\clH^{m}_{+})$.
  For any $Q'\in\bigl[0,\Qbar_{\FBSDE}(\sigma)\bigr)$, we have
  \begin{equation}
    \Qbar_{\FBSDE}(\sigma)=Q'+\Qbar_{\FBSDE}\bigl(J_{\sigma}(Q',\anon)\bigr).\label{eq:Qbarsagree}
  \end{equation}
\end{cor}

\begin{proof}
  We have already proved the ``$\ge$'' direction of \cref{eq:Qbarsagree}
  in \cref{eq:qbarsigmagtqbarjsigma}, so it remains to prove that
  \begin{equation}
    \Qbar_{\FBSDE}(\sigma)\le Q'+\Qbar_{\FBSDE}\bigl(J_{\sigma}(Q',\anon)\bigr).\label{eq:qbarsagree-1}
  \end{equation}
  By \cref{lem:chainJs}, we have
  \begin{equation}
    J_{\sigma}(Q_{2},b)=J_{J_{\sigma}(Q',\anon)}(Q_{2}-Q',b)\qquad\text{for any }b\in\RR^m\label{eq:applychainJs}
  \end{equation}
  whenever $Q_{2}-Q'<\Qbar_{\FBSDE}\bigl(J_{\sigma}(Q',\anon)\bigr)$.
  Suppose that $\Qbar_{\FBSDE}\bigl(J_{\sigma}(Q',\anon)\bigr)<\infty$;
  if this is not the case then \cref{eq:qbarsagree-1} holds vacuously.
  By \cref{thm:FBSDEwellposed}, we have
  \begin{equation*}
      \sup\left\{\Lip\bigl(J_{J_{\sigma}(Q',\anon)}(Q,\anon)\bigr)\ :\ Q\in[0,\Qbar_{\FBSDE}(J_{\sigma}(Q',\anon)))\right\}=\infty.
  \end{equation*}
  By \cref{eq:applychainJs}, this implies that
  \begin{equation*}
      \sup\left\{\Lip\bigl(J_{\sigma}(Q,\anon)\bigr)\ :\  Q\in[Q',Q'+\Qbar_{\FBSDE}(J_{\sigma}(Q',\anon)))\right\}=\infty.
  \end{equation*}
  By \cref{eq:Juniflipschitz}, this implies \cref{eq:qbarsagree-1}.
\end{proof}
\cref{cor:Qbarsagree} lets us drop the assumption \cref{eq:addhyp} in
\cref{lem:chainJs}.
\begin{cor}
  \label{cor:chainJs-nohyp}%
  Let $\sigma\in\Lip(\RR^m;\clH^{m}_{+})$.
  Whenever $0\le Q_{1}\le Q_{2}<\Qbar_{\FBSDE}(\sigma)$, we
  have
  \begin{equation}
    J_{\sigma}(Q_{2},b)=J_{J_{\sigma}(Q_{1},\anon)}(Q_{2}-Q_{1},b)\qquad\text{for all }b\in\RR^m.\label{eq:Jsubst-nohyp}
  \end{equation}
\end{cor}

\begin{proof}
  By \cref{cor:Qbarsagree}, we have
  \begin{equation*}
    \Qbar_{\FBSDE}\bigl(J_{\sigma}(Q_{1},\anon)\bigr)=\Qbar_{\FBSDE}(\sigma)-Q_{1}>Q_{2}-Q_{1},
  \end{equation*}
  so the hypothesis \cref{eq:addhyp} is satisfied, and the result follows
  from \cref{lem:chainJs}.
\end{proof}

\subsection{Properties of the FBSDE solution}

We now prove some properties of solutions to the FBSDE \zcref[range]{eq:FBSDE,eq:Jeqn}
that will be useful in later arguments.

\subsubsection{Bounding the root decoupling function}

Here we prove a bound on the root decoupling function $J_{\sigma}$. We
recall the definition \cref{eq:Sigmaplusdef} of $\quadset(M,\beta)$
and note that
\begin{equation}
  \bigcup_{M,\beta\in(0,\infty)^{2}}\quadset(M,\beta)=\Lip(\RR^m;\clH^m_+).\label{eq:sigmaplusexhaust}
\end{equation}

\begin{prop}
  \label{prop:Jsigmaub}Fix $M,\beta\in(0,\infty)$ and $\sigma\in\quadset(M,\beta)$.
  Then we have
  \begin{equation}
    J_{\sigma}(Q,\anon)\in\quadset\bigl((1-\beta^{2}Q)^{-2}M,(\beta^{-2}-Q)^{-1/2}\bigr)\qquad\text{for all }Q\in\bigl[0,\beta^{-2}\wedge\Qbar_{\FBSDE}(\sigma)\bigr).\label{eq:Jsigmaub}
  \end{equation}
\end{prop}

\begin{proof}
  Fix parameters $\Qbar\in[0,\beta^{-2}\wedge\Qbar_{\FBSDE}(\sigma))$
  and $\eps\in\bigl(0,(\Qbar\beta^{2})^{-1/2}-1\bigr]$, so $\Qbar\le[(1+\eps)\beta]^{-2}$.
  For each $Q\in[0,\Qbar]$, we define
  \[
    \beta_{Q}=\inf\bigl\{\beta'\ge\beta\suchthat\exists M'>0\text{ s.t. }J_{\sigma}(Q,\anon)\in\quadset(M',\beta')\bigr\}
  \]
  and
  \[
    \overline{M}_{Q}=\inf\{M'>0\suchthat J_{\sigma}(Q,\anon)\in\quadset(M',(1+\eps)\beta_{Q})\}.
  \]
  We observe that $0<\beta\le\beta_{Q}<\infty$ by \cref{eq:sigmaplusexhaust}
  and that $\overline{M}_{Q}<\infty$ because $(1+\eps)\beta_{Q}>\beta_{Q}$.
  Also, we have
  \begin{equation}
      J_{\sigma}(Q,\anon) \in \quadset(\overline M_Q,(1+\eps)\beta_Q)\qquad\text{for each $Q\in [0,\overline Q]$}.\label{eq:JsigmaQbd}
  \end{equation}
  Recalling the equation \cref{eq:Jeqn} for $J_{\sigma}$ and using the
  assumption that $\sigma\in\quadset(M,\beta)$, we have for each
  $Q\in[0,\Qbar]$ that
  \begin{equation}
    |J_{\sigma}(Q,b)|_{\Frob}^{2}=\EE |\sigma(\Gamma_{b,Q}^{\sigma}(Q))|_{\Frob}^{2}\le M+\beta^{2}\EE |\Gamma_{b,Q}^{\sigma}(Q)|^{2}.\label{eq:Jsigmabd}
  \end{equation}
  Also, Itô's isometry applied to \cref{eq:FBSDE} yields
  \[
      \EE |\Gamma_{b,Q}^{\sigma}(q)|^{2}=b^{2}+\int_{0}^{q}\EE |J_{\sigma}(Q-r,\Gamma_{b,Q}^{\sigma}(r))|_{\Frob}^{2}\,\dif r\overset{\cref{eq:JsigmaQbd}}\le b^{2}+\int_{0}^{q}\left(\overline{M}_{Q-r}+[(1+\eps)\beta_{Q-r}]^{2}\EE |\Gamma_{b,Q}^{\sigma}(r)|^{2}\right)\,\dif r
  \]
  whenever $0\le q\le Q\le\overline Q$.
  By Grönwall's inequality, this implies that
  \begin{equation}
    \EE |\Gamma_{b,Q}^{\sigma}(Q)|^{2}\le\left(b^{2}+\int_{0}^{Q}\overline{M}_{q}\,\dif q\right)\exp\left\{ \int_{0}^{Q}[(1+\eps)\beta_{q}]^{2}\,\dif q\right\} .\label{eq:Gamma-L2}
  \end{equation}
  Combining \cref{eq:Jsigmabd,eq:Gamma-L2}, we find
  \begin{equation}
    |J_{\sigma}(Q,b)|_{\Frob}^{2}\le M+\beta^{2}\left(b^{2}+\int_{0}^{Q}\overline{M}_{q}\,\dif q\right)\exp\left\{ \int_{0}^{Q}[(1+\eps)\beta_{q}]^{2}\,\dif q\right\} .\label{eq:J2-bd}
  \end{equation}
  It follows that
  \[
    \beta_{Q}^{2}\le\beta^{2}\exp\left\{ \int_{0}^{Q}[(1+\eps)\beta_{q}]^{2}\,\dif q\right\} .
  \]
  Multiplying both sides by $(1+\eps)^{2}$ and applying \cref{lem:fQrecursive}
  with $f(q)=[(1+\eps)\beta_{q}]^{2}$, we find
  \begin{equation}
    \beta_{Q}\le(1+\eps)^{-1}\bigl([(1+\eps)\beta]^{-2}-Q\bigr)^{-1/2}\qquad\text{for all }Q\in[0,\Qbar].\label{eq:betaepsbd}
  \end{equation}

  We next define, for all $Q\in[0,\Qbar]$,
  \[
    \widetilde{M}_{Q}\coloneqq\inf\bigl\{M'>0\suchthat J_{\sigma}(Q,\anon)\in\quadset\bigl(M',\bigl([(1+\eps)\beta]^{-2}-Q\bigr)^{-1/2}\bigr)\bigr\},
  \]
  which is finite by \cref{eq:betaepsbd} and the definitions. Mimicking
  the argument that led to \cref{eq:J2-bd}, we have
  \begin{align*}
    |J_{\sigma}(Q,b)|_{\Frob}^{2} & \le M+\beta^{2}\left(b^{2}+\int_{0}^{Q}\widetilde{M}_{q}\,\dif q\right)\exp\left\{ \int_{0}^{Q}\frac{\dif r}{[(1+\eps)\beta]^{-2}-r}\right\} \\
                                  & =M+\frac{\beta^{2}}{1-[(1+\eps)\beta]^{2}Q}\left(b^{2}+\int_{0}^{Q}\widetilde{M}_{q}\,\dif q\right)                                          \\
                                  & \le M+\frac{\beta^{2}}{1-[(1+\eps)\beta]^{2}Q}\int_{0}^{Q}\widetilde{M}_{q}\,\dif q+\frac{b^{2}}{[(1+\eps)\beta]^{-2}-Q},
  \end{align*}
  so
  \begin{equation}
    \widetilde{M}_{Q}\le M+\frac{\beta^{2}}{1-[(1+\eps)\beta]^{2}Q}\int_{0}^{Q}\widetilde{M}_{q}\,\dif q\label{eq:M-pre-Gronwall}
  \end{equation}
  for all $Q\in[0,\Qbar]$. Now let $g(Q)=\bigl([(1+\eps)\beta]^{-2}-Q\bigr)\widetilde{M}_{Q}$,
  so \cref{eq:M-pre-Gronwall} yields the inequality
  \[
    g(Q)\le[(1+\eps)\beta]^{-2}M+\int_{0}^{Q}\frac{g(r)}{[(1+\eps)\beta]^{-2}-r}\,\dif r.
  \]
  Applying Grönwall's inequality, we find
  \[
    \widetilde{M}_{Q}=\bigl([(1+\eps)\beta]^{-2}-Q\bigr)^{-1}g(Q)\le\bigl(1-[(1+\eps)\beta]^{2}Q\bigr)^{-2}M.
  \]
  By the definition of $\widetilde{M}_{Q}$, we get
  \[
    |J_{\sigma}(Q,b)|^2_\Frob\le\bigl(1-[(1+\eps)\beta]^{2}Q\bigr)^{-2}M+\bigl([(1+\eps)\beta]^{-2}-Q\bigr)^{-1}b^{2}\qquad\text{for all }b\in\RR^m\text{ and all }Q\in[0,\Qbar].
  \]
  Taking $\eps\searrow0$, we obtain \cref{eq:Jsigmaub}.
\end{proof}

\subsubsection{Time-regularity of the root decoupling function}

We now control the time-regularity of the root decoupling function
$J_{\sigma}$. This strengthens a result from \cref{step:spacetimecontinuity}
in the proof of \cref{prop:localwellposed}.
We begin with a small-time continuity result.
\begin{lem}
  Let $\sigma\in\Lip(\RR^m;\clH^m_+)$.
  Suppose that $Q<\Qbar_{\FBSDE}(\sigma)$
  and suppose that $M,\beta\in(0,\infty)$ are such that $J_{\sigma}(q,\anon)\in\quadset(M,\beta)$
  for all $q\in[0,Q]$. Then we have, for all $q\in[0,Q]$ and $b\in\RR^m$, that
  \begin{equation}
    \EE |\Gamma_{b,Q}^{\sigma}(q)|^{2}\le(b^{2}+Mq)\e^{\beta^{2}q}\label{eq:gammasecondmomentbd}
  \end{equation}
  and that
  \begin{equation}
    \EE |\Gamma_{b,Q}^{\sigma}(q)-b|^{2}\le (M + M \e^{\beta^2 q} + \beta^{2}b^{2})q.\label{eq:gammavarbd}
  \end{equation}
\end{lem}

\begin{proof}
  We have by \cref{eq:FBSDE} and Itô's formula that
  \begin{align*}
    \EE |\Gamma_{b,Q}^{\sigma}(q)|^{2} & =b^{2}+\int_{0}^{q}\EE \bigl\lvert J_{\sigma}\bigl(Q-r,\Gamma_{a,Q}^{\sigma}(r)\bigr)\bigr\rvert _{\Frob}^{2}\,\dif r\le b^{2}+Mq+\beta^{2}\int_{0}^{q}\EE |\Gamma_{a,Q}^{\sigma}(r)|^{2}\,\dif r,
  \end{align*}
  so by Grönwall's inequality we obtain \cref{eq:gammasecondmomentbd}.
  Then, again using \cref{eq:FBSDE} and Itô's formula, we obtain
  \begin{align*}
    \EE |\Gamma_{b,Q}^{\sigma}(q)-b|^{2} & =\int_{0}^{q}\EE\bigl\lvert J_{\sigma}\bigl(Q-r,\Gamma_{a,Q}^{\sigma}(r)\bigr)\bigr\rvert _{\Frob}^{2}\,\dif r\le Mq+\beta^{2}\int_{0}^{q}\EE |\Gamma_{a,Q}^{\sigma}(r)|^{2}\,\dif r \\
                                         & \le Mq+\beta^{2}\int_{0}^{q}(b^{2}+Mr)\e^{\beta^{2}r}\,\dif r\le(M+\beta^{2}b^{2} + M \e^{\beta^2 q})q,
  \end{align*}
  which is \cref{eq:gammavarbd}.
\end{proof}
The next proposition is our main time-regularity result for the decoupling function.
\begin{prop}
  \label{prop:Jsigmatimereg-1}For all $M,\beta,\lambda,Q\in(0,\infty)$,
  there is a constant $C=C(M,\beta,\lambda,Q)<\infty$ such that the following
  holds. Suppose $\sigma\in\Lip(\RR^m;\clH^m_+)$ is such that $\Qbar_\FBSDE>Q$ %
  and for all $q \in [0, Q]$, we have $J_{\sigma}(q,\anon)\in\quadset(M,\beta)$ and $\Lip\bigl(J_{\sigma}(q,\anon)\bigr)\le\lambda$.
  Let $Q_{1},Q_{2}\in[0,Q)$ satisfy $|Q_{1}-Q_{2}|\le\beta^{-2}\wedge1$.
  Then for any $q\in[0,Q_{1}\wedge Q_{2}]$ and $a\in\RR^m$, we have
  \begin{equation}
    \clW_{2}\bigl(\Gamma_{a,Q_{1}}^{\sigma}(Q_{1}-q),\Gamma_{a,Q_{2}}^{\sigma}(Q_{2}-q)\bigr)\le C\langle a\rangle|Q_{2}-Q_{1}|^{1/2}\label{eq:changetheendtime}
  \end{equation}
  and
  \begin{equation}
    |J_{\sigma}(Q_{2},a)-J_{\sigma}(Q_{1},a)|_{\Frob}\le C\langle a\rangle|Q_{2}-Q_{1}|^{1/2}.\label{eq:Jtimecontinuity}
  \end{equation}
\end{prop}

\begin{proof}
  Assume without loss of generality that $Q_{2}>Q_{1}$. We recall that
  $\Gamma_{a,Q_{2}}^{\sigma}$ satisfies the SDE
  \begin{align*}
    \dif\Gamma_{a,Q_{2}}^{\sigma}(q) & =J_{\sigma}\bigl(Q_{2}-q,\Gamma_{a,Q_{2}}^{\sigma}(q)\bigr)\dif B(q),\qquad q\in(0,Q_{2}); \\
    \Gamma_{a,Q_{2}}^{\sigma}(0)     & =a.
  \end{align*}
  Let $\Upsilon$ satisfy the SDE
  \begin{align*}
    \dif\Upsilon(q)       & =J_{\sigma}\bigl(Q_{2}-q,\Upsilon(q)\bigr)\dif B(q),\qquad q\in(Q_{2}-Q_{1},Q_{2}); \\
    \Upsilon(Q_{2}-Q_{1}) & =a.
  \end{align*}
  Then we have
  \begin{equation}
    \bigl(\Upsilon(Q_{2}-Q_{1}+q)\bigr)_{q\in[0,Q_{1}]}\overset{\mathrm{law}}{=}\bigl(\Gamma_{a,Q_{1}}^{\sigma}(q)\bigr)_{q\in[0,Q_{1}]}.\label{eq:lawsthesame}
  \end{equation}
  Moreover, for all $q\in[Q_{2}-Q_{1},Q_{2}]$ we have
  \begin{align*}
    \EE |\Gamma_{a,Q_{2}}^{\sigma}(q)-\Upsilon(q)|^{2} & =\EE |\Gamma_{a,Q_{2}}^{\sigma}(Q_{2}-Q_{1})-a|^{2} + \int_{Q_{2}-Q_{1}}^{q}\EE \bigl\lvert J_{\sigma}\bigl(Q_{2}-p,\Gamma_{a,Q_{2}}^{\sigma}(p)\bigr)-J_{\sigma}\bigl(Q_{2}-p,\Upsilon(p)\bigr)\bigr\rvert _{\Frob}^{2}\,\dif p \\
                                                       & \le\EE |\Gamma_{a,Q_{2}}^{\sigma}(Q_{2}-Q_{1})-a|^{2} + \lambda^{2}\int_{Q_{2}-Q_{1}}^{q}\EE |\Gamma_{a,Q_{2}}^{\sigma}(p)-\Upsilon(p)|^{2}\,\dif p.
  \end{align*}
  Hence Grönwall's inequality and \cref{eq:gammavarbd} yield
  \begin{align*}
    \EE |\Gamma_{a,Q_{2}}^{\sigma}(q)-\Upsilon(q)|^{2} & \le \e^{\lambda^{2}(q-(Q_{2}-Q_{1}))}\EE |\Gamma_{a,Q_{2}}^{\sigma}(Q_{2}-Q_{1})-a|^{2}                                            \\
                                                       & \le\e^{\lambda^{2}(q-(Q_{2}-Q_{1}))}\left[M + M \e^{\beta^{2}(Q_{2}-Q_{1})} + \beta^{2}a^{2}\right] (Q_{2}-Q_{1}).
  \end{align*}
  Then \cref{eq:changetheendtime} follows from the equality of laws
  \cref{eq:lawsthesame} (meaning that we have a coupling over which the $\mathcal{W}_2$ metric is an infimum) and our assumption $|Q_{1}-Q_{2}|\le\beta^{-2}\wedge1$.

  To prove \cref{eq:Jtimecontinuity}, we note that, by \cref{prop:matrixvaluedreversetriangleinequality},
  we have
  \begin{align*}
    |J_{\sigma}(Q_{2},a)-J_{\sigma}(Q_{1},a)|_{\Frob}^{2} & \le\clW_{2}\bigl(\sigma\bigl(\Gamma_{a,Q_{2}}^{\sigma}(Q_{2})\bigr),\sigma\bigl(\Gamma_{a,Q_{1}}^{\sigma}(Q_{1})\bigr)\bigr) \\
                                                          & \le\lambda\clW_{2}\bigl(\Gamma_{a,Q_{2}}^{\sigma}(Q_{2}),\Gamma_{a,Q_{1}}^{\sigma}(Q_{1})\bigr),
  \end{align*}
  and then conclude using \cref{eq:changetheendtime} with $q\setto0$.
\end{proof}

\subsubsection{Rescaling the FBSDE solution}

We will sometimes need to scale solutions to the FBSDE.
We record a scaling symmetry in the following proposition. A similar rescaling symmetry holds for the SPDEs \zcref{eq:u-SPDE,eq:SPDE}, except that for those equations the spatial rescaling also changes the mollification scale. The FBSDE, as the limiting problem, has no mollification scale and thus enjoys an exact rescaling symmetry.
\begin{prop}
  Let $\sigma\in\Lip(\RR^m;\clH^m_+)$.
  For any $\zeta>0$,
  we have
  \begin{equation}
    \label{eq:scaleexistencetimes}
    \Qbar_{\FBSDE}(\zeta\sigma)=\zeta^{-2}\Qbar_{\FBSDE}(\sigma)
  \end{equation}
  and, for all $q\in\bigl[0,\zeta^{-2}\Qbar_{\FBSDE}(\sigma)\bigr)$
  and $b\in\RR^m$,
  \begin{equation}
    J_{\zeta\sigma}(q,b)=\zeta J_{\sigma}(\zeta^{2}q,b).\label{eq:Jrescale}
  \end{equation}
\end{prop}

\begin{proof}
  To study $J_{\zeta\sigma}$, we recall its associated FBSDE:
  \begin{align}
    \dif\Gamma_{a,Q}^{\zeta\sigma}(q) & =J_{\zeta\sigma}\bigl(Q-q,\Gamma_{a,Q}^{\zeta\sigma}(q)\bigr)\dif\widetilde{B}(q),\qquad q\in(0,Q];\label{eq:FBSDE-forrescale} \\
    \Gamma_{a,Q}^{\zeta\sigma}(0)     & =a;\label{eq:FBSDEic-forrescale}                                                                                     \\
    J_{\zeta\sigma}(Q,b)              & =\zeta\bigl[\EE \sigma^{2}\bigl(\Gamma_{b,Q}^{\zeta\sigma}(Q)\bigr)\bigr]^{1/2}\label{eq:Jeqn-forrescale}
  \end{align}
  with a standard Brownian motion $\widetilde{B}$. We couple $\widetilde{B}$
  to another standard $\RR^{m}$-valued Brownian motion $B$
  by $\widetilde{B}(q)=\zeta^{-1}B(\zeta^{2}q)$. Let $(\Gamma_{a,Q}^{\sigma})_{a,Q}$
  and $J_{\sigma}$ solve the FBSDE \zcref[range]{eq:FBSDE,eq:Jeqn}
  on $\bigl[0,\Qbar_{\FBSDE}(\sigma)\bigr)$ with noise $\dif B$.
  Then we can solve \zcref[range]{eq:FBSDE-forrescale,eq:Jeqn-forrescale}
  by setting $J_{\zeta\sigma}(Q,b)=\zeta J_{\sigma}(\zeta^{2}Q,b)$
  and $\Gamma_{a,Q}^{\zeta\sigma}(q)=\Gamma_{a,\zeta^{2}Q}^{\sigma}(\zeta^{2}q)$
  for $Q<\zeta^{-2}\Qbar_{\FBSDE}(\sigma)$. Conversely,
  any solution to \zcref[range]{eq:FBSDE-forrescale,eq:Jeqn-forrescale}
  on $\bigl[0,\Qbar_{\FBSDE}(\zeta\sigma)\bigr)$ yields a solution
  to \zcref[range]{eq:FBSDE,eq:Jeqn} on $\bigl[0,\zeta^{-2}\Qbar_{\FBSDE}(\zeta\sigma)\bigr)$.
  From these considerations we conclude \cref{eq:scaleexistencetimes,eq:Jrescale}.
\end{proof}
\subsubsection{An inductive lemma and zeros of the nonlinearity}

The renormalization flow preserves certain properties of $\sigma$.
A key example is that if $\sigma(b)=0$, then $J(q,b)=0$ for all $q$, as shown in \cref{prop:zerosstayzero} below.
We first prove the following more abstract statement that will be useful in other contexts.
\begin{prop}
  \label{prop:preserve}
  Let $A$ be a closed subset of $\clX$ such that for all ${\sigma\in\Lip(\RR^m;\clH^m_+)\cap A}$, $Q>0$, and $g\in\clA_Q$ such that $g(q,\anon)\in A$ for each $q\in [0,Q]$, we have $\clQ_\sigma g(q,\cdot)\in A$ for each $q\in [0,Q]$ as well.
  Then, for any $\sigma\in\Lip(\RR^m;\clH^m_+) \cap A$, we have $J_\sigma(q,\cdot)\in A$ for all $q\in \bigl[0,\overline Q_\FBSDE(\sigma)\bigr)$.
\end{prop}
\begin{proof}
  Suppose for the sake of contradiction that there exists some $Q\in \bigl[0,\overline Q_\FBSDE(\sigma)\bigr)$ such that $J_\sigma(Q,\cdot)\not\in A$.
  By \cref{eq:Jtimecontinuity}, $q\mapsto J_\sigma(q,\cdot)$ is continuous as a map $[0,Q]\to\clX$.
  Since $A \subset \clX$ is closed and $J_\sigma(0, \anon) = \sigma \in A$, the set $\{q\in [0,Q]\suchthat J_\sigma(q,\cdot)\in A\}$ is closed and nonempty.
  It thus has a maximum $q_0$, which is strictly less than $Q$ by hypothesis.

  By \cref{cor:chainJs-nohyp}, for all $q\in \bigl[q_0,\overline Q_\FBSDE(\sigma)\bigr)$ and $b\in\RR$, we have
  \begin{equation}
    \label{eq:J-semigroup}
    J_\sigma(q,b) = J_{J_\sigma(q_0,\cdot)}(q-q_0,b).
  \end{equation}
  Fix $p\in (0,\Lip[J_\sigma(q_0,\cdot)]^{-2})$ and view $J_\sigma(q_0,\cdot)$ as a $q$-independent function in $\m{A}_p$.
  By \cref{eq:nfoldlimit}, we have
  \begin{equation}
    \label{eq:applynfoldlimit}
    \lim_{n\to\infty} \|(\clQ^n_{J_\sigma(q_0,\cdot)}[J_\sigma(q_0,\cdot)]-J_{J_\sigma(q_0,\cdot)})(p,\cdot)\|_\clX =0.
  \end{equation}
  Inducting on our hypothesis for $A$, $\clQ^n_{J_\sigma(q_0,\cdot)}[J_\sigma(q_0,\cdot)](q,\cdot)\in A$ for each $n \in \N$ and $q\in [0,p]$.
  Since $A$ is closed, \cref{eq:applynfoldlimit} yields $J_{J_\sigma(q_0,\cdot)}(p,\cdot)\in A$ and hence $J_\sigma(q_0+p,\cdot)\in A$ by \cref{eq:J-semigroup}.
  This contradicts the maximality of $q_0$.
\end{proof}

We can use the previous proposition to prove that the decoupling flow preserves the zeros of the nonlinearity.
\begin{prop}\label{prop:zerosstayzero}%
    Let $\sigma\in\Lip(\RR^m;\clH^m_+)$. For each $Q\in \bigl[0,\Qbar_\FBSDE(\sigma)\bigr)$, we have \[\{b\in\RR^m\suchthat J_\sigma(Q,b)=0\} = \{b\in\RR^m\suchthat\sigma(b)=0\}.\]
\end{prop}

\begin{proof}
  First suppose $\sigma(b) = 0$ and let $A_b \coloneqq \{\chi \in \m{X} \suchthat \chi(b) = 0\}$, which is closed.
  Given $Q > 0$, suppose $g \in \m{A}_Q$ and $g(q, \anon) \in A_b$ for all $q \in [0, Q]$, so $g(q,b) = 0$.
  From the SDE \zcref[range]{eq:thetaSDE,eq:thetaIC}, we see that $\Theta_{b,q}^{g}(q)=b$ almost surely.
  Then if $\chi \in A_b$, \cref{eq:Qdef} yields $\m{Q}_\chi g(q,b) = 0$.
  That is, $\m{Q}_\chi g(q, \anon) \in A_b$ for each $q \in [0, Q]$.
  Hence by \cref{prop:preserve}, $J_\sigma(q,b) = 0$ for all $q\in\bigl[0,\overline Q_\FBSDE(\sigma)\bigr)$.

  Now consider $b\in\RR^m$ such that $\sigma(b)\ne 0$.
  By continuity, there exists $\eps>0$ such that $\sigma(a)\ne 0$ whenever $|a-b|<\eps$. For any $Q>0$, we have $\PP(|\Gamma^\sigma_{b,Q}(Q)-b|<\eps)>0$, so $\PP\bigl[\sigma^2\bigl(\Gamma^\sigma_{b,Q}(Q)\bigr)\ne 0\bigr]>0$. Since $\sigma^2(b)$ is nonnegative-definite for all $b$, \cref{eq:Jeqn} implies that $J_\sigma(Q,b)\ne 0$.
\end{proof}

\subsubsection{Coupling equations with the same noise}

As an important special case, we consider several equations driven with the same noise but different nonlinearities or initial conditions.
This will prove useful in our study of stability and universality in \cref{sec:universal}, and is thereby a major motivation for our consideration of vector-valued solutions.

Given $j\in\NN$ and $\sigma_1,\ldots,\sigma_j\in\Lip(\RR^m;\clH^m_+)$, we define the block matrix
\begin{equation}
  \label{eq:Sigma0def}
  \hat\sigma_0(b_1,\ldots,b_j)\coloneqq
  \begin{pmatrix}
    \sigma_1(b_1)&0&\cdots&0\\
    \sigma_2(b_2)&0&\cdots&0\\
    \vdots&\vdots&\ddots&\vdots\\
    \sigma_j(b_j)&0&\cdots&0
  \end{pmatrix}
\end{equation}
and then set
\begin{equation*}
  \hat\sigma \coloneqq (\hat\sigma_0\hat\sigma_0^{\mathrm{T}})^{1/2}.
\end{equation*}
By \cref{eq:SPDE}, the former corresponds to $j$ equations with nonlinearities $\sigma_1,\ldots,\sigma_j$ all driven by the same white noise.
Then $\sigma$ is a symmetrization yielding the same law for $v$.
As noted in the introduction, $\hat\sigma$ is Lipschitz by \cite[Theorem VII.5.7]{Bha97}.
We wish to show that the FBSDE for $\hat\sigma$ remains well-posed as long as those for $\sigma_1,\ldots,\sigma_j$.
More precisely, we have the following two propositions:
\begin{prop}\label{prop:multigrowth}
  If $\sigma_i\in \quadset(M_i,\beta_i)$ for each $i \in \{1,\ldots,j\}$, then $\hat\sigma\in\quadset\left(\sum\limits_{i=1}^j M_i,\max\limits_{i\in\{1,\ldots,j\}}\beta_i\right)$.
\end{prop}
\begin{prop}\label{prop:multivarwellposedforsametime}
  We have $\overline{Q}_{\FBSDE}(\hat\sigma) = \min\limits_{i\in\{1,\ldots,j\}}\overline{Q}_{\FBSDE}(\sigma_i)$.
\end{prop}
The first is quite simple:
\begin{proof}[Proof of \cref{prop:multigrowth}]
  For any $b_1,\ldots,b_j\in\RR^m$, we have
  \begin{equation*}
    \lvert\hat\sigma(b_1,\ldots,b_j)\rvert_\Frob^2 = \tr \h \sigma^2(b_1,\ldots,b_j) = \sum_{i=1}^j \tr \sigma_i^2(b_i) \le \sum_{i=1}^j (M_i + \beta_i^2\lvert b_i\rvert^2) \le \sum_{i=1}^j M_i + \left(\max_{i\in \{1,\ldots,j\}}\beta_i\right) |b|^2.\qedhere
  \end{equation*}
\end{proof}
The proof of \cref{prop:multivarwellposedforsametime} is somewhat more involved.
Given $i\in\{1,\ldots,j\}$, let $\pi_i\colon\RR^{jm}\to\RR^m$ denote the projection onto the $i$th block.
We begin with the following lemma.
\begin{lem}
  For any $q\in\bigl[0,\overline{Q}_\FBSDE(\hat\sigma)\bigr)$, $a=(a_1,\ldots,a_j)\in \RR^{jm}$, and $i\in\{1,\ldots,j\}$, we have
  \begin{equation}\label{eq:Jprojection}
    \pi_i J_{\hat\sigma}(q,a)^2\pi_i^{\mathrm{T}} = J_{\sigma_i}(q,a_i)^2.
  \end{equation}
  Moreover,
  \begin{equation}\label{eq:QFBSDEeasydirection}
    \overline Q_\FBSDE(\hat\sigma)\le \overline Q_\FBSDE(\sigma_i),
  \end{equation}
  and for any $Q\in \bigl[0,\overline{Q}_\FBSDE(\hat\sigma)\bigr)$, we have
  \begin{equation}\label{eq:piiGamma}
    \bigl(\pi_i \Gamma^{\hat\sigma}_{a,Q}(q)\bigr)_{q\in [0,Q],a\in\RR^{mj}}\overset{\mathrm{law}} = \bigl(\Gamma^{\sigma_i}_{\pi_i a,Q}(q)\bigr)_{q\in [0,Q],a\in \RR^{jm}}.
  \end{equation}
\end{lem}
\begin{proof}
  Let $A$ be the set of $\chi\in\clX$ such that $b\mapsto\pi_i\chi^2(b)\pi_i^{\mathrm{T}}$ depends only on $\pi_ib$. It is straightforward to see that $A$ is a closed subset of $\clX$.
  If $g \in \m{A}_Q$ and $g(q,\cdot)\in A$ for all $q\in [0,Q]$, then there exists $g_i \in \m{A}_Q$ such that \begin{equation}\label{eq:pig}\pi_i g(q,a)^2\pi_i^{\mathrm{T}} = g_i(q,\pi_i a)^2.\end{equation}
  Applying Itô's formula to \cref{eq:thetaSDE}, we see that $\pi_i\Theta^g_{a,Q}$ satisfies the SDE
  \begin{equation}
    \label{eq:block-projected}
    \begin{aligned}
      \dif \pi_i\Theta^g_{a,Q}(q) &= \pi_i g\bigl(Q-q,\Theta^g_{a,Q}(q)\bigr) \ds B(q),\qquad q\in[0,Q];\\
      \pi_i \Theta^g_{a,Q}(0) &= a_i.
    \end{aligned}
  \end{equation}
  By \cref{eq:pig}, this SDE corresponds to the same martingale problem as $\Theta^{g_i}_{a_i,Q}$, so
  \begin{equation}
    \bigl(\pi_i\Theta^{g}_{a,Q}(q)\bigr)_{q\in [0,Q]}\overset{\mathrm{law}}=\bigl(\Theta^{g_i}_{a_i,Q}(q)\bigr)_{q\in [0,Q]}.\label{eq:thetaseqlaw}
  \end{equation}
  If $\chi\in A$, there is a function $\chi_i \in \m{X}$ such that
  \begin{equation}
    \pi_i\chi^2(a)\pi_i^{\mathrm{T}} = \chi_i(\pi_i a)^2 \ForAll a\in\RR^{jm}.
    \label{eq:pichieqn}
  \end{equation}
  We therefore have
  \begin{equation*}
    \pi_i \clQ_\chi g(q,a)^2 \pi_i^{\mathrm{T}} \overset{\cref{eq:Qdef}}= \EE\bigl[\pi_i \chi^2\bigl(\Theta^g_{a,q}(q)\bigr)\pi_i^{\mathrm{T}}\bigr]\overset{\cref{eq:pichieqn}}=\EE\bigl[\chi_i^2\bigl(\pi_i \Theta^g_{a,q}(q)\bigr)\bigr] \overset{\cref{eq:thetaseqlaw}}= \EE\bigl[\chi_i^2\bigl(\Theta^{g_i}_{\pi_i a,q}(q)\bigr)\bigr],
  \end{equation*}
  which indeed only depends on $\pi_i a$.
  Hence $\clQ_\chi g$ takes values in $A$ as well.

  It is easy to check that $\hat\sigma\in A$, so by \cref{prop:preserve}, $J_{\hat\sigma}(q,\cdot)\in A$ for all $q\in \bigl[0,\overline Q_\FBSDE(\hat\sigma)\bigr)$.
  Hence there is some $h_i$ such that, for all $a\in\RR^{jm}$ and $q\in\overline Q_\FBSDE(\hat\sigma)$,
  \begin{equation*}
    h_i(q,\pi_i a)^2 = \pi_i J_{\hat\sigma}(q,a)^2\pi_i^{\mathrm{T}} \overset{\cref{eq:root decouplingholds}}= \EE\bigl[\pi_i \hat\sigma^2\bigl(\Theta^{J_{\hat\sigma}}_{a,q}(q)\bigr)\pi_i^{\mathrm{T}}\bigr] = \EE\bigl[\sigma^2_i\bigl(\pi_i \Theta^{J_{\hat\sigma}}_{a,q}(q)\bigr)\bigr] \overset{\cref{eq:thetaseqlaw}}= \EE\bigl[\sigma^2_i\bigl(\Theta^{g_i}_{\pi_i a,q}(q)\bigr)\bigr].
  \end{equation*}
  Hence $h_i$ satisfies \cref{eq:root decouplingholds}, so $\overline Q_\FBSDE(\sigma_i)\ge \overline Q_\FBSDE(\hat\sigma)$ and $h_i = J_\sigma$ on $\bigl[0,\overline Q_\FBSDE(\hat\sigma)\bigr)$.
  Thus we have proved \cref{eq:Jprojection,eq:QFBSDEeasydirection}.
  Then \cref{eq:piiGamma} follows from \cref{eq:thetaseqlaw,eq:GammaisTheta}.
\end{proof}
Now we can prove \cref{prop:multivarwellposedforsametime}.
\begin{proof}[Proof of \cref{prop:multivarwellposedforsametime}]
  The ``$\le$'' direction has already been proved as \cref{eq:QFBSDEeasydirection}, so we only need to prove the ``$\ge$'' direction.
  Suppose for the sake of contradiction that $\overline Q_{\FBSDE}(\hat\sigma)<\min_i \overline Q_{\FBSDE}(\sigma_i)$, and take $q \in \bigl[0, \overline Q_\FBSDE(\hat\sigma)\bigr)$.
  Then we can compute
  \begin{align}
    \notag|J_{\hat\sigma}(q,b_2)-J_{\hat\sigma}(q,b_1)|^2_{\Frob}
    \notag \overset{\cref{eq:Jeqn}}&= \Bigl|\bigl[\EE\hat\sigma^2\bigl(\Gamma^{\hat\sigma}_{b_2,q}(q)\bigr)\bigr]^{1/2}-\bigl[\EE\hat\sigma^2\bigl(\Gamma^{\hat\sigma}_{b_1,q}(q)\bigr)\bigr]^{1/2}\Bigr|^2_\Frob\\
    \notag \overset{\cref{eq:matrixvaluedreversetriangleinequality}}&\le \EE\bigl|\hat\sigma\bigl(\Gamma^{\hat\sigma}_{b_2,q}(q)\bigr)-\hat\sigma\bigl(\Gamma^{\hat\sigma}_{b_1,q}(q)\bigr)\bigr|^2_\Frob\\
    \notag&\le \Lip(\hat{\sigma})^2\EE\bigl|\Gamma^{\hat\sigma}_{b_2,q}(q)-\Gamma^{\hat\sigma}_{b_1,q}(q)\bigl|^2\\
    \notag &= \Lip(\hat{\sigma})^2\sum_{i=1}^j\EE\bigl|\pi_i\Gamma^{\hat\sigma}_{b_2,q}(q)-\pi_i\Gamma^{\hat\sigma}_{b_1,q}(q)\bigr|^2\\
    \notag \overset{\cref{eq:piiGamma}}&= \Lip(\hat{\sigma})^2\sum_{i=1}^j\EE\bigl|\Gamma^{\sigma_i}_{\pi_i b_2,q}(q)-\Gamma^{\sigma_i}_{\pi_i b_1,q}(q)\bigr|^2 \notag \overset{\cref{eq:SDEctsininitialconditions}} \leq C |b_2 -b_1|^2,
  \end{align}
  where $C<\infty$ is a constant that may depend on $\sigma_1,\ldots,\sigma_j$ but is uniform over all $q<\overline Q_\FBSDE(\hat\sigma)$, since $\overline Q_\FBSDE(\hat\sigma) < \min_i\overline Q_\FBSDE(\sigma)$ by assumption.
  Taking $q$ sufficiently close to $\overline Q_\FBSDE(\hat\sigma)$, we obtain a contradiction to \cref{thm:FBSDEwellposed}.
\end{proof}
We now specialize to the case in which all $\sigma_i$ and $b_i$ are identical.
Let $D \subset \R^{jm}$ denote the diagonal subspace consisting of vectors $b = (b_0,\ldots,b_0)$ for some $b_0 \in \R^m$.
\begin{lem}
  \label{lem:diagonal}
  Let $\sigma \in \Lip(\R^m;\m{H}_+^m)$ and take $\sigma_1 = \cdots = \sigma_j = \sigma$ in \cref{eq:Sigma0def}.
  For all $q \in \big[0, \Qbar_\FBSDE(\sigma)\big)$ and $b \in D$, $J_{\h \sigma}(q, b) \colon \R^{jm} \to D$.
\end{lem}
\noindent
In words, $J_{\h \sigma}$ sends the diagonal to matrices that project to the diagonal.
\begin{proof}
  The set $A \subset \m{X}$ of matrices sending the diagonal to diagonal projectors is closed.
  Suppose $g \in \m{A}_Q$ satisfies $g(q, \anon) \in A$ for all $q \in [0, Q]$.
  Then $g|_D \colon \R^{jm} \to D$, so if $a \in D$, each component $\pi_i \Theta_{a,Q}^g$ in \cref{eq:block-projected} experiences the same noise and begins from the same data.
  It follows that $\Theta_{a,Q}^g \in D$.
  The diagonal projectors are convex in the space of matrices, so if $\sigma \in A$, then $\m{Q}_\sigma g(q, \anon) \in A$ for all $q \in [0, Q]$.
  Moreover, $\h \sigma \in A$.
  Thus the lemma follows from \cref{prop:preserve,prop:multivarwellposedforsametime}.
\end{proof}
In fact one can say more: by \cref{eq:Jprojection} from \cref{prop:multivarwellposedforsametime}, $J_{\h \sigma}|_D$ is composed of $j^2$ identical blocks, each given by $J_\sigma$.
We will not require this degree of precision.

\section{The PDE for the decoupling function\label{sec:PDE}}

In this section, we relate the decoupling function $J^2_{\sigma}$ to the parabolic PDE \zcref[range]{eq:HPDE-1,eq:Hic-1}.
For $m=1$, we study this PDE in the companion paper \cite{DG25a}.
At the end of this section, we show how \cref{prop:satisfiesPDE} and the results of \cite{DG25a} can be combined to prove \cref{thm:PDEthm}.

As noted in the introduction, the degeneracy of \zcref[range]{eq:HPDE-1,eq:Hic-1} permits solutions that are not $\m{C}^2$.
We therefore adopt a slightly more general notion of solution.
\begin{defn}\label{defn:almost-classical}
  An \emph{almost classical solution} to \zcref[range]{eq:HPDE-1,eq:Hic-1} on a time interval $[0, Q_0)$ is a continuous function $H\colon[0,Q_0)\times\RR^{m}\to\clH_{+}^{m}$
  such that the following conditions hold:
  \begin{enumerate}[label = (\roman*)]
\item For every compact $\mathcal{K} \subset(0,Q_0)$ and bounded open $U\subset\R^m$,
  \begin{equation*}
      H|_{\mathcal{K}\times U}\in L^1\bigl(\mathcal{K};W^{2,\infty}(U;\clH_{+}^{m})\bigr)\cap\clC^{1}(\mathcal{K}\times U;\clH_{+}^{m}).
  \end{equation*}
  Here, $W^{2,\infty}(U;\clH_+^m)$ is the Sobolev space of functions on $U$ taking values in $\clH_+^m$ with weak second derivative measurable and essentially bounded.
\item We have $H(0,b)=\sigma^{2}(b)$ for all $b\in\RR^{m}$.
\item We have \begin{equation}
    \partial_{q}H(q,b)=\frac{1}{2}[H(q,b):\nabla_{b}^{2}]H(q,b)\qquad\text{for almost all } (q, b) \in(0,Q_{0}) \times \RR^{m}.\label{eq:derivalmosteverywhere}
  \end{equation}
\end{enumerate}
\end{defn}
We show that an almost classical solution satisfying suitable growth bounds coincides with the decoupling function.
In the following, let $U_R = \{b\suchthat\abs{b} < R\}$ denote the open ball in $\R^m$ of radius $R$ centered at the origin.
\nomenclature[UR]{$U_R$}{open ball of radius $R$ in $\R^m$ centered at the origin}
\begin{prop}
  \label{prop:satisfiesPDE}Fix $Q_{0}\in\bigl(0,\Qbar_{\FBSDE}(\sigma)\bigr)$.
  Suppose $H$ is an almost classical solution of \textup{\zcref[range]{eq:HPDE-1,eq:Hic-1}} on $[0,Q_0]$ such that $\sup_{q\in[0,Q_0]}\Lip\sqrt{H(q,\anon)}<\infty$.\vspace{2pt}
  Also assume that for each compact $\mathcal{K}\subset(0,Q_0)$, there exists a constant $C(\mathcal{K})<\infty$ such that for all $R > 0,$
  \begin{equation}
    \bigl\|H|_{\m K \times U_R}\bigr\|_{L^1(\mathcal{K};W^{2,\infty}(U_R;\clH_{+}^{m}))\cap\clC^{1}(\mathcal{K}\times U_R;\clH_{+}^{m})}\le CR^C.\label{eq:polybd}
  \end{equation}
  Then $\sqrt{H}=J_{\sigma}$ on $[0, Q_0] \times \R^m$.
\end{prop}
\begin{proof}
    Recalling the discussion following \cref{def:decouplingfn}, it suffices to
show that
\begin{equation}
    H(Q,a)=\EE \sigma^{2}\bigl(\Theta_{a,Q}^{\sqrt{H}}(Q)\bigr)\qquad\text{for any }a\in\RR^{m}\text{ and }Q\in[0,Q_{0}].\label{eq:solvesPDEgoal}
\end{equation}
We wish to apply Itô's formula, but since the solution $H$ is not guaranteed to be smooth, we begin with a mollification.
Fix a mollifier $\eta\in\mathcal{C}_{\cc}^{\infty}(\RR^{m})$
with $\int\eta=1$, $\eta \geq 0$, and $\supp\eta\subseteq[-1,1]$.
Given $\eps > 0$, define $\eta^{\eps}(b)=\eps^{-m}\eta(\eps^{-1}b)$ and
\begin{equation*}
  H^{\eps}(q,a)=\int\eta^{\eps}(a-b)H(q,b)\,\dif b.
\end{equation*}
We also let $H^{0}=H$.
Then we can write
\begin{align*}
\frac{1}{2} & [H(q,a):\nabla_{b}^{2}]H^{\eps}(q,a)=\frac{1}{2}\int\eta^{\eps}(a-b)[H(q,a):\nabla_{b}^{2}]H(q,b)\,\dif b\\
 & =\frac{1}{2}\int \eta^{\eps}(a-b)[H(q,b):\nabla_{b}^{2}] H(q,b)\,\dif b+\frac{1}{2}\int\eta^{\eps}(a-b)\left([H(q,a)-H(q,b)]:\nabla_{b}^{2}\right)H(q,b)\,\dif b.
\end{align*}
We can develop the first integral using \cref{eq:derivalmosteverywhere}
as
\[
\frac{1}{2}\int\eta^{\eps}(a-b)[H(q,b):\nabla_{b}^{2}]H(q,b)\,\dif b=\frac{1}{2}\int\eta^{\eps}(a-b)\partial_{q}H(q,b)\,\dif b=\partial_{q}H^{\eps}(q,a)
\]
for almost every $q\in(0,Q_{0})$.
Thus
\begin{equation}
\frac{1}{2}(H(q,a):\nabla_{b}^{2})H^{\eps}(q,a)=\partial_{q}H^{\eps}(q,a)+\frac{1}{2}\int\eta^{\eps}(a-b)\left([H(q,a)-H(q,b)]:\nabla_{b}^{2}\right)H(q,b)\,\dif b.\label{eq:NLH-Heps}
\end{equation}

Now we fix $a\in\RR^{m}$ and $Q\in[0,Q_{0}]$ and define,
for $\eps\ge0$ and $q\in[0,Q]$,
\begin{equation}
    N_{q}^{\eps}=H^{\eps}\bigl(Q-q,\Theta_{a,Q}^{\sqrt{H}}(q)\bigr).\label{eq:Nqdef}
\end{equation}
When $\eps>0$, the function $H^{\eps}$ is smooth in space and continuously
differentiable at positive times. Therefore, we can use Itô's formula
(see, e.g.,\ \cite[Theorem~5.10]{CW14} for a statement at this regularity)
to compute $N_P^\eps - N_0^\eps$ for $\eps>0$ and $P\in(0,Q)$:
\begin{align}
    N_{P}^{\eps}&-N_{0}^{\eps}  =-\int_{0}^{P}\partial_{q}H^{\eps}\bigl(Q-q,\Theta_{a,Q}^{\sqrt{H}}(q)\bigr)\,\dif q+\frac{1}{2}\int_{0}^{P}\nabla_{b}^{2}H^{\eps}\bigl(Q-q,\Theta_{a,Q}^{\sqrt{H}}(q)\bigr):\dif[\Theta_{a,Q}^{\sqrt{H}}]_{q}+R_{P}^{\eps}\nonumber \\
 & =\int_{0}^{P}\left[-\partial_{q}H^{\eps}\bigl(Q-q,\Theta_{a,Q}^{\sqrt{H}}(q)\bigr)+ \frac{1}{2}\left[H\bigl(Q-q,\Theta_{a,Q}^{\sqrt{H}}(q)\bigr):\nabla_{q}^{2}\right]H^{\eps}\bigl(Q-q,\Theta_{a,Q}^{\sqrt{H}}(q)\bigr)\right]\,\dif q+R_{P}^{\eps}\nonumber \\
  \overset{\cref{eq:NLH-Heps}}&{=}\frac{1}{2}\int_{0}^{P}\!\!\!\!\int\eta^{\eps}\bigl(\Theta_{a,Q}^{\sqrt{H}}(q)-b\bigr)\left(\left[H\bigl(Q-q,\Theta_{a,Q}^{\sqrt{H}}(q)\bigr)-H(Q-q,b)\right]:\nabla_{b}^{2}\right)H(Q-q,b)\,\dif b\,\dif q+R_{P}^{\eps},\label{eq:apply-ito}
\end{align}
where we have defined
\begin{align}
R_{p}^{\eps} & =\int_{0}^{p}\partial_{b}H^{\eps}\bigl(Q-q,\Theta_{a,Q}^{\sqrt{H}}(q)\bigr)\,\dif\Theta_{a,Q}^{\sqrt{H}}(q)\nonumber \\
              \overset{\cref{eq:thetaSDE}}&=\int_{0}^{p} \partial_{b}H^{\eps}\bigl(Q-q,\Theta_{a,Q}^{\sqrt{H}}(q)\bigr)H\bigl(Q-q,\Theta_{a,Q}^{\sqrt{H}}(q)\bigr)^{1/2}\,\dif B(q).\label{eq:Repsdef}
\end{align}
It is clear from the definition that $(R_{p}^{\eps})_{p}$ is a local
martingale. We want to show that it is in fact a martingale. From
\cref{eq:Repsdef}, we can bound the trace of the expected quadratic variation
at time $P$ by
\begin{align*}
\tr\EE [R^{\eps}]_{P} & =\int_{0}^{P}\EE \left\lvert  \partial_{b}H^{\eps}\bigl(Q-q,\Theta_{a,Q}^{\sqrt{H}}(q)\bigr)H\bigl(Q-q,\Theta_{a,Q}^{\sqrt{H}}(q)\bigr)^{1/2}\right\rvert  _{\mathrm{F}}^{2}\,\dif q,
\end{align*}
and the right side is finite by \cref{eq:polybd,eq:finite-moment}. %
This implies that $(R_{p}^{\eps})_{p\in[0,P]}$ is in fact a martingale (see, e.g.,\ \cite[IV.1.23]{RY99}).
In particular, $\EE (R_{P}^{\eps})=0$,
so taking expectations in \cref{eq:apply-ito}, we obtain
\begin{align*}
| & \EE N_{P}^{\eps}-N_{0}^{\eps}|\\
 & \le\frac{1}{2}\int_{0}^{P}\!\!\!\!\int\EE \left\lvert  \eta^{\eps}\bigl(\Theta_{a,Q}^{\sqrt{H}}(q)-b\bigr)\left(\left[H\bigl(Q-q,\Theta_{a,Q}^{\sqrt{H}}(q)\bigr)-H(Q-q,b)\right]:\nabla_{b}^{2}\right)H(Q-q,b)\right\rvert  \,\dif b\,\dif q.
\end{align*}
The integral on the right side goes to $0$ as $\eps\searrow0$ by \cref{eq:polybd}.
Therefore, we can take $\eps\searrow0$ to see that
\[
    H(Q,a)\overset{\cref{eq:Nqdef}}=N_{0}^{0}=\EE (N_{P}^{0})\overset{\cref{eq:Nqdef}}=\EE \bigl[H\bigl(Q-P,\Theta_{a,Q}^{\sqrt{H}}(P)\bigr)\bigr].
\]
We also have 
\[
\lim_{P\nearrow Q}\EE \bigl[H\bigl(Q-P,\Theta_{a,Q}^{\sqrt{H}}(P)\bigr)\bigr]=\EE \bigl[H\bigl(0,\Theta_{a,Q}^{\sqrt{H}}(Q)\bigr)\bigr]=\EE \bigl[\sigma^{2}\bigl(\Theta_{a,Q}^{\sqrt{H}}(Q)\bigr)\bigr]
\]
by the continuity of $H$ and \cref{eq:Hic-1}.
Together, the last two displays show that
\[
H(Q,a)=\EE \bigl[\sigma^{2}\bigl(\Theta_{a,Q}^{\sqrt{H}}(Q)\bigr)\bigr],
\]
completing the proof of \cref{eq:solvesPDEgoal}.
\end{proof}

We conclude by showing that \cref{prop:satisfiesPDE} and the results of \cite{DG25a} imply \cref{thm:PDEthm}.
\begin{proof}[Proof of \cref{thm:PDEthm}]
  The hypothesis of the theorem ensures that $\sigma^2$ satisfies Hypothesis~\textup{H}($\beta^2$,$\gamma$) of \cite{DG25a}.
  Therefore, by Theorem~1.2 of \cite{DG25a}, \zcref[range]{eq:HPDE-1,eq:Hic-1} admits a \emph{strong solution} $H : [0,\beta^{-2})\times \RR\to\RR_{\ge 0}$ in the sense of \cite[Definition~1.2]{DG25a}.
  Moreover, for each $Q \in [0, \beta^{-2})$ and compact $\m{K} \subset (0, \beta^{-2})$, we have
  \begin{equation}
    \label{eq:lipbdd}
    \sup_{q\in[0,Q]}\Lip\sqrt{H(q,\anon)}<\infty \quad \text{and} \quad \norm{\partial_b^2H}_{L^\infty(\m{K} \times \R)} <\infty.
  \end{equation}
  In combination with \cite[Proposition~2.10]{DG25a}, these facts imply that $H$ is an almost classical solution in the sense of \cref{defn:almost-classical}.
  Therefore the hypotheses of \cref{prop:satisfiesPDE} are satisfied for any $Q_0<\overline{Q}_\FBSDE(\sigma)\wedge \beta^{-2}$, so
  \begin{equation}
    \label{eq:JissqrtH}
    J_\sigma(q,b) = \sqrt{H(q,b)}\qquad\text{for all $b\in\RR$ and $q\in [0,\overline{Q}_\FBSDE(\sigma)\wedge\beta^{-2})$}.
  \end{equation}

  Now suppose for the sake of contradiction that $\overline{Q}_\FBSDE(\sigma)<\beta^{-2}$.
  Then we can take $Q\leftsquigarrow\overline{Q}_\FBSDE(\sigma)$ in \cref{eq:lipbdd} and use \cref{eq:JissqrtH} to see that $\lambda \coloneqq \liminf\limits_{q\nearrow \overline{Q}_\FBSDE(\sigma)}J_\sigma(q,\anon)<\infty$.
  Then, by \cref{thm:FBSDEwellposed}, we have $\overline{Q}_\FBSDE(\sigma)\ge\overline{Q}_\FBSDE(\sigma)+\lambda^{-2}>\overline{Q}_\FBSDE(\sigma)$, a contradiction.
\end{proof}

\section{Basic properties of the SPDE}\label{sec:basic-properties}
In this section we establish some preliminary facts about solutions to the SPDE \cref{eq:SPDE}.
In \cref{subsec:exp-scale}, we define some parameters to keep track of the exponential scales inherent to the problem.
With this framework in place, we show in \cref{subsec:A-fixed-point-argument} that on short time scales, solutions to \cref{eq:SPDE} can be characterized as fixed points of a certain operator.
We prove moment estimates on the SPDE solution in \cref{subsec:moment-estimates}.
Finally, in \cref{subsec:regularitygaussianaverages} we show regularity results for Gaussian averages of the solutions.

None of the results in this section rely on bounds on the Lipschitz constant of $\sigma$.
This contrasts with \cref{sec:local-SPDE-estimates}, in which the Lipschitz constant plays a pivotal role.

\subsection{Exponential scale parameters}\label{subsec:exp-scale}

As noted in \cref{subsec:MGexptimescale},
an intrinsic feature of the SPDE \cref{eq:SPDE} is the ``exponential timescale'' $t \approx T-\rho^{q}$ for some fixed ``final time'' $T$.
In this subsection, we introduce notation adapted to this scale.
We recall the notation $\sfL(\tau)=\log(1+\tau)$ from \cref{eq:Ldef}, and note for later use that
\begin{equation}
  \sfL(\tau_{1}\tau_{2})=\log(1+\tau_{1}\tau_{2})\le\log(1+\tau_{1})+\log(1+\tau_{2})=\sfL(\tau_{1})+\sfL(\tau_{2})\qquad\text{for all }\tau_{1},\tau_{2}\ge0.\label{eq:Lmult}
\end{equation}
Given $\rho > 0$ and $\tau \geq 0$, define
\nomenclature[L z rho]{$\sfL_\rho(\tau)$}{$\sfL(\tau)/\sfL(1/\rho)\approx \log_{1/\rho}\tau$, \cref{eq:Lrhodef}}
\begin{equation}
  \sfL_{\rho}(\tau)\coloneqq\frac{\sfL(\tau)}{\sfL(1/\rho)}\label{eq:Lrhodef}
\end{equation}
and
\nomenclature[S rho]{$\sfS_\rho(\tau)$}{$\sfL_\rho(\tau/\rho)\approx1-\log_\rho\tau$, \cref{eq:Srhodef}}
\begin{equation}
  \sfS_{\rho}(\tau)\coloneqq\sfL_{\rho}(\tau/\rho)=\frac{\log(\tau/\rho+1)}{\log(1/\rho+1)}=\frac{1}{\sfL(1/\rho)}\int_{0}^{\tau}\frac{\dif r}{\tau+\rho-r}.\label{eq:Srhodef}
\end{equation}
We note that when $\rho \ll \tau \wedge 1$, we have
\begin{equation}
  \sfS_{\rho}(\tau)\approx1-\log_{\rho}\tau.\label{eq:Sapprox}
\end{equation}
Given $q\ge0$, we also define
\nomenclature[T rho]{$\sfT_\rho(q)$}{$\sfS^{-1}_\rho(q)=\rho[(1+1/\rho)^q-1]\approx\rho^{1-q}$, \cref{eq:Tdef}}
\begin{equation}
  \sfT_{\rho}(q)=\sfS_{\rho}^{-1}(q)=\rho[(1+1/\rho)^{q}-1]\label{eq:Tdef}
\end{equation}
and observe that for $\rho\ll1$ and $q > 0$ of order $1$, we have
\begin{equation}
  \sfT_{\rho}(q)\approx\rho^{1-q}.\label{eq:Tapprox}
\end{equation}
We note for future use that, whenever $q_{1}\le q_{2}$, we have
\begin{align}
  \sfL_{\rho}\left(\frac{\sfT_{\rho}(q_{2})-\sfT_{\rho}(q_{1})}{\sfT_{\rho}(q_{1})+\rho}\right) & =q_{2}-q_{1}.\label{eq:Lbdq1q2}
\end{align}

\subsection{A fixed-point argument for the SPDE on short time scales\label{subsec:A-fixed-point-argument}}

We do not present a full well-posedness theory for \cref{eq:SPDE}
because it follows from rather standard techniques. We do, however,
provide a fixed-point theory on short time scales, since we will need
the explicit estimate later.
In the following, let $v=(v_{r}(x))_{r\ge s,x\in\RR^{2}}$ be a random field adapted
to the filtration $\{\scrF_{t}\}_{t}$.
Given $\sigma\in\Lip(\RR^m;\clH^m_+)$ %
and $s\in\RR$, define the operator $\clT_{s}^{\sigma,\rho}$ by
\nomenclature[T cl]{$\clT^{\sigma,\rho}_s$}{map for which the SPDE solution is a fixed point, \cref{eq:clTdef}}
\begin{equation}
  (\clT_{s}^{\sigma,\rho}v)_{t}=\clG_{t-s}v_{s}+\gamma_{\rho}\int_{s}^{t}\clG_{t+\rho-r}[\sigma(v_{r})\,\dif W_{r}].\label{eq:clTdef}
\end{equation}
As we work in a mild framework, solutions of \cref{eq:SPDE} for $t\ge s$ are exactly fixed points of $\clT_{s}^{\sigma,\rho}$.

For $s<t$, we define a norm $\clM_{s,t}$ on random fields
on $[s,t]\times\RR^{2}$ by
\nomenclature[M st]{$\clM_{s,t}$}{$\sup_{r\in[s,t]}\vvvert v_{r}\vvvert_{2}$, \cref{eq:scrMstdef}}
\begin{equation}
  \clM_{s,t}(v)=\sup_{r\in[s,t]}\vvvert v_{r}\vvvert_{2}=\sup_{\substack{r\in[s,t]\\
      x\in\RR^{2}
    }
  }(\EE |v_{r}(x)|^{2})^{1/2},\label{eq:scrMstdef}
\end{equation}
recalling the definition \cref{eq:triplenormdef}.
\begin{lem}
  Let $\sigma\in\Lip(\RR^m;\clH^m_+)$ and $t > s$.
  If $v,\widetilde{v}$ are random fields on $[s,t]\times\RR^{2}$
  adapted to the filtration $(\scrF_{s})$ and satisfying $v_s = \widetilde{v}_s$ as well as
  $\clM_{s,t}(v),\clM_{s,t}(\widetilde{v})<\infty$, then
  \begin{equation}
    \clM_{s,t}(\clT_{s}^{\sigma,\rho}v-\clT_{s}^{\sigma,\rho}\widetilde{v})\le\Lip(\sigma)\sfS_{\rho}(t-s)^{1/2}\clM_{s,t}(v-\widetilde{v}).\label{eq:Mbd}
  \end{equation}
\end{lem}

\begin{proof}
  For all $r\in[s,t]$ and $x \in \R^2$, we have (using the assumption that $v_s = \widetilde{v}_s$) that
  \begin{align*}
    \EE (\clT_{s}^{\sigma,\rho}v-\clT_{s}^{\sigma,\rho}\widetilde{v})_{r}(x)^{2} & =\frac{4\pi}{\sfL(1/\rho)}\int_{s}^{r}\!\!\!\!\int G_{r+\rho-r'}^{2}(x-y)\EE \bigl\lvert \sigma\bigl(v_{r'}(y)\bigr)-\sigma\bigl(\widetilde{v}_{r'}(y)\bigr)\bigr\rvert _{\Frob}^{2}\,\dif y\,\dif r'                                                                \\
                                                                                 & \le\frac{\Lip(\sigma)^{2}}{\sfL(1/\rho)}\int_{s}^{r}\frac{\vvvert v_{r'}-\widetilde v_{r'}\vvvert_{2}^{2}}{r+\rho-r'}\,\dif r'\overset{\cref{eq:Srhodef}}{\le}\clM_{s,t}(v-\widetilde{v})^{2}\Lip(\sigma)^{2}\sfS_{\rho}(r-s),
  \end{align*}
  which implies \cref{eq:Mbd}.
\end{proof}
\begin{prop}
  \label{prop:solnapprox}Suppose $\sigma\in\Lip(\RR^m;\clH^m_+)$ and $s<t$ satisfy
  \begin{equation}
    \Lip(\sigma)\sfS_{\rho}(t-s)^{1/2}<1.\label{eq:llambdalt1}
  \end{equation}
  Let $v$ be a random field on $[s,t]\times\RR^{2}$ adapted to the filtration $(\s F_s)$ with $\clM_{s,t}(v)<\infty$.
  Then we have
  \begin{equation}
    \clM_{s,t}(v-\clV_{s,\cdot}^{\sigma,\rho}v_{s})\le\frac{\clM_{s,t}(v-\clT_{s}^{\sigma,\rho}v)}{1-\Lip(\sigma)\sfS_{\rho}(t-s)^{1/2}}.\label{eq:Mclose}
  \end{equation}
\end{prop}

\begin{proof}
  By \cref{eq:Mbd} under the assumption \cref{eq:llambdalt1}, we see that
  $\clT_{s}^{\sigma,\rho}$ is a contraction in the $\clM_{s,t}$
  norm with fixed point $\clV_{s,\anon}^{\sigma,\rho}v_{s}$ and contraction
  factor $\Lip(\sigma)\sfS_\rho(t-s)^{1/2}$.
  This implies \cref{eq:Mclose}.
\end{proof}

\subsection{Moment estimates}\label{subsec:moment-estimates}

We now control the moments of solutions to \cref{eq:SPDE}. %
Recall the definition of $\scrX_{s}^{\ell}$ and its associated
norm $\vvvert\cdot\vvvert_{\ell}$ on \zcpageref{eq:triplenormdef}. Some
estimates in the sequel concern a problem with the
noise ``turned off'' on parts of space-time. We introduce notation for this operation now.
\begin{defn}
  \label{def:turnoffnoise}For a Borel set $\ttB\subseteq\RR^{2}$, let $\dif\mathring{W}_{t}^{\ttB}=\mathbf{1}_{\ttB}\dif W_{t}$,
  so $\dif\mathring{W}_{t}^{\ttB}|_{\ttB}=\dif W_{t}|_{\ttB}$ but $\dif\mathring{W}_{t}^{\ttB}|_{\ttB^{\cc}}\equiv0$.
  We let $\mathring{\clV}^{\sigma,\rho,\ttB}_{s,t}$ be identical to the propagator
  $\clV^{\sigma,\rho}_{s,t}$ for the SPDE \cref{eq:SPDE}, but with the noise
  $\dif W_{t}$ replaced by $\dif\mathring{W}_{t}^{\ttB}$.
\end{defn}
\nomenclature[W circ]{$\dif\mathring W_t^\ttB$}{$\mathbf{1}_\ttB\dif W_t$, \cref{def:turnoffnoise}}
\nomenclature[V ring]{$\mathring{\clV}^{\sigma,\rho,\ttB}_{s,t}$}{propagator for the SPDE with the noise turned off on $\ttB^\cc$ ($\sigma,\rho$ often omitted), \cref{def:turnoffnoise}}

We establish moment estimates that are uniform in $\ttB$.
For $\ell\ge2$, define
\nomenclature[zzzgreek η]{$\eta(\beta,\ell)$}{$[2^{-2/\ell}\ell(\ell-1)\left(\beta^{2}+1-2/\ell\right)]^{1/2}$, \cref{eq:etadef}}
\begin{equation}
  \eta(\beta,\ell)=[2^{-2/\ell}\ell(\ell-1)\left(\beta^{2}+1-2/\ell\right)]^{1/2}.\label{eq:etadef}
\end{equation}
Note that $\eta(\beta,2)=\beta$ and that $\eta(\beta,\ell)$ is continuous and strictly increasing in $\ell$ and $\beta$.
\begin{prop}
  \label{prop:momentbd}Fix $\ell\ge2$, $\beta\in(0,\infty)$, $\ttB\subseteq\RR^{2}$,
  and $s\in\RR$, and let $v_{s}\in\scrX_{s}^{\ell}$ and
  $\sigma\in\quadset(M,\beta)$. For all $T\in\bigl[s,s+\sfT_{\rho}(\eta(\beta,\ell)^{-2})\bigr)$,
  we have
  \begin{equation}
    \vvvert\mathring{\clV}_{s,T}^{\sigma,\rho,\ttB}v_{s}\vvvert_{\ell}^{\ell}\le\frac{\vvvert v_{s}\vvvert_{\ell}^{\ell}+2^{1-2/\ell}(\ell-1)\sfS_{\rho}(T-s)M^{\ell/2}}{1-\eta(\beta,\ell)^{2}\sfS_{\rho}(T-s)}.\label{eq:momentbd}
  \end{equation}
\end{prop}

\begin{rem}
  \label{rem:momentbdvswhatwehadbefore}A result similar to \cref{prop:momentbd}
  was proved for the linear problem $\sigma(u)=\beta u$ using hypercontractivity
  in \cite[(5.11)]{CSZ20}. In \cite[Proposition 3.1]{DG22}, this was
  applied to a nonlinear problem using the moment comparison principle
  of \cite{CK20}. This comparison principle is not available in the vector-valued
  case (nor, indeed, in the scalar-valued case without the assumption $\sigma(0)=0$).
  The proof of \cref{prop:momentbd}
  we give here is self-contained and uses only It\^o calculus and simple
  convexity inequalities.
\end{rem}

\begin{rem}
  If $s=0$ and $T$ is positive and independent
  of $\rho$, then $\sfS_{\rho}(T)\to1$ as $\rho\searrow0$.
  This means that if $\beta<1$ and $\ell$ is sufficiently close to
  $2$ but independent of $\rho$, then the hypothesis $T<\sfT_{\rho}(\eta(\beta,\ell)^{-2})$ of \cref{prop:momentbd}
  is satisfied when $\rho$ is sufficiently small. The works 
  \cite{LZ21,CZ21} showed that in the discrete directed polymer model
  (corresponding to our SPDE in the
   scalar-valued linear case $\sigma(u)=\beta u$), the condition
  that $\ell$ is sufficiently close to $2$ is not necessary: all positive
  moments can be controlled, uniformly as $\rho\searrow0$, whenever
  $\sfS_{\rho}(T)\beta^{2}$ is bounded above by a constant less than $1$. (See also \cite{LZ22} for corresponding results on higher multipoint moments.) It seems
  likely that their techniques could be applied to our nonlinear setting
  as well to strengthen \cref{eq:momentbd}.
\end{rem}

\begin{proof}[Proof of \cref{prop:momentbd}]
  We will only consider the case $\ttB=\RR^{2}$; the case when $\ttB$ is
  a proper subset of $\RR^{2}$ is essentially the same but requires
  extra notation, so we leave it to the reader. We also assume, without
  loss of generality, that $s=0$. Then the assumptions on $T$, $\ell$, and $\beta$ imply that
  \begin{equation}
      \sfS_\rho(T)2^{-2/\ell}\ell(\ell-1)(\beta^2+1-2/\ell)<1.\label{eq:ourTcond}
  \end{equation}

  Let $v_{t}=\clV_{t}^{\sigma,\rho}v_{0}$ and
  define $V_{t}=V_{t}^{\sigma,\rho,T}(x)=\clG_{T-t}v_{t}(x)$
  as in \cref{eq:Vdef}. Define $Y_{t}=|V_{t}(x)|^{\ell}$. %
  Using the derivative computations \zcref[range]{eq:lpowergradient,eq:lpowerHessian}, by Itô's formula we obtain
  \begin{equation}
    \dif Y_{t}=\ell|V_{t}|^{\ell-2}V_{t}\cdot\dif V_{t}+\frac{1}{2}\tr\left[\left(\ell(\ell-2)|V_{t}|^{\ell-4}V_{t}^{\otimes2}+\ell|V_{t}|^{\ell-2}\Id_m\right)\dif[V]_{t}\right].\label{eq:ito}
  \end{equation}
  Now we integrate \cref{eq:ito} from $0$ to $T$ and take expectations.
  It follows from standard techniques that the stochastic
  integral in the result contributes zero to the expectation. We also have $\EE Y_{0}\le\vvvert v_{0}\vvvert_{\ell}^{\ell}$.
  Hence \cref{eq:VQV} yields
  \begin{equation}
    \begin{aligned}\EE & |v_{T}(x)|^{\ell}=\EE Y_{T}\le\vvvert v_{0}\vvvert_{\ell}^{\ell}+\frac{1}{2}\EE \left[\int_{0}^{T}\tr\left[\left(\ell(\ell-2)|V_{t}|^{\ell-4}V_{t}^{\otimes2}+\ell|V_{t}|^{\ell-2}\Id_m\right)\dif[V]_{t}\right]\right]                                                    \\
                   & =\vvvert v_{0}\vvvert_{\ell}^{\ell}+\frac{1}{2\sfL(1/\rho)}\int_{0}^{T}\frac{\EE \left(\tr\left[\left(\ell(\ell-2)|V_{t}|^{\ell-4}V_{t}^{\otimes2}+\ell|V_{t}|^{\ell-2}\Id_m\right)\clG_{\frac{1}{2}[T+\rho-t]}(\sigma^{2}\circ v_{t})(x)\right]\right)}{T+\rho-t}\,\dif t.
    \end{aligned}
    \label{eq:pthmomentexpectations}
  \end{equation}

  We now estimate the quantity inside the expectation on the
  right side of \cref{eq:pthmomentexpectations}.
  For the first term, we have
  \begin{align*}
    \tr\left[V_{t}^{\otimes2}\clG_{\frac{1}{2}[T+\rho-t]}(\sigma^{2}\circ v_{t})(x)\right]\overset{\cref{eq:simpleholderschatten}} & {\le}|V_{t}^{\otimes2}|_{\op}\clG_{\frac{1}{2}[T+\rho-t]}(\tr\circ\sigma^{2}\circ v_{t})(x) \\&=|V_{t}|^{2}\clG_{\frac{1}{2}[T+\rho-t]}(|\sigma|_{\Frob}^{2}\circ v_{t})(x),
  \end{align*}
  This means that
  \begin{equation}
    \begin{aligned}
      \tr & \left[\left(\ell(\ell-2)|V_{t}|^{\ell-4}V_{t}^{\otimes2}+\ell|V_{t}|^{\ell-2}\Id_m\right)\clG_{\frac{1}{2}[T+\rho-t]}(\sigma^{2}\circ v_{t})(x)\right] \\&\le\ell(\ell-1)|V_{t}|^{\ell-2}\clG_{\frac{1}{2}[T+\rho-t]}(|\sigma|_{\Frob}^{2}\circ v_{t})(x).
    \end{aligned}\label{eq:combinethetwotraceterms}
  \end{equation}
  Using Young's product inequality and Jensen's inequality, for all $A > 0$ we have
  \begin{align}
      |V_{t}|^{\ell-2} & \clG_{\frac{1}{2}[T+\rho-t]}(|\sigma|_{\Frob}^{2}\circ v_{t})(x)\overset{\cref{eq:Vdef}}=|\clG_{T-t}v_{t}(x)|^{\ell-2}\clG_{\frac{1}{2}[T+\rho-t]}(|\sigma|_{\Frob}^{2}\circ v_{t})(x)\nonumber    \\
                     & \le(1-2/\ell)A^{\frac{\ell}{\ell-2}}|\clG_{T-t}v_{t}(x)|^{\ell}+\frac{2}{\ell A^{\ell/2}}\left[\clG_{\frac{1}{2}[T+\rho-t]}(|\sigma|_{\Frob}^{2}\circ v_{t})(x)\right]^{\ell/2}\nonumber \\
                     & \le(1-2/\ell)A^{\frac{\ell}{\ell-2}}\clG_{T-t}(|v_{t}|^{\ell})(x)+\frac{2}{\ell A^{\ell/2}}\clG_{\frac{1}{2}[T+\rho-t]}(|\sigma|_{\Frob}^{\ell}\circ v_{t})(x)\label{eq:Jensenparty-pre}
  \end{align}
  For the second term, we use the fact that $\sigma\in\quadset(M,\beta)$
  (recalling the definition \cref{eq:Sigmaplusdef}) along with Hölder's
  inequality to write
  \begin{equation}
      |\sigma(v)|_{\Frob}^{\ell}\le[M+\beta^{2}|v|^{2}]^{\ell/2}\le2^{\ell/2-1}[M^{\ell/2}+\beta^{\ell}|v|^{\ell}].\label{eq:holder2}
\end{equation}
Using \cref{eq:holder2} in \cref{eq:Jensenparty-pre}, we see that
  \begin{align*}
    |V_{t}|^{\ell-2} & \clG_{\frac{1}{2}[T+\rho-t]}(|\sigma|_{\Frob}^{2}\circ v_{t})(x)                                                                                                                                       \\
                     & \le(1-2/\ell)A^{\frac{\ell}{\ell-2}}\clG_{T-t}(|v_{t}|^{\ell})(x)+\frac{2^{\ell/2}}{\ell A^{\ell/2}}\left[M^{\ell/2}+\beta^{\ell}\clG_{\frac{1}{2}[T+\rho-t]}(|v_{t}|^{\ell})(x)\right].
  \end{align*}
  Taking expectations, we obtain
  \[
    \EE \left[|V_{t}|^{\ell-2}\clG_{\frac{1}{2}[T+\rho-t]}(|\sigma|_{\Frob}^{2}\circ v_{t})(x)\right]\le(1-2/\ell)A^{\frac{\ell}{\ell-2}}\vvvert v_{t}\vvvert_{\ell}^{\ell}+\frac{2^{\ell/2}}{\ell A^{\ell/2}}\beta^{\ell}\vvvert v_{t}\vvvert_{\ell}^{\ell}+\frac{2^{\ell/2}}{\ell A^{\ell/2}}M^{\ell/2}.
  \]
  We minimize the right side by taking
  \[
    A=2^{(\ell-2)^{2}/\ell^{2}}\left(\beta^{\ell}+\frac{M^{\ell/2}}{\vvvert v_{t}\vvvert_{\ell}^{\ell}}\right)^{2(\ell-2)/\ell^{2}}
  \]
  to obtain
  \begin{align*}
    \EE & \left[|V_{t}|^{\ell-2}\clG_{\frac{1}{2}[T+\rho-t]}(|\sigma|_{\Frob}^{2}\circ v_{t})(x)\right]\le2^{1-2/\ell}\left(\beta^{\ell}\vvvert v_{t}\vvvert_{\ell}^{\ell}+M^{\ell/2}\right)^{2/\ell}\vvvert v_{t}\vvvert_{\ell}^{\ell-2}      \\
        & \le2^{1-2/\ell}\left(\beta^{2}\vvvert v_{t}\vvvert_{\ell}^{\ell}+M\vvvert v_{t}\vvvert_{\ell}^{\ell-2}\right)\le2^{1-2/\ell}\left[\left(\beta^{2}+1-2/\ell\right)\vvvert v_{t}\vvvert_{\ell}^{\ell}+\frac{2}{\ell}M^{\ell/2}\right],
  \end{align*}
  with the last inequality by Young's product inequality. Using this, \cref{eq:combinethetwotraceterms},
  and \cref{eq:Srhodef} in \cref{eq:pthmomentexpectations}, we get
  \begin{equation}
    \EE |v_{T}(x)|^{\ell}  \le\vvvert v_{0}\vvvert_{\ell}^{\ell}+2^{-2/\ell}\ell(\ell-1)\sfS_{\rho}(T)\left[\left(\beta^{2}+1-2/\ell\right)\Upsilon^{\ell}+\frac{2}{\ell}M^{\ell/2}\right],\label{eq:boundwithupsilon}
  \end{equation}
  where we have defined $\Upsilon\coloneqq\sup\limits_{t\in[0,T]}\vvvert v_{t}\vvvert_{\ell}$.
  Since the right side of \cref{eq:boundwithupsilon} is increasing in $T$, and $T$ and $x$ are
  arbitrary, we in fact have
  \begin{align*}
    \Upsilon^{\ell} & \le\vvvert v_{0}\vvvert_{\ell}^{\ell}+2^{-2/\ell}\ell(\ell-1)\sfS_{\rho}(T)\left(\left(\beta^{2}+1-2/\ell\right)\Upsilon^{\ell}+\frac{2}{\ell}M^{\ell/2}\right),
  \end{align*}
  which we can rearrange to obtain
  \[
    \Upsilon^{\ell}\le\frac{\vvvert v_{0}\vvvert_{\ell}^{\ell}+2^{1-2/\ell}(\ell-1)\sfS_{\rho}(T)M^{\ell/2}}{1-2^{-2/\ell}\ell(\ell-1)\sfS_{\rho}(T)\left(\beta^{2}+1-2/\ell\right)}
  \]
  noting that the denominator is positive by \cref{eq:ourTcond}.
  Recalling the definition \cref{eq:etadef} of $\eta(\beta,\ell)$, we see that we have proved
  \cref{eq:momentbd}.
\end{proof}
We now give a name to a quantity that is bounded by \cref{prop:momentbd}, and which will appear frequently in the estimates in the sequel.
\begin{defn}\label{def:momentbdnotation}
For $t\ge s$, $\ell\ge2$, and $v_{s}\in\scrX_{s}^{\ell}$,
we define
\nomenclature[zzzgreek ξ  ]{$\Xi_{\ell;s,t}^{\sigma,\rho}[v_{s}]$}{maximal norm of the SPDE solution, see \cref{eq:momentbdnotation}}
\begin{equation}
  \Xi_{\ell;s,t}^{\sigma,\rho}[v_{s}]=\adjustlimits\sup_{\ttB\subseteq\RR^{2}}\sup_{r\in[s,t]}\left(\vvvert\mathring{\clV}_{s,r}^{\sigma,\rho,\ttB}v_{s}\vvvert_{\ell}\vee\vvvert\sigma(\mathring{\clV}_{s,r}^{\sigma,\rho,\ttB}v_{s})\vvvert_{\ell}\right).\label{eq:momentbdnotation}
\end{equation}
We will often abbreviate $\Xi_{\ell;t}^{\sigma,\rho}=\Xi_{\ell;0,t}^{\sigma,\rho}$.
\nomenclature[zzzgreek ξ ]{$\Xi_{\ell;t}^{\sigma,\rho}[v_{s}]$}{abbreviation for $\Xi_{\ell;0,t}^{\sigma,\rho}[v_{s}]$}
\end{defn}

\subsection{Regularity estimates for Gaussian averages of the solution\label{subsec:regularitygaussianaverages}}
In this section we prove bounds on the differences between spatial averages of the solution centered at different space-time points. These estimates are proved using $L^2$ estimates on the mild solution formula.
\begin{lem}
  \label{lem:VTcontinuity}Suppose that $(v_{t})$ solves \cref{eq:SPDE}
  and define $V_{t}^{\rho,T}$ as in \cref{eq:Vdef}. We have,
  for $t_{1}\le t_{2}\le T$ and $x\in\RR^2$, that
  \[
    \EE |V_{t_{2}}^{\rho,T}(x)-V_{t_{1}}^{\rho,T}(x)|^{2}\le\Xi_{2;t_{1},t_{2}}^{\rho}[v_{t_{1}}]^{2}\sfL_{\rho}\left(\frac{t_{2}-t_{1}}{T+\rho-t_{2}}\right).
  \]
\end{lem}

\begin{proof}
  By \cref{eq:Vtintegral}, we have
  \[
    V_{t_{2}}^{\rho,T}(x)-V_{t_{1}}^{\rho,T}(x)=\gamma_{\rho}\int_{t_{1}}^{t_{2}}\clG_{T+\rho-r}[\sigma(v_{r})\,\dif W_{r}](x).
  \]
  So, by the Itô isometry,
  \begin{align*}
    \EE & |V_{t_{2}}^{\rho,T}(x)-V_{t_{1}}^{\rho,T}(x)|^{2}=\frac{4\pi}{\sfL(1/\rho)}\int_{t_{1}}^{t_{2}}\!\!\!\!\int G_{T+\rho-r}^{2}(x-y)\EE \bigl\lvert \sigma\bigl(v_{r}(y)\bigr)\bigr\rvert _{\Frob}^{2}\,\dif y\,\dif r                                                     \\
        & \le\frac{\Xi_{2;t_{1},t_{2}}^{\sigma,\rho}[v_{t_{1}}]^{2}}{\sfL(1/\rho)}\int_{t_{1}}^{t_{2}}\frac{\dif r}{T+\rho-r}=\Xi_{2;t_{1},t_{2}}^{\sigma,\rho}[v_{t_{1}}]^{2}\sfL_{\rho}\left(\frac{t_{2}-t_{1}}{T+\rho-t_{2}}\right).\qedhere
  \end{align*}
\end{proof}
\begin{prop}
  \label{prop:continuitybd}There is an absolute constant $C<\infty$
  such that the following holds. Suppose $(v_{t})_t$ solves \cref{eq:SPDE}
  with initial condition $v_{s_{0}}\in\scrX_{s_{0}}^{2}$.
  Then for all $s_{1},s_{2}>s_{0}$, $t_{1},t_{2}\ge0$, and $x_{1},x_{2}\in\RR^{2}$, we have
  \begin{equation}
    \begin{aligned}(\EE & |\clG_{t_{2}}v_{s_{2}}(x_{2})-\clG_{t_{1}}v_{s_{1}}(x_{1})|^{2})^{1/2}                                                                                               \\
                    & \le C\Xi_{2;s_{0},s_1 \vee s_2}^{\sigma,\rho}[v_{s_0}]\Biggl[\sfL_{\rho}\left(\frac{|s_{2}-s_{1}|+|t_{1}-t_{2}|+|x_{1}-x_{2}|^{2}/2}{t_{1}\wedge t_2+\rho}\right)^{1/2}        \\
                    & \qquad\qquad\qquad\qquad\hfill+\left(\frac{|s_{1}-s_{2}|+|t_{1}-t_{2}|+|x_{1}-x_{2}|^{2}/2}{s_{1}\wedge s_{2}-s_{0}}\right)^{1/2}+\sfL_{\rho}\bigl(1+\sfL(1/\rho)\bigr)^{1/2}\Biggr].
    \end{aligned}
    \label{eq:continuitybd}
  \end{equation}
\end{prop}
\begin{proof}
  Assume without loss of generality that $s_{1}\le s_{2}$. For $i=1,2$,
  define $T_{i}=t_{i}+s_{i}$. Fix $s<s_{1}$, to be chosen later. We
  have, by the triangle inequality, that
  \begin{align}
    (\EE |\clG_{t_{1}}v_{s_{1}}(x_{1})-\clG_{t_{2}}v_{s_{2}}(x_{2})|^{2})^{1/2} & \le(\EE |\clG_{t_{1}}v_{s_{1}}(x_{1})-\clG_{T_{1}-s}v_{s}(x_{1})|^{2})^{1/2}\nonumber                       \\
                                                                                & \qquad+(\EE |\clG_{T_{1}-s}v_{s}(x_{1})-\clG_{T_{2}-s}v_{s}(x_{2})|^{2})^{1/2}\nonumber                     \\
                                                                                & \qquad+(\EE |\clG_{T_{2}-s}v_{s}(x_{2})-\clG_{t_{2}}v_{s_{2}}(x_{2})|^{2})^{1/2}.\label{eq:continuitysplit}
  \end{align}
  For the first and third terms of \cref{eq:continuitysplit}, we have
  by \cref{lem:VTcontinuity} (using that $s<s_{1}\le s_{2}$) that
  \begin{align}
    \EE |\clG_{t_{i}}v_{s_{i}}(x_{i})-\clG_{t_{i}+s_{i}-s}v_{s}(x_{i})|^{2} & =\EE |V_{s_{i}}^{\rho,t_{i}+s_{i}}(x_{i})-V_{s}^{\rho,t_{i}+s_{i}}(x_{i})|^{2}\nonumber                               \\
                                                                            & \le\Xi_{2;s_{0},s_{i}}^{\sigma,\rho}[v_{s_{0}}]^{2}\sfL_{\rho}\left(\frac{s_{i}-s}{t_{i}+\rho}\right).\label{eq:firstandthirdterms}
  \end{align}
  For the second term of \cref{eq:continuitysplit}, we have
  \begin{align}
    \EE & |\clG_{T_{1}-s}v_{s}(x_{1})-\clG_{T_{2}-s}v_{s}(x_{2})|^{2} = \EE \left\lvert  \int[G_{T_{1}-s}(x_{1}-y)-G_{T_{2}-s}(x_{2}-y)]v_{s}(y)\,\dif y\right\rvert  ^{2}\nonumber                                                                                               \\
        & \leq\iint\EE [|v_{s}(y_{1})||v_{s}(y_{2})|]\prod_{j=1}^{2}|G_{T_{1}-s}(x_{1}-y_{j})-G_{T_{2}-s}(x_{2}-y_{j})|\,\dif y_{1}\,\dif y_{2}\nonumber                                                                                                               \\
        & \le\vvvert v_{s}\vvvert_{2}^{2}\|G_{T_{1}-s}(x_{1}-\cdot)-G_{T_{2}-s}(x_{2}-\cdot)\|_{L^{1}}^{2}\nonumber                                                                                                                                                 \\
        & \le\Xi_{2;s_{0},s}^{\sigma,\rho}[v_{s_{0}}]^{2}\frac{|T_{1}-T_{2}|+\frac{1}{2}|x_{1}-x_{2}|^{2}}{(T_{1}-s)\wedge(T_{2}-s)}\le\Xi_{2;s_{0},s}^{\sigma,\rho}[v_{s_{0}}]^{2}\frac{|T_{1}-T_{2}|+\frac{1}{2}|x_{1}-x_{2}|^{2}}{s_{1}-s},\label{eq:middleterm}
  \end{align}
  with the penultimate inequality by \cref{eq:L1ofgaussians}.

  Now take
  \begin{equation}
    s\coloneqq\max\{s_{0},s_{1}-\sfL(1/\rho)[|T_{1}-T_{2}|+|x_{1}-x_{2}|^{2}/2]\}.\label{eq:schoice}
  \end{equation}
  Then \cref{eq:firstandthirdterms} means that
  \begin{align}
    ( & \EE |\clG_{t_{i}}v_{s_{i}}(x_{i})-\clG_{t_{i}+s_{i}-s}v_{s}(x_{i})|^{2})^{1/2}\le\Xi_{2;s_{0},s_{i}}^{\sigma,\rho}[s_0]\sfL_{\rho}\left(\frac{s_{2}-s_{1}+\sfL(1/\rho)[|T_{1}-T_{2}|+|x_{1}-x_{2}|^{2}/2]}{t_{i}+\rho}\right)^{1/2}\nonumber \\
      & \le\Xi_{2;s_{0},s_{i}}^{\sigma,\rho}[s_0]\left[\sfL_{\rho}\left(\frac{s_{2}-s_{1}+|T_{1}-T_{2}|+|x_{1}-x_{2}|^{2}/2}{t_{i}+\rho}\right)^{1/2}+\sfL_{\rho}\bigl(\sfL(1/\rho)\bigr)^{1/2}\right]\nonumber                                                \\
      & \le\Xi_{2;s_{0},s_{i}}^{\sigma,\rho}[s_0]\left[\sfL_{\rho}\left(\frac{2(s_{2}-s_{1})+|t_{1}-t_{2}|+|x_{1}-x_{2}|^{2}/2}{t_{i}+\rho}\right)^{1/2}+\sfL_{\rho}\bigl(\sfL(1/\rho)\bigr)^{1/2}\right]\label{eq:firstthirdtermours}
  \end{align}
  with the second inequality by \cref{eq:Lmult}. Also, with the choice
  \cref{eq:schoice} of $s$, we have by \cref{eq:middleterm} that
  \begin{align}
    \EE |\clG_{T_{1}-s}v_{s}(x_{1})-\clG_{T_{2}-s}v_{s}(x_{2})|^{2} & \le\Xi_{2;s_{0},s}^{\sigma,\rho}[v_{s_{0}}]^{2}\frac{|T_{1}-T_{2}|+\frac{1}{2}|x_{1}-x_{2}|^{2}}{s_{1}-\max\{s_{0},s_{1}-\sfL(1/\rho)[|T_{1}-T_{2}|+|x_{1}-x_{2}|^{2}/2]\}}\nonumber                    \\
                                                                    & =\Xi_{2;s_{0},s}^{\sigma,\rho}[v_{s_{0}}]^{2}\max\left\{ \frac{|T_{1}-T_{2}|+\frac{1}{2}|x_{1}-x_{2}|^{2}}{s_{1}-s_{0}},\frac{1}{\sfL(1/\rho)}\right\} \nonumber                                        \\
                                                                    & \le\Xi_{2;s_{0},s}^{\sigma,\rho}[v_{s_{0}}]^{2}\left(\frac{|t_{1}-t_{2}|+|s_{1}-s_{2}|+\frac{1}{2}|x_{1}-x_{2}|^{2}}{s_{1}\wedge s_{2}-s_{0}}+\frac{1}{\sfL(1/\rho)}\right).\label{eq:middeltermours-1}
  \end{align}
  Using \cref{eq:firstthirdtermours} and \cref{eq:middeltermours-1} in
  \cref{eq:continuitysplit}, we get \cref{eq:continuitybd}.
\end{proof}

\section{Local SPDE estimates\label{sec:local-SPDE-estimates}}

The estimates in this section are largely analogous to those proved
in \cite{DG22}, which focused on the case $\Lip(\sigma)<1$ (in the notation of this paper).
We no longer assume $\Lip(\sigma)<1$, so the estimates here will only hold on a short time scale depending on $\Lip(\sigma)$.
We define this short
time scale as
\nomenclature[T bar]{$\overline{T}_{\rho}(\lambda)$}{$\sfT_{\rho}\bigl(\lambda^{-2}-2/\sfL(1/\rho)\bigr)\approx\rho^{1-\lambda^{-2}}$, \cref{eq:mtbardef}}
\begin{equation}
  \overline{T}_{\rho}(\lambda)\coloneqq\sfT_{\rho}\bigl(\lambda^{-2}-2\sfL(1/\rho)^{-1}\bigr),\label{eq:mtbardef}
\end{equation}
where we will generally take $\lambda$ to be (an upper bound on) $\Lip(\sigma)$.
Recalling \cref{eq:Tapprox}, we have $\overline{T}_{\rho}(\lambda) \approx \rho^{1 - \lambda^{-2}}$ when $\rho \ll 1$.
Naturally, many bounds will blow up as the time scale approaches $\overline{T}_{\rho}(\lambda)$; to facilitate quantitative estimates, we define the parameter
\nomenclature[K lambda rho]{$K_{\lambda,\rho,t}$}{error bound term equal to $[1-\lambda^2(\sfS_\rho(t)+2/\sfL(1/\rho))]^{-1}$, \cref{eq:Kdef-1}}
\begin{equation}
  K_{\lambda,\rho,t}\coloneqq\bigl(1-\lambda^{2}\bigl[\sfS_{\rho}(t)+2\sfL(1/\rho)^{-1}\bigr]\bigr)^{-1}\qquad\text{for }t\in\bigl[0,\overline{T}_{\rho}(\lambda)\bigr).\label{eq:Kdef-1}
\end{equation}
This is order 1 when $t \approx \rho^{1 - q}$ for $q < \lambda^{-2}$, but diverges as $t \nearrow \overline{T}_{\rho}(\lambda)$.

If $\lambda<1$, then $\overline{T}_{\rho}(\lambda)\to\infty$
as $\rho\searrow0$ and the estimates proved in this section will hold on macroscopic time scales;
this is the case considered in \cite{DG22}. When $\lambda>1$, however, $\overline{T}_{\rho}(\lambda)\to0$ as $\rho\searrow0$, so the estimates only work on short time scales.
Thus the results of this section are ``local,'' in analogy with \cref{subsec:FBSDE-local} for the FBSDE.

Throughout, we use \cref{prop:momentbd} to control moments of $v$.
However, \cref{prop:momentbd} is conceptually different from the results of this section, as it only requires $\sigma\in\quadset(M,\beta)$ rather than $\Lip(\sigma)\le\lambda$.
We therefore isolate the moments in subsequent estimates via the $\Xi$ notation defined in \cref{def:momentbdnotation}.
We will not consider multiple $\sigma$s or $\rho$s, so we abbreviate $\clV_{s,t}=\clV_{s,t}^{\sigma,\rho}$, $\Xi_{\ell;s,t}=\Xi_{\ell;s,t}^{\sigma,\rho}$, and $\Xi_{\ell;t}=\Xi_{\ell;0,t}=\Xi_{\ell;0,t}^{\sigma,\rho}$.

\subsection{Turning off the noise and flattening\label{subsec:noise-flatten}}
In this section, we explore the impact of two operations on the solution $v$.
In the first, we turn off the noise on a ``quiet interval,'' causing \cref{eq:SPDE} to reduce to the heat equation.
In the second operation, we replace the solution by a constant (``flatten'' it) at a certain moment in time.
Their combination plays a major role in our proof of \cref{thm:mainthm-singlepoint}.

\subsubsection{Turning off the noise field}

We begin by studying the effect of turning off the noise.
Precisely, we compare the operators $\clV_{\tau_{0},\tau_{2}}$ and $\clV_{\tau_{1},\tau_{2}}\clG_{\tau_{1}-\tau_{0}}$.
The first represents the dynamics \cref{eq:SPDE} on the time interval $[\tau_{0},\tau_{2}]$.
The second runs the deterministic heat equation on the quiet interval $[\tau_{0},\tau_{1}]$, followed by the full random dynamics \cref{eq:SPDE} on $[\tau_{1},\tau_{2}]$.
This is equivalent to solving \cref{eq:SPDE} with the noise $\dif W_t$ replaced by $\tbf{1}_{[\tau_1,\tau_2]}(t)\dif W_t$.

The following proposition is analogous to \cite[Proposition 4.1]{DG22} and is proved in the same way.
\begin{prop}
  \label{prop:turnoffnoise}
  Suppose $\lambda\in(0,\infty)$, $\sigma\in\lipset(\lambda)$, and $\tau_{0}\le\tau_{1} < \tau_{2}$ satisfy $\tau_{2}-\tau_{1} < \overline{T}_{\rho}(\lambda)$.
  For all $v_{\tau_{0}}\in\scrX_{\tau_{0}}^{2}$, we have
  \[
    \vvvert\clV_{\tau_{0},\tau_{2}}v_{\tau_{0}}-\clV_{\tau_{1},\tau_{2}}\clG_{\tau_{1}-\tau_{0}}v_{\tau_{0}}\vvvert_{2}^{2}\le K_{\lambda,\rho,\tau_{2}-\tau_{1}}\Xi_{2;\tau_0,\tau_{2}}[v_{\tau_0}]^{2}\left[\sfL_{\rho}\left(\frac{\tau_{1}-\tau_{0}}{\tau_{2}-\tau_{1}+\rho}\right)+\frac{\sfL(2)K_{\lambda,\rho,\tau_{2}-\tau_{1}}}{\sfL(1/\rho)}\right].
  \]
\end{prop}

\begin{proof}
  Assume without loss of generality that $\tau_{0}=0$. For $t\ge\tau_{1}$,
  define $v_{t}=\clV_{0,t}v_{0}$ and $\widetilde{v}_{t}=\clV_{\tau_{1},t}\clG_{\tau_{1}}v_{0}$.
  Using \cref{eq:mildsoln}, we have
  \[
    v_{t}-\widetilde{v}_{t}=\gamma_{\rho}\int_{0}^{\tau_{1}}\clG_{t+\rho-r}[\sigma(v_{r})\,\dif W_{r}]+\gamma_{\rho}\int_{\tau_{1}}^{t}\clG_{t+\rho-r}\bigl([\sigma(v_{r})-\sigma(\widetilde{v}_{r})]\,\dif W_{r}\bigr).
  \]
  Taking second moments and using the Itô isometry, we can estimate
  \begin{align*}
    \EE |(v_{t}-\widetilde{v}_{t})(x)|^{2} & =\frac{4\pi}{\sfL(1/\rho)}\int_{0}^{\tau_{1}}\!\!\!\!\int G_{t+\rho-r}^{2}(x-y)\EE \bigl\lvert \sigma\bigl(v_{r}(y)\bigr)\bigr\rvert _{\Frob}^{2}\,\dif y\,\dif r\\
                                           & \qquad+\frac{4\pi}{\sfL(1/\rho)}\int_{\tau_{1}}^{t}\!\!\!\int G_{t+\rho-r}^{2}(x-y)\EE \bigl\lvert \sigma\bigl(v_{r}(y)\bigr)-\sigma\bigl(\widetilde{v}_{r}(y)\bigr)\bigr\rvert _{\Frob}^{2}\,\dif y\,\dif r\\
                                           & \le\frac{\Xi_{2;\tau_1}[v_0]^2}{\sfL(1/\rho)}\log\frac{t+\rho}{t+\rho-\tau_{1}}+\frac{\lambda^{2}}{\sfL(1/\rho)}\int_{\tau_{1}}^{t}\frac{\vvvert v_{r}-\widetilde{v}_{r}\vvvert_{2}^{2}}{t+\rho-r}\,\dif r,
  \end{align*}
  where in the last inequality we used the Lipschitz property of $\sigma$.
  If we define
  \begin{equation}
    f(\tau)\coloneqq\vvvert v_{\tau_{1}+\tau}-\widetilde{v}_{\tau_{1}+\tau}\vvvert_{2}^{2} \quad \text{for } \tau \in [0, \tau_2 - \tau_1],\label{eq:fdef}
  \end{equation}
  this means that
  \[
    f(\tau)\le\frac{\Xi_{2;\tau_1}[v_0]^2}{\sfL(1/\rho)} \sfL\left(\frac{\tau_1}{\tau + \rho}\right) + \frac{\lambda^{2}}{\sfL(1/\rho)}\int_{0}^{\tau}\frac{f(r)}{\tau+\rho -r}\,\dif r.
  \]
  Therefore, by \cref{lem:turnoffnoiselemma} below (with $A\setto\frac{\Xi_{2;\tau_{1}}[v_{0}]^{2}}{\sfL(1/\rho)}$
  and $\zeta\setto\tau_{1}$), we obtain
  \begin{align*}
    \vvvert v_{\tau_{2}}-\widetilde{v}_{\tau_{2}}\vvvert_{2}^{2}\overset{\cref{eq:fdef}}{=}f(\tau_{2}-\tau_{1}) & \le\frac{\Xi_{2;\tau_{1}}[v_{0}]^{2}}{\sfL(1/\rho)}K_{\lambda,\rho,\tau_{2}-\tau_{1}}\left[\sfL\left(\frac{\tau_{1}}{\tau_{2}-\tau_{1}+\rho}\right)+\sfL(2)K_{\lambda,\rho,\tau_{2}-\tau_{1}}\right]       \\
                                                                                                                & =\Xi_{2;\tau_{1}}[v_{0}]^{2}K_{\lambda,\rho,\tau_{2}-\tau_{1}}\left[\sfL_{\rho}\left(\frac{\tau_{1}}{\tau_{2}-\tau_{1}+\rho}\right)+\frac{\sfL(2)K_{\lambda,\rho,\tau_{2}-\tau_{1}}}{\sfL(1/\rho)}\right],
  \end{align*}
  completing the proof.
\end{proof}
\begin{lem}
  \label{lem:turnoffnoiselemma}Fix $A,\zeta,T\in[0,\infty)$. Let $f:[0,T]\to[0,\infty)$
  be a bounded function such that, for all $t\in[0,T]$, we have the
  bound
  \begin{equation}
    f(t)\le A\sfL\left(\frac{\zeta}{t+\rho}\right)+\frac{\lambda^{2}}{\sfL(1/\rho)}\int_{0}^{t}\frac{f(s)}{t+\rho -s}\,\dif s.\label{eq:frecursive}
  \end{equation}
  Then we have for all
  \begin{equation}
    t\in\bigl[0,T\wedge\overline{T}_{\rho}(\lambda)\bigr)\label{eq:tnottoobig}
  \end{equation}
  that
  \[
    f(t)\le AK_{\lambda,\rho,t}\left[\sfL\left(\frac{\zeta}{t+\rho}\right)+\sfL(2)K_{\lambda,\rho,t}\right].
  \]
\end{lem}

\begin{proof}
  The proof is essentially the same as that of \cite[Lemma 4.3]{DG22}.
  We make the \emph{ansatz}
  \begin{equation}
    f(t)\le B_{1}\sfL\left(\frac{\zeta}{t+\rho}\right)+B_{2}\qquad\text{for all }t\in\bigl[0,T\wedge\overline{T}_{\rho}(\lambda)\bigr).\label{eq:fansatz}
  \end{equation}
  Suppose for now that \cref{eq:fansatz} holds for some values of $B_{1}$
  and $B_{2}$. We plug \cref{eq:fansatz} into \cref{eq:frecursive} to
  get
  \begin{align}
    f(t) & \le A\sfL\left(\frac{\zeta}{t+\rho}\right)+\frac{\lambda^{2}}{\sfL(1/\rho)}\int_{0}^{t}\frac{B_{1}\sfL\bigl(\zeta(s+\rho)^{-1}\bigr)+B_{2}}{t+\rho-s}\,\dif s\nonumber                                            \\
         & =A\sfL\left(\frac{\zeta}{t+\rho}\right)+\frac{B_{1}\lambda^{2}}{\sfL(1/\rho)}\int_{0}^{t}\frac{\sfL\bigl(\zeta(s+\rho)^{-1}\bigr)}{t+\rho-s}\,\dif s+B_{2}\lambda^{2}\sfS_{\rho}(t).\label{eq:readyformiddleterm}
  \end{align}
  We divide the integral in \cref{eq:readyformiddleterm} into two parts.
  For the first half of the integration domain we write
  \begin{align*}
    \int_{0}^{t/2} & \frac{\sfL\bigl(\zeta(s+\rho)^{-1}\bigr)}{t+\rho-s}\,\dif s\le\frac{1}{t/2+\rho}\int_{0}^{t/2}\log\left(\frac{s+\rho+\zeta}{s+\rho}\right)\,\dif s\le\frac{1}{t/2+\rho}\int_{0}^{t}\log\left(\frac{t+\rho+\zeta}{s+\rho}\right)\,\dif s \\
                   & =\frac{1}{t/2+\rho}\left(t\left[1+\sfL\left(\frac{\zeta}{t+\rho}\right)\right]-\rho\sfL(t/\rho)\right)\le2+2\sfL\left(\frac{\zeta}{t+\rho}\right).
  \end{align*}
  For the second half we write
  \[
    \int_{t/2}^{t}\frac{\sfL\bigl(\zeta(s+\rho)^{-1}\bigr)}{t+\rho-s}\,\dif s\le\sfL\left(\frac{\zeta}{t/2+\rho}\right)\int_{t/2}^{t}\frac{\dif s}{t+\rho-s}\le\left[\sfL\left(\frac{\zeta}{t+\rho}\right)+\sfL(2)\right]\sfL(t/\rho),
  \]
  with the last inequality by \cref{eq:Lmult}. Combining the last two
  displays (and using that $\sfL(2)\ge1$) we obtain
  \begin{equation}
    \label{eq:two-halves}
    \int_{0}^{t}\frac{\sfL\bigl(\zeta(s+\rho)^{-1}\bigr)}{t+\rho-s}\,\dif s \le\left[\sfL\left(\frac{\zeta}{t+\rho}\right)+\sfL(2)\right]\left[\sfL(t/\rho)+2\right].
  \end{equation}
  Define
  \begin{equation}
    \alpha\coloneqq\lambda^{2}\left[\sfS_{\rho}(t)+\frac{2}{\sfL(1/\rho)}\right],\label{eq:alphadef}
  \end{equation}
  which satisfies $\al < 1$ by \cref{eq:tnottoobig} (recalling \cref{eq:mtbardef}).
  Using \cref{eq:two-halves} in \cref{eq:readyformiddleterm}, we then obtain
  \begin{align*}
    f(t) & \le A\sfL\left(\frac{\zeta}{t+\rho}\right)+\al B_{1}\left[\sfL\left(\frac{\zeta}{t+\rho}\right)+\sfL(2)\right]+B_{2}\lambda^{2}\sfS_{\rho}(t) \le(A+\alpha B_{1})\sfL\left(\frac{\zeta}{t+\rho}\right)+\sfL(2)B_1+\alpha B_{2}.
  \end{align*}
  Thus if \cref{eq:fansatz} holds with $B_{1}\setto B_{1}^{(n)}$
  and $B_{2}\setto B_{2}^{(n)}$, then it also holds with
  \begin{align}
    B_{1}\setto B_{1}^{(n+1)} & \coloneqq A+\alpha B_{1}^{(n)},\label{eq:B1recursion}  \\
    B_{2}\setto B_{2}^{(n+1)} & \coloneqq\sfL(2)B_{1}^{(n)}+\alpha B_{2}^{(n)}.\label{eq:B2recursion}
  \end{align}
  Since we assumed that $f$ is bounded, it is clearly possible to
  choose $B_{2}^{(0)}$ finite such that \cref{eq:fansatz} holds with
  $B_{1}\setto0$ and $B_{2}\setto B_{2}^{(0)}$.
  Then from \cref{eq:B1recursion} we see that
  \begin{equation}
    B_{1}^{(n+1)}=A\sum_{k=0}^{n+1}\alpha^{k}<AK_{\lambda,\rho,t},\label{eq:B1bd}
  \end{equation}
  recalling the definitions \cref{eq:alphadef} and \cref{eq:Kdef-1}. Using
  \cref{eq:B1bd} in \cref{eq:B2recursion}, we get
  \[
    B_{2}^{(n+1)}\le AK_{\lambda,\rho,t}\sfL(2)+\alpha B_{2}^{(n)},
  \]
  so (again summing a geometric series) we see that
  \begin{equation}
    \limsup_{n\to\infty}B_{2}^{(n)}\le AK_{\lambda,\rho,t}^{2}\sfL(2).\label{eq:B2to0}
  \end{equation}
  Using \cref{eq:B1bd} and \cref{eq:B2to0} in \cref{eq:fansatz} completes
  the proof.
\end{proof}

\subsubsection{Flattening}

Given $X\in\RR^{2}$, we define the operator $\clZ_{X}$ by
\nomenclature[Z cl]{$\clZ_X$}{``flattens'' a field to its value at $X \in \R^2$, \cref{eq:Zcaldef}}
\begin{equation}
  (\clZ_{X}v)(x)=v(X)\qquad\text{for all }x\in\RR^{2}.\label{eq:Zcaldef}
\end{equation}
So $\clZ_{X}$ ``flattens'' the field $v$ to its value at the spatial location $X$.

We now compare the two operators $\clV_{\tau_{1},\tau_{2}}\clG_{\tau_{1}-\tau_{0}}$ and $\clV_{\tau_{1},\tau_{2}}\clZ_{X}\clG_{\tau_{1}-\tau_{0}}$.
The first turns off the noise on the quiet interval $[\tau_0,\tau_1]$, as discussed in the previous section.
In the second operator, the field is moreover flattened to its value at $X$ after the quiet dynamics.
For appropriate parameter choices, this should have little effect on the solution close to $X$, because the heat equation smooths the small scales of $v$.

The following estimate is an analogue of \cite[Proposition 5.1]{DG22}.
\begin{prop}
  \label{prop:spatialbd}
  Suppose $\lambda\in[0,\infty)$, $\sigma\in\lipset(\lambda)$, and $\tau_{0}\le\tau_{1}<\tau_{2}$ satisfy $\tau_{2}-\tau_{1}<\overline{T}_{\rho}(\lambda)$.
  Then, for all $x,X\in\RR^{2}$ and all $v_{\tau_{0}}\in\scrX_{\tau_{0}}^{2}$,
  we have
  \begin{equation}
    \begin{aligned}\EE |(\clV_{\tau_{1},\tau_{2}}\clG_{\tau_{1}-\tau_{0}}v_{\tau_{0}}&-\clV_{\tau_{1},\tau_{2}}\clZ_{X}\clG_{\tau_{1}-\tau_{0}}v_{\tau_{0}})(x)|^{2}                                                                                                                \\
                   & \le K_{\lambda,\rho,\tau_{2}-\tau_{1}}\vvvert v_{\tau_0}\vvvert_{2}^{2}\left[\frac{\tau_{2}-\tau_{1}+\frac{1}{2}|x-X|^{2}}{\tau_{1}-\tau_{0}}+\frac{\lambda^{2}(\tau_{2}-\tau_{1})K_{\lambda,\rho,\tau_{2}-\tau_{1}}}{2(\tau_{1}-\tau_{0})\sfL(1/\rho)}\right].
    \end{aligned}
    \label{eq:spatialbd}
  \end{equation}

\end{prop}

\begin{proof}
  Assume without loss of generality that $\tau_{0}=0$.
  For all $\tau\in[0,\tau_{2}-\tau_{1}]$, we have
  \begin{align*}
    (\clV_{\tau_{1},\tau_{1}+\tau} & \clG_{\tau_{1}}v_{0}-\clV_{\tau_{1},\tau_{1}+\tau}\clZ_{X}\clG_{\tau_{1}}v_{0})(x)                                                                                                                                            \\
                                   & =\int[G_{\tau_{1}+\tau}(x-y)-G_{\tau_{1}}(X-y)]v_0(y)\,\dif y                                                                                                                                                               \\
                                   & \qquad+\gamma_{\rho}\int_{0}^{\tau}\clG_{\tau+\rho-r}\left(\left[\sigma(\clV_{\tau_{1},\tau_{1}+r}\clG_{\tau_{1}}v_{0})-\sigma(\clV_{\tau_{1},\tau_{1}+r}\clZ_{X}\clG_{\tau_{1}}v_{0})\right]\,\dif W_{\tau_{1}+r}\right)(x).
  \end{align*}
  Taking second moments, we obtain
  \begin{align}
    \EE & |(\clV_{\tau_{1},\tau_{1}+\tau}\clG_{\tau_{1}}v_{0}-\clV_{\tau_{1},\tau_{1}+\tau}\clZ_{X}\clG_{\tau_{1}}v_{0})(x)|^{2}\nonumber                                                                                                                               \\
        & \le\iint\EE \prod_{i=1}^{2}\left(|G_{\tau_{1}+\tau}(x-y_{i})-G_{\tau_{1}}(X-y_{i})||v_{s}(y_{i})|\right)\,\dif y_{1}\,\dif y_{2}\nonumber                                                                                                                    \\
        & \qquad+\frac{4\pi}{\sfL(1/\rho)}\int_{0}^{\tau}\!\!\!\!\int\EE \left\lvert  \sigma\bigl(\clV_{\tau_{1},\tau_{1}+r}\clG_{\tau_{1}}v_{0}(y)\bigr)-\sigma\bigl(\clV_{\tau_{1},\tau_{1}+r}\clZ_{X}\clG_{\tau_{1}}v_{0}(y)\bigr)\right\rvert  _{\Frob}^{2}G_{\tau+\rho-r}^{2}(x-y)\,\dif y\,\dif r\nonumber \\
        & \eqqcolon I_{1}+I_{2}.\label{eq:splitupfreezeintegral}
  \end{align}
  To estimate $I_{1}$, we use the Cauchy--Schwarz inequality on the
  probability space to write
  \begin{align}
    I_{1} & \le\iint\prod_{i=1}^{2}\left(|G_{\tau_{1}+\tau}(x-y_{i})-G_{\tau_{1}}(X-y_{i})|\left(\EE |v_{0}(y_{i})|^{2}\right)^{1/2}\right)\,\dif y_{1}\,\dif y_{2}\nonumber                                  \\
          & \le\vvvert v_{0}\vvvert_{2}^{2}\|G_{\tau_{1}+\tau}(x-\anon)-G_{\tau_{1}}(X-\anon)\|_{L^{1}(\RR^{2})}^{2}\le\frac{\tau+\frac{1}{2}|X-x|^{2}}{\tau_{1}}\vvvert v_{0}\vvvert_{2}^{2},\label{eq:I1bd}
  \end{align}
  with the last inequality by \cref{eq:L1ofgaussians}.

  To estimate $I_{2}$, we use the Lipschitz property of $\sigma$ and
  \cref{eq:GT2} to write
  \begin{equation}
    I_{2}\le\frac{\lambda^{2}}{\sfL(1/\rho)}\int_{0}^{\tau}\!\!\!\!\int\frac{G_{\frac{1}{2}[\tau+\rho-r]}(x-y)f(r,y)}{\tau+\rho-r}\,\dif y\,\dif r,\label{eq:I2bd}
  \end{equation}
  where we have defined
  \[
    f(r,y)\coloneqq\EE |(\clV_{\tau_{1},\tau_{1}+r}\clG_{\tau_{1}}v_{0}-\clV_{\tau_{1},\tau_{1}+r}\clZ_{X}\clG_{\tau_{1}}v_{0})(y)|^{2}.
  \]
  Now we plug \cref{eq:I1bd} and \cref{eq:I2bd} into \cref{eq:splitupfreezeintegral}
  to obtain, for all $\tau\in[\tau_1,\tau_2]$,
  \begin{align*}
    f(\tau,x) & \le\frac{\tau+\frac{1}{2}|X-x|^{2}}{\tau_{1}}\vvvert v_{0}\vvvert_{2}^{2}+\frac{\lambda^{2}}{\sfL(1/\rho)}\int_{0}^{\tau}\!\!\!\!\int\frac{G_{\frac{1}{2}[\tau+\rho-r]}(x-y)f(r,y)}{\tau+\rho-r}\,\dif y\,\dif r              \\
              & \le\frac{\tau_{1}-\tau_{2}+\frac{1}{2}|X-x|^{2}}{\tau_{1}}\vvvert v_{0}\vvvert_{2}^{2}+\frac{\lambda^{2}}{\sfL(1/\rho)}\int_{0}^{\tau}\!\!\!\!\int\frac{G_{\frac{1}{2}[\tau+\rho-r]}(x-y)f(r,y)}{\tau+\rho-r}\,\dif y\,\dif r.
  \end{align*}
  Then we can apply \cref{lem:fspatialbd} below with $T\setto\tau_{2}-\tau_{1}$,
  \[
    A_{1}\setto\frac{\tau_{2}-\tau_{1}}{\tau_{1}}\vvvert v_{0}\vvvert_{2}^{2}\qquad\text{and}\qquad A_{2}\setto\frac{1}{2\tau_{1}}\vvvert v_{0}\vvvert_{2}^{2}
  \]
  to obtain
  \[
    f(\tau_{2}-\tau_{1},x)\le K_{\lambda,\rho,\tau_{2}-\tau_{1}}\vvvert v_{0}\vvvert_{2}^{2}\left[\frac{\tau_{2}-\tau_{1}+\frac{1}{2}|x-X|^{2}}{\tau_{1}}+\frac{\lambda^{2}(\tau_{2}-\tau_{1})K_{\lambda,\rho,\tau_{2}-\tau_{1}}}{2\tau_{1}\sfL(1/\rho)}\right],
  \]
  which implies \cref{eq:spatialbd}.
\end{proof}
\begin{lem}
  \label{lem:fspatialbd}Fix $X\in\RR^{2}$ and $A_{1},A_{2},\lambda,T\in[0,\infty)$
  such that $T<\overline{T}_{\rho}(\lambda)$. Let $f:[0,T]\times\RR^{2}\to[0,\infty)$
  be bounded and satisfy
  \begin{equation}
    f(\tau,x)\le A_{1}+A_{2}|x-X|^{2}+\frac{\lambda^{2}}{\sfL(1/\rho)}\int_{0}^{\tau}\!\!\!\!\int\frac{G_{\frac{1}{2}[\tau+\rho-r]}(x-y)f(r,y)}{\tau+\rho-r}\,\dif y\,\dif r.\label{eq:fspatialrecursion}
  \end{equation}
  for all $\tau \in [0, T]$.
  Then
  \begin{equation}
    f(T,x)\le K_{\lambda,\rho,T}\left[A_{1}+A_{2}|x-X|^{2}+\frac{\lambda^{2}TK_{\lambda,\rho,T}A_{2}}{\sfL(1/\rho)}\right].\label{eq:fspatialbd}
  \end{equation}
\end{lem}

\begin{proof}
  The proof is the same as that of \cite[Lemma 5.2]{DG22}. We start
  with the \emph{ansatz}
  \begin{equation}
    f(\tau,x)\le B_{1}+B_{2}|x-X|^{2}.\label{eq:fspatialansantz}
  \end{equation}
  We assume for the moment that \cref{eq:fspatialansantz} holds for some
  $B_{1}$ and $B_{2}$. Using \cref{eq:fspatialansantz} in \cref{eq:fspatialrecursion},
  and recalling \cref{eq:Srhodef}, we obtain
  \begin{align}
    f(\tau,x) & \le A_{1}+A_{2}|x-X|^{2}+\frac{\lambda^{2}}{\sfL(1/\rho)}\int_{0}^{\tau}\!\!\!\!\int\frac{G_{\frac{1}{2}[\tau+\rho-r]}(x-y)[B_{1}+B_{2}|y-X|^{2}]}{\tau+\rho-r}\,\dif y\,\dif r\nonumber                                       \\
              & =A_{1}+A_{2}|x-X|^{2}+B_{1}\lambda^{2}\sfS_{\rho}(\tau)+\frac{B_{2}\lambda^{2}}{\sfL(1/\rho)}\int_{0}^{\tau}\frac{\int G_{\frac{1}{2}[\tau+\rho-r]}(x-y)|y-X|^{2}\,\dif y}{\tau+\rho-r}\,\dif r.\label{eq:applyansatz}
  \end{align}
  We can write
  \begin{align*}
    \int G_{\frac{1}{2}[\tau+\rho-r]}(x-y)|y-X|^{2}\,\dif y & =|x-X|^{2}+\int G_{\frac{1}{2}[\tau+\rho-r]}(x-y)|y-x|^{2}\,\dif y \\
                                                            & =|x-X|^{2}+\tau+\rho-r,
  \end{align*}
  where in the second identity we recalled that we are working in spatial
  dimension $2$. Substituting this into \cref{eq:applyansatz} and again
  recalling \cref{eq:Srhodef}, we get
  \begin{align*}
    f(\tau,x) & \le A_{1}+B_{1}\lambda^{2}\sfS_{\rho}(\tau)+\frac{B_{2}\lambda^{2}\tau}{\sfL(1/\rho)}+ [A_{2}+\lambda^{2}\sfS_{\rho}(\tau)B_{2}]|x-X|^{2} \\
              & \le A_{1}+B_{1}\lambda^{2}\sfS_{\rho}(T)+\frac{B_{2}\lambda^{2}T}{\sfL(1/\rho)}+[A_{2}+\lambda^{2}\sfS_{\rho}(T)B_{2}]|x-X|^{2}.
  \end{align*}

  We have thus shown that if \cref{eq:fspatialansantz} holds with $B_{1}\setto B_{1}^{(n)}$
  and $B_{2}\setto B_{2}^{(n)}$, then it also holds with
  \begin{align}
    B_{1}\setto B_{1}^{(n+1)} & \coloneqq A_{1}+\lambda^{2}\sfS_{\rho}(T)B_{1}^{(n)}+\frac{\lambda^{2}T B_{2}^{(n)}}{\sfL(1/\rho)},\label{eq:B1spatial} \\
    B_{2}\setto B_{2}^{(n+1)} & \coloneqq A_{2}+\lambda^{2}\sfS_{\rho}(T)B_{2}^{(n)}.\label{eq:B2spatial}
  \end{align}
  Since $f$ is bounded above, we can take $B_{2}^{(0)}=0$ and $B_{1}^{(0)}$
  sufficiently large. Then \cref{eq:B2spatial} implies that
  \[
    B_{2}^{(n)}\le K_{\lambda,\rho,T}A_{2}\qquad\text{for all }n\ge0.
  \]
  Plugging this into \cref{eq:B1spatial}, we obtain
  \[
    B_{1}^{(n+1)}\le A_{1}+\lambda^{2}\sfS_{\rho}(T)B_{1}^{(n)}+\frac{K_{\lambda,\rho,T}A_{2}\lambda^{2}T}{\sfL(1/\rho)},
  \]
  which means that
  \[
    B_{1}^{(n)}\le K_{\lambda,\rho,T}\left[A_{1}+\frac{K_{\lambda,\rho,T}A_{2}\lambda^{2}T}{\sfL(1/\rho)}\right]\qquad\text{for all }n\ge0.
  \]
  Using these inequalities in \cref{eq:fansatz}, we get \cref{eq:fspatialbd}.
\end{proof}
Combining \cref{prop:turnoffnoise,prop:spatialbd}, we can collect the effects of turning off the noise and flattening.
\begin{prop}
  \label{prop:flattenic}
  Suppose $\lambda\in(0,\infty)$, $\sigma\in\lipset(\lambda)$,
  $x,X\in\RR^{2}$, and $\tau_{0}\le\tau_{1}<\tau_{2}$ satisfy $\tau_{2}-\tau_{1} < \overline{T}_{\rho}(\lambda)$.
  Then for any $v_{\tau_{0}}\in\scrX_{\tau_{0}}^{2}$, we have
  \begin{align*}
     \EE |\clV_{\tau_{0},\tau_{2}}v_{\tau_{0}}(x)-&\clV_{\tau_{1},\tau_{2}}\clZ_{X}\clG_{\tau_{1}-\tau_{0}}v_{\tau_{0}}(x)|^{2}\\
     & \le4K_{\lambda,\rho,\tau_{2}-\tau_{1}}\Xi_{2;\tau_{0},\tau_{1}}[v_{\tau_{0}}]^{2}\left[\sfL_{\rho}\left(\frac{\tau_{1}-\tau_{0}}{\tau_{2}-\tau_{1}+\rho}\right)+\frac{(1+\lambda^{2}\delta)(\tau_{2}-\tau_{1})+|x-X|^{2}}{\tau_{1}-\tau_{0}}+\delta\right],
  \end{align*}
  where $\delta=K_{\lambda,\rho,\tau_{2}-\tau_{1}}\sfL(1/\rho)^{-1}$.
\end{prop}

In the sequel, we will apply \cref{prop:flattenic} in the regime when the right side is small, so when $(\tau_2-\tau_1)/(\tau_1-\tau_0)$ is both $o(1)$ and $\rho^{o(1)}$, and also $|x-X|^2/(\tau_1-\tau_0) =o(1)$. See \cite[§6.1]{DG22}.

\subsection{Local-in-space dependence of the solution on the noise}

In this section we show that a solution to our SPDE at a point only
depends on the noise in a spatial neighborhood of the point. We follow
\cite[Section 9.1]{DG22}. For $\ttB\subseteq\RR^{2}$, recall
\cref{def:turnoffnoise} of $\dif\mathring{W}_{t}^{\ttB}$
and $\mathring{\clV}_{s,t}^{\ttB}$, in which the noise is restricted
to the set $\ttB$. Also, let $\dif W_{t}^{\ttB}$ be a space-time white
noise such that $\dif W_{t}^{\ttB}|_{\ttB}=\dif W_{t}|_{\ttB}$ but $\dif W_{t}^{\ttB}|_{\ttB^{\cc}}$
is independent of everything else (including of $\dif W_{t}^{\widetilde{\ttB}}|_{\widetilde{\ttB}^{\cc}}$
for other subsets $\widetilde{\ttB}$). We let $\clV_{s,t}^{\ttB}$
be identical to the propagator
$\clV_{s,t}$ for the SPDE \cref{eq:SPDE}, but with the noise
$\dif W_{t}$ replaced by $\dif W_{t}^{\ttB}$. %
\nomenclature[V cl B st]{$\clV^{\sigma,\rho,\ttB}_{s,t}$}{propagator for the SPDE with the noise resampled on $\ttB^\cc$ ($\sigma,\rho$ often omitted)}

The main goal of this section is to prove the following proposition.
\begin{prop}
  \label{prop:resamplenoiseoffrectangle}There is an absolute constant
  $C<\infty$ such that the following holds. Let $\lambda\in(0,\infty)$
  and suppose $\sigma\in\lipset(\lambda)$. Let $\ttR\subseteq\RR^{2}$
  be a rectangle. Then for all $x\in\ttR$ and $\tau_{0}< \tau_{1}$
  such that $\tau_{1}-\tau_{0} < \overline{T}_{\rho}(\lambda)$, we have
  \begin{equation}
  \begin{aligned}
    \lambda^{2}\EE & |(\clV_{\tau_{0},\tau_{1}}v_{\tau_{0}}-\clV_{\tau_{0},\tau_{1}}^{\ttR}v_{\tau_{0}})(x)|^{2}                                                                                                                               \\
                   & \le C\Xi_{2;\tau_{0},\tau_{1}}[v_{\tau_{0}}]^{2}K_{\lambda,\rho,\tau_{1}-\tau_{0}}\exp\left\{ -\frac{1}{2}\xi[(\tau_{1}-\tau_{0})\vee\rho]^{-1/2}\left(1\wedge\log\frac{1}{\lambda^{2}\sfS_{\rho}(\tau_{1}-\tau_{0})}\right)\right\},
\end{aligned}\label{eq:resamplenoiseoffrectangle}
\end{equation}
  where $\xi=\dist(x,\ttR^{\cc})$.
\end{prop}
The reason for the $\lambda^{2}$ on the left side is that, by a happy
coincidence, this factor simplifies both the proof and the application.
We note that the right side of \cref{eq:resamplenoiseoffrectangle} is small when $\xi^2\gg \rho\vee(\tau_1-\tau_0)$ and $\sfS_\rho(\tau_1-\tau_0)<\lambda^{-2}$.
The key step in the proof of \cref{prop:resamplenoiseoffrectangle} is the following.
\begin{prop}
  \label{prop:turnoffhyperplane}There is an absolute constant $C<\infty$
  so that the following holds. Let $\lambda\in(0,\infty)$ and
  that $\sigma\in\lipset(\lambda)$. Let $\ttB\subseteq\RR^{2}$ and let
  $\ttH$ be a half-plane in $\RR^{2}$. 
  Let $\tau_{0} < \tau_{1}$ satisfy
  \begin{equation}
    \tau_{1}-\tau_{0} < \overline{T}_{\rho}(\lambda).\label{eq:tau1tau0smallenough}
  \end{equation}
  Then for all $x\in \ttH^{\cc}$ and $v_{0}\in\scrX_{\tau_{0}}^{2}$, we have
  \begin{equation}
    \begin{aligned}\lambda^{2}\EE  |(\mathring{\clV}_{\tau_{0},\tau_{1}}^{\ttB}&v_{\tau_{0}}-\mathring{\clV}_{\tau_{0},\tau_{1}}^{\ttB\setminus \ttH}v_{\tau_{0}})(x)|^{2}                                                                                                           \\
                              & \le C\Xi_{2;\tau_{0},\tau_{1}}[v_{\tau_{0}}]^{2}K_{\lambda,\rho,\tau_{1}-\tau_{0}}\exp\left\{ -\frac{1}{2}\xi[(\tau_{1}-\tau_{0})\vee\rho]^{-1/2}\left(1\wedge\log\frac{1}{\lambda^{2}\sfS_{\rho}(\tau_{1}-\tau_{0})}\right)\right\} ,
    \end{aligned}
    \label{eq:turnoffspatialnoise}
  \end{equation}
  where $\xi=\dist(x,\ttH)$.
\end{prop}

Before we prove \cref{prop:turnoffhyperplane}, we show how to use it
to prove \cref{prop:resamplenoiseoffrectangle}.
\begin{proof}[Proof of \cref{prop:resamplenoiseoffrectangle}.]
  We can write the complement of $\ttR$ as the union of four half-spaces
  $\ttH_{1}$, $\ttH_{2}$, $\ttH_{3}$, and $\ttH_{4}$. Then $\xi\coloneqq\dist(x,\ttR^{c})=\min_i \dist(x,\ttH_{i})$.
  Applying \cref{prop:turnoffhyperplane} four times and using the triangle
  inequality, we can bound $\EE |(\clV_{\tau_{0},\tau_{1}}v_{\tau_{0}}-\mathring{\clV}_{\tau_{0},\tau_{1}}^{\ttR}v_{\tau_{0}})(x)|^{2}$
  by $16$ times the right side of \cref{eq:turnoffspatialnoise}. In
  the same way, we can bound $\EE |(\clV_{\tau_{0},\tau_{1}}^{\ttR}v_{\tau_{0}}-\mathring{\clV}_{\tau_{0},\tau_{1}}^{\ttR}v_{\tau_{0}})(x)|^{2}$
  by the same quantity. Then we conclude by the triangle inequality again.
\end{proof}
\begin{proof}[Proof of \cref{prop:turnoffhyperplane}.]
  The proof uses arguments similar to those used in the proofs of \cite[Lemmas 9.2 and 9.3]{DG22},
  with adjustments because here we consider the vector-valued case and do not assume that $\sigma(0)=0$.
  Assume without loss of generality that $\tau_{0}=0$. We divide the
  proof into two steps.
  \begin{thmstepnv}
    \item \emph{A recursive estimate on the difference.} Subtracting the respective
    mild solution formulas, we have for all $t\in[0,\tau_{1}]$ that
    \begin{align*}
      (\mathring{\clV}_{t}^{\ttB}v_{0}-\mathring{\clV}_{t}^{\ttB\setminus \ttH}v_{0})(x) & =\gamma_{\rho}\int_{0}^{t}\!\!\!\!\int_{\ttB\cap \ttH}G_{t+\rho-r}(x-y)\sigma\bigl(\mathring{\clV}_{r}^{\ttB}v_{0}(y)\bigr)\,\dif W_{r}(y)                                                                             \\
                                                                                & \qquad+\gamma_{\rho}\int_{0}^{t}\!\!\!\!\int_{\ttB\setminus \ttH}G_{t+\rho-r}(x-y)\left[\sigma\bigl(\mathring{\clV}_{r}^{\ttB}v_{0}(y)\bigr)-\sigma\bigl(\mathring{\clV}_{r}^{\ttB\setminus \ttH}v_{0}(y)\bigr)\right]\,\dif W_{r}(y).
    \end{align*}
    We can take second moments to obtain
    \begin{align}
      \EE |( & \mathring{\clV}_{t}^{\ttB}v_{0}-\mathring{\clV}_{t}^{\ttB\setminus \ttH}v_{0})(x)|^{2}\nonumber                                                                                                                                                           \\
             & =\frac{4\pi}{\sfL(1/\rho)}\int_{0}^{t}\!\!\!\!\int_{\ttB\cap \ttH}G_{t+\rho-r}^{2}(x-y)\EE \bigl\lvert \sigma\bigl(\mathring{\clV}_{r}^{\ttB}v_{0}(y)\bigr)\bigr\rvert _{\Frob}^{2}\,\dif y\,\dif r\nonumber                                                                                      \\
             & \qquad+\frac{4\pi}{\sfL(1/\rho)}\int_{0}^{t}\!\!\!\!\int_{\ttB\setminus \ttH}G_{t+\rho-r}^{2}(x-y)\EE \bigl\lvert  \sigma\bigl(\mathring{\clV}_{r}^{\ttB}v_{0}(y)\bigr)-\sigma\bigl(\mathring{\clV}_{r}^{\ttB\setminus \ttH}v_{0}(y)\bigr)\bigr\rvert  _{\Frob}^{2}\,\dif y\,\dif r\nonumber             \\
             & \le\frac{\Xi_{2;t}[v_{0}]^{2}}{\sfL(1/\rho)}\int_{0}^{t}\frac{\clG_{\frac{1}{2}[t+\rho-r]}\mathbf{1}_{\ttH}(x)}{t+\rho-r}\,\dif r\nonumber                                                                                                          \\
             & \qquad+\frac{\lambda^{2}}{\sfL(1/\rho)}\int_{0}^{t}\frac{\int G_{\frac{1}{2}[t+\rho-r]}(x-y)\EE |\mathring{\clV}_{r}^{\ttB}v_{0}(y)-\mathring{\clV}_{r}^{\ttB\setminus \ttH}v_{0}(y)|^{2}\,\dif y}{t+\rho-r}\,\dif r.\label{eq:takesecondmonents-spatial}
    \end{align}
    Now by \cite[(9.10)]{DG22} we have, for $x\in \ttH^{\cc}$ and
    $0\le r\le t$, that
    \[
      \clG_{\frac{1}{2}[t+\rho-r]}\mathbf{1}_{\ttH}(x)\le\clG_{\frac{1}{2}[t+\rho]}\mathbf{1}_{\ttH}(x).
    \]
    This reflects the fact that a Gaussian with smaller variance puts
    its mass closer to the origin. Using this in \cref{eq:takesecondmonents-spatial},
    we obtain, for all $x\in \ttH^{\cc}$,
    \begin{equation}
      f_{t}(x)\le\clG_{\frac{1}{2}[t+\rho]}\mathbf{1}_{\ttH}(x)\lambda^{2}\sfS_{\rho}(t)\Xi_{2;t}[v_{0}]^{2}+\frac{\lambda^{2}}{\sfL(1/\rho)}\int_{0}^{t}\frac{\clG_{\frac{1}{2}[t+\rho-r]}f_{r}(x)}{t+\rho-r}\,\dif r,\label{eq:fturnoffrecursive}
    \end{equation}
    where we have defined
    \begin{equation}
      f_{t}(x)\coloneqq\lambda^{2}\EE |\mathring{\clV}_{t}^{\ttB}v_{0}(x)-\mathring{\clV}_{t}^{\ttB\setminus \ttH}v_{0}(x)|^{2}.\label{eq:fturnoffdef}
    \end{equation}
  \item \emph{Analysis of the recursive estimate.} We now seek to derive a
    non-recursive bound from \cref{eq:fturnoffrecursive}. Define
    \begin{equation}
      \alpha\coloneqq\lambda^{2}\sfS_{\rho}(\tau_{1})\overset{\cref{eq:tau1tau0smallenough}}{<}1\label{eq:alphadef-1}
    \end{equation}
    and, for $k\ge1$,
    \begin{equation}
      B_{2}^{(k)}\coloneqq\alpha^{k}\Xi_{2;\tau_{1}}[v_{0}]^{2}.\label{eq:B2kdef}
    \end{equation}
    Now suppose that, for all $t\in[0,\tau_{1}]$, we have
    \begin{equation}
      f_{t}(x)\le B_{1}^{(n)}+\sum_{k=1}^{n}B_{2}^{(k)}\clG_{\frac{1}{2}(t+k\rho)}\mathbf{1}_{\ttH}(x)\qquad\text{for all }x\in \ttH^{\cc}.\label{eq:fturnoffansatz}
    \end{equation}
    For $n=0$ this can be achieved by choosing $B_{1}^{(0)}$ sufficiently
    large. Now we plug \cref{eq:fturnoffansatz} into \cref{eq:fturnoffrecursive}
    to obtain
    \begin{align*}
      f_{t}(x) & \le\clG_{\frac{1}{2}[t+\rho]}\mathbf{1}_{\ttH}(x)\alpha\Xi_{2;\tau_{1}}[v_{0}]^{2}+\frac{\lambda^{2}}{\sfL(1/\rho)}\int_{0}^{t}\frac{\clG_{\frac{1}{2}[t+\rho-r]}\left(B_{1}^{(n)}+\sum\limits _{k=1}^{n}B_{2}^{(k)}\clG_{\frac{1}{2}(r+k\rho)}\mathbf{1}_{\ttH}\right)(x)}{t+\rho-r}\,\dif r \\
               & \le\clG_{\frac{1}{2}[t+\rho]}\mathbf{1}_{\ttH}(x)\alpha\Xi_{2;\tau_{1}}[v_{0}]^{2}+\lambda^{2}\sfS_{\rho}(\tau_{1})B_{1}^{(n)}+\lambda^{2}\sfS_{\rho}(\tau_{1})\sum\limits _{k=1}^{n}B_{2}^{(k)}\clG_{\frac{1}{2}[t+(k+1)\rho]}\mathbf{1}_{\ttH}(x)                                           \\
               & =\alpha B_{1}^{(n)}+\sum\limits _{k=1}^{n+1}B_{2}^{(k)}\clG_{\frac{1}{2}[t+k\rho]}\mathbf{1}_{\ttH}(x),
    \end{align*}
    where in the last identity we used \cref{eq:alphadef-1} and \cref{eq:B2kdef}.
    Hence \cref{eq:fturnoffansatz} holds for $n+1$ with $B_{1}^{(n+1)}\coloneqq\alpha B_{1}^{(n)}$.
    Since $\alpha<1$, we can take $n\to\infty$ to obtain
    \[
      f_{t}(x)\le\sum\limits _{k=1}^{\infty}B_{2}^{(k)}\clG_{\frac{1}{2}[t+k\rho]}\mathbf{1}_{\ttH}(x).
    \]
    Recalling that $\xi=\dist(x,\ttH)$, we can use standard Gaussian tail
    bounds (see, e.g., \cite[(2.7)]{Wai19}) to write
    \[
      \clG_{\frac{1}{2}(t+k\rho)}\mathbf{1}_{\ttH}(x)\le4\e^{-\xi^{2}/(t+k\rho)},
    \]
    so we conclude that
    \begin{equation}
      f_{t}(x)\le4\sum\limits _{k=1}^{\infty}B_{2}^{(k)}\e^{-\xi^{2}/(t+k\rho)}=4\Xi_{2;\tau_{1}}[v_{0}]^{2}\sum_{k=1}^{\infty}\alpha^{k}\e^{-\xi^{2}/(t+k\rho)}.\label{eq:ftbdfinalsum}
    \end{equation}
    Fix $N=\lfloor\xi\rho^{-1/2}\rfloor$, so $\xi\rho^{-1/2}-1\le N\le\xi\rho^{-1/2}$
    and thus we have
    \begin{multline}
      \sum_{k=1}^{\infty}\alpha^{k}\e^{-\xi^{2}/(t+k\rho)}=\sum_{k=1}^{N}\alpha^{k}\e^{-\xi^{2}/(t+k\rho)}+\sum_{k=N+1}^{\infty}\alpha^{k}\e^{-\xi^{2}/(t+k\rho)}\le\e^{-\xi^{2}/(t+N\rho)}\sum_{k=1}^{N}\alpha^{k}+\sum_{k=N+1}^{\infty}\alpha^{k}\\
      \le\frac{\e^{-\xi^{2}/(t+N\rho)}+\alpha^{N+1}}{1-\alpha}\le K_{\lambda,\rho,\tau_{1}}\left(\e^{-\xi^{2}/(t+\xi\rho^{1/2})}+\alpha^{\xi\rho^{-1/2}}\right).\label{eq:sortoutthesum}
    \end{multline}
    Now if we let $\zeta=(t\vee\rho)^{1/2}/\xi$, then
    \[
      \e^{-\xi^{2}/(t+\xi\rho^{1/2})}\le\e^{-1/(\zeta+\zeta^{2})}\le C\e^{-1/(2\zeta)}
    \]
    for an absolute constant $C<\infty$.
    Hence
    \begin{align}
      \e^{-\xi^{2}/(t+\xi\rho^{1/2})}+\alpha^{\xi\rho^{-1/2}} & \le C\e^{-1/(2\zeta)}+\e^{\zeta^{-1}\log\alpha}\le2C\exp\left\{ -\frac{1\wedge\log\frac{1}{\alpha}}{2\zeta}\right\} .\label{eq:dealwithzeta}
    \end{align}
    Combining \cref{eq:fturnoffdef}, \cref{eq:ftbdfinalsum}, \cref{eq:sortoutthesum},
    and \cref{eq:dealwithzeta}, and recalling the definitions of $\alpha$
    and $\zeta$, we obtain \cref{eq:turnoffspatialnoise}.\qedhere
  \end{thmstepnv}
\end{proof}

\subsection{The approximate root decoupling function \texorpdfstring{$J_{\sigma,\rho}$}{J\_{σ,ρ}}: regularity}

For $a\in\RR^m$ and $q\ge0$, we define, recalling the definition \cref{eq:Tdef} of $\sfT_\rho(q)$,
\nomenclature[J sigma rho]{$J_{\sigma,\rho}$}{approximate root decoupling function, \cref{eq:Jdef-expscale}}
\begin{equation}
  J_{\sigma,\rho}(q,a)=\bigl[\EE \sigma^{2}\bigl(\clV_{\sfT_{\rho}(q)}^{\sigma,\rho}a(x)\bigr)\bigr]^{1/2}.\label{eq:Jdef-expscale}
\end{equation}
Note that by the spatial stationarity of the noise, this definition
is independent of $x$. We call $J_{\sigma,\rho}$ the ``approximate
root decoupling function'' because we will later show (in \cref{prop:Japprox}
below) that it converges to the root decoupling function $J_{\sigma}$
in the FBSDE \zcref[range]{eq:FBSDE,eq:Jeqn} as $\rho\searrow0$. To
do this, we will set up an approximate version of the FBSDE at the
level of the SPDE, for which $J_{\sigma,\rho}$ will be the analogue
of the root decoupling function.
We will conclude using the convergence
of the approximate problem to the FBSDE problem \zcref[range]{eq:FBSDE,eq:Jeqn} and using the well-posedness
of the FBSDE.

In our approximation schemes, we will need estimates on the regularity of $J_{\sigma,\rho}$.
First we consider its spatial regularity.
\begin{prop}
  \label{prop:approxJlipschitz}Let $\lambda\in(0,\infty)$ and
  $\sigma\in\lipset(\lambda)$.
  Then for any $q\in(0,\lambda^{-2})$ and $a_{1},a_{2}\in\RR^m$, we have
  \begin{equation}
    |J_{\sigma,\rho}(q,a_{2})-J_{\sigma,\rho}(q,a_{1})|_{\Frob}\le\frac{|a_{2}-a_{1}|}{(\lambda^{-2}-q)^{1/2}}.\label{eq:approxJlipschitz}
  \end{equation}
\end{prop}

\begin{proof}
  Let $a_{1},a_{2}\in\RR^m$. We have by \cref{prop:matrixvaluedreversetriangleinequality}
  that
  \begin{align}
    |J_{\sigma,\rho}(q,a_{2})-J_{\sigma,\rho}(q,a_{1})|_{\Frob}^{2} &\le\sup_{x\in\RR^{2}} \E \bigl\lvert \sigma\bigl(\clV_{\sfT_{\rho}(q)}a_{2}(x)\bigr)-\sigma\bigl(\clV_{\sfT_{\rho}(q)}a_{1}(x)\bigr)\bigr\rvert _{\Frob}^{2}\nonumber\\
                                                                    &\le\lambda^{2}\vvvert\clV_{\sfT_{\rho}(q)}a_{2}-\clV_{\sfT_{\rho}(q)}a_{1}\vvvert_{2}^{2},\label{eq:approxJlipschitzfirststep}
  \end{align}
  with the last inequality by the Lipschitz bound on $\sigma$. Then
  we can write
  \[
    (\clV_{\sfT_{\rho}(q)}a_{2}-\clV_{\sfT_{\rho}(q)}a_{1})(x)=a_{2}-a_{1}+\gamma_{\rho}\int_{0}^{\sfT_{\rho}(q)}\clG_{\sfT_{\rho}(q)+\rho-r}\bigl[\bigl(\sigma(\clV_{r}a_{2})-\sigma(\clV_{r}a_{1})\bigr)\,\dif W_{r}\bigr](x),
  \]
  so
  \begin{align*}
    \EE & |(\clV_{\sfT_{\rho}(q)}a_{2}-\clV_{\sfT_{\rho}(q)}a_{1})(x)|^{2}                                                                                                                                \\
        & =|a_{2}-a_{1}|^{2}+\frac{4\pi}{\sfL(1/\rho)}\int_{0}^{\sfT_{\rho}(q)}\!\!\!\!\int G_{\sfT_{\rho}(q)+\rho-r}^{2}(x-y)\EE \bigl\lvert \sigma\bigl(\clV_{r}a_{2}(y)\bigr)-\sigma\bigl(\clV_{r}a_{1}(y)\bigr)\bigr\rvert _{\Frob}^{2}\,\dif y\,\dif r \\
        & \le|a_{2}-a_{1}|^{2}+\frac{\lambda^{2}}{\sfL(1/\rho)}\int_{0}^{\sfT_{\rho}(q)}\frac{\vvvert\clV_{r}a_{2}-\clV_{r}a_{1}\vvvert_{2}^{2}}{\sfT_{\rho}(q)+\rho-r}\,\dif r.
  \end{align*}
  Therefore, if we define $A\coloneqq\sup\limits _{r\in[0,\sfT_{\rho}(q)]}\vvvert\clV_{r}a_{2}-\clV_{r}a_{1}\vvvert_{2}^{2}$,
  we have
  \[
    A\le|a_{2}-a_{1}|^{2}+\lambda^{2}\sfS_{\rho}\bigl(\sfT_{\rho}(q)\bigr)A=|a_{2}-a_{1}|^{2}+\lambda^{2}qA.
  \]
  Rearranging, we see that $A\le\frac{|a_{2}-a_{1}|^{2}}{1-\lambda^{2}q}$.
  Together with \cref{eq:approxJlipschitzfirststep}, this implies \cref{eq:approxJlipschitz}.
\end{proof}
We next consider the temporal regularity of $J_{\sigma,\rho}$.
\begin{prop}
  \label{prop:approxJtemporalregularity}Let $\lambda\in(0,\infty)$ and $\sigma \in \lipset(\lambda)$.
  For any $q_{1},q_{2}\ge0$ such that $q_{1}\wedge q_{2}<\lambda^{-2}-2\sfL(1/\rho)^{-1}$, and any $a\in\RR^m$, we have
    \begin{equation*}
    | J_{\sigma,\rho}(q_{2},a)-J_{\sigma,\rho}(q_{1},a)|_{\Frob} \le\lambda K_{\lambda,\rho,\sfT_{\rho}(q_{1}\wedge q_{2})}^{1/2}\Xi_{2;|\sfT_{\rho}(q_{2})-\sfT_{\rho}(q_{1})|}[a]\left(|q_{2}-q_{1}|^{1/2}+\frac{\sfL(2)^{1/2}K_{\lambda,\rho,\sfT_{\rho}(q_{1}\wedge q_{2})}^{1/2}}{\sfL(1/\rho)^{1/2}}\right).
  \end{equation*}
\end{prop}

\begin{proof}
  Assume without loss of generality that $q_{2}\ge q_{1}$. From the
  definition \cref{eq:Jdef-expscale}, we write
  \begin{align*}
    |  J_{\sigma,\rho}(q_{2},a)-J_{\sigma,\rho}&(q_{1},a)|_{\Frob}^{2}\le \sup_{x\in\RR^{2}}\E \bigl\lvert \sigma\bigl(\clV_{\sfT_{\rho}(q_{2})}a(x)\bigr)-\sigma\bigl(\clV_{\sfT_{\rho}(q_{2})-\sfT_{\rho}(q_{1}),\sfT_{\rho}(q_{2})}a(x)\bigr)\bigr\rvert _{\Frob}^{2}                                            \\
      & \le\lambda^{2}\vvvert\clV_{\sfT_{\rho}(q_{2})}a-\clV_{\sfT_{\rho}(q_{2})-\sfT_{\rho}(q_{1}),\sfT_{\rho}(q_{2})}a\vvvert_{2}^{2}                                                                                                                               \\
      & \le\lambda^{2}K_{\lambda,\rho,\sfT_{\rho}(q_{1})}\Xi_{2;\sfT_{\rho}(q_{2})-\sfT_{\rho}(q_{1})}[a]^{2}\left(\sfL_{\rho}\left(\frac{\sfT_{\rho}(q_{2})-\sfT_{\rho}(q_{1})}{\sfT_{\rho}(q_{1})+\rho}\right)+\frac{\sfL(2)K_{\lambda,\rho,\sfT_{\rho}(q_{1})}}{\sfL(1/\rho)}\right) \\
      & =\lambda^{2}K_{\lambda,\rho,\sfT_{\rho}(q_{1})}\Xi_{2;\sfT_{\rho}(q_{2})-\sfT_{\rho}(q_{1})}[a]^{2}\left(q_{2}-q_{1}+\frac{\sfL(2)K_{\lambda,\rho,\sfT_{\rho}(q_{1})}}{\sfL(1/\rho)}\right),
  \end{align*}
  with the first inequality by \cref{prop:matrixvaluedreversetriangleinequality},
  the second by the Lipschitz bound on $\sigma$, the third by \cref{prop:turnoffnoise},
  and the final identity by \cref{eq:Lbdq1q2}.
\end{proof}
For convenience later on, we combine the previous two results into
a single statement.
\begin{cor}
  \label{cor:approxJregularity}Let $\lambda\in(0,\infty)$ and 
  $\sigma\in\lipset(\lambda)$. For any $q_{1},q_{2}\in(0,\lambda^{-2})$ and $a_{1},a_{2}\in\RR^m$, we have
  \begin{align*}
    | J_{\sigma,\rho}&(q_{2},a_{2})-J_{\sigma,\rho}(q_{1},a_{1})|_{\Frob}                                                                                                                                                                                                                \\
      & \le\lambda K_{\lambda,\rho,\sfT_{\rho}(q_{1}\wedge q_{2})}^{1/2}\left[\Xi_{2;|\sfT_{\rho}(q_{2})-\sfT_{\rho}(q_{1})|}[a_{2}]\left(|q_{2}-q_{1}|^{1/2}+\frac{\sfL(2)^{1/2}K_{\lambda,\rho,\sfT_{\rho}(q_{1}\wedge q_{2})}^{1/2}}{\sfL(1/\rho)^{1/2}}\right)+|a_{2}-a_{1}|\right].
  \end{align*}
\end{cor}

\begin{proof}
  Assume without loss of generality that $q_{2}\ge q_{1}$. We have
  \begin{align*}
    | J_{\sigma,\rho}&(q_{2},a_{2})-J_{\sigma,\rho}(q_{1},a_{1})|_{\Frob}                                                                                                                                                                                                                                     \\
      & \le|J_{\sigma,\rho}(q_{2},a_{2})-J_{\sigma,\rho}(q_{1},a_{2})|_{\Frob}+|J_{\sigma,\rho}(q_{1},a_{2})-J_{\sigma,\rho}(q_{1},a_{1})|_{\Frob}                                                                                                                                                             \\
      & \le\lambda K_{\lambda,\rho,\sfT_{\rho}(q_{1})}^{1/2}\Xi_{2;|\sfT_{\rho}(q_{2})-\sfT_{\rho}(q_{1})|}[a_{2}]\left(|q_{2}-q_{1}|^{1/2}+\frac{\sfL(2)^{1/2}K_{\lambda,\rho,\sfT_{\rho}(q_{1})}^{1/2}}{\sfL(1/\rho)^{1/2}}\right)+\frac{|a_{2}-a_{1}|}{(\lambda^{-2}-q_{1})^{1/2}} \\
      & \le\lambda K_{\lambda,\rho,\sfT_{\rho}(q_{1})}^{1/2}\left[\Xi_{2;|\sfT_{\rho}(q_{2})-\sfT_{\rho}(q_{1})|}[a_{2}]\left(|q_{2}-q_{1}|^{1/2}+\frac{\sfL(2)^{1/2}K_{\lambda,\rho,\sfT_{\rho}(q_{1})}^{1/2}}{\sfL(1/\rho)^{1/2}}\right)+|a_{2}-a_{1}|\right],
  \end{align*}
  where the first inequality is by the triangle inequality, the second
  is by \cref{prop:approxJlipschitz,prop:approxJtemporalregularity},
  and the third is by the definition \cref{eq:Kdef-1} of $K_{\lambda,\rho,\sfT_{\rho}(q_{1})}$.
\end{proof}

\subsection{Concentration of averages of the \texorpdfstring{$\sigma^{2}$}{σ²} field\label{subsec:concentrate-averages}}

Let $\ttR$ be a square in $\RR^{2}$ of side length $\xi$.
Given deterministic $a\in\RR^m$ and $\tau_{0}\le\tau_{1}$, \cref{eq:Jdef-expscale} yields
\begin{equation}
  \EE \left(\fint_{\ttR}\sigma^{2}\bigl(\clV_{\tau_{0},\tau_{1}}^{\sigma,\rho}a(x)\bigr)\,\dif x\right)=J_{\sigma,\rho}^{2}(\sfS_{\rho}(\tau_{1}-\tau_{0}),a).\label{eq:J2istheexpectation}
\end{equation}
Here we use the shorthand $\fint_{\ttR} f = \abs{\ttR}^{-1} \int_{\ttR}f$ for a spatial average.

In this section we will show that if $t\ll\xi^{2}$, then the term
inside the expectation in \cref{eq:J2istheexpectation} is in fact concentrated
about its mean.
To see this, we express the average over $\ttR$ as an average of sub-averages taken over many squares of side-length $\approx \sqrt{t}$.
Most of the sub-averages are approximately independent of one other, so a form of the law of large numbers holds.
\begin{prop}
  \label{prop:LLNbd}Let $\ell\in(2,4]$.
  There is a constant $C=C(\ell,m)<\infty$ such that the following holds.
  Let $\lambda\in(0,\infty)$ and $\sigma\in\lipset(\lambda)$.
  For any $a\in\RR^m$, $\xi>0$, $s < t$ such that
  $t-s<\xi^{2}\wedge\overline{T}_{\rho}(\lambda)$, $\ell\in(2,4]$,
  and any square $\ttR\subset\RR^{2}$ with side length $\xi$, we have
  \begin{equation}
    \begin{aligned}\EE \Biggl\lvert  \left(\fint_{\ttR}\sigma^{2}\bigl(\clV_{s,t}a(x)\bigr)\,\dif x\right)^{1/2}&-J_{\sigma,\rho}(\sfS_{\rho}(t-s),a)\Biggr\rvert  _{\Frob}^{2}                                                               \\
                   & \le C\Xi_{\ell;s,t}[a]^{2}\left(\frac{2-4/\ell}{1\wedge\log\frac{1}{\lambda^{2}\sfS_{\rho}(t-s)}}\cdot\frac{\log\bigl(\xi[(t-s)\vee\rho]^{-1/2}\bigr)}{\xi[(t-s)\vee\rho]^{-1/2}}\right)^{2-4/\ell}.
    \end{aligned}
    \label{eq:LLNbd}
  \end{equation}
\end{prop}

\begin{rem}
  \label{rem:interpretLLNbd}
  In interpreting \cref{eq:LLNbd}, we recall
  that the exponent $2-4/\ell$ is positive because of the assumption
  $\ell>2$. Also, the assumption that $t-s<\overline{T}_{\rho}(\lambda)$
  means that $\frac{1}{\lambda^{2}\sfS_{\rho}(t-s)}>1$.
  We note
  that if $c<\lambda^{-2}$, then there is a constant $C=C(\ell,\lambda,c)$
  such that
  \begin{equation}
    \left(\frac{2-4/\ell}{1\wedge\log\frac{1}{\lambda^{2}\sfS_{\rho}(t-s)}}\right)^{2-4/\ell}\le C(\lambda,\ell,c)\qquad\text{whenever }\sfS_{\rho}(t-s)\le c.\label{eq:boundthelogthing}
  \end{equation}
\end{rem}

\begin{proof}
  Assume without loss of generality that $s=0$. Fix $N\in\NN$,
  to be chosen later. Write $\ttR=\bigsqcup\limits _{i\in\clI}\ttR_{i}$,
  where the $\ttR_{i}$s are disjoint translates of $[0,\xi/N)^{2}$ and
  $\clI$ is an index set with $|\clI|=N^{2}$.
  For $i\in\clI$, let $\hat{\ttR}_{i}$ be a box with the same
  center as $\ttR_{i}$, but with side length $3\xi/N$. Let $(\clI_{j})_{j\in\clJ}$
  be a partition of $\clI$ such that $|\clJ|\le16$ and,
  for each $j\in\clJ$, we have $\hat{\ttR}_{i_{1}}\cap\hat{\ttR}_{i_{2}}=\emptyset$
  for all $i_{1},i_{2}\in\clI_{j}$.
  Given $x\in \ttR$, let $i(x)$
  denote the index such that $x\in \ttR_{i(x)}$.

  In this proof we use the notation $\clV_{s,t}^{\ttB}$ introduced
  prior to \cref{prop:resamplenoiseoffrectangle}. Applying
  \cref{prop:matrixvaluedreversetriangleinequality} with the probability
  space $(\Omega,\PP)$ replaced by $(\ttR,|\ttR|^{-1}\dif x)$, we find
  \begin{align}
    \EE & \left\lvert  \left(\fint_{\ttR}\sigma^{2}\bigl(\clV_{t}^{\hat{\ttR}_{i}(x)}a(x)\bigr)\,\dif x\right)^{1/2}-\left(\fint_{\ttR}\sigma^{2}\bigl(\clV_{t}a(x)\bigr)\,\dif x\right)^{1/2}\right\rvert  _{\Frob}^{2}\nonumber                         \\
        & \le\fint_{\ttR}\EE \bigl\lvert \sigma\bigl(\clV_{t}^{\hat{\ttR}_{i}(x)}a(x)\bigr)-\sigma\bigl(\clV_{t}a(x)\bigr)\bigr\rvert _{\Frob}^{2}\,\dif x\le\lambda^{2}\fint_{\ttR}\EE \bigl\lvert \clV_{t}^{\hat{\ttR}_{i}(x)}a(x)-\clV_{t}a(x)\bigr\rvert _{\Frob}^{2}\,\dif x\nonumber \\
        & \le CK_{\lambda,\rho,t}\Xi_{2;t}[a]^{2}\exp\left\{ -\frac{\xi}{2N(t\vee\rho)^{1/2}}\left(1\wedge\log\frac{1}{\lambda^{2}\sfS_{\rho}(t)}\right)\right\} \label{eq:applyresampling}
  \end{align}
  The final inequality follows from \cref{prop:resamplenoiseoffrectangle} (yielding the absolute constant $C$) and
  the fact that $\inf\limits _{x\in \ttR}\dist(x,\hat{\ttR}_{i(x)})=\xi/N$.
  Now define
  \[
    X_{i}=\fint_{\ttR_{i}}\sigma^{2}\bigl(\clV_{t}^{\hat{\ttR}_{i}}a(x)\bigr)\,\dif x.
  \]
  Note that the $X_{i}$s are identically distributed random variables
  taking values in $\clH_{+}^{m}$. Also, for each $j\in\clJ$,
  the family of random variables $(X_{i})_{i\in\clI_{j}}$ is
  independent.
  Indeed, $X_{i}$ depends on the original white noise
  only inside $\hat{\ttR}_{i}$ (along with a new white noise that is independent
  of all the other white noises outside of $\hat{\ttR}_{i}$), and the
  $\hat{\ttR}_{i}$s for $i\in\clI_{j}$ do not overlap by construction.
  Furthermore, we have
  \begin{align}
    \EE |X_{i}-\EE X_{i}|_{*}^{\ell/2}\le C(\ell)\EE |X_{i}|_{*}^{\ell/2}\le C(\ell)\EE \bigl\lvert \sigma^{2}&\bigl(\clV_{t}^{\hat{\ttR}_{i}}a(x)\bigr)\bigr\rvert _{*}^{\ell/2}\nonumber\\
                                                                                                       &=C(\ell)\EE \bigl\lvert \sigma\bigl(\clV_{t}^{\hat{\ttR}_{i}}a(x)\bigr)\bigr\rvert _{\Frob}^{\ell}\le C(\ell)\Xi_{\ell;t}[a]^{2},\label{eq:normofXminusEx}
  \end{align}
  by Jensen's inequality, where $C(\ell)$ is a constant depending only
  on $\ell$.

  We now use a Hilbert space-valued version of the quantitative Law of
  Large Numbers proved by von Bahr and Esseen in \cite{vBE65}, e.g.\ 
  \cite[Corollary 11]{Nag87}, which tells us that
  \[
    \EE \Biggl|\sum_{i\in\clI_{j}}(X_{i}-\EE X_{i})\Biggr|_{\Frob}^{\ell/2}\le C(\ell)|\clI_{j}|\EE \left\lvert  X_{i}-\EE X_{i}\right\rvert  _{\Frob}^{\ell/2}
  \]
  for a (new) constant $C(\ell)$ depending only on $\ell$. By the equivalence
  of norms on finite-dimensional spaces, this implies that
  \begin{equation}
    \EE \Biggl|\sum_{i\in\clI_{j}}(X_{i}-\EE X_{i})\Biggr|_{*}^{\ell/2}\le C(\ell,m)|\clI_{j}|\EE |X_{i}-\EE X_{i}|_{*}^{\ell/2}\overset{\cref{eq:normofXminusEx}}{\le}C(\ell,m)|\clI_{j}|\Xi_{\ell;t}[a]^{2}.\label{eq:useequivalenceofnorms}
  \end{equation}
  We can thus estimate
  \begin{align*}
    \EE & \left\lvert  \left(N^{-2}\sum_{i\in\clI}X_{i}\right)^{1/2}-(\EE X_{i})^{1/2}\right\rvert  _{\Frob}^{2}\overset{\cref{eq:powerstormer}}{\le}\frac{1}{N^{2}}\EE \left\lvert  \sum_{i\in\clI}(X_{i}-\EE X_{i})\right\rvert  _{*}\le\frac{1}{N^{2}}\sum_{j\in\clJ}\EE \left\lvert  \sum_{i\in\clI_{j}}(X_{i}-\EE X_{i})\right\rvert  _{*}      \\
        & \le\frac{1}{N^{2}}\sum_{j\in\clJ}\left(\EE \left\lvert  \sum_{i\in\clI_{i}}(X_{i}-\EE X_{i})\right\rvert  _{*}^{\ell/2}\right)^{2/\ell}\overset{\cref{eq:useequivalenceofnorms}}{\le}C(\ell,m)\frac{\Xi_{\ell;t}[a]^{2}}{N^{2}}\sum_{j\in\clJ}|\clI_{j}|^{2/\ell}\le C(\ell,m)\Xi_{\ell;t}[a]^{2}N^{4/\ell-2},
  \end{align*}
  where the constant $C(\ell,m)$ has been allowed to change from
  expression to expression but always depends only on $\ell$ and $m$.
  Now we note that
  \[
    N^{-2}\sum_{i\in\clI}X_{i}=\fint_{\ttR}\sigma^{2}\bigl(\clV_{t}^{\hat{\ttR}_{i(x)}}a(x)\bigr)\,\dif x
  \]
  and that
  \[
    \EE X_{i}=J_{\sigma,\rho}^{2}(\sfS_{\rho}(t),a).
  \]
  Combining the last three displays, we see that
  \begin{equation}
    \EE \left\lvert  \left(\fint_{\ttR}\sigma^{2}\bigl(\clV_{t}^{\hat{\ttR}_{i(x)}}a(x)\bigr)\,\dif x\right)^{1/2}-J_{\sigma,\rho}(\sfS_{\rho}(t),a)\right\rvert  ^{2}_\Frob\le C(\ell,m)\Xi_{\ell;t}[a]^{2}N^{4/\ell-2}.\label{eq:applyLLN}
  \end{equation}

  Combining \cref{eq:applyresampling} and \cref{eq:applyLLN} and using
  the fact that $\Xi_{2;t}[a]\le\Xi_{\ell;t}[a]$ by Jensen's inequality,
  we see that
  \begin{align*}
    \EE & \left\lvert  \left(\fint_{\ttR}\sigma^{2}\bigl(\clV_{t}a(x)\bigr)\,\dif x\right)^{1/2}-J_{\sigma,\rho}(\sfS_{\rho}(t),a)\right\rvert  _{\Frob}^{2}\nonumber                                                               \\
        & \le C\Xi_{\ell;t}[a]^{2}\left[\exp\left\{ -\frac{\xi}{2N(t\vee\rho)^{1/2}}\left(1\wedge\log\frac{1}{\lambda^{2}\sfS_{\rho}(t)}\right)\right\} +N^{4/\ell-2}\right],%
  \end{align*}
  for a possibly larger constant $C$. Now we take
  \[
    N=\left\lceil \frac{\xi\left(1\wedge\log\frac{1}{\lambda^{2}\sfS_{\rho}(t)}\right)}{4(2-4/\ell)(t\vee\rho)^{1/2}\log\frac{\xi}{(t\vee\rho)^{1/2}}}\right\rceil,
  \]
  so the inequality becomes
  \begin{align*}
    \EE & \left\lvert  \left(\fint_{\ttR}\sigma^{2}(\clV_{t}a(x))\,\dif x\right)^{1/2}-J_{\sigma,\rho}(\sfS_{\rho}(t),a)\right\rvert  _{\Frob}^{2}\nonumber                                                                               \\
        & \le C\Xi_{\ell;t}[a]^{2}\left(\frac{1\wedge\log\frac{1}{\lambda^{2}\sfS_{\rho}(t)}}{2-4/\ell}\cdot\frac{\xi(t\vee\rho)^{-1/2}}{\log[\xi(t\vee\rho)^{-1/2}]}\right)^{-(2-4/\ell)},%
  \end{align*}
  again for a possibly larger constant $C$.
\end{proof}
To complement \cref{prop:LLNbd}, we prove the following simple bound
on the same quantity, which is useful for very short times $(\sfS_{\rho}(t-s)\ll1$).
\begin{prop}
  \label{prop:smalltimebd}Let $\lambda\in(0,\infty)$ and  $\sigma\in\lipset(\lambda)$.
  For any $a\in\RR^m$, $\ttR\subseteq\RR^{2}$, and $s\le t$, we have
  \begin{equation*}
    \EE \left\lvert  \left(\fint_{\ttR}\sigma^{2}\bigl(\clV_{s,t}a(x)\bigr)\,\dif x\right)^{1/2}-J_{\sigma,\rho}(\sfS_{\rho}(t-s),a)\right\rvert  _{\Frob}^{2}\le4\Xi_{2;t-s}[a]^{2}\lambda^{2}\sfS_{\rho}(t-s).
  \end{equation*}
\end{prop}

\begin{proof}
  Assume without loss of generality that $s=0$. We can estimate, using
  \cref{prop:matrixvaluedreversetriangleinequality},
  \begin{align*}
    \EE \left\lvert  \left(\fint_{\ttR}\sigma^{2}\bigl(\clV_{t}a(x)\bigr)\,\dif x\right)^{1/2}-\sigma(a)\right\rvert  _{\Frob}^{2}\le\fint_{\ttR}\EE \bigl\lvert \sigma\bigl(\clV_{t}a(x)\bigr)-\sigma(a)\bigr\rvert _{\Frob}^{2}\,\dif x & \le\lambda^{2}\vvvert\clV_{t}a-a\vvvert_{2}^{2}.
  \end{align*}
  Also, we have
  \begin{align*}
    \left\lvert  J_{\sigma,\rho}(\sfS_{\rho}(t),a)-\sigma(a)\right\rvert  _{\Frob}^{2}=\left\lvert  [\EE \sigma^{2}\bigl(\clV_{t}a(x)\bigr)]^{1/2}-\sigma(a)\right\rvert  ^{2} & \le\sup_{x\in\RR^{2}}\bigl\lvert \sigma\bigl(\clV_{t}a(x)\bigr)-\sigma(a)\bigr\rvert _{\Frob}^{2} \\
                                                                                                                                                    & \le\lambda^{2}\vvvert\clV_{t}a-a\vvvert_{2}^{2}.
  \end{align*}
  We can bound the right side of each of the last two displays by noting
  that
  \[
    \EE |\clV_{t}a(x)-a|^{2}=\frac{4\pi}{\sfL(1/\rho)}\int_{0}^{t}\!\!\!\!\int G_{t+\rho-r}^{2}(x-y)\EE \bigl\lvert \sigma\bigl(\clV_{r}a(y)\bigr)\bigr\rvert _{\Frob}^{2}\,\dif y\,\dif r\le\Xi_{2;t}[a]^{2}\sfS_{\rho}(t),
  \]
  and the conclusion follows from the triangle inequality.
\end{proof}

\subsection{Approximating averages of the \texorpdfstring{$\sigma^{2}$}{σ²} field in terms of averages
  of the solution}

In this section (namely in \cref{thm:mainapproxthm} below),
we relate spatial averages of $\sigma^{2}\circ v_{t}$
to spatial averages of $v_{t}$. For technical reasons (mainly to do with the proof of \cref{prop:coupling} below),
the former averages will be uniform averages over squares, while the
latter will be averages against Gaussians. We do this by first using the results of \cref{subsec:noise-flatten} to replace the field at a slightly earlier time $s$ by a flattened version, and then using the results of \cref{subsec:concentrate-averages} to replace the spatial average at time $t$ by an expectation conditional on the spatial average at time $s$. The spatial average at time $s$ is then seen to be close to the spatial average at time $t$ by the results of \cref{subsec:regularitygaussianaverages}.

We begin with some notation.
Given $\zeta\in(0,\infty)$ and $k\in\RR^{2}$, we define the
square
\nomenclature[Bzetak]{$\ttB_{\zeta,k}$}{box with side length $\zeta$ and lower left corner $k\in\RR^2$, \cref{eq:boxdef}}
\begin{equation}
  \ttB_{\zeta,k}\coloneqq\zeta\cdot(k+[0,1)^{2}).\label{eq:boxdef}
\end{equation}
Given $x,z\in\RR^{2}$, let $k_{\zeta,z}(x)$
be the unique integral vector $k \in \Z^2$ such that $\ttB_{\zeta,k + z}\ni x$.\nomenclature[kzetaz]{$k_{\zeta,z}(x)$}{$k$ such that $\ttB_{\zeta,k+z}\ni x$}
Given a field $f$ on
$\RR^{2}$, $\zeta\in(0,\infty)$, and $z\in\RR^{2}$, define
\nomenclature[S cal]{$\clS_{\zeta,z}$}{averages a field over $\ttB_{\zeta,z}$, \cref{eq:Scaldef}}
\nomenclature[S calzero]{$\clS_{\zeta}$}{$\clS_{\zeta,0}$, \cref{eq:Scaldef}}
\begin{equation}
  \clS_{\zeta,z}f(x)\coloneqq\fint_{\ttB_{\zeta,k_{\zeta,z}(x)+z}}f(y)\,\dif y=\frac{1}{\zeta^{2}}\int_{\ttB_{\zeta,k_{\zeta,z}(x)+z}}f(y)\,\dif y;\qquad\qquad\clS_{\zeta}\coloneqq\clS_{\zeta,0}.\label{eq:Scaldef}
\end{equation}
Finally, given $\ell>2$, define
\nomenclature[zzzgreek κ]{$\kappa_{\ell,\rho}$}{small parameter, equal to $\sfL(1/\rho)^{-1/(1-2/\ell)}$, see \cref{eq:kappadef}}
\begin{equation}
  \kappa_{\ell,\rho}=\sfL(1/\rho)^{-1/(1-2/\ell)}.\label{eq:kappadef}
\end{equation}
We note that
\begin{equation*}
  \kappa_{\ell,\rho}\le\sfL(1/\rho)^{-1}\qquad\text{for all \ensuremath{\rho>0}\text{ and all }\ensuremath{\ell>2}}.
\end{equation*}
We think of $\ell$ as slightly greater than $2$, in which case the
exponent $-1/(1-2/\ell)$ in \cref{eq:kappadef} is a large negative constant, and thus $\kappa_{\ell,\rho}\ll 1$ for small $\rho$.
We use $\kappa$ as a fudge-factor to modestly expand or contract certain time or space scales.
The exponent in \cref{eq:kappadef} is $-2$ times the inverse of the exponent in
\cref{eq:LLNbd}. The precise value of the exponent is important only to obtain the estimate \cref{eq:E2firstbound} below.

We recall the parameter $\eta(\beta,\ell)$ from \cref{eq:etadef}.
\begin{thm}
  \label{thm:mainapproxthm}Fix $M,\beta,\lambda\in(0,\infty)$, $\ell\in(2,4]$,
  $\overline{\beta}>\eta(\beta,\ell)$, and $\overline{q}\in(0,\lambda^{-2})$.
  There is a constant $C=C(M,\beta,\lambda,\ell,\overline{\beta},\overline{q})\in(0,\infty)$
  such that the following holds. Suppose that $\rho\in(0,C^{-1})$.
  Let $\sigma\in\quadset(M,\beta)$ satisfy $\Lip(\sigma)\le\lambda$
  and let $(v_{t})_{t\ge0}$ solve \cref{eq:mildsoln} with initial condition
  $v_{0}\in\scrX_{0}^{\ell}$. Let $\xi\in[0,\overline{T}_{\rho}(\overline{q}^{-1/2})]$,
  \begin{equation}
    t\in\bigl[\bigl([1+\sfL(1/\rho)]\kappa_{\ell,\rho}^{-1}+\kappa_{\ell,\rho}\bigr)\xi^{2},\sfT_{\rho}(\overline{\beta}^{-2})\bigr],\label{eq:mainapproxthmtcond}
  \end{equation}
  $x,z\in\RR^{2}$, and $\iota_{1},\iota_{2}\in[0,1]$. Then we have
  \begin{equation}
    \EE \left\lvert  \left(\clS_{\xi,z}[\sigma^{2}\circ v_{t}](x)\right)^{1/2}-J_{\sigma,\rho}\bigl(\sfS_{\rho}(\sfL(1/\rho)^{\iota_{1}}\xi^{2}),\clG_{\sfL(1/\rho)^{\iota_{2}}\xi^{2}}v_{t}(x)\bigr)\right\rvert  _{\Frob}^{2}\le C\langle\vvvert v_{0}\vvvert_{\ell}\rangle^{2}\frac{\sfL(\sfL(1/\rho))}{\sfL(1/\rho)}.\label{eq:mainapproxthm-conclusion}
  \end{equation}
\end{thm}

\begin{proof}
  The proof combines the earlier results of this section
  with specific choices of parameters. We proceed in several steps.
  \begin{thmstepnv}
    \item \emph{Choosing the parameters.} Let
    \begin{equation}
      \ttR=\ttB_{\xi,k_{\xi,z}(x) + z},\quad\tau_{0}=t-(\kappa_{\ell,\rho}^{-1}+\kappa_{\ell,\rho})\xi^{2},\quad\tau_{1}=\tau_{0}+\kappa_{\ell,\rho}^{-1}\xi^{2},\quad\tau_{2}=\tau_{1}+\kappa_{\ell,\rho}\xi^{2}=t\label{eq:choosexietc}
    \end{equation}
    and let $X \in \R^d$ denote the center of $\ttR$.
    Note that
    \begin{equation}
      \tau_{2}\ge\tau_{1}\ge\tau_{0}\ge\sfL(1/\rho)\kappa_{\ell,\rho}^{-1}\xi^{2}\ge0\label{eq:tauestimates}
    \end{equation}
    by the assumption \cref{eq:mainapproxthmtcond} on $t$.

    \item \emph{Flattening the solution.} We have by \cref{prop:matrixvaluedreversetriangleinequality}
    and \cref{prop:flattenic} that there is an absolute constant $C<\infty$
    such that (as long as $\rho$ is small enough that $\overline{q}<\lambda^{-2}-2\sfL(1/\rho)^{-1}$)
    \begin{align*}
      E_{1} & \coloneqq\EE \left\lvert  \left(\fint_{\ttR}\sigma^{2}\bigl(v_{t}(y)\bigr)\,\dif y\right)^{1/2}-\left(\fint\sigma^{2}\bigl(\clV_{\tau_{1},\tau_{2}}\clZ_{X}\clG_{\tau_{1}-\tau_{0}}v_{\tau_{0}}(y)\bigr)\,\dif y\right)^{1/2}\right\rvert  _{\Frob}^{2}                                        \\
            & \le\fint_{\ttR}\EE \left\lvert  \sigma\bigl(\clV_{\tau_{0},\tau_{2}}v_{\tau_{0}}(y)\bigr)-\sigma\bigl(\clV_{\tau_{1},\tau_{2}}\clZ_{X}\clG_{\tau_{1}-\tau_{0}}v_{\tau_{0}}(y)\bigr)\right\rvert  _{\Frob}^{2}\,\dif y                                                                          \\
            & \le\lambda^{2}\fint_{\ttR}\EE \left\lvert  \clV_{\tau_{0},\tau_{2}}v_{\tau_{0}}(y)-\clV_{\tau_{1},\tau_{2}}\clZ_{X}\clG_{\tau_{1}-\tau_{0}}v_{\tau_{0}}(y)\right\rvert  _{\Frob}^{2}\,\dif y                                                                               \\
            & \le CK_{\lambda,\rho,\tau_{2}-\tau_{1}}\Xi_{2;\tau_{0},\tau_{1}}[v_{\tau_{0}}]^{2}\left[\sfL_{\rho}\left(\frac{\tau_{1}-\tau_{0}}{\tau_{2}-\tau_{1}+\rho}\right)+\frac{(1+\lambda^{2}\delta)(\tau_{2}-\tau_{1})+\xi^{2}}{\tau_{1}-\tau_{0}}+\delta\right]
    \end{align*}
    where $\delta=K_{\lambda,\rho,\tau_{2}-\tau_{1}}\sfL(1/\rho)^{-1}.$
    Substituting in the parameters \cref{eq:choosexietc}, we get
    \[
      \delta=K_{\lambda,\rho,\kappa_{\ell,\rho}\xi^{2}}\sfL(1/\rho)^{-1}\le K_{\lambda,\rho,\xi^{2}}\sfL(1/\rho)^{-1},
    \]
    and so
    \begin{align}
      E_{1} & \le CK_{\lambda,\rho,\xi^{2}}\Xi_{2;\tau_{0},\tau_{1}}[v_{\tau_{0}}]^{2}\left[\sfL_{\rho}(\kappa_{\ell,\rho}^{-2})+\kappa_{\ell,\rho}[(1+\lambda^{2}\delta)\kappa_{\ell,\rho}+1]+\frac{K_{\lambda,\rho,\xi^{2}}}{\sfL(1/\rho)}\right].\label{eq:E1bd}
    \end{align}
    From \cref{eq:E1bd} and the definitions \cref{eq:kappadef} of $\kappa_{\ell,\rho}$
    and \cref{eq:Kdef-1} of $K_{\lambda,\rho,\xi^{2}}$, it is clear that
    there if $\rho$ is sufficiently small (depending on $\overline{q}$
    and $\lambda$), then we have
    \begin{equation}
      E_{1}\le C\Xi_{2;\tau_{0},\tau_{1}}[v_{\tau_{0}}]^{2}\cdot\frac{\sfL(\sfL(1/\rho))}{\sfL(1/\rho)}\label{eq:E1OK}
    \end{equation}
    for a constant $C$ depending only on $\overline{q}$ and $\lambda$.
    \item \emph{Concentration.} Define
    \[
      E_{2}\coloneqq\EE \left\lvert  \left(\fint_{\ttR}\sigma^{2}\bigl(\clV_{\tau_{1},\tau_{2}}\clZ_{X}\clG_{\tau_{1}-\tau_{0}}v_{\tau_{0}}(y)\bigr)\,\dif y\right)^{1/2}-J_{\sigma,\rho}\bigl(\sfS_{\rho}(\tau_{2}-\tau_{1}),\clG_{\tau_{1}-\tau_{0}}v_{\tau_{0}}(x)\bigr)\right\rvert  _{\Frob}^{2}.
    \]
    We consider two cases.
    \begin{casenv}
      \item If $\tau_{2}-\tau_{1}=\kappa_{\ell,\rho}\xi^{2}\ge\rho$, then we
      use \cref{prop:LLNbd}, applied with $s\setto\tau_{1}$, $t\setto\tau_{2}$,
      and $a\setto\clG_{\tau_{1}-\tau_{0}}v_{\tau_{0}}(X)$.
      Because $\tau_2 - \tau_1 > \rho$, the term $(\tau_{2}-\tau_{1})\vee\rho$ becomes simply $\tau_{2}-\tau_{1}$.
      Using \cref{eq:boundthelogthing} of \cref{rem:interpretLLNbd}, we find %
      \begin{align*}
        E_{2} & \le C\Xi_{\ell;\tau_{1},\tau_{2}}[\clG_{\tau_{1}-\tau_{0}}v_{\tau_{0}}(X)]^{2}\left(\frac{\log[\xi(\tau_{2}-\tau_{1})^{-1/2}]}{\xi(\tau_{2}-\tau_{1})^{-1/2}}\right)^{2-4/\ell} \\
              & =C\Xi_{\ell;\kappa_{\ell,\rho}\xi^{2}}[\clG_{\tau_{1}-\tau_{0}}v_{\tau_{0}}(X)]^{2}\left(\kappa_{\ell,\rho}^{1/2}\log[\kappa_{\ell,\rho}^{-1/2}]\right)^{2-4/\ell}.
      \end{align*}
      Using the definition \cref{eq:kappadef} of $\kappa_{\ell,\rho}$, we
      see that there is a constant $C=C(\ell)<\infty$ such that if $\rho$ is sufficiently small (depending only
      on $\overline{q}$ and $\lambda$), then
      \begin{equation}
        E_{2}\le C\Xi_{\ell;\kappa_{\ell,\rho}\xi^{2}}[\clG_{\tau_{1}-\tau_{0}}v_{\tau_{0}}(x)]^{2}\frac{\sfL(\sfL(1/\rho))^{2-4/\ell}}{\sfL(1/\rho)}\le C\Xi_{\ell;\kappa_{\ell,\rho}\xi^{2}}[\clG_{\tau_{1}-\tau_{0}}v_{\tau_{0}}(x)]^{2}\frac{\sfL(\sfL(1/\rho))}{\sfL(1/\rho)},\label{eq:E2firstbound}
      \end{equation}
      where in the last inequality we used the fact that $\ell\le4$.
      \item If $\tau_{2}-\tau_{1}=\kappa_{\ell,\rho}\xi^{2}<\rho$, then we use
      \cref{prop:smalltimebd}, applied with $s\setto\tau_{1}$,
      $t\setto\tau_{2}$, and $a\setto\clG_{\tau_{1}-\tau_{0}}v_{\tau_{0}}(X)$,
      to obtain
      \begin{equation*}
        E_{2}\le4\Xi_{2;\tau_{1},\tau_{2}}[\clG_{\tau_{1}-\tau_{0}}v_{\tau_{0}}(X)]^{2}\lambda^{2}\sfS_{\rho}(\rho)=4\Xi_{2;\kappa_{\ell,\rho}\xi^{2}}[\clG_{\tau_{1}-\tau_{0}}v_{\tau_{0}}(X)]^{2}\lambda^{2}\frac{\sfL(2)}{\sfL(1/\rho)}.
      \end{equation*}
    \end{casenv}
    In either case, we therefore have
    \begin{equation}
      E_{2}\le C\Xi_{\ell;\tau_{1},\tau_{2}}[\clG_{\tau_{1}-\tau_{0}}v_{\tau_{0}}(X)]^{2}\frac{\sfL(\sfL(1/\rho))}{\sfL(1/\rho)}\label{eq:E2finalbd}
    \end{equation}
    for a constant $C<\infty$ depending only on $\ell$ and $\lambda$.
    \item \emph{\label{step:adjustJsigmarho}Adjusting the arguments of $J_{\sigma,\rho}$.}
    Let $\iota_{1},\iota_{2}\in[0,1]$. By \cref{cor:approxJregularity} and the triangle inequality on $L^2(\Omega,\P)$, we have
    \begin{align}
       & \left(\EE \left\lvert  J_{\sigma,\rho}\bigl(\sfS_{\rho}(\kappa_{\ell,\rho}\xi^{2}),\clG_{\kappa_{\ell,\rho}^{-1}\xi^{2}}v_{\tau_{0}}(X)\bigr)-J_{\sigma,\rho}\bigl(\sfS_{\rho}(\sfL(1/\rho)^{\iota_{1}}\xi^{2}),\clG_{\sfL(1/\rho)^{\iota_{2}}\xi^{2}}v_{t}(X)\bigr)\right\rvert  _{\Frob}^{2}\right)^{1/2}\nonumber                                             \\
       & \quad\le\lambda K_{\lambda,\rho,\xi^{2}}^{1/2}\Biggl[\Xi_{2;\sfL(1/\rho)\xi^{2}}[\clG_{\kappa_{\ell,\rho}^{-1}\xi^{2}}v_{\tau_{0}}(X)]^{2}\Biggl(|\sfS_{\rho}(\sfL(1/\rho)^{\iota_{1}}\xi^2)-\sfS_{\rho}(\kappa_{\ell,\rho}\xi^2)|^{1/2}+\frac{\sfL(2)^{1/2}K_{\lambda,\rho,\xi^{2}}^{1/2}}{\sfL(1/\rho)^{1/2}}\Biggr)\nonumber \\
       & \hspace{18em}+(\EE |\clG_{\kappa_{\ell,\rho}^{-1}\xi^{2}}v_{\tau_{0}}(X)-\clG_{\sfL(1/\rho)^{\iota_{2}}\xi^{2}}v_{t}(X)|^{2})^{1/2}\Biggr]\label{eq:applytimeregularity}
    \end{align}
    Recalling \cref{eq:Srhodef}, for $\rho$ sufficiently small we have
    \begin{multline}
      0\le\sfS_{\rho}(\sfL(1/\rho)^{\iota_{1}}\xi^{2})-\sfS_{\rho}(\kappa_{\ell,\rho}\xi^{2})\le\sfS_{\rho}(\sfL(1/\rho)\xi^{2})-\sfS_{\rho}(\kappa_{\ell,\rho}\xi^{2})=\frac{1}{\sfL(1/\rho)}\log\frac{1+\sfL(1/\rho)\rho^{-1}\xi^{2}}{1+\kappa_{\ell,\rho}\rho^{-1}\xi^{2}}\\
      \le\frac{\log[\sfL(1/\rho)\kappa_{\ell,\rho}^{-1}]}{\sfL(1/\rho)}\le\frac{C\sfL(\sfL(1/\rho))}{\sfL(1/\rho)}\label{eq:Srhodiff}
    \end{multline}
    for a constant $C$ depending only on $\ell$.
    The penultimate
    inequality holds because $\kappa_{\ell,\rho}<\sfL(1/\rho)$ when $\rho$ is sufficiently small.
    Also, using \cref{prop:continuitybd} and $\kappa_{\ell,\rho}^{-1} > \sfL(1/\rho)$, there is an absolute constant
    $C<\infty$ such that (for sufficiently small $\rho$)
    \begin{align}
      ( & \EE |\clG_{\kappa_{\ell,\rho}^{-1}\xi^{2}}v_{\tau_{0}}(X)-\clG_{\sfL(1/\rho)^{\iota_{2}}\xi^{2}}v_{t}(X)|^{2})^{1/2}\nonumber                                                                                                                                                                                                                                \\
        & \le C\Xi_{2,t}[v_0]\Biggl[\sfL_{\rho}\left(\frac{\kappa_{\ell,\rho}^{-1}-\sfL(1/\rho)^{\iota_{2}} + \kappa_{\ell,\rho}^{-1} + \kappa_{\ell,\rho}}{\sfL(1/\rho)^{\iota_{2}}}\right)^{1/2}+\left(\frac{[\kappa_{\ell,\rho}^{-1}-\sfL(1/\rho)^{\iota_{2}} + \kappa_{\ell,\rho}^{-1} + \kappa_{\ell,\rho}]\xi^{2}}{\tau_{0}}\right)^{1/2}\nonumber\\
      &\hspace{10cm}+\left(\frac{\sfL(\sfL(1/\rho))}{\sfL(1/\rho)}\right)^{1/2}\Biggr]\nonumber \\
        \overset{\cref{eq:tauestimates}}&{\le}
      C\Xi_{2,t}[v_0]\left[\sfL_{\rho}(2\kappa_{\ell,\rho}^{-1})^{1/2}+\frac{1}{\sfL(1/\rho)^{1/2}}+\left(\frac{\sfL(\sfL(1/\rho))}{\sfL(1/\rho)}\right)^{1/2}\right]\nonumber                                                                    \\
        & \le C\Xi_{2,t}[v_0]\left[\frac{\sfL(\sfL(1/\rho))}{\sfL(1/\rho)}\right]^{1/2}\label{eq:EGdiff}
    \end{align}
    Here $C$ has been allowed to change from expression to expression.
    Using \cref{eq:Srhodiff} and \cref{eq:EGdiff} in \cref{eq:applytimeregularity},
    we obtain
    \begin{equation}
      \begin{aligned}\bigl(\EE \bigl\lvert J_{\sigma,\rho} & \bigl(\sfS_{\rho}(\kappa_{\ell,\rho}\xi^{2}),\clG_{\kappa_{\ell,\rho}^{-1}\xi^{2}}v_{\tau_{0}}(X)\bigr)-J_{\sigma,\rho}\bigl(\sfS_{\rho}(\sfL(1/\rho)^{\iota_{1}}\xi^{2}),\clG_{\sfL(1/\rho)^{\iota_{2}}\xi^{2}}v_{\tau_{0}}(X)\bigr)\bigr\rvert _{\Frob}^{2}\bigr)^{1/2} \\
                                     & \le C\left(\Xi_{2;\sfL(1/\rho)\xi^2}[\clG_{\tau_{0}-\tau_{1}}v_{\tau_{0}}(X)]+\Xi_{2;t}[v_0]\right)\left[\frac{\sfL(\sfL(1/\rho))}{\sfL(1/\rho)}\right]^{1/2}
      \end{aligned}
      \label{eq:Jcontbd}
    \end{equation}
    for a constant $C$ depending on $\ell$, $\lambda$, $\overline{q}$.
  \item \emph{Conclusion. }
    Recall that $0 < \tau_0 < \tau_1 < t \leq \sfT_\rho(\bar{\beta}^{-2})$.
    Thus by \cref{prop:momentbd} and the hypothesis $\bar{\beta} > \eta(\beta, \ell)$, we have
    \begin{equation}
      \Xi_{2;\tau_{0},\tau_{1}}[v_{\tau_{0}}]^{2}+\Xi_{\ell;\tau_{1},\tau_{2}}[\clG_{\tau_{1}-\tau_{0}}v_{\tau_{0}}(X)]^{2} + \Xi_{2;\sfL(1/\rho)\xi^2}[\clG_{\tau_{0}-\tau_{1}}v_{\tau_{0}}(X)] + \Xi_{2;t}[v_0] \le C\langle\vvvert v_{0}\vvvert_{\ell}\rangle^{2},\label{eq:applymomentbd}
    \end{equation}
    where the constant $C$ here depends on $M,\beta,\overline{\beta},\ell$.
    Combining \cref{eq:E1OK}, \cref{eq:E2finalbd}, and \cref{eq:Jcontbd}
    using the triangle inequality, Cauchy--Schwarz, and \cref{eq:applymomentbd},
    we see that
    \begin{align*}
      \EE & \left\lvert  \left(\fint_{\ttR}\sigma^{2}\bigl(v_{t}(y)\bigr)\,\dif y\right)^{1/2}-J_{\sigma,\rho}\bigl(\sfS_{\rho}(\sfL(1/\rho)^{\iota_{1}}\xi^{2}),\clG_{\sfL(1/\rho)^{\iota_{2}}\xi^{2}}v_{\tau_{0}}(x)\bigr)\right\rvert  _{\Frob}^{2}\le C\langle\vvvert v_{0}\vvvert_{\ell}\rangle^{2}\frac{\sfL(\sfL(1/\rho))}{\sfL(1/\rho)}
    \end{align*}
    for a constant $C<\infty$ depending only on $q,M,\beta,\lambda,\ell$.
    Recalling the definition of the square $\ttR$, we see that this implies
    \cref{eq:mainapproxthm-conclusion}.\qedhere
  \end{thmstepnv}
\end{proof}

\subsection{The approximate root decoupling function approximates the root decoupling function}

The goal of this subsection is to prove (in \cref{prop:Japprox} below) that the approximate root decoupling function $J_{\sigma,\rho}$ really is an approximation of the root decoupling function $J_\sigma$.

First we introduce some notation. We define, for $T_{0}<T_{1}$, the quantities
\nomenclature[R sf]{$\sfR_{T_0,T_1,\rho}$}{change of variables, inverse to $\sfU_{T_0,T_1,\rho}$, see \cref{eq:sfRdef}}
\begin{equation}
  \sfR_{T_{0},T_{1},\rho}(q)\coloneqq T_{0}+(1-\e^{-q\sfL(1/\rho)})(T_{1}+\rho-T_{0})=T_{1}-\sfT_{\rho}(\sfS_{\rho}(T_{1}-T_{0})-q)\label{eq:sfRdef}
\end{equation}
and
\nomenclature[U sf]{$\sfU_{T_0,T_1,\rho}$}{change of variables, inverse to $\sfR_{T_0,T_1,\rho}$, see \cref{eq:sfUdef}}
\begin{equation}
  \sfU_{T_{0},T_{1},\rho}(t)\coloneqq\frac{1}{\sfL(1/\rho)}\log\frac{T_{1}-T_{0}+\rho}{T_{1}-t+\rho}=\sfS_{\rho}(T_{1}-T_{0})-\sfS_{\rho}(T_{1}-t),\label{eq:sfUdef}
\end{equation}
The latter is the exponential time scale introduced in \cref{eq:introchgvar} and the former is its inverse:
\begin{equation*}
  \sfR_{T_{0},T_{1},\rho}\bigl(\sfU_{T_{0},T_{1},\rho}(t)\bigr)=t\qquad\text{for all }t\in[T_{0},T_{1}].
\end{equation*}
We note that
\begin{equation}
  \sfR_{T_{0},T_{1},\rho}(0)=T_{0}\qquad\text{and}\qquad\sfR_{T_{0},T_{1},\rho}(\sfS_{\rho}(T_{1}-T_{0}))=T_{1}.\label{eq:Rendpoints}
\end{equation}
We recall the definition \cref{eq:Xnorm} of the norm $\|\anon\|_{\clX}$.
\begin{prop}
  \label{prop:Japprox}Let $M,\lambda\in(0,\infty)$, $\sigma \in \lipset(M, \lambda)$, and $Q_{0}\in(0,\lambda^{-2})$.
  There exists a constant $C=C(M,\lambda,Q_{0})<\infty$ such
  that
  \begin{equation}
    \sup_{Q\in[0,Q_{0}]}\|(J_{\sigma,\rho}-J_{\sigma})(Q,\cdot)\|_{\clX}\le C\left[\frac{\sfL(\sfL(1/\rho))}{\sfL(1/\rho)}\right]^{1/2}.\label{eq:Japprox-concl}
  \end{equation}
\end{prop}

\begin{proof}
  We proceed in several steps.
  \begin{thmstepnv}
    \item \emph{Choosing the parameters. }Let $\beta=\lambda$ and select $\overline{\beta}>\beta$
    such that $Q_{0}<\overline{\beta}^{-2}$. Choose $\ell\in(2,4]$ such
    that $\eta(\beta,\ell)<\overline{\beta}$, where $\eta(\beta,\ell)$
    is defined in \cref{eq:etadef}. Let $a\in\RR^m$ and
    let $(v_{t})_{t\ge0}$ solve \cref{eq:SPDE} with constant initial
    condition $v_{0}\equiv a$. Let $Q\in[0,Q_{0}]$ and define
    \begin{equation}
      T_{0}=\sfL(1/\rho)\bigl([1+\sfL(1/\rho)]\kappa_{\ell,\rho}^{-1}+\kappa_{\ell,\rho}\bigr)\sfT_{\rho}(Q)\qquad\text{and}\qquad T_{1}=T_{0}+\sfT_{\rho}(Q).\label{eq:T1T0defs}
    \end{equation}
    Recalling \cref{eq:Rendpoints}, we note that $\sfR_{T_{0},T_{1},\rho}(0)=T_{0}$
    and $\sfR_{T_{0},T_{1},\rho}(Q)=T_{1}$. Define $V_{t}^{\rho,T_{1}}$
    as in \cref{eq:Vdef}. Fix $x\in\RR^{2}$ arbitrarily. We note by \cref{eq:Vtintegral}
    that, for $t\in[T_{0},T_{1}]$, we have
    \begin{equation}
      V_{t}^{\rho,T_{1}}(x)=\clG_{T_{1}-t}v_{t}(x)=V_{T_{0}}^{\rho,T_{1}}(x)+\gamma_{\rho}\int_{T_{0}}^{t}\clG_{T_{1}+\rho-r}[\sigma(v_{r})\,\dif W_{r}](x),\label{eq:VTintegralrestate}
    \end{equation}
    and we recall that
    \begin{equation}
      V_{T_{1}}^{\rho,T_{1}}=v_{T_{1}}.\label{eq:atthefinaltime}
    \end{equation}
    \item \emph{The initial condition.}\label{step:twiddleic} Our choice of $T_{0}$ in \cref{eq:T1T0defs}
    will allow us to satisfy the hypothesis \cref{eq:mainapproxthmtcond}
    of \cref{thm:mainapproxthm}.
    (Absent this restriction, we would have preferred to take $T_{0}=0$.)
    Here we show that the difference is not important.
    Using \cref{lem:VTcontinuity}, we have
    \begin{equation}
      \begin{aligned}\EE & |V_{T_{0}}^{\rho,T_{1}}(x)-a|^{2}=\EE |V_{T_{0}}^{\rho,T_{1}}(x)-V_{0}^{\rho,T_{1}}(x)|^{2}\le\Xi_{2;0,T_{0}}[v_{0}]^{2}\sfL_{\rho}\left(\frac{T_{0}}{T_{1}+\rho-T_{0}}\right)                               \\
                   & \overset{\cref{eq:T1T0defs}}{\le}\Xi_{2;0,T_{0}}[v_{0}]^{2}\sfL_{\rho}\left(\sfL(1/\rho)\bigl([1+\sfL(1/\rho)]\kappa_{\ell,\rho}^{-1}+\kappa_{\ell,\rho}\bigr)\right)\le C\Xi_{2;0,T_{0}}[v_{0}]^{2}\frac{\sfL(\sfL(1/\rho))}{\sfL(1/\rho)}
      \end{aligned}
      \label{eq:fudgethebeginning}
    \end{equation}
    for sufficiently small $\rho$, with a constant $C$ depending only
    on $\ell$.
    \item \emph{The new noise.} Define
    \begin{equation}
      \zeta_{r}=[\sfL(1/\rho)^{-1}(T_{1}-r)]^{1/2}.\label{eq:zetarappldef}
    \end{equation}
    Applying \cref{prop:coupling} with $\eta_{r}=T_{1}+\rho-r$, there exists a space-time
    white noise $\dif\widetilde{W}_{t}$ such that for all $t\in[T_{0},T_{1}]$,
    we have
    \begin{align}
       & \EE \left\lvert  \gamma_{\rho}\int_{\widetilde{T}_{0}}^{t}\clG_{T_{1}+\rho-r}[\sigma(v_{r})\,\dif W_{r}-(\clS_{\zeta_{r}}[\sigma^{2}\circ v_{r}])^{1/2}\,\dif\widetilde{W}_{r}](x)\right\rvert  ^{2}\le\frac{C\Xi_{2;T_{0},T_{1}}[v_{T_{0}}]^{2}}{\sfL(1/\rho)^{2}}\int_{T_{0}}^{t}\frac{\dif r}{T_{1}+\rho-r}\nonumber \\
       & \qquad\qquad\le\frac{C\Xi_{2;T_{0},T_{1}}[v_{T_{0}}]^{2}}{\sfL(1/\rho)}\sfS_{\rho}(T_{1}-T_{0})\overset{\cref{eq:T1T0defs}}{=}\frac{CQ\Xi_{2;T_{0},T_{1}}[v_{T_{0}}]^{2}}{\sfL(1/\rho)},\label{eq:applycoupling}
    \end{align}
    for an absolute constant $C$.
    \item \emph{The closed equation.} We would like to replace
    $(\clS_{\zeta_{r}}[\sigma^{2}\circ v_{r}])^{1/2}$ on the left
    side of \cref{eq:applycoupling} by a quantity only depending on $V_{r}^{\rho,T_{1}}(x)$.
    This will allow us to create a closed equation for $(V_{t}^{\rho,T_{1}}(x))_{t}$.
    \begin{thmstepnv}
    \item \emph{Concentration.}
      Let $Q_1 = (Q_0 + \lambda^{-2})/2,$ so that $Q_1 \in (Q_0, \lambda^{-2})$.
      We apply \cref{thm:mainapproxthm} with $\xi\setto\zeta_{r}$
      (defined in \cref{eq:zetarappldef}), $\overline{q}\setto Q_1$,
      $t\setto r$, $\iota_{1},\iota_{2}\setto1$, and $z = 0$.
      Using also the definition \cref{eq:Vdef} of $V_{r}^{\rho,T_{1}}(x)$,
      we obtain for all $y\in\RR^{2}$, all $r\in[T_{0},T_{1}]$, and all
      sufficiently small $\rho$ that
      \begin{align}
        \EE & \left\lvert  \left(\clS_{\zeta_{r}}[\sigma^{2}\circ v_{r}](y)\right)^{1/2}-J_{\sigma,\rho}\bigl(\sfS_{\rho}(T_{1}-r),V_{r}^{\rho,T_{1}}(y)\bigr)\right\rvert  _{\Frob}^{2}\le C\langle\vvvert v_{0}\vvvert_{\ell}\rangle^{2}\frac{\sfL(\sfL(1/\rho))}{\sfL(1/\rho)}\label{eq:applyconcentration}
      \end{align}
      with a constant $C=C(M,\lambda,\beta,\ell,\overline{\beta},Q_0)<\infty$.
      The condition \cref{eq:mainapproxthmtcond} to apply \cref{thm:mainapproxthm}
      holds since $T_{0}\ge([1+\sfL(1/\rho)]\kappa_{\ell,\rho}^{-1}+\kappa_{\ell,\rho})\zeta_{r}^{2}$
      for all $r\in[T_{0},T_{1}]$ by the definitions \cref{eq:T1T0defs}
      and \cref{eq:zetarappldef}.
      Likewise, the condition $\xi \leq \bar{T}_\rho(\bar{q}^{-1/2})$ holds for sufficiently small $\rho$ because $\bar{q} = Q_1 > Q_0 \geq Q$, and the resulting constant depends on $Q_0$ and $\lambda$ through $Q_1$.

      \item \emph{Removing the spatial dependence.} The quantity $J_{\sigma,\rho}\bigl(\sfS_{\rho}(T_{1}-r),V_{r}^{\rho,T_{1}}(y)\bigr)$
      appearing on the left side of \cref{eq:applyconcentration} depends
      on $y$. Since we ultimately seek a one-dimensional problem, we would
      like to eliminate this dependence. By \cref{prop:approxJlipschitz,prop:continuitybd,prop:momentbd},
      we have, still
      for all $y\in\RR^2$ and all $r\in[T_{0},T_{1}]$, that
      \begin{equation}
        \begin{aligned}\EE & \left\lvert  J_{\sigma,\rho}\bigl(\sfS_{\rho}(T_{1}-r),V_{r}^{\rho,T_{1}}(y)\bigr)-J_{\sigma,\rho}\bigl(\sfS_{\rho}(T_{1}-r),V_{r}^{\rho,T_{1}}(x)\bigr)\right\rvert  _{\Frob}^{2}                                                                                                 \\
                   & \le\frac{\EE \left[\clG_{T_{1}-r}v_{r}(y)-\clG_{T_{1}-r}v_{r}(x)\right]^{2}}{\lambda^{-2}-\sfS_{\rho}(T_{1}-r)}\le C\langle\vvvert v_{0}\vvvert_{\ell}\rangle^{2}\left[\sfL_{\rho}\left(\frac{|x-y|^{2}}{2(T_{1} + \rho -r)}\right)+\frac{\abs{x - y}^2}{r}+\sfL_{\rho}\bigl(1+\sfL(1/\rho)\bigr)\right]
        \end{aligned}
        \label{eq:eliminatespatialdependence}
      \end{equation}
      for a (possibly larger) constant $C=C(M,\lambda,\beta,Q)<\infty$.
  
      \item \emph{Creating the closed equation.} Combining \cref{eq:applyconcentration}
      and \cref{eq:eliminatespatialdependence} via the triangle inequality,
      we obtain %
      \begin{align}
        \EE & \left\lvert  \left(\clS_{\zeta_{r}}[\sigma^{2}\circ v_{r}](y)\right)^{1/2}-J_{\sigma,\rho}\bigl(\sfS_{\rho}(T_{1}-r),V_{r}^{\rho,T_{1}}(x)\bigr)\right\rvert  _{\Frob}^{2}\nonumber                                                                                                                                                       \\
            & \le\frac{C\langle\vvvert v_{0}\vvvert_{\ell}\rangle^{2}}{\sfL(1/\rho)}\left[\sfL\left(\frac{|x-y|^{2}}{2(T_{1}+\rho-r)}\right) + \frac{\sfL(1/\rho)\abs{x - y}^2}{r} +\sfL(\sfL(1/\rho))\right]\nonumber\\
            & \le\frac{C\langle\vvvert v_{0}\vvvert_{\ell}\rangle^{2}}{\sfL(1/\rho)}\left[\frac{|x-y|^{2}}{T_{1}+\rho-r} + \frac{\sfL(1/\rho)\abs{x - y}^2}{r} + \sfL(\sfL(1/\rho))\right]\label{eq:concentrationdone}
      \end{align}
      for sufficiently small $\rho$. We wish to continue from \cref{eq:applycoupling},
      so we write
      \begin{align}
        \EE & \left\lvert  \gamma_{\rho}\int_{T_{0}}^{t}\clG_{T_{1}+\rho-r}\left[\left((\clS_{\zeta_{r}}[\sigma^{2}\circ v_{r}])^{1/2}-J_{\sigma,\rho}\bigl(\sfS_{\rho}(T_{1}-r),V_{r}^{\rho,T_{1}}(x)\bigr)\right)\,\dif\widetilde{W}_{r}\right](x)\right\rvert  ^{2}\nonumber      \\
            & \le\frac{4\pi}{\sfL(1/\rho)}\int_{T_{0}}^{t}\!\!\int G_{T_{1}+\rho-r}^{2}(x-y)\EE \left\lvert  (\clS_{\zeta_{r}}[\sigma^{2}\circ v_{r}](y))^{1/2}-J_{\sigma,\rho}\bigl(\sfS_{\rho}(T_{1}-r),V_{r}^{\rho,T_{1}}(x)\bigr)\right\rvert  _{\Frob}^{2}\,\dif y\,\dif r\nonumber \\
            & \le\frac{C\langle\vvvert v_{0}\vvvert_{\ell}\rangle^{2}}{\sfL(1/\rho)^{2}}\int_{T_{0}}^{t}\!\!\int\frac{G_{\frac{1}{2}[T_{1}+\rho-r]}(x-y)}{T_{1}+\rho-r}\left[\frac{|x-y|^{2}}{T_{1}+\rho-r}+ \frac{\sfL(1/\rho)\abs{x - y}^2}{r}+\sfL(\sfL(1/\rho))\right]\,\dif y\,\dif r\nonumber          \\
            & =\frac{C\langle\vvvert v_{0}\vvvert_{\ell}\rangle^{2}[1+\sfL(\sfL(1/\rho))]}{\sfL(1/\rho)^{2}}\int_{T_{0}}^{t}\frac{\dif r}{T_{1}+\rho-r} + \frac{C\langle\vvvert v_{0}\vvvert_{\ell}\rangle^{2}}{\sfL(1/\rho)} \int_{T_0}^t \frac{\dif r}{r}\nonumber                                                                                                 \\
            & =\frac{C\langle\vvvert v_{0}\vvvert_{\ell}\rangle^{2}}{\sfL(1/\rho)}\left([1+\sfL(\sfL(1/\rho))]\sfS_{\rho}(t-T_{0}) + \log \frac{T_1}{T_0}\right)\le C\langle\vvvert v_{0}\vvvert_{\ell}\rangle^{2}\frac{\sfL(\sfL(1/\rho))}{\sfL(1/\rho)}\label{eq:comparetoclosedthing}
      \end{align}
      with the second inequality by \cref{eq:concentrationdone} and \cref{eq:GT2}
      and the last inequality by \cref{eq:T1T0defs} (since the constant $C$
      is allowed to depend on $Q_{0}$).
      We emphasize that in the quantity
      $J_{\sigma,\rho}\bigl(\sfS_{\rho}(T_{1}-r),V_{r}^{\rho,T_{1}}(x)\bigr)$
      appearing on the left side, the point $x\in\RR^{2}$ is \emph{fixed},
      so this quantity is treated as a constant from the perspective of
      the operator $\clG_{T_{1}+\rho-r}$. In particular we have
      \begin{equation}
        \gamma_{\rho}\int_{T_{0}}^{t}\clG_{T_{1}+\rho-r}\left[J_{\sigma,\rho}\bigl(\sfS_{\rho}(T_{1}-r),V_{r}^{\rho,T_{1}}(x)\bigr)\,\dif\widetilde{W}_{r}\right](x)=\int_{T_{0}}^{t}J_{\sigma,\rho}\bigl(\sfS_{\rho}(T_{1}-r),V_{r}^{\rho,T_{1}}(x)\bigr)\,\dif\hat{B}(r),\label{eq:Jatxisaconstant}
      \end{equation}
      where we have defined
      \begin{equation}
        \hat{B}(t)=\gamma_{\rho}\int_{T_{0}}^{t}(\clG_{T_{1}+\rho-r}\,\dif\widetilde{W}_{r})(x)\qquad\text{for }t\ge T_{0}.\label{eq:Bhat}
      \end{equation}
      Combining \cref{eq:VTintegralrestate}, \cref{eq:fudgethebeginning},
      \cref{eq:applycoupling}, \cref{eq:comparetoclosedthing}, and \cref{eq:Jatxisaconstant},
      we see that
      \begin{equation}
        \begin{aligned}\EE  & \left\lvert  V_{t}^{\rho,T_{1}}(x)-a-\int_{T_{0}}^{t}J_{\sigma,\rho}\bigl(\sfS_{\rho}(T_{1}-r),V_{r}^{\rho,T_{1}}(x)\bigr)\,\dif\hat{B}(r)\right\rvert  ^{2}\le C\langle\vvvert v_{0}\vvvert_{\ell}\rangle^{2}\frac{\sfL(\sfL(1/\rho))}{\sfL(1/\rho)},\end{aligned}
        \label{eq:concludestep2}
      \end{equation}
      with a new constant $C=C(M,\lambda,Q)$. This holds
      for all $t\in[T_{0},T_{1}]$, and we note that the left side only
      involves the one-dimensional process $\bigl(V_{t}^{\rho,T_{1}}(x)\bigr)_{t\in[T_{0},T_{1}]}$
      and the one-dimensional noise $\hat{B}$. Thus we have established
      an approximate closed equation for this one-dimensional stochastic
      process.
    \end{thmstepnv}
    \item \emph{\label{step:chgvar}The change of variables.} The approximate
    equation \cref{eq:concludestep2} for $\bigl(V_{t}^{\rho,T_{1}}(x)\bigr)_{t\in[T_{0},T_{1}]}$
    is closed, but it depends on $\rho$ in a singular way as $\rho\searrow0$.
    We now change variables so that as $\rho\searrow0$,
    we see a limiting SDE.
    As a first step, we compute the quadratic
    variation of the process $\hat{B}(r)$.
    Directly from the definition
    \cref{eq:Bhat}, we have
    \begin{align}
        [\hat{B}]_{t}=\frac{4\pi\Id_m}{\sfL(1/\rho)}\int_{T_{0}}^{t}\|G_{T_{1}+\rho-r}\|_{L^{2}(\RR^{2})}^{2}\,\dif r\overset{\cref{eq:GtauL2norm}}{=}\frac{\Id_m}{\sfL(1/\rho)}\int_{T_{0}}^{t}\frac{\dif r}{T_{1}+\rho-r} & =\frac{\Id_m}{\sfL(1/\rho)}\log\frac{T_{1}-T_{0}+\rho}{T_{1}-t+\rho}\nonumber \\
             \overset{\cref{eq:sfUdef}}                                                                                                                                                                                 & {=}\Id_m\sfU_{T_{0},T_{1},\rho}(t).\label{eq:BhatQV}
    \end{align}
    This means that there is a \emph{standard} $\RR^m$-valued Brownian motion $B$ such
    that
    \begin{equation*}
      \hat{B}(t)=B(\sfU_{T_{0},T_{1},\rho}(t)).
    \end{equation*}
    Making a change of variables $r=\sfR_{T_{0},T_{1},\rho}(q)$
    in the integral appearing in \cref{eq:concludestep2}, we obtain, for
    all $t\in[T_{0},T_{1}]$, that
    \begin{align}
      \int_{T_{0}}^{t}J_{\sigma,\rho}\bigl(\sfS_{\rho}(T_{1}-r),V_{r}^{\rho,T_{1}}(x)\bigr)\,\dif\hat{B}(r) & =\int_{0}^{\sfU_{T_{0},T_{1},\rho}(t)}J_{\sigma,\rho}\bigl(\sfS_{\rho}\bigl(T_{1}-\sfR_{T_{0},T_{1},\rho}(q)\bigr),V_{\sfR_{T_{0},T_{1},\rho}(q)}^{\rho,T_{1}}(x)\bigr)\,\dif B(q)\nonumber \\
                                                                                                         & =\int_{0}^{\sfU_{T_{0},T_{1},\rho}(t)}J_{\sigma,\rho}\bigl(Q-q,V_{\sfR_{T_{0},T_{1},\rho}(q)}^{\rho,T_{1}}(x)\bigr)\,\dif B(q),\label{eq:dothechgvar}
    \end{align}
    with the last identity since $\sfS_{\rho}(T_{1}-\sfR_{T_{0},T_{1},\rho}(q))=Q-q$
    by \cref{eq:sfRdef} and \cref{eq:T1T0defs}.
    Let
    \begin{equation}
      \Upsilon(p)=V_{\sfR_{T_{0},T_{1},\rho}(p)}^{\rho,T_{1}}(x),\label{eq:Upsilondef-1-2}
    \end{equation}
    Using \cref{eq:dothechgvar}
    in \cref{eq:concludestep2} with $t=\sfR_{T_{0},T_{1},\rho}(p)$, we obtain
    \begin{equation}
      \sup_{p\in[0,Q]}\EE \left[\Upsilon(p)-a-\int_{0}^{p}J_{\sigma,\rho}\bigl(Q-q,\Upsilon(q)\bigr)\,\dif B(q)\right]^{2}\le C\langle\vvvert v_{0}\vvvert_{\ell}\rangle^{2}\frac{\sfL(\sfL(1/\rho))}{\sfL(1/\rho)}.\label{eq:makeVeqn-conclusion-2}
    \end{equation}
    \item \emph{Using the SDE solution theory.}
    The estimate \cref{eq:makeVeqn-conclusion-2} says that $\Upsilon$
    approximately satisfies an SDE. We will now use the SDE solution theory
    of \cref{subsec:SDE-solution-theory} to show that this means that $\Upsilon$
    is close to a solution of that SDE. Recalling the definition \cref{eq:Rdef},
    we have
    \[
      \clR_{a,Q}^{J_{\sigma,\rho}}\Upsilon(p)=a+\int_{0}^{p}J_{\sigma,\rho}\bigl(Q-q,\Upsilon(q)\bigr)\,\dif B(q).
    \]
    Because $v_0 \equiv a$, \cref{eq:makeVeqn-conclusion-2} becomes
    \begin{equation}
      \EE \left\lvert  \Upsilon(p)-\clR_{a,Q}^{J_{\sigma,\rho}}\Upsilon(p)\right\rvert  ^{2}\le C\langle a\rangle^{2}\frac{\sfL(\sfL(1/\rho))}{\sfL(1/\rho)},\label{eq:makeVeqn-conclusion-1-1}
    \end{equation}
    with a constant $C=C(M,\lambda,Q_{0})<\infty$. Now by the bound
    \cref{eq:SDEclose} of \cref{prop:SDEwellposed} (recalling that $\Lip(J_{\sigma,\rho}(q,\cdot))\le(\lambda^{-2}-q)^{-1/2}$
    for all $q\in[0,\lambda^{-2})$ by \cref{prop:approxJlipschitz}) in
    conjunction with \cref{eq:makeVeqn-conclusion-1-1}, we obtain
    \[
      \sup_{p\in[0,Q]}\EE |\Upsilon(p)-\Theta_{a,Q}^{J_{\sigma,\rho}}(p)|^{2}\le C\langle a\rangle^{2}\frac{\sfL(\sfL(1/\rho))}{\sfL(1/\rho)}
    \]
    for a new constant $C$, depending on the same parameters. Here, $\Theta_{a,Q}^{J_{\sigma,\rho}}$
    is as in \zcref[range]{eq:thetaSDE,eq:thetaIC}. Taking $p=Q$ and
    noting that
    \[
      \Upsilon(Q)\overset{\cref{eq:Upsilondef-1-2}}{=}V_{\sfR_{T_{0},T_{1},\rho}(Q)}^{\rho,T_{1}}\overset{\cref{eq:Rendpoints}}{=}V_{T_{1}}^{\rho,T_{1}}\overset{\cref{eq:atthefinaltime}}{=}v_{T_{1}}(x),
    \]
    we obtain in particular that
    \begin{equation}
      \EE |v_{T_{1}}(x)-\Theta_{a,Q}^{J_{\sigma,\rho}}(Q)|^{2}\le C\langle a\rangle^{2}\frac{\sfL(\sfL(1/\rho))}{\sfL(1/\rho)}.\label{eq:vclosetoTheta-1}
    \end{equation}
    \item \emph{Using the FBSDE solution theory}. We recall the definition \cref{eq:Qdef}
    of the operator $\clQ_{\sigma}$ and the definition \cref{eq:Jdef-expscale}
    of the approximate root decoupling function $J_{\sigma,\rho}$. Using these definitions along
    with \cref{prop:matrixvaluedreversetriangleinequality}, we obtain,
    for all $Q\in[0,Q_{0}]$, that
    \begin{align*}
      | & J_{\sigma,\rho}(\sfS_{\rho}(T_{1}),a)-\clQ_{\sigma}J_{\sigma,\rho}(Q,a)|_{\Frob}^{2}=\bigl\lvert \bigl[\EE \sigma^{2}\bigl(v_{T_{1}}(x)\bigr)\bigr]^{1/2}-\bigl[\EE \sigma^{2}\bigl(\Theta_{a,Q}^{J_{\sigma,\rho}}(Q)\bigr)\bigr]^{1/2}\bigr\rvert _{\Frob}^{2}                                                             \\
        & \le\EE \bigl\lvert \sigma\bigl(v_{T_{1}}(x)\bigr)-\sigma\bigl(\Theta_{a,Q}^{J_{\sigma,\rho}}(Q)\bigr)\bigr\rvert _{\Frob}^{2}\le\lambda^{2}\EE |v_{T_{1}}(x)-\Theta_{a,Q}^{J_{\sigma,\rho}}(Q)|^{2}\overset{\cref{eq:vclosetoTheta-1}}{\le}C\langle a\rangle^{2}\frac{\sfL(\sfL(1/\rho))}{\sfL(1/\rho)}
    \end{align*}
    for a new constant $C=C(M,\lambda,Q)<\infty$. On the other
    hand, we have by \cref{prop:approxJtemporalregularity} that
    \[
      |J_{\sigma,\rho}(\sfS_{\rho}(T_{1}),a)-J_{\sigma,\rho}(Q,a)|_{\Frob}\le\lambda K_{\lambda,\rho,T_{1}}^{1/2}\Xi_{2;T_{1}}[a]\left(|\sfS_{\rho}(T_{1})-Q|^{1/2}+\frac{\sfL(2)^{1/2}K_{\lambda,\rho,T_{1}}^{1/2}}{\sfL(1/\rho)^{1/2}}\right)
    \]
    and
    \[
      0\le\sfS_{\rho}(T_{1})-Q\overset{
        \cref{eq:Lmult},\\
        \cref{eq:T1T0defs}
      }{\le}\sfL_{\rho}\bigl(\sfL(1/\rho)\bigl([1+\sfL(1/\rho)]\kappa_{\ell, \rho}^{-1} + \kappa_{\ell, \rho}\bigr)+1\bigr)\le\frac{C\sfL(\sfL(1/\rho))}{\sfL(1/\rho)}.
    \]
    Combining the last three displays (along with \cref{prop:momentbd}
    to control the moment) we see that
    \[
      |J_{\sigma,\rho}(Q,a)-\clQ_{\sigma}J_{\sigma,\rho}(Q,a)|_{\Frob}^{2}\le C\langle a\rangle^{2}\frac{\sfL(\sfL(1/\rho))}{\sfL(1/\rho)}.
    \]
    Now we apply \cref{eq:onestepclose} of \cref{prop:localwellposed}, with
    $g\setto J_{\sigma,\rho}$, to see that
    \[
      \sup_{Q\in[0,Q_{0}]}\|(J_{\sigma,\rho}-J_{\sigma})(Q,\cdot)\|_{\clX}\le C\sup_{\substack{Q\in[0,Q_0]\\
          a\in\RR
        }
      }\frac{|J_{\sigma,\rho}(Q,a)-\clQ_{\sigma}J_{\sigma,\rho}(Q,a)|_{\Frob}}{\langle a\rangle}\le C\left[\frac{\sfL(\sfL(1/\rho))}{\sfL(1/\rho)}\right]^{1/2},
    \]
    which is \cref{eq:Japprox-concl}.\qedhere
  \end{thmstepnv}
\end{proof}

\section{Induction on scales\label{sec:extendtheestimates}}

The results of \cref{sec:local-SPDE-estimates} were ``local'' in
the sense that they required the time scale of interest to be shorter
than $\overline{T}_{\rho}(\lambda)\approx\sfT_{\rho}(\lambda^{-2})\approx\rho^{1-\lambda^{-2}}$,
where $\lambda$ is the Lipschitz constant of the nonlinearity $\sigma$.
The goal of this section is to extend these results to longer time
scales, much as \cref{subsec:extendingFBSDE} extended
the results of \cref{subsec:FBSDE-local}. A typical
instance of the phenomenology, which we will use as the
hypothesis in an inductive argument, is that the conclusion of \cref{thm:mainapproxthm}
may hold with $\overline{q}\ge\lambda^{-2}$, even though \cref{thm:mainapproxthm}
only asserts that it is true for $\overline{q}<\lambda^{-2}$. This
is roughly what we encode the following definition.

We recall $\Qbar_{\FBSDE}(\sigma)$ from \cref{def:QbarFBSDE} and let\nomenclature[zzzgreek β star]{$\beta_*(\sigma)$}{infimal $\beta$ such that $\sigma\in\quadset(M,\beta)$ for some $M$}%
\begin{equation}
  \label{eq:minimal-beta}
  \beta_*(\sigma) \coloneqq \inf\bigl\{\beta \in \R_+ \suchthat \sigma \in \quadset(M, \beta) \text{ for some } M \in \R_+\bigr\}.
\end{equation}
\begin{defn}
  \label{def:QSPDE}Given $\sigma\in\Lip(\RR^m; \clH_+^m)$, we let $\Qbar_{\SPDE}(\sigma)$ denote the supremum of all
  \begin{equation*}
    \overline{q} \in \bigl(0, \Qbar_{\FBSDE}(\sigma) \wedge \beta_*(\sigma)^{-2}\bigr)
  \end{equation*}
  such that
  for all $\overline{\beta} \in \bigl(\beta_*(\sigma),\, \bar{q}^{-1/2}\bigr)$ and $\ell \in (2, 4],$
  there is a constant $C=C(\sigma,\overline{q},\overline{\beta},\ell)\in(0,\infty)$
  and a function $\omega=\omega_{\sigma,\overline{q},\overline{\beta},\ell}\colon(0,C^{-1})\to[0,\infty)$ such that the following conditions hold.
  \begin{enumerate}[label = (\roman*)]
  \item For all $\rho \in (0, C^{-1})$, we have
    \begin{equation}
      \label{eq:omegaissmallenough}
      \sfL\bigl(\omega(\rho)\bigr)\le C\sfL\bigl(\sfL(1/\rho)\bigr)
    \end{equation}
    and
    \begin{equation}
      \label{eq:omegaisbigenough}
      \omega(\rho)\ge C^{-1}\sfL(1/\rho)^{2}
    \end{equation}
    as well as $\omega(\rho)\sfT_{\rho}(\bar q) \leq \sfT_{\rho}\bigl(\overline{\beta}^{-2}\bigr)$.

  \item
	  If $\rho\in(0,C^{-1})$ and $(v_{t})_{t\ge0}$ solves \cref{eq:SPDE} with initial condition $v_{0}\in\scrX_{0}^{\ell}$, then for all $q\in[0,\overline{q}]$, \begin{equation}t \in \bigl[\omega(\rho)\sfT_{\rho}(q), \sfT_{\rho}\bigl(\overline{\beta}^{-2}\bigr)\bigr],\label{eq:QSPDEdefinterval}\end{equation} and $x,z\in\RR^{2}$, we have
    \begin{equation}
      \label{eq:QSPDEapprox-3}
      \EE \left\lvert  \bigl(\clS_{\sfT_{\rho}(q)^{1/2},z}[\sigma^{2}\circ v_{t}](x)\bigr)^{1/2}-J_{\sigma}\bigl(q,\clG_{\sfT_{\rho}(q)}v_{t}(x)\bigr)\right\rvert  _{\Frob}^{2}\le\frac{C\langle\vvvert v_{0}\vvvert_{\ell}\rangle^{2}\sfL(\sfL(1/\rho))}{\sfL(1/\rho)}.
    \end{equation}
  \end{enumerate}
\end{defn}
\nomenclature[QbarSPDE]{$\Qbar_\SPDE(\sigma)$}{supremal scale at which certain estimates on the approximation of the SPDE hold, see \cref{def:QSPDE}}
The bound \cref{eq:QSPDEapprox-3}, which we term the ``key approximation,'' is the heart of \cref{def:QSPDE}.
It parallels the conclusion \cref{eq:mainapproxthm-conclusion} of \cref{thm:mainapproxthm}, and essentially says that the root decoupling function $J_\sigma$ tracks the differential quadratic variation of the martingale $V_t^{\sigma,\rho,T}$ on the exponential scale.
For technical reasons (originating in the proof of \cref{prop:coupling}), we average $\sigma^2\circ v$ over a square rather than with respect to a Gaussian.

The function $\omega$ in \cref{def:QSPDE} plays the role of the factor $(1+\sfL(1/\rho))\kappa^{-1}_{\ell,\rho}+\kappa_{\ell,\rho}$ appearing in the interval of validity \cref{eq:mainapproxthmtcond} of \cref{thm:mainapproxthm}. As we inductively prove that \cref{def:QSPDE} holds for successively larger values of $\overline q$, the size of $\omega$ will grow somewhat: see \cref{eq:omegabardef} in the proof of \cref{prop:extendQSPDE-1} below. The bound \cref{eq:omegaissmallenough} requires that the factor $\omega$ appearing in \cref{eq:QSPDEdefinterval} is negligible in terms of the exponential time scale.
\begin{rem}
    To establish a lower bound $\bar{q}$ on $\Qbar_\SPDE(\sigma)$, it suffices to find, for every $\overline{\beta}\in \bigl(\beta_*(\sigma),\overline{q}^{-1/2}\bigr)$, an $\ell_0=\ell_0(\overline\beta)\in (2,4]$ such that the conditions in \cref{def:QSPDE} hold for all $\ell\in (2,\ell_0]$.
    This is because the spaces $\clX^{\ell}_0$ only get smaller (and the norms $\vvvert \anon\vvvert_\ell$ get larger) as $\ell$ increases.\label{rem:considerjustellsmall}
\end{rem}
We note that
\begin{equation*}
  \Qbar_{\SPDE}(\sigma)\le\Qbar_{\FBSDE}(\sigma)
\end{equation*}
by definition. The goal of this section will be to prove the following
theorem.
\begin{thm}
  \label{thm:QSPDEgreater}Let $M\in(0,\infty)$, $\beta\in(0,1)$,
  and $\sigma\in\quadset(M,\beta)$. If $\Qbar_{\FBSDE}(\sigma)>1$,
  then $\Qbar_{\SPDE}(\sigma)>1$.
\end{thm}

We will prove \cref{thm:QSPDEgreater} through an inductive procedure. The
base case, which will also be used in the inductive step, is the following.
\begin{prop}
  \label{prop:QSPDEbasecase}For any $\sigma\in\Lip(\RR^m;\clH_+^m)$,
  we have $\Qbar_{\SPDE}(\sigma)\ge\Lip(\sigma)^{-2}$.
\end{prop}

\begin{proof}
  Let $\ell\in (2,4]$ and let $\beta$ and $\overline\beta$ satisfy $\beta_*(\sigma)<\beta<\eta(\beta,\ell)<\overline\beta$.
  (It suffices to consider the case when the last inequality holds by \cref{rem:considerjustellsmall}.)
  Take $M\in(0,\infty)$ such that $\sigma\in\quadset(M,\beta)$ and let $\lambda=\Lip(\sigma)$.
  Fix $\overline{q}\in[0,\lambda^{-2})$ and let $\h{q} = (\bar{q}+\lambda^{-2})/2$, so that $\h{q} \in (\bar{q}, \lambda^{-2})$.
  Take $q \in [0, \bar{q}]$ and let $(v_{t})_{t\ge0}$ solve \cref{eq:SPDE} with initial condition $v_{0}\in\scrX_{0}^{\ell}$.
  Taking $\bar{q} \setto \h{q}$, $\xi\setto\sfT_{\rho}(q)^{1/2}$, and $\iota_{1},\iota_{2}\setto0$ in \cref{thm:mainapproxthm},
  we see that there exists $C=C(M,\lambda,\beta,\overline{\beta},\ell,\overline{q})<\infty$
  such that for sufficiently small $\rho$ and any $x,z\in\RR^{2}$ and $t\in[([1+\sfL(1/\rho)]\kappa_{{\ell},\rho}^{-1}+\kappa_{{\ell},\rho})\sfT_{\rho}(q),\sfT_{\rho}(\overline{\beta}^{-2})]$
  (in accordance with \cref{eq:mainapproxthmtcond}),
  we have
  \begin{align*}
    \sup_{z\in\RR^{2}}\EE \Big\lvert  \bigl(\clS_{\sfT_{\rho}(q)^{1/2},z}[\sigma^{2}\circ v_{t}](x)\bigr)^{1/2}-J_{\sigma,\rho}\bigl(q,\clG_{\sfT_{\rho}(q)}v_{t}(x)\bigr)\Big\rvert  _{\Frob}^{2} & \le C\langle\vvvert v_{0}\vvvert_{{\ell}}\rangle^{2}\frac{\sfL(\sfL(1/\rho))}{\sfL(1/\rho)}, %
  \end{align*}
  On the other hand, \cref{prop:Japprox,prop:momentbd} yield
  \begin{align}
    \EE \bigl\lvert J_{\sigma,\rho}\bigl(q,\clG_{\sfT_{\rho}(q)}v_{t}(x)\bigr)-J_{\sigma}\bigl(q,\clG_{\sfT_{\rho}(q)}v_{t}(x)\bigr)\bigr\rvert _{\Frob}^{2} & \le\frac{C\sfL(\sfL(1/\rho))}{\sfL(1/\rho)}\EE \langle\clG_{\sfT_{\rho}(q)}v_{t}(x)\rangle^{2}\nonumber \\
														   & \le\frac{C\sfL(\sfL(1/\rho))}{\sfL(1/\rho)}\EE \langle\vvvert v_{0}\vvvert_{\ell}\rangle^{2}\label{eq:approxJsigmarhobyJsigma}
  \end{align}
  for a constant $C<\infty$ depending only on $\sigma$ and $\overline{q}$.
  Combining the last two displays yields \cref{eq:QSPDEapprox-3} with
  $\omega(\rho)=[1+\sfL(1/\rho)]\kappa_{{\ell},\rho}^{-1}+\kappa_{{\ell},\rho}$.
  This evidently satisfies \zcref[range]{eq:omegaissmallenough,eq:omegaisbigenough} by the definition \cref{eq:kappadef} of $\kappa_{{\ell},\rho}$, completing the proof.
\end{proof}
The following proposition is the inductive step, and the main result
of this section.
\begin{prop}
  \label{prop:extendQSPDE-1}
  Let $\sigma\in\Lip(\RR^m;\clH_+^m)$ and $Q_{1}\in(0,\Qbar_{\SPDE}(\sigma) \wedge 1).$
  Then we have
  \begin{equation*}
    \Qbar_{\SPDE}(\sigma)\ge\bigl[Q_{1}+\Lip\bigl(J_{\sigma}(Q_{1},\anon)\bigr)^{-2}\bigr]\wedge\beta_*(\sigma)^{-2}.
  \end{equation*}
\end{prop}

Before proving \cref{prop:extendQSPDE-1}, we show how it and \cref{prop:QSPDEbasecase} imply \cref{thm:QSPDEgreater}.
\begin{proof}[Proof of \cref{thm:QSPDEgreater}.]
  Suppose for the sake of contradiction
  that $\Qbar_{\SPDE}(\sigma)\le1<\Qbar_{\FBSDE}(\sigma)\wedge\beta^{-2}$.
  Taking $\Qbar=\Qbar_{\SPDE}(\sigma)$ in \cref{def:QbarFBSDE},
  we find a constant $L\in(0,\infty)$ such that
  \[
    \Lip[J_{\sigma}(Q,\anon)]\le L\qquad\text{for all }Q\in[0,\Qbar_{\SPDE}(\sigma)].
  \]
  \cref{prop:extendQSPDE-1} then tells us that
  \begin{equation}
    \Qbar_{\SPDE}(\sigma)\ge(Q+L^{-2})\wedge\beta^{-2}\qquad\text{for all }Q\in\bigl(0,\Qbar_{\SPDE}(\sigma)\bigr).\label{eq:applyextendQSPDE}
  \end{equation}
  By \cref{prop:QSPDEbasecase}, the interval $\bigl(0,\Qbar_{\SPDE}(\sigma)\bigr)$ is nonempty.
  Therefore, taking $Q\nearrow\Qbar_{\SPDE}(\sigma)$
  in \cref{eq:applyextendQSPDE}, we see that $\Qbar_{\SPDE}(\sigma)\ge[\Qbar_{\SPDE}(\sigma)+L^{-2}]\wedge\beta^{-2}$,
  so either $\Qbar_{\SPDE}(\sigma)\ge\Qbar_{\SPDE}(\sigma)+L^{-2}$,
  which is absurd, or $\Qbar_{\SPDE}(\sigma)\ge\beta^{-2}$,
  contradicting our starting assumption.
\end{proof}
Now we prove \cref{prop:extendQSPDE-1}. Let us first give a brief sketch of the proof. In each stage of the induction, we will apply the key approximation \zcref{eq:QSPDEapprox-3}, which essentially forms the inductive hypothesis, twice. First, we use it ``in the small'', i.e.\ on the original field but on shorter time scales than we are ultimately interested in. This, along with the coupling in \zcref{prop:coupling}, will let us approximate the original problem by a coarse-grained problem with an effective nonlinearity. We then show that the requirements to apply the inductive hypothesis are satisfied with this effective nonlinearity, so that apply the inductive hypothesis to the coarse-grained problem to conclude the overall approximation.
\begin{proof}[Proof of \cref{prop:extendQSPDE-1}.]
  Take $\bar\beta > \beta > \beta_*(\sigma)$ and choose $Q_{2}$ and $Q_{2}'$ such that
  \begin{equation}
    Q_{1}<Q_{2}<Q_{2}'<\bigl[Q_{1}+\Lip\bigl(J_{\sigma}(Q_{1},\anon)\bigr)^{-2}\bigr]\wedge\beta^{-2}.\label{eq:Q2choice-1}
  \end{equation}
  Let $\ell\in(2,4]$.
  As explained in \cref{rem:considerjustellsmall}, we are free to assume that $\ell$ is close enough to $2$ that $\overline{\beta}>\eta(\beta,\ell)$. %
  By \cref{prop:momentbd}, this implies that there is a $C=C(\beta,\overline\beta,\ell,M)<\infty$ such that, if $\widetilde\ell\in [2,\ell]$, then
  \begin{equation}
    \Xi^{\sigma,\rho}_{\widetilde\ell;t}[v_0] %
    \le C\langle\vvvert v_0\vvvert_{\widetilde\ell}\rangle\qquad\text{for all }t\in [0,\sfT_\rho(\overline\beta^{-2})].\label{eq:applymomentbd-again}
  \end{equation}

  Our goal is to show that there is a constant $C\in(0,\infty)$ and
  a function $\overline{\omega}\colon(0,C^{-1})\to[0,\infty)$ satisfying
\zcref[range]{eq:omegaissmallenough,eq:omegaisbigenough} such that,
  if $\rho\in(0,C^{-1})$, $(v_{t})_{t\ge0}$ solves \cref{eq:SPDE}
  with initial condition $v_{0}\in\scrX_{0}^{\ell}$, $q\in[0,Q_{2}]$,
  $t\in\bigl[\overline{\omega}(\rho)\sfT_{\rho}(q),\sfT_{\rho}\bigl(\overline{\beta}^{-2}\bigr)\bigr]$,
  and $x,z\in\RR^{2}$, then
  \begin{equation}
    \EE \left\lvert  \bigl(\mathcal{S}_{\sfT_{\rho}(q)^{1/2},z}[\sigma^{2}\circ v_{t}](x)\bigr)^{1/2}-J_{\sigma}\bigl(q,\clG_{\sfT_{\rho}(q)}v_{t}(x)\bigr)\right\rvert  _{\Frob}^{2}\le\frac{C\langle\vvvert v_{0}\vvvert_{\ell}\rangle^2\sfL(\sfL(1/\rho))}{\sfL(1/\rho)}.\label{eq:thegoal}
  \end{equation}
  We proceed in several steps.
  \begin{thmstepnv}
    \item\emph{Base case.} \label{step:delta} Since $Q_{1}<\Qbar_{\SPDE}(\sigma)$, we can find a
    $\delta>0$ such that the claim holds for $q\in[0,Q_{1}+\delta]$. So we can assume that $q\in[Q_{1}+\delta,Q_{2}]$.
  \item \emph{Defining parameters for the multiscale argument.} The crux of the proof will be to show that a smoothing of the SPDE solution approximately satisfies the same SPDE with different values of $\rho$ and $\sigma$.
    In this step, we introduce the parameters to carry out this analysis. Define
    \begin{equation}
      \widetilde{\sigma}(a)\coloneqq(1-Q_{1})^{1/2}J_{\sigma}(Q_{1},a).\label{eq:sigmatildedef-1}
    \end{equation}
    Also define
    \begin{equation}
      \widetilde{\rho}\coloneqq\frac{\rho}{\sfT_{\rho}(1-Q_{1})}\qquad\text{and}\qquad\tau\coloneqq\widetilde{\rho}-\rho.\label{eq:rhotildeandtaudef-1}
    \end{equation}
    We smooth the solution at scale $\tau$, and the ``effective'' equation will have $(\widetilde\sigma,\widetilde\rho)$ in place of $(\sigma,\rho)$.
    We note that $\widetilde{\rho},\tau\approx\rho^{1-Q_{1}}$ for small
    $\rho$. (For this to be true, we need the assumption that $Q_{1}<1$.)
    In particular, it is not difficult to check that for all $\eps \in (0, 1)$, there is a  $C(\eps)<\infty$ such that
    \begin{equation}
      |\sfS_{\rho}(\tau)-Q_{1}|\le\frac{C(\eps)}{\sfL(1/\rho)}\qquad\text{for all } Q_1\in[0,1-\eps)\text{ and }\rho\in(0,C(\eps)^{-1}).\label{eq:STbarest-1}
    \end{equation}
    Now, we have the identity
    \begin{equation*}
      \sfL(1/\widetilde{\rho})\overset{\substack{\cref{eq:Ldef},\\
          \cref{eq:rhotildeandtaudef-1}
        }
      }{=}\log\left(1+\frac{\sfT_{\rho}(1-Q_{1})}{\rho}\right)\overset{\cref{eq:Tdef}}{=}\log[(1+1/\rho)^{1-Q_{1}}]\overset{\cref{eq:Ldef}}{=}(1-Q_{1})\sfL(1/\rho),
    \end{equation*}
    so \cref{eq:gammarhodef} yields
    \begin{equation*}
      \gamma_{\rho}=(1-Q_{1})^{1/2}\gamma_{\widetilde{\rho}}.
    \end{equation*}
    In particular, this means that
    \begin{equation}
      \gamma_{\rho}J_{\sigma}(Q_{1},a)=\gamma_{\widetilde{\rho}}\widetilde{\sigma}(a)\qquad\text{for all }a\in\RR^m.\label{eq:addtildes}
    \end{equation}
    Put
    \begin{equation}
      T_{0}=\omega(\rho)\tau,\label{eq:altogetherT0def}
    \end{equation}
    where
    \begin{equation*}
      \omega=\omega_{\sigma,Q_{1},\overline{\beta},\ell}
    \end{equation*}
    is as in \cref{def:QSPDE}, and
    \begin{equation}
      T_{1}=\sfT_{\rho}(\overline{\beta}^{-2}).\label{eq:altogetherT1def}
    \end{equation}
    We note that, if $\rho$ is sufficiently small, then
    \begin{equation}
      \sfS_{\widetilde{\rho}}(T_{1}-T_{0})\le\sfS_{\widetilde{\rho}}(T_{1})<\frac{\beta^{-2}-Q_{1}}{1-Q_{1}}\label{eq:SrhoT1T0approx}
    \end{equation}
    because $\bar\beta > \beta$.
    \item \emph{Replacing the noise.} By applying $\clG_{\tau}$ to both
    sides of \cref{eq:mildsoln}, we obtain, for all $t\ge t_{0}\ge T_{0}$,
    \begin{equation}
      \clG_{\tau}v_{t}=\clG_{t-t_{0}}\clG_{\tau}v_{t_{0}}+\gamma_{\rho}\int_{t_{0}}^{t}\clG_{t+\widetilde{\rho}-r}[\sigma(v_{r})\,\dif W_{r}].\label{eq:convolvetheSPDE-1}
    \end{equation}
    We study the integral on the right side. For it to look like an equation of the form \cref{eq:mildsoln} for $\clG_\tau v_t$, we would like to replace $\gamma_\rho\sigma(v_r)$ by $\gamma_{\widetilde\rho}\widetilde\sigma(\clG_\tau v_r)$. This replacement will  be a good approximation in distribution but \emph{not} almost surely, so we need to use a coupling. (See the final approximation in \cref{step:approximatespde} below.) To this end, let $\dif\widetilde{W}_{t}$
    be the space-time white noise constructed in \cref{prop:coupling} applied
    with $\zeta_{r}\setto\tau^{1/2}$ (independent of $r$) and
    $w_{r}\setto\sigma\circ v_{r}$. By the estimate \cref{eq:propcouplingconclusion},
    we have for any $t\ge0$ that
    \begin{align}
      \EE & \left\lvert  \gamma_{\rho}\int_{t_{0}}^{t}\clG_{t+\widetilde{\rho}-r}\left[\sigma(v_{r})\,\dif W_{r}-(\clS_{\tau^{1/2}}[\sigma^{2}\circ v_{r}])^{1/2}\,\dif\widetilde{W}_{r}\right](x)\right\rvert  ^{2}\nonumber                                                                                                                                       \\
          & \le\frac{C\tau\Xi_{2;t}^{\sigma,\rho}[v_{0}]^{2}}{\sfL(1/\rho)}\int_{t_{0}}^{t}(t+\widetilde{\rho}-r)^{-2}\,\dif r\le\frac{C\tau\Xi_{2;t}^{\sigma,\rho}[v_{0}]^{2}}{\widetilde{\rho}\sfL(1/\rho)}\overset{\cref{eq:applymomentbd-again}}{\le}\frac{C\langle\vvvert v_{0}\vvvert_{2}\rangle^{2}}{\sfL(1/\rho)},\label{eq:newnoise-concl-1}
    \end{align}
    where in the last estimate we also used %
    the fact
    that $\tau\le\widetilde{\rho}$ by the definition \cref{eq:rhotildeandtaudef-1}.
  \item \emph{Applying the key approximation in the small.}
    We now use the inductive hypothesis (namely, that $Q_1<\overline{Q}_\SPDE(\sigma)$) to approximate the term $(\clS_{\tau^{1/2}}[\sigma^{2}\circ v_{r}])^{1/2}$ appearing in the left side of \cref{eq:newnoise-concl-1}.
    We call this approximation ``in the small'' because it applies the approximation to the original field $v_r$ but only over short time intervals, i.e.\ those of length $\rho^{Q_1}$.
    We relate it to the root decoupling function $J_\sigma$ via the key approximation \cref{eq:QSPDEapprox-3}.
Whenever $T_{0}\le t_{0}\le t\le T_{1}$
    we have
    \begin{align}
      \EE                     & \left\lvert  \gamma_{\rho}\int_{t_{0}}^{t}\clG_{t+\widetilde{\rho}-r}\left[\left((\clS_{\tau^{1/2}}[\sigma^{2}\circ v_{r}])^{1/2}-J_{\sigma}(Q_{1},\clG_{\tau}v_{r})\right)\,\dif\widetilde{W}_{r}\right](x)\right\rvert  ^{2}\nonumber                                                                \\
                              & =\frac{4\pi}{\sfL(1/\rho)}\int_{t_{0}}^{t}\!\!\!\int\EE \left\lvert  \bigl(\clS_{\tau^{1/2}}[\sigma^{2}\circ v_{r}](y)\bigr)^{1/2}-J_{\sigma}\bigl(Q_{1},\clG_{\tau}v_{r}(y)\bigr)\right\rvert  _{\Frob}^{2}G_{t+\widetilde{\rho}-r}^{2}(x-y)\,\dif y\,\dif r\nonumber                                                           \\
      \overset{\cref{eq:GT2}} & {=}\frac{1}{\sfL(1/\rho)}\int_{t_{0}}^{t}\!\!\!\int\EE \left\lvert  \bigl(\clS_{\tau^{1/2}}[\sigma^{2}\circ v_{r}](y)\bigr)^{1/2}-J_{\sigma}\bigl(Q_{1},\clG_{\tau}v_{r}(y)\bigr)\right\rvert  _{\Frob}^{2}\frac{G_{\frac{1}{2}[t+\widetilde{\rho}-r]}(x-y)}{t+\widetilde{\rho}-r}\,\dif y\,\dif r.\label{eq:starttoapplyindhyp}
    \end{align}
    We now apply \cref{eq:QSPDEapprox-3} with $q\setto\sfS_{\rho}(\tau)$,
    which is allowed as long as $\rho$ is sufficiently small by the assumption
    that $Q_{1}<\Qbar_{\SPDE}(\sigma)$ in light of \cref{eq:STbarest-1}.
    From this we see that, for all
    \[
      r\in[\omega(\rho)\tau,\sfT_{\rho}(\overline{\beta}^{-2})]=[T_{0},T_{1}]
    \]
    (recalling the definitions \zcref[range]{eq:altogetherT0def,eq:altogetherT1def}),
    we have, for all $y\in\RR^{2}$, that
    \begin{equation}
      \EE \left\lvert  \bigl(\clS_{\tau^{1/2}}[\sigma^{2}\circ v_{r}](y)\bigr)^{1/2}-J_{\sigma}\bigl(Q_{1},\clG_{\tau}v_{r}(y)\bigr)\right\rvert  _{\Frob}^{2}\le\frac{C\langle\vvvert v_{0}\vvvert_{\ell}\rangle^{2}\sfL(\sfL(1/\rho))}{\sfL(1/\rho)}.\label{eq:smallinductivehyp}
    \end{equation}
    In deriving \cref{eq:smallinductivehyp}, we also used \cref{eq:STbarest-1},
    \cref{prop:Jsigmatimereg-1}, and \cref{prop:momentbd} to replace $\sfS_{\rho}(\tau)$
    by $Q_{1}$ in the first argument of $J_{\sigma}$. Using \cref{eq:smallinductivehyp}
    in \cref{eq:starttoapplyindhyp}, we get
    \begin{align}
      \EE & \left\lvert  \gamma_{\rho}\int_{t_{0}}^{t}\clG_{t+\widetilde{\rho}-r}\left[\left((\clS_{\tau^{1/2}}[\sigma^{2}\circ v_{r}])^{1/2}-J_{\sigma}(Q_{1},\clG_{\tau}v_{r})\right)\,\dif\widetilde{W}_{r}\right](x)\right\rvert  _{\Frob}^{2}\nonumber                                                                                                       \\
          & \le\frac{C\langle\vvvert v_{0}\vvvert_{\ell}\rangle^{2}\sfL(\sfL(1/\rho))}{\sfL(1/\rho)^{2}}\int_{t_{0}}^{t}\frac{\dif r}{t+\widetilde{\rho}-r}\nonumber                                                                                                                                                                                \\
          & =\frac{C\langle\vvvert v_{0}\vvvert_{\ell}\rangle^{2}\sfL(\sfL(1/\rho))}{\sfL(1/\rho)^{2}}\log\frac{t-t_{0}+\widetilde{\rho}}{\widetilde{\rho}}\overset{\cref{eq:Srhodef}}{=}\frac{C\langle\vvvert v_{0}\vvvert_{\ell}\rangle^{2}\sfL(\sfL(1/\rho))\sfL(1/\widetilde{\rho})}{\sfL(1/\rho)^{2}}\sfS_{\widetilde{\rho}}(t-t_{0})\nonumber \\
      \overset{\substack{\cref{eq:SrhoT1T0approx},                                                                                                                                                                                                                                                                                                  \\
          \cref{eq:altogetherT1def}
        }
      }   & {\le}\frac{C\langle\vvvert v_{0}\vvvert_{\ell}\rangle^{2}\sfL(\sfL(1/\rho))}{\sfL(1/\rho)},\label{eq:applyindhyp}
    \end{align}
    where the constant $C$ depends on $Q_{1}$ and $\beta$.
    \item \emph{The approximate SPDE. }\label{step:approximatespde}Combining \cref{eq:addtildes,eq:newnoise-concl-1,eq:applyindhyp}, we see that, whenever $T_{0}\le t_{0}\le t\le T_{1}$,
    we have
    \begin{equation*}
       \left\vvvert \gamma_{\rho}\int_{t_{0}}^{t}\clG_{t+\widetilde{\rho}-r}\left[\sigma(v_{r})\,\dif W_{r}\right]-\gamma_{\widetilde{\rho}}\int_{t_{0}}^{t}\clG_{t+\widetilde{\rho}-r}\left[\widetilde{\sigma}(\clG_{\tau}v_{r})\,\dif\widetilde{W}_{r}\right]\right\vvvert _{2}^{2} \le\frac{C\langle\vvvert v_{0}\vvvert_{\ell}\rangle^{2}\sfL(\sfL(1/\rho))}{\sfL(1/\rho)}.
    \end{equation*}
    In light of \cref{eq:convolvetheSPDE-1}, this means that
    \begin{equation}
      \left\vvvert \clG_{\tau}v_{t}-\clG_{t-t_{0}}\clG_{\tau}v_{t_{0}}-\gamma_{\widetilde{\rho}}\int_{t_{0}}^{t}\clG_{t+\widetilde{\rho}-r}\left[\widetilde{\sigma}(\clG_{\tau}v_{r})\,\dif\widetilde{W}_{r}\right]\right\vvvert _{2}^{2}\le\frac{C\langle\vvvert v_{0}\vvvert_{\ell}\rangle^{2}\sfL(\sfL(1/\rho))}{\sfL(1/\rho)}.\label{eq:approx}
    \end{equation}
    Now define
    \begin{equation}
      t_{1}=t_{0}+\sfT_{\widetilde{\rho}}\left(\frac{Q_{2}'-Q_{1}}{1-Q_{1}}\right).\label{eq:t1def}
    \end{equation}
    The estimate \cref{eq:approx} means that $(\clG_{\tau}v_{t})_{t\ge t_{0}}$
    is approximately a solution to the SPDE \cref{eq:SPDE} with $\rho\setto\widetilde{\rho}$,
    $\sigma\setto\widetilde{\sigma}$, and $\dif W_{r}\setto\dif\widetilde{W}_{r}$,
    in the sense of the norm $\clM_{t_{0},t_{1}}$ defined in \cref{eq:scrMstdef}.
    We thus seek to apply \cref{prop:solnapprox} with these choices of
    parameters. We note that
    \[
      \Lip(\widetilde{\sigma})\sfS_{\widetilde{\rho}}(t_{1}-t_{0})^{1/2}\overset{\substack{\cref{eq:sigmatildedef-1},\\
          \cref{eq:t1def}
        }
      }{=}(1-Q_{1})^{1/2}\Lip\bigl(J_{\sigma}(Q_{1},\anon)\bigr)\left(\frac{Q_{2}'-Q_{1}}{1-Q_{1}}\right)^{1/2}=(Q_{2}'-Q_{1})^{1/2}\Lip\bigl(J_{\sigma}(Q_{1},\anon)\bigr)\overset{\cref{eq:Q2choice-1}}{<}1.
    \]
    Therefore, \cref{prop:solnapprox} tells us that, for all $t\in[t_{0},t_{1}]$,
    we have
    \begin{equation}
      \vvvert\clG_{\tau}v_{t}-\widetilde\clV_{t_{0},t}^{\widetilde{\sigma},\widetilde{\rho}}\clG_{\tau}v_{t_{0}}\vvvert_{2}^{2}\le\frac{C\langle\vvvert v_{0}\vvvert_{2}\rangle^{2}\sfL(\sfL(1/\rho))}{[1-(Q_{2}'-Q_{1})^{1/2}\Lip(J_{\sigma}(Q_{1},\anon))]\sfL(1/\rho)},\label{eq:swapthings}
    \end{equation}
    where $\widetilde{\m{V}}$ denotes the evolution operator for \cref{eq:SPDE} with noise $\dn \widetilde{W}$ in place of $\dn W$.
  \item \emph{Applying the key approximation in the large. }
    Now that we have approximated $\clG_\tau v_t$ by a solution to the original SPDE with $\rho\setto\widetilde\rho$, we seek to leverage the fact that this ``coarsened'' SPDE will again satisfy the key approximation \cref{eq:QSPDEapprox-3} due to the choice of parameters \cref{eq:Q2choice-1}.
    We continue
    to fix $t_{0}\in[T_{0},T_{1}]$. We wish to apply \cref{eq:QSPDEapprox-3}
    again, this time with time shifted by $t_{0}$ and the choices
    \begin{equation*}
      \overline{q}\setto\frac{Q_{2}'-Q_{1}}{1-Q_{1}},\qquad\rho\setto\widetilde{\rho},\qquad\dif W_{t}\setto\dif\widetilde{W}_{t},\qquad\sigma\setto\widetilde{\sigma}.
    \end{equation*}
    We call this approximation ``in the large'' because it is applied to the coarse-grained field.
    We need to check the hypotheses. By \cref{prop:QSPDEbasecase}, we have
    \[
      Q_{\SPDE}(\widetilde{\sigma})\ge\Lip(\widetilde{\sigma})^{-2}\overset{\cref{eq:sigmatildedef-1}}{=}\frac{\Lip(J_{\sigma}(Q_{1},\anon))^{-2}}{1-Q_{1}}\overset{\cref{eq:Q2choice-1}}{>}\frac{Q_{2}'-Q_{1}}{1-Q_{1}}.
    \]
    We also note that by \cref{prop:Jsigmaub}, we have
    \begin{equation*}
      \widetilde{\sigma}\in\quadset(\widetilde{M},\widetilde{\beta}),\quad\text{with }\enspace\widetilde{M}\coloneqq\frac{1-Q_{1}}{(1-\beta^{2}Q_{1})^{2}}M\enspace\text{ and }\enspace\widetilde{\beta}%
      \coloneqq\left(\frac{1-Q_{1}}{\beta^{-2}-Q_{1}}\right)^{1/2}.
    \end{equation*}
    It follows that $\beta_*(\widetilde\sigma) \leq \widetilde \beta$.
    We choose
    \begin{equation}
        \widebartilde{\beta}\in\left(\widetilde\beta\,%
      ,\left(\frac{1-Q_{1}}{Q_{2}'-Q_{1}}\right)^{1/2}\right)\label{eq:betatildeoverlinechoice}
  \end{equation}
  and take $\overline{\beta}\setto\widebartilde{\beta}$ in \cref{def:QSPDE}.
  Note that the interval in \cref{eq:betatildeoverlinechoice} is nonempty because \cref{eq:Q2choice-1} yields $Q_2'<\beta^{-2}$.
    Then \cref{eq:QSPDEapprox-3} gives us a function $\widetilde{\omega}=\omega_{\widetilde{\sigma},Q_{2}',\widebartilde{\beta},\ell}$
    such that for all $p\in[0,(Q_{2}'-Q_{1})/(1-Q_{1})]$, whenever
    \[
      t-t_{0}\in\Big[\widetilde{\omega}\bigl(\widetilde{\rho}\bigr)\sfT_{\widetilde{\rho}}(p),\sfT_{\widetilde{\rho}}\bigl(\widebartilde{\beta}^{-2}\bigr)\Big]
    \]
    and $x,\widetilde{z}\in\RR^{2}$, we have %
    \begin{align}
       & \EE \left\lvert  \bigl(\clS_{\sfT_{\widetilde{\rho}}(p)^{1/2},\widetilde{z}}\bigl[\widetilde{\sigma}^{2}\circ\widetilde{\clV}_{t_{0},t}^{\widetilde{\sigma},\widetilde{\rho}}\clG_{\tau}v_{t_{0}}\bigr](x)\bigr)^{1/2}-J_{\widetilde{\sigma}}\bigl(p,\clG_{\sfT_{\widetilde{\rho}}(p)}\widetilde{\clV}_{t_{0},t}^{\widetilde{\sigma},\widetilde{\rho}}\clG_{\tau}v_{t_{0}}(x)\bigr)\right\rvert  _{\Frob}^{2} \nonumber \\&\qquad\le\frac{C\sfL(\sfL(1/\widetilde{\rho}))}{\sfL(1/\widetilde{\rho})}\langle\vvvert \clG_\tau v_{t_0}\vvvert_{\ell}\rangle^{2}  \overset{\cref{eq:applymomentbd-again}}{\le}\frac{C\sfL(\sfL(1/\widetilde{\rho}))}{\sfL(1/\widetilde{\rho})}\langle\vvvert v_{0}\vvvert_{\ell}\rangle^{2}
      \overset{\cref{eq:rhotildeandtaudef-1}}{\le}\frac{C\sfL(\sfL(1/\rho))}{\sfL(1/\rho)}\langle\vvvert v_{0}\vvvert_{\ell}\rangle^{2}\label{eq:applyindhypinthelarge}
    \end{align}
    for a new choice of constant $C=C(\sigma,Q_{1},Q_{2}',M,\beta,\widebartilde{\beta},\ell)$.
    \item \emph{Change of variables.} Now we wish
    to change variables in \cref{eq:applyindhypinthelarge} by setting
    \[
      p=\sfS_{\widetilde{\rho}}(\sfT_{\rho}(q)-\tau),
    \]
    in order to return to the original scale based on $\rho$ instead of $\widetilde\rho$.
    It is not difficult to check that
    \begin{equation}
      \left\lvert  p-\frac{q-Q_{1}}{1-Q_{1}}\right\rvert  \le\frac{C}{\sfL(1/\rho)}.\label{eq:pclosetoqminusQ1over1minusQ}
    \end{equation}
    for a constant $C$ depending on $Q_{1}$. 
    For all $a\in\RR^m$, we compute
    \begin{equation*}
      \begin{aligned}J_{\widetilde{\sigma}}\left(\frac{q-Q_{1}}{1-Q_{1}},a\right)\overset{\cref{eq:sigmatildedef-1}}{=}J_{(1-Q_{1})^{1/2}J_{\sigma}(Q_{1},\anon)}\left(\frac{q-Q_{1}}{1-Q_{1}},a\right)  \overset{\cref{eq:Jrescale}} & {=}(1-Q_{1})^{1/2}J_{J_{\sigma}(Q_{1},\anon)}(q-Q_{1},a) \\
               \overset{\cref{eq:Jsubst-nohyp}}                                                                                                                                                                                 & {=}(1-Q_{1})^{1/2}J_{\sigma}(q,a).
      \end{aligned}
    \end{equation*}
    Using this and \cref{eq:sigmatildedef-1} 
    in \cref{eq:applyindhypinthelarge},
    and also using
    \cref{prop:Jsigmatimereg-1,prop:momentbd,eq:pclosetoqminusQ1over1minusQ} to make the substitution $p \setto (q - Q_1)/(1 - Q_1)$ in the first argument of $J_\sigma$,
    we obtain
    \begin{equation}
      \begin{aligned}\EE & \left\lvert  \bigl(\clS_{[\sfT_{\rho}(q)-\tau]^{1/2},\widetilde{z}}\bigl[J_{\sigma}^{2}(Q_{1},\anon)\circ\widetilde{\clV}_{t_{0},t}^{\widetilde{\sigma},\widetilde{\rho}}\clG_{\tau}v_{t_{0}}\bigr](x)\bigr)^{1/2}-J_{\sigma}\bigl(q,\clG_{\sfT_{\rho}(q)-\tau}\widetilde{\clV}_{t_{0},t}^{\widetilde{\sigma},\widetilde{\rho}}\clG_{\tau}v_{t_{0}}(x)\bigr)\right\rvert  _{\Frob}^{2} \\
                   & \hspace{9cm}\le\frac{C\sfL(\sfL(1/\rho))}{\sfL(1/\rho)}\langle\vvvert v_{0}\vvvert_{2}\rangle^{2}
      \end{aligned}
      \label{eq:afterchgvarptoq}
    \end{equation}
    for all $q\in[Q_{1},Q_{2}']$, still whenever $t_{0}\in[T_{0},T_{1}]$
    and
    \begin{align*}
      t-t_{0}\in\left[\widetilde{\omega}\bigl(\widetilde{\rho}\bigr)\bigl[\sfT_{\rho}(q)-\tau\bigr],\sfT_{\widetilde{\rho}}\bigl(\widebartilde{\beta}^{-2}\bigr)\right] & \supseteq\left[\widetilde{\omega}\bigl(\widetilde{\rho}\bigr)\sfT_{\rho}(q),\sfT_{\widetilde{\rho}}\bigl(\widebartilde{\beta}^{-2}\bigr)\right].
    \end{align*}
  \item \emph{Proof of \cref{eq:thegoal}. }
    \label{item:thegoal}
    Take
    $q\in[Q_{1},Q_{2}']$, $t_{0}\in[T_{0},T_{1}]$, and (recalling \cref{eq:t1def})
    \begin{equation}
      t-t_{0}\in\left[\widetilde{\omega}(\widetilde{\rho})\sfT_{\rho}(q),\sfT_{\widetilde{\rho}}\left(\widebartilde{\beta}^{-2}\wedge\frac{Q_{2}'-Q_{1}}{1-Q_{1}}\right)\right]\overset{\cref{eq:betatildeoverlinechoice}}{=}\left[\widetilde{\omega}(\widetilde{\rho})\sfT_{\rho}(q),\sfT_{\widetilde{\rho}}\left(\frac{Q_{2}'-Q_{1}}{1-Q_{1}}\right)\right].\label{eq:tcond-2}
    \end{equation}
    Using \cref{eq:swapthings} twice on the left side of \cref{eq:afterchgvarptoq}, we see that
    \begin{equation}
      \begin{aligned}\EE & \left\lvert  \bigl(\clS_{[\sfT_{\rho}(q)-\tau]^{1/2},\widetilde{z}}[J_{\sigma}^{2}(Q_{1},\anon)\circ\clG_{\tau}v_{t}](x)\bigr)^{1/2}-J_{\sigma}\bigl(q,\clG_{\sfT_{\rho}(q)}v_{t}(x)\bigr)\right\rvert  _{\Frob}^{2}                                                               \\
                   & \le C\langle\vvvert v_{0}\vvvert_{2}\rangle^{2}\frac{\sfL(\sfL(1/\rho))}{\sfL(1/\rho)}\bigl[1+\Lip\bigl(J_{\sigma}(q,\anon)\bigr)\bigr]\le C\langle\vvvert v_{0}\vvvert_{\ell}\rangle^{2}\frac{\sfL(\sfL(1/\rho))}{\sfL(1/\rho)}
      \end{aligned}
      \label{eq:afterchgvarptoq-1}
    \end{equation}
    with a constant $C=C(\sigma,Q_{1},Q_{2}',M,\beta,\widebartilde{\beta},\ell)$.
    Finally, combining this with \cref{eq:smallinductivehyp}, we see that
    \begin{equation}
      \EE \left\lvert  \bigl(\clS_{[\sfT_{\rho}(q)-\tau]^{1/2},\widetilde{z}}\bigl[\clS_{\tau^{1/2}}[\sigma^{2}\circ v_{t}]\bigr](x)\bigr)^{1/2}-J_{\sigma}\bigl(q,\clG_{\sfT_{\rho}(q)}v_{t}(x)\bigr)\right\rvert  _{\Frob}^{2}\le C\langle\vvvert v_{0}\vvvert_{\ell}\rangle^{2}\frac{\sfL(\sfL(1/\rho))}{\sfL(1/\rho)}.\label{eq:afterchgvarptoq-1-1}
    \end{equation}
    Now let $z\in\RR^{2}$. By \cref{lem:composeaverages} (and also
    using \cref{eq:applymomentbd-again}), there is a choice of $\widetilde{z}\in\RR^{2}$
    such that
    \begin{align}
      \EE \left\lvert  \clS_{[\sfT_{\rho}(q)-\tau]^{1/2},\widetilde{z}}\bigl[\clS_{\tau^{1/2}}[\sigma^{2}\circ v_{t}]\bigr](x)-\clS_{\sfT_{\rho}(q)^{1/2},z}[\sigma^{2}\circ v_{t}]\right\rvert  _{\Frob}^{2} & \le\frac{C\langle\vvvert v_{0}\vvvert_{2}\rangle^{2}\tau}{\sfT_{\rho}(q)-\tau}. %
      \label{eq:needdeltanow}
    \end{align}
    Recall the $\delta>0$ fixed in \cref{step:delta}. %
    The %
    right side of
    \cref{eq:needdeltanow} is (quite crudely) bounded above by $C\frac{\langle\vvvert v_{0}\vvvert_{2}\rangle^{2}}{\sfL(1/\rho)}$
    uniformly over $q\in[Q_{1}+\delta,Q_{2}']$, where the constant $C$
    can now depend also on $\delta$. Combining this with \cref{eq:afterchgvarptoq-1-1},
    we see that
    \begin{equation}
      \label{eq:almostthegoal}
      \EE \left\lvert  \clS_{\sfT_{\rho}(q)^{1/2},z}[\sigma^{2}\circ v_{t}]-J_{\sigma}\bigl(q,\clG_{\sfT_{\rho}(q)}v_{t}(x)\bigr)\right\rvert  _{\Frob}^{2}\le C\langle\vvvert v_{0}\vvvert_{\ell}\rangle^{2}\frac{\sfL(\sfL(1/\rho))}{\sfL(1/\rho)}
    \end{equation}
    still uniformly over $q\in[Q_{1}+\delta,Q_{2}']$, and for all $t_{0}\in[T_{0},T_{1}]$ and $t$ satisfying \cref{eq:tcond-2}.
    This is \cref{eq:thegoal} with a constant $C(\sigma,Q_1,Q_2',M,\beta,\widebartilde{\beta},\ell,\delta)$.
    We can easily define $Q_2'$, $M$, $\widebartilde{\beta}$, and $\delta$ as functions of $(\sigma,Q_1,Q_2)$, so it suffices to choose $Q_1$ as a function of $(\sigma,Q_2)$.
    To this end, define $h(q;\sigma,Q_2) \coloneqq q + \Lip[J_\sigma(q, \anon)]^{-2} - Q_2$ on the interval $I \coloneqq [0,Q_2 \wedge \Qbar_\SPDE(\sigma) \wedge 1)$.
    By the choice of $Q_2$, $\sup h > 0$.
    Because $J_\sigma$ is continuous, $q \mapsto \Lip[J_\sigma(q, \anon)]$ is lower semicontinuous.
    It follows that $h$ is upper semicontinuous, and hence $\{q \suchthat h(q) \geq \sup h/2\}$ is closed in $I$.
    If we choose $Q_1$ to be the minimum of this set, it is a function of $\sigma$ and $Q_2$ and satisfies the required properties.

    So in fact \cref{eq:almostthegoal} is \cref{eq:thegoal}.
    It only remains to check that it holds for an appropriate range of $t$.
    
    \item \emph{The range of $t$.} We have proved that \cref{eq:thegoal} holds
    whenever $t_{0}\in[T_{0},T_{1}]$, $q\in[Q_{1}+\delta,Q_{2}']$, and
    $t$ satisfies \cref{eq:tcond-2}. Let
    \begin{equation}
      \overline{\omega}(\rho)=\omega(\rho)+\widetilde{\omega}(\widetilde{\rho}).\label{eq:omegabardef}
    \end{equation}
    It is clear that this choice of $\overline{\omega}$ satisfies \zcref[range]{eq:omegaissmallenough,eq:omegaisbigenough}.
    Suppose that $q\in[Q_{1}+\delta,Q_{2}]$ and
    \begin{equation}
      t\in[\overline{\omega}(\rho)\sfT_{\rho}(q),\sfT_{\rho}(\overline{\beta}^{-2})].\label{eq:trange}
    \end{equation}
    The proof will be complete if we can find a $t_{0}\in[T_{0},T_{1}]$
    such that \cref{eq:tcond-2} holds. We note that \cref{eq:trange}, \cref{eq:omegabardef},
    and the assumption on $q$ imply that
    \begin{equation}
      t\ge\widetilde{\omega}(\widetilde{\rho})\sfT_{\rho}(q)+\omega(\rho)\sfT_{\rho}(q)\ge\widetilde{\omega}(\widetilde{\rho})\sfT_{\rho}(q)+\omega(\rho)\sfT_{\rho}(Q_{1}+\delta).\label{eq:tlb}
    \end{equation}
    Therefore, we can define
    \begin{equation*}
      t_{0}=t-\widetilde{\omega}(\widetilde{\rho})\sfT_{\rho}(q),%
    \end{equation*}
    and note from \cref{eq:tlb} that, at least as long as $\rho$ is sufficiently
    small, we have
    \[
      t_{0}\ge\omega(\rho)\tau\overset{\cref{eq:altogetherT0def}}{=}T_{0}.
    \]
    Since we also know that $t_{0}\le t\le\sfT_{\rho}(\overline{\beta}^{-2})$,
    this implies that $t_{0}\in[T_{0},T_{1}]$. Now since $q\le Q_{2}$,
    we have
    \begin{equation}
      t-t_{0}=\widetilde{\omega}(\widetilde{\rho})\sfT_{\rho}(q)<\sfT_{\widetilde{\rho}}\left(\frac{Q_{2}'-Q_{1}}{1-Q_{1}}\right)\label{eq:tminust0nottoobig}
    \end{equation}
    when $\rho$ is sufficiently small, where the required smallness is
    uniform over all $q\in[Q_{1}+\delta,Q_{2}]$ since $Q_{2}<Q_{2}'$.
    Together, \cref{eq:tlb} and \cref{eq:tminust0nottoobig} imply that \cref{eq:tcond-2}
    holds, and the proof is complete.\qedhere
  \end{thmstepnv}
\end{proof}
We next slightly adjust the parameters appearing in \cref{def:QSPDE}. This is analogous
to what was done in \cref{step:adjustJsigmarho} in the proof of \cref{thm:mainapproxthm}.
The proof is the same except that since we now work with $J_{\sigma}$
rather than $J_{\sigma,\rho}$, we use \cref{prop:Jsigmatimereg-1}
and the Lipschitz bound on $J_{\sigma}$ instead of \cref{cor:approxJregularity}.
\begin{lem}
  \label{lem:applyIH}
  Suppose $\sigma \in \Lip(\R^m; \m{H}_+^m)$, $\overline{q}\in\bigl(0,\Qbar_{\SPDE}(\sigma)\bigr)$, $\overline\beta \in \bigl(\beta_*(\sigma), \bar{q}^{-1/2}\bigr)$, and $\ell\in(2,4]$.
  Let $\omega=\omega_{\sigma,\overline{q},\overline{\beta},\ell}$ be as in \cref{def:QSPDE}.
  There is a constant $C=C(\sigma,\overline{q},\overline{\beta},\ell)\in(0,\infty)$
  such that, whenever $\rho\in(0,C^{-1})$, if $(v_{t})_{t\ge0}$ solves
  \cref{eq:SPDE} with initial condition $v_{0}\in\scrX_{0}^{\ell}$,
  then for any $q\in[0,\overline{q}]$, any
  \begin{equation}
    t\in[\omega(\rho)\sfT_{\rho}(q),\sfT_{\rho}(\overline{\beta}^{-2})],\label{eq:tconddefn-1-1}
  \end{equation}
  and any $x,z\in\RR^{2}$, we have
  \begin{equation*}
    \begin{aligned}\EE & \left\lvert  \bigl(\clS_{\sfT_{\rho}(q)^{1/2},z}[\sigma^{2}\circ v_{t}](x)\bigr)^{1/2}-J_{\sigma}\bigl(\sfS_{\rho}\bigl(\sfL(1/\rho)\sfT_{\rho}(q)\bigr), \, \clG_{\sfL(1/\rho)\sfT_{\rho}(q)}v_{t}(x)\bigr)\right\rvert  _{\Frob}^{2} \\
                   & \hspace{8cm}\le\frac{C\langle\vvvert v_{0}\vvvert_{\ell}\rangle^{2}\sfL(\sfL(1/\rho))}{\sfL(1/\rho)}.
    \end{aligned}
  \end{equation*}
\end{lem}

\begin{proof}
  By \cref{eq:QSPDEapprox-3} and the triangle inequality, it suffices
  to prove that
  \begin{equation}
    \EE \left\lvert  J_{\sigma}\bigl(\sfS_{\rho}\bigl(\sfL(1/\rho)\sfT_{\rho}(q)\bigr),\clG_{\sfL(1/\rho)\sfT_{\rho}(q)}v_{t}(x)\bigr)-J_{\sigma}\bigl(q,\clG_{\sfT_{\rho}(q)}v_{t}(x)\bigr)\right\rvert  _{\Frob}^{2}\le\frac{C\langle\vvvert v_{0}\vvvert_{\ell}\rangle^{2}\sfL(\sfL(1/\rho))}{\sfL(1/\rho)}.\label{eq:Jsigmasclose}
  \end{equation}
  We prove this in two steps. First we note that, by \cref{eq:Jtimecontinuity}
  of \cref{prop:Jsigmatimereg-1}, we have
  \begin{multline*}
    \left\lvert  J_{\sigma}\bigl(\sfS_{\rho}\bigl(\sfL(1/\rho)\sfT_{\rho}(q)\bigr),\clG_{\sfT_{\rho}(q)}v_{t}(x)\bigr)-J_{\sigma}\bigl(q,\clG_{\sfT_{\rho}(q)}v_{t}(x)\bigr)\right\rvert  _{\Frob}\\
    \le C\langle\clG_{\sfT_{\rho}(q)}v_{t}(x)\rangle\bigl\lvert \sfS_{\rho}\bigl(\sfL(1/\rho)\sfT_{\rho}(q)\bigr)-q\bigr\rvert ^{1/2}\le C\langle\clG_{\sfT_{\rho}(q)}v_{t}(x)\rangle\left(\frac{\sfL(\sfL(1/\rho))}{\sfL(1/\rho)}\right)^{1/2}.
  \end{multline*}
  The second inequality follows from an elementary estimate on $|\sfS_{\rho}(\sfL(1/\rho)\sfT_{\rho}(q))-q|$ using \cref{eq:Srhodef} and \cref{eq:Tdef}. 
  Taking second moments and using \cref{prop:momentbd}, 
  we see that
  \begin{equation}
    \EE \left\lvert  J_{\sigma}\bigl(\sfS_{\rho}\bigl(\sfL(1/\rho)\sfT_{\rho}(q)\bigr),\clG_{\sfT_{\rho}(q)}v_{t}(x)\bigr)-J_{\sigma}\bigl(q,\clG_{\sfT_{\rho}(q)}v_{t}(x)\bigr)\right\rvert  _{\Frob}^{2}\le\frac{C\langle\vvvert v_{0}\vvvert_{\ell}\rangle^{2}\sfL(\sfL(1/\rho))}{\sfL(1/\rho)}.\label{eq:firstpiece}
  \end{equation}
  Second, we can estimate
  \begin{align}
     & \left[\EE \left\lvert  J_{\sigma}\bigl(\sfS_{\rho}\bigl(\sfL(1/\rho)\sfT_{\rho}(q)\bigr),\clG_{\sfL(1/\rho)\sfT_{\rho}(q)}v_{t}(x)\bigr)-J_{\sigma}\bigl(\sfS_{\rho}\bigl(\sfL(1/\rho)\sfT_{\rho}(q)\bigr),\clG_{\sfT_{\rho}(q)}v_{t}(x)\bigr)\right\rvert  _{\Frob}^{2}\right]^{1/2}\nonumber \\
     & \qquad\overset{\cref{eq:Juniflipschitz}}{\le}C\left[\EE \left\lvert  \clG_{\sfL(1/\rho)\sfT_{\rho}(q)}v_{t}(x)-\clG_{\sfT_{\rho}(q)}v_{t}(x)\right\rvert  ^{2}\right]^{1/2}\nonumber                                                                   \\
     & \qquad\overset{\substack{\cref{eq:continuitybd},                                                                                                                                                                                         \\
        \cref{eq:momentbd}
      }
    }{\le}C\langle\vvvert v_{0}\vvvert_{\ell}\rangle\left[\left(\frac{|\sfL(1/\rho)-1|\sfT_{\rho}(q)}{t}\right)^{1/2}+\sfL_{\rho}(1+\sfL(1/\rho))^{1/2}\right]\nonumber                                                                     \\
     & \qquad\overset{\cref{eq:omegaisbigenough}}{\le}C\langle\vvvert v_{0}\vvvert_{\ell}\rangle\left(\frac{\sfL(\sfL(1/\rho))}{\sfL(1/\rho)}\right)^{1/2}\label{eq:secondpiece}
  \end{align}
  for $\rho$ sufficiently small. Combining \cref{eq:firstpiece} and
  \cref{eq:secondpiece} using the triangle inequality, we obtain \cref{eq:Jsigmasclose}.
\end{proof}
To close the section, we observe that \cref{eq:SPDE} is approximately closed under renormalization, up to a change of noise and nonlinearity.
This is the rigorous statement behind the informal \cref{thm:renorm}.
\begin{cor}
  \label{cor:approx-SPDE}
  Fix $\sigma\in\Lip(\RR^m;\clH_+^m)$ and $Q \in(0,\Qbar_{\SPDE}(\sigma) \wedge 1).$
  Let\vspace{2pt} $\widetilde{\sigma} \coloneqq (1 - Q)^{1/2} J_\sigma(Q,\anon)$ and $\widetilde\rho \coloneqq \rho^{1 - Q}$.
  Then there exists a white noise $\dn \widetilde{W}$ with mild solution operator $\widetilde{\m{T}}$ such that for each $\beta < \beta_*(\sigma)$ and $\ell \in (2, 4]$, there exists a constant $C(\sigma,Q,\ell,\beta) > 0$ such that for all $v_0 \in \s{X}_0^\ell$,
  \begin{equation}
    \label{eq:approx-SPDE}
    \m{M}_{0,\sfT_\rho(\beta^{-2})}\big(\widetilde{\m{T}}_0^{\widetilde{\rho},\widetilde{\sigma}} \m{G}_{\widetilde{\rho}} v - \m{G}_{\widetilde{\rho}} v\big)^2 \leq C \langle \vvvert v_0 \vvvert_\ell\rangle^2 \frac{\sfL\bigl(\sfL(1/\rho)\bigr)}{\sfL(1/\rho)}.
  \end{equation}
\end{cor}
\begin{proof}
  This is \cref{eq:approx} with a few adjustments.
  First, we have taken $\widetilde{\rho} = \rho^{1 - Q}$ rather than $\rho/\sfT_\rho(1 - Q)$ and we have smoothed to scale $\widetilde{\rho}$ rather than $\widetilde{\rho} - \rho$.
  Using $Q \in (0, 1)$, \cref{eq:Jtimecontinuity}, and \cref{lem:VTcontinuity}, it is straightforward to check that the error from these adjustments is no larger than that in \cref{eq:approx-SPDE}.
  
  Second, we can only begin the mild expression in \cref{eq:approx} at $T_0 = \omega(\rho) \widetilde{\rho}$ rather than $0$.
  Suppose $t \geq T_0$.
  We can adjust $\m{G}_{t - t_0} \m{G}_{\widetilde{\rho}} v_{t_0}$ to $\m{G}_t \m{G}_{\widetilde\rho} v_0$ through \cref{prop:continuitybd}.
  For the integral in \cref{eq:approx}, we independently extend $\dn \widetilde{W}$ to the time interval $[0, T_0]$.
  Using \cref{prop:Jsigmaub,prop:momentbd}, we can bound the integral over $[0,T_0]$ by
  \begin{align*}
    \E \abs{\gamma_{\widetilde\rho} \int_0^{T_0} \m{G}_{t + \widetilde\rho - r}[\widetilde\sigma(\m{G}_{\widetilde\rho}v_r) \ds \widetilde{W}_r](x)}^2 \leq C \langle \vvvert v_0 \vvvert_2\rangle^2 \gamma_{\widetilde\rho}^2 \int_0^{T_0} \frac{\dn r}{T_0 + \widetilde\rho - r}
    &\leq C \langle \vvvert v_0 \vvvert_2\rangle^2 \frac{1}{\sfL(1/\rho)} \log \omega(\rho)\\
    \overset{\cref{eq:omegaissmallenough}}&\leq C \langle \vvvert v_0 \vvvert_2\rangle^2 \frac{\sfL\bigl(\sfL(1/\rho)\bigr)}{\sfL(1/\rho)}.
  \end{align*}
  Similar calculations hold when $t \in [0, T_0)$; we omit the details.
\end{proof}

\section{Proof of the main result\label{sec:proofofmainresults}}

The proof of \cref{thm:QSPDEgreater} was the most technical step of
the paper. With this theorem in hand, we can prove our main results.
\subsection{The approximate stochastic differential equation}
The first result of this section shows that the martingale $V^{\rho,T_1}(x)$, defined in \zcref{eq:Vdef}, approximately satisfies a stochastic differential equation. This is similar to \cref{eq:concludestep2} in the proof of \cref{prop:Japprox}, but with two key upgrades. First, the approximate root decoupling function $J_{\sigma,\rho}$ is replaced by the true root decoupling function $J_\sigma$; in fact this was done in the proof of \cref{prop:QSPDEbasecase} (see \cref{eq:approxJsigmarhobyJsigma}) using the conclusion of \cref{prop:Japprox} itself. Second, the statement \cref{eq:concludestep2} is extended to larger time scales. This  upgrade relies on the multiscale analysis carried out in \cref{sec:extendtheestimates}.

We recall the definition \cref{eq:minimal-beta} of $\beta_*(\sigma)$. %
\begin{prop}
  \label{prop:laydowntheSDE}Let $\sigma \in \Lip(\R^m; \m H_+^m)$, $Q_{0}\in(0,\Qbar_{\SPDE}(\sigma))$, $\overline{\beta}>\beta_*(\sigma)$, and  $\ell\in(2,4]$.
  There is a constant $C=C(\sigma,Q_0,\overline{\beta},\ell)<\infty$
  such that the following holds. Let $\omega=\omega_{\sigma,Q_{0},\overline{\beta},\ell}$
  be as in \cref{def:QSPDE}. Let $(v_{t})_{t}$ solve \cref{eq:SPDE}
  with initial condition $v_{0}\in\scrX_{0}^{\ell}$. Let $T_{0},T_{1}$ satisfy
  \begin{equation}
    \bigl[\sfL(1/\rho)\omega(\rho)^{-1}+1\bigr]^{-1}T_{1}\le T_{0}<T_{1}\le\sfT_{\rho}\bigl(\overline{\beta}^{-2}\bigr)\label{eq:T0T1nottoosmallorbig}
  \end{equation}
  and
  \begin{equation}
    \sfS_{\rho}(T_{1}-T_{0})\le Q_{0}.\label{eq:T1T0nottoofar}
  \end{equation}
  As in \cref{eq:Bhat}, we define, for $t\ge T_{0}$,
  \begin{equation}
    \hat{B}(t)=\gamma_{\rho}\int_{T_{0}}^{t}[\clG_{T_{1}+\rho-r}\,\dif\widetilde{W}_{r}](x),\label{eq:Bhatdef}
  \end{equation}
  where $\dif\widetilde{W}$ is the space-time white noise constructed
  in \cref{prop:coupling}, with the choices $w_{r}\setto\sigma\circ v_{r}$
  and $\zeta_{r}=[\sfL(1/\rho)^{-1}(T_{1}-r)]^{1/2}$. Then we
  have, for all $t\in[T_{0},T_{1}]$, that
  \begin{equation}
    \EE \left\lvert  V_{t}^{\rho,T_{1}}(x)-V_{T_{0}}^{\rho,T_{1}}(x)-\int_{T_{0}}^{t}J_{\sigma}\bigl(\sfS_{\rho}(T_{1}-r),V_{r}^{\rho,T_{1}}(x)\bigr)\,\dif\hat{B}(r)\right\rvert  ^{2}\le C\langle\vvvert v_{0}\vvvert_{\ell}\rangle^{2}\frac{\sfL(\sfL(1/\rho))}{\sfL(1/\rho)}.\label{eq:oneguyclosetoSDE}
  \end{equation}
\end{prop}

\begin{proof}
  The proof follows many of the steps of the proof of \cref{prop:Japprox}.
  However, we now have at our disposal \cref{lem:applyIH}, which we use
  in place of \cref{thm:mainapproxthm}. Since \cref{lem:applyIH} concerns
  $J_{\sigma}$ rather than the approximation $J_{\sigma,\rho}$, much
  of the work has already been done. That is, since we can approximate
  the SPDE solution directly in terms of the root decoupling function $J_{\sigma}$,
  we effectively need to show that the solution approximately solves
  an SDE rather than an FBSDE.
  Throughout, we allow the constant $C$ to change from line to line.

  We recall from \cref{eq:Vtintegral} (as we did in \cref{eq:VTintegralrestate})
  that, for all $t\in[T_{0},T_{1}]$,
  \begin{equation}
    V_{t}^{\rho,T_{1}}(x)=V_{T_{0}}^{\rho,T_{1}}(x)+\gamma_{\rho}\int_{T_{0}}^{t}\clG_{T_{1}+\rho-r}[\sigma(v_{r})\,\dif W_{r}](x).\label{eq:Vmild}
  \end{equation}
  Let $\zeta_{r}=[\sfL(1/\rho)^{-1}(T_{1}-r)]^{1/2}$. We apply
  \cref{prop:coupling} with $w_{r}\setto\sigma\circ v_{r}$,
  with this choice of $\zeta_{r}$, and with $\eta_{r}\setto T_{1}+\rho-r$
  (and also use \cref{prop:momentbd} to obtain a moment bound) to find
  a space-time white noise $\dif\widetilde{W}$ such that, for
  any $x\in\RR^{2}$,
  \begin{align}
    \EE & \left\lvert  \gamma_{\rho}\int_{T_{0}}^{t}\clG_{T_{1}+\rho-r}\left[\sigma(v_{r})\,\dif W_{r}-(\clS_{\zeta_{r}}[\sigma^{2}\circ v_{r}])^{1/2}\,\dif\widetilde{W}_{r}\right](x)\right\rvert  ^{2}\le\frac{C\langle\vvvert v_{0}\vvvert_{2}\rangle^{2}}{\sfL(1/\rho)^{2}}\int_{T_{0}}^{t}\frac{T_{1}-r}{(T_{1}+\rho-r)^{2}}\,\dif r\nonumber \\
        & \le\frac{C\langle\vvvert v_{0}\vvvert_{2}\rangle^{2}\sfS_{\rho}(T_{1}-T_{0})}{\sfL(1/\rho)}\overset{\cref{eq:T1T0nottoofar}}{\le}\frac{C\langle\vvvert v_{0}\vvvert_{2}\rangle^{2}}{\sfL(1/\rho)}.\label{eq:usecoupling}
  \end{align}

  Now we apply \cref{lem:applyIH}, with $t\setto r$ and $q\setto\sfS_{\rho}(\zeta_{r}^{2})=\sfS_{\rho}(\sfL(1/\rho)^{-1}(T_{1}-r))$.
  We note that as long as $r\in[T_{0},T_{1}]$, then \cref{eq:tconddefn-1-1}
  is satisfied by \cref{eq:T0T1nottoosmallorbig}. So for any $r\in[T_{0},T_{1}]$
  and any $y\in\RR^{2}$ we have
  \[
    \EE \left\lvert  \bigl(\clS_{\zeta_{r}}[\sigma^{2}\circ v_{r}](y)\bigr)^{1/2}-J_{\sigma}\bigl(\sfS_{\rho}(T_{1}-r),V_{r}^{\rho,T_{1}}(y)\bigr)\right\rvert  _{\Frob}^{2}\le\frac{C\langle\vvvert v_{0}\vvvert_{\ell}\rangle^{2}\sfL(\sfL(1/\rho))}{\sfL(1/\rho)}.
  \]
  We note that there is a constant $C=C(\sigma,Q_{0})<\infty$ such
  that $\Lip\bigl(J_{\sigma}(q,\anon)\bigr)\le C^{1/2}$ for all $q\in[0,Q_{0}]$,
  and for all $r\in[T_{0},T_{1}]$ we have $0\le\sfS_{\rho}(T_{1}-r)\le\sfS_{\rho}(T_{1}-T_{0})\le Q_{0}$
  by \cref{eq:T1T0nottoofar}. From these observations we can conclude
  that, for all $x,y\in\RR^{2}$, we have
  \begin{align*}
    \EE & \left\lvert  J_{\sigma}\bigl(\sfS_{\rho}(T_{1}-r),V_{r}^{\rho,T_{1}}(y)\bigr)-J_{\sigma}\bigl(\sfS_{\rho}(T_{1}-r),V_{r}^{\rho,T_{1}}(x)\bigr)\right\rvert  _{\Frob}^{2}\le C\EE \left\lvert  V_{r}^{\rho,T_{1}}(y)-V_{r}^{\rho,T_{1}}(x)\right\rvert  ^{2} \\
        & \hspace{4cm}\le\frac{C\langle\vvvert v_{0}\vvvert_{2}\rangle^{2}}{\sfL(1/\rho)}\left[\sfL\left(\frac{|y-x|^{2}}{T_{1}+\rho-r}\right)+\sfL(1/\rho)\frac{|y-x|^{2}}{2r}+\sfL\bigl(1+\sfL(1/\rho)\bigr)\right]
  \end{align*}
  with the second inequality by \cref{prop:continuitybd,prop:momentbd}.
  Combining the last two displays, we see that
  \begin{equation}
    \begin{aligned}\EE & \left\lvert  \bigl(\clS_{\zeta_{r}}[\sigma^{2}\circ v_{r}](y)\bigr)^{1/2}-J_{\sigma}\bigl(\sfS_{\rho}(T_{1}-r),V_{r}^{\rho,T_{1}}(x)\bigr)\right\rvert  _{\Frob}^{2}                                            \\
                   &\hspace{2cm} \le\frac{C\langle\vvvert v_{0}\vvvert_{\ell}\rangle^{2}}{\sfL(1/\rho)}\left[\sfL\left(\frac{|x-y|^{2}}{T_{1}+\rho-r}\right)+\sfL(1/\rho)\frac{|y-x|^{2}}{2r}+\sfL\bigl(1+\sfL(1/\rho)\bigr)\right].
    \end{aligned}
    \label{eq:Ss2xJy}
  \end{equation}
  Now, proceeding as in \cref{eq:comparetoclosedthing} and recalling
  the definition \cref{eq:Bhatdef} of $\hat B(t)$, we estimate
  \begin{align}
    \EE & \left\lvert  \gamma_{\rho}\int_{T_{0}}^{t}\clG_{T_{1}+\rho-r}\left[\bigl(\clS_{\zeta_{r},z}[\sigma^{2}\circ v_{r}]\bigr)^{1/2}\,\dif\widetilde{W}_{r}\right](x)-\int_{T_{0}}^{t}J_{\sigma}\bigl(\sfS_{\rho}(T_{1}-r),V_{r}^{\rho,T_{1}}(x)\bigr)\,\dif\hat{B}(t)\right\rvert  ^{2}\nonumber                     \\
        & =\frac{4\pi}{\sfL(1/\rho)}\int_{T_{0}}^{t}\!\!\int G_{T_{1}+\rho-r}(x-y)^{2}\EE \left\lvert  \bigl(\clS_{\zeta_{r},z}[\sigma^{2}\circ v_{r}](y)\bigr)^{1/2}-J_{\sigma}\bigl(\sfS_{\rho}(T_{1}-r),V_{r}^{\rho,T_{1}}(x)\bigr)\right\rvert  _{\Frob}^{2}\,\dif y\,\dif s\nonumber                                      \\
        & \overset{\substack{\cref{eq:Ss2xJy},                                                                                                                                                                                                                                                  \\
        \cref{eq:GT2}
      }
    }{\le}\frac{C\langle\vvvert v_{0}\vvvert_{\ell}\rangle^{2}}{\sfL(1/\rho)^{2}}\int_{T_{0}}^{t}\!\!\int\frac{G_{\frac{T_{1}+\rho-r}{2}}(x-y)}{T_{1}+\rho-r}\left[\sfL\left(\frac{|x-y|^{2}}{T_{1}+\rho-r}\right)+\sfL(1/\rho)\frac{|y-x|^{2}}{2r}+\sfL\bigl(1+\sfL(1/\rho)\bigr)\right]\,\dif y\,\dif s\nonumber \\
        & \le\frac{C\langle\vvvert v_{0}\vvvert_{\ell}\rangle^{2}}{\sfL(1/\rho)}\left[(\sfL\bigl(1+\sfL(1/\rho)\bigr)\sfS_{\rho}(T_{1}-T_{0})+\log \frac{T_1}{T_0}\right]\overset{\cref{eq:T1T0nottoofar}}{\le}C\langle\vvvert v_{0}\vvvert_{\ell}\rangle^{2}\frac{\sfL(\sfL(1/\rho))}{\sfL(1/\rho)}.\label{eq:dotheapproximation}
  \end{align}
  The last inequality holds as long as $\rho$ is sufficiently small.
  Combining \cref{eq:Vmild,eq:usecoupling,eq:dotheapproximation}, %
  we obtain \cref{eq:oneguyclosetoSDE}.
\end{proof}

\subsection{Statement of the main theorem}

The following theorem is the precise general statement that encapsulates \cref{thm:mainthm-singlepoint,thm:mainthm-multipoint}.
On first reading, the reader may prefer to consider only the $N=1$
case (which is \cref{thm:mainthm-singlepoint}). Define
\nomenclature[zzzgreek νrho]{$\nu_\rho$}{large parameter equal to $2\sfL(\sfL(1/\rho))$}
\begin{equation}
  \nu_{\rho}=2\sfL\bigl(\sfL(1/\rho)\bigr).\label{eq:nurhodef}
\end{equation}
\begin{thm}
  \label{thm:thebigtheorem}Fix $\overline{T}>1$, $\beta\in(0,1)$,
  $M\in(0,\infty)$, $\ell\in(2,4]$, and $N\in\NN$. Suppose $\sigma\in\quadset(M,\beta)$
  satisfies $\Qbar_{\FBSDE}(\sigma)>1$. Let $\omega=\omega_{\sigma,1,(1+\beta)/2,\ell}$ be as in \cref{def:QSPDE}. There
  exists $C=C(\sigma,M,\beta,\ell,\overline{T},N)\in(0,\infty)$ such
  that the following holds. Let $(t^{(i)},R^{(i)},x^{(i)})\in[\overline{T}^{-1},\overline{T}]\times[0,\overline{T}]\times\RR^{2}$ for each $i\in\{1,\ldots,N\}$.
  For each $i\in\{1,\ldots,N\}$, define %
  \begin{gather}
    T_{1}^{(i)}\coloneqq t^{(i)}+R^{(i)},\qquad T_{0}^{(i)}\coloneqq \bigl[\sfL(1/\rho)\omega(\rho)^{-1}+1\bigr]^{-1}T_{1}^{(i)},\label{eq:T1idef}\\
    \sfR_{\rho}^{(i)}\coloneqq\sfR_{T_{0}^{(i)},T_{1}^{(i)},\rho},\qquad\sfU_{\rho}^{(i)}\coloneqq\sfU_{T_{0}^{(i)},T_{1}^{(i)},\rho},\text{ and}\label{eq:RiUidefs}\\
    \hat{Q}^{(i)}\coloneqq\sfS_{\rho}(T_{1}^{(i)}-T_{0}^{(i)})=\sfU_{\rho}^{(i)}(T_{1}^{(i)}),\label{eq:Qidef-1}
  \end{gather}
  recalling the definitions \textup{\zcref[range]{eq:sfRdef,eq:sfUdef}}. For each
  $i,j\in\{1,\ldots,N\}$, define
  \begin{align}
    D^{(i,j)}     & \coloneqq |T_{1}^{(i)}-T_{1}^{(j)}|\vee\frac{|x^{(i)}-x^{(j)}|^{2}}{32\nu_{p}},\label{eq:Dijdef}     \\
    T_{*}^{(i,j)} & \coloneqq\begin{cases}
                               \max\left\{ T_{1}^{(i)}\wedge T_{1}^{(j)}-D^{(i,j)},T_{0}^{(i)}\right\} & \text{if }i\ne j; \\
                               T_{0}^{(i)}                                                             & \text{if }i=j.
                             \end{cases}\label{eq:T0ij-1}
  \end{align}
  Let
  \begin{equation}
	  j_{*}(i)\coloneqq\argmax_{j\le i}T_{*}^{(i,j)},\qquad T_{*}^{(i)}\coloneqq T_{*}^{(i,j_{*}(i))}=\max_{j\le i}T_{*}^{(i,j)},\qquad\text{and}\qquad \label{eq:jstaridef}
    q_{*}^{(i)}\coloneqq\sfU_{\rho}^{(i)}\bigl(T_{*}^{(i)}\bigr).%
  \end{equation}
  (If the $\argmax$ contains multiple elements,
  choose one arbitrarily.)
  Let $v_0\in\scrX^\ell_0$ and let $\Upsilon^{(1)},\ldots,\Upsilon^{(N)}$
  solve the system of SDEs
  \begin{align}
	  \dif\Upsilon^{(i)}(q)       & =J_{\sigma}\bigl(\hat{Q}^{(i)}-q,\Upsilon^{(i)}(q)\bigr)\dif B^{(i)}(q)&&\text{for }q \in \bigl(q_{*}^{(i)}, \h{Q}^{(i)}\bigr);\label{eq:dUpsiloni} \\
    \Upsilon^{(i)}(q) & =\begin{cases}
	    \Upsilon^{(j_*(i))}\bigl(\sfU_\rho^{(j_*(i))}\bigl(\sfR_\rho^{(i)}(q)\bigr)\bigr)%
	    & \text{if }j_{*}(i)<i; \\
                                     \clG_{T_{1}^{(i)}}v_{0}(x^{(i)})       & \text{if }j_{*}(i)=i,
			     \end{cases}&&\text{for }q\in [0,q_*^{(i)}], \label{eq:Upsiloniic}
  \end{align}
  where $B^{(1)},\ldots,B^{(N)}$ are independent Brownian motions (also independent of $v_0$).
  If $(v_t)_t$ satisfies \cref{eq:SPDE} with initial condition $v_0$ at time $0$, then
  \begin{equation}
    \clW_{2}\bigl(\bigl(\clG_{R^{(i)}}v_{t^{(i)}}(x^{(i)})\bigr)_{i\in\{1,\ldots,N\}}, \, \bigl(\Upsilon^{(i)}\bigl(\sfU_{\rho}^{(i)}(t^{(i)})\bigr)\bigr)_{i\in\{1,\ldots,N\}}\bigr)\le C\langle\vvvert v_{0}\vvvert_{\ell}\rangle\sqrt{\frac{\sfL(\sfL(1/\rho))}{\sfL(1/\rho)}}.\label{eq:thebigthm-conclusion}
  \end{equation}
\end{thm}
Let us explain the parameters introduced the statement of \cref{thm:thebigtheorem}. In the final statement \cref{eq:thebigthm-conclusion}, we are concerned with $\clG_{R^{(i)}}v_{t^{(i)}}(x^{(i)}) = V^{\rho,T^{(i)}_1}_{t^{(i)}}(x^{(i)})$, so $T^{(i)}_1$ defined in \cref{eq:T1idef} is the final time of the martingale we are interested in. We will think of starting this martingale at time $T^{(i)}_0$: from the perspective of the martingale, this would be the same as starting it at time $0$, but starting it later allows us to satisfy the hypothesis \cref{eq:tconddefn-1-1} of \cref{lem:applyIH} (completely analogously to the considerations in \cref{step:twiddleic} in the proof of \cref{prop:Japprox}).

The key to the phenomenology of the multipoint statistics is that the martingales $(V^{\rho,T^{(i)}_1}_{t}(x^{(i)}))_t$ and $(V^{\rho,T^{(j)}_1}_{t}(x^{(j)}))_t$ will be approximately equal to one another until $t$ reaches a time such that $T^{(i)}\wedge T^{(j)}-t$ is comparable to the parabolic distance between $(t^{(i)},x^{(i)})$ and $(t^{(j)},x^{(j)})$. This is because they are averages of approximately the same field (see \zcref{subsec:regularitygaussianaverages}). After this time, the martingale increments will be approximately independent, since they will be subject to approximately independent pieces of the noise field. We represent the parabolic distance by $D^{(i,j)}$. (The adjustment factor $\nu_\rho$ in the definition is for essentially technical reasons.) This distance is then used in the definition of $T^{(i,j)}_*$, which is the time at which the increments of the two martingales switch from being essentially identical to essentially independent.

The functions $\sfR_{\rho}^{(i)}\coloneqq\sfR_{T_{0}^{(i)},T_{1}^{(i)},\rho}$ and $\sfU_{\rho}^{(i)}\coloneqq\sfU_{T_{0}^{(i)},T_{1}^{(i)},\rho}$ are used to perform changes of variables exactly as in \cref{step:chgvar} of the proof of \cref{prop:Japprox}. This system is set up with a triangular structure, depending on the ordering of the indices, although the order is of course arbitrary. For each $i$, the index $j_*(i)$ is the last preceding index from which the martingale $(V^{\rho,T^{(i)}_1}_{t}(x^{(i)}))_t$ ``separates,'' and $q^{(i)}_*$, $\overline{q}^{(i)}_*$ are the corresponding times after the changes of variables for $i$ and for $j_*(i)$, respectively.
The final system \zcref[range]{eq:dUpsiloni,eq:Upsiloniic} of SDE problems is then based on  this triangular structure: the problem for $i$ starts at the time $\overline{q}_*^{(i)}$, with initial condition taken from the problem for $j_*(i)$ at time $q_*^{(i)}$.

Before we prove \cref{thm:thebigtheorem}, we show how it implies \cref{thm:mainthm-singlepoint,thm:mainthm-multipoint}.
\begin{proof}[Proof of \cref{thm:mainthm-singlepoint}.]
  We apply \cref{thm:thebigtheorem}
  in the case $N=1$ and $R^{(1)}=0$. Since $N=1$, we omit all ``$^{(1)}$'' superscripts in what follows. We have
  \[
    q_{*}\overset{\cref{eq:jstaridef}}{=}\sfU_{\rho}(T_{*})\overset{\text{\zcref[range]{eq:RiUidefs,eq:jstaridef}}}{=}\sfU_{T_{0},T_{1},\rho}(T_{0})\overset{\cref{eq:sfUdef}}{=}0.
  \]
  Therefore, $\Upsilon$ solves the SDE
  \begin{align*}
    \dif\Upsilon(q) & =J_{\sigma}\bigl(\hat{Q}-q,\Upsilon(q)\bigr)\dif B(q)\qquad\text{for }q\in \bigl[0, \h Q\bigr]; \\
    \Upsilon(0)     & =\clG_{T_{1}}v_{0}(x).
  \end{align*}
  Moreover, since
  \[
    \sfU_{\rho}(t)\overset{\cref{eq:T1idef}}{=}\sfU_{\rho}(T_{1})\overset{\cref{eq:Qidef-1}}{=}\hat{Q},
  \]
  the estimate \cref{eq:thebigthm-conclusion} becomes
  \begin{equation}
    \clW_{2}\bigl(v_{t}(x),\Upsilon(\hat{Q})\bigr)\le C\langle\vvvert v_{0}\vvvert_{\ell}\rangle\sqrt{\frac{\sfL(\sfL(1/\rho))}{\sfL(1/\rho)}}.\label{eq:usethemainthm}
  \end{equation}
  To conclude, we need to use the fact that $\hat{Q}$ is very
  close to $1$. Indeed, we have
  \begin{align*}
    |\hat{Q}-1|= |\sfS_{\rho}(T_{1}-T_{0})-1|&=\bigl\lvert \sfS_{\rho}\bigl(\sfL(1/\rho)\omega(\rho)^{-1}T_{0}\bigr)-1\bigr\rvert \\
                                                            &= \sfL(1/\rho)^{-1}\left\lvert  \log\frac{\sfL(1/\rho)\omega(\rho)^{-1}T_{0}/\rho+1}{1/\rho+1}\right\rvert   \\
                                                            &= \sfL(1/\rho)^{-1}\left\lvert  \log\frac{\sfL(1/\rho)\omega(\rho)^{-1}T_{0}+\rho}{1+\rho}\right\rvert  .
  \end{align*}
  Now $T_0 \geq t/2$ for $\rho$ sufficiently small, so $(2\overline{T})^{-1}\le T_{0}\le\overline{T}$ then.
  We can thus use \cref{eq:omegaissmallenough} to write
  \[
    \left\lvert  \log\frac{\sfL(1/\rho)\omega(\rho)^{-1}T_{0}+\rho}{1+\rho}\right\rvert  \le C\sfL(\sfL(1/\rho))
  \]
  for small $\rho$ and a constant $C$ depending on $\omega$ and $\overline{T}$.
  Using \cref{eq:changetheendtime} of \cref{prop:Jsigmatimereg-1}, we
  can thus conclude that
  \begin{equation}
    \clW_{2}\bigl(\Upsilon\bigl(\hat{Q}),\Gamma_{\clG_{T_{1}}v_{0}(x),1}^{\sigma}(1)\bigr)\le C\langle\vvvert v_{0}\vvvert_{\ell}\rangle\sqrt{\frac{\sfL(\sfL(1/\rho))}{\sfL(1/\rho)}}.\label{eq:fixtheendtime}
  \end{equation}
  Combining \cref{eq:usethemainthm} and \cref{eq:fixtheendtime} (using
  the triangle inequality for $\clW_{2}$) completes the proof.
\end{proof}
\begin{proof}[Proof of \cref{thm:mainthm-multipoint}.]
  \pagetarget{proofofmainthmmultipoint}
  We apply \cref{thm:thebigtheorem}
  with $t^{(i)}\setto t_{\rho}^{(i)}$, $R^{(i)}\setto R_{\rho}^{(i)}$,
  and $x^{(i)}\setto x_{\rho}^{(i)}$.
  We similarly add a subscript $\rho$ to the objects appearing in the statement of \cref{thm:thebigtheorem}, namely
  $\hat{Q}_\rho^{(i)}$, $q_{*,\rho}^{(i)}$, and $\Upsilon^{(i)}_\rho$.
  From the assumption that $\overline{T}^{-1}\le t_{\rho}^{(i)}\le\overline{T}$,
  we see that
  \begin{equation}
    \lim_{\rho\searrow0}\hat{Q}_{\rho}^{(i)}=1.\label{eq:Qhatito1}
  \end{equation}
  We can also compute
  \begin{equation*}
	  \lim_{\rho\to 0} \sup_{q\in [0,q_{*,\rho}^{(i)}]}\bigl\lvert \sfU_\rho^{(j_*(i))}\bigl(\sfR_\rho^{(i)}(q)\bigr)-q\bigr\rvert  =0,
  \end{equation*}
  and if we define (with $p^{(i,j)}$ as in \cref{eq:dijqilimits}) 
  \begin{equation*}
    p_*^{(i)} \coloneqq 0\vee \max_{j<i}p^{(i,j)},
  \end{equation*}
  then we have
  \begin{equation*}
	  \lim_{\rho\to 0} q^{(i)}_{*,\rho} = p_*^{(i)}.
  \end{equation*}
  Furthermore, if $p^{(i,j)}>0$, then
  \[
    \lim_{\rho\searrow0}|t_{\rho}^{(i)}+R_{\rho}^{(i)}-(t_{\rho}^{(j)}+R_{\rho}^{(j)})|=0\qquad\text{and}\qquad\lim_{\rho\searrow0}|x_{\rho}^{(i)}-x_{\rho}^{(j)}|=0,
  \]
  which means that
  \begin{equation}
	  \lim_{\rho\searrow0}|\clG_{t_{\rho}^{(i)}+R_{\rho}^{(i)}}v_{0}(x_{\rho}^{(i)})-\clG_{t_{\rho}^{(j)}+R_{\rho}^{(j)}}v_{0}(x_{\rho}^{(j)})|=0.\label{eq:icsconverging}
\end{equation}

  We wish to compare solutions $\Psi^{(i)}$ of the problem \zcref[range]{eq:upsiloniintroeqn,eq:upsiloniintroic} to solutions $\Upsilon^{(i)}$ of the problem \zcref[range]{eq:dUpsiloni,eq:Upsiloniic}. We define \begin{equation}j_*(i) \coloneqq\begin{cases}i&\text{if }p^{(i)}_* = 0;\\ \argmax\limits_{j< i} p^{(i,j)}&\text{otherwise.}\end{cases} \label{eq:jstardef}\end{equation}%
  Then we claim that \zcref[range]{eq:upsiloniintroeqn,eq:upsiloniintroic} is equivalent to the system of SDEs
  \begin{align}
	  \dif \Psi^{(i)}(q) &= J_\sigma\bigl(1-q,\Psi^{(i)}(q)\bigr)\dif B^{(i)}(q),&&\text{for }q\in (p_*^{(i)},1);\label{eq:newupsiloneqn}\\
	  \Psi^{(i)}(q) &= \begin{cases}
	  		\Psi^{(j_*(i))}(q)&\text{if }j_*^{(i)}<i;\\
	  		\mathcal{G}_{T^{(i)}}v_0(x^{(i)})&\text{otherwise.}
		\end{cases}
			&&\text{for }q\in [0,p^{(i)}_*].\label{eq:newupsilonic}
  \end{align}
  This is simply another way of representing the tree structure discussed after \cref{thm:mainthm-multipoint } and depicted in \cref{fig:treestructure}.
  Rather than representing the tree structure through the correlations of the driving Brownian motions, we represent coincident paths by the path with the least index, starting the new path $\Psi^{(i)}$ at the time $p^{(i)}_*$ that it diverges from the last path with a lower index.

  To see the equivalence formally, note that if $j_*(i)<i$, then $x^{(i)} = x^{(j_*(i))}$ and $T^{(i)} = T^{(j_*(i))}$ by \zcref[range]{eq:dparabolicdistancedef,eq:xiandtilimits}, so $\clG_{T^{(i)}}v_0(x^{(i)}) = \clG_{T^{(j_*(i))}}v_0(x^{(j_*(i))})$. Therefore, by the correlation structure \cref{eq:Bijcorrelations}, in a solution to \zcref[range]{eq:upsiloniintroeqn,eq:upsiloniintroic}, we have $\Psi^{(i)}(q) = \Psi^{(j_*(i))}(q)$ whenever $q\le p^{(i,j)} = p_*^{(i)}$, and so \cref{eq:newupsilonic} is satisfied. On the other hand, if $j_*(i)=i$, then $p_*^{(i)} = 0$ by \cref{eq:jstardef}, and so \cref{eq:newupsilonic} matches the initial condition \cref{eq:upsiloniintroic}.

The limits %
\zcref[range]{eq:Qhatito1,eq:icsconverging} show that the SDE problem \zcref[range]{eq:dUpsiloni,eq:Upsiloniic} converges to the problem \zcref[range]{eq:newupsiloneqn,eq:newupsilonic}.
Thus, by standard well-posedness/continuity
  results for SDEs, we see that the law of $(\Upsilon_{\rho}^{(1)},\ldots,\Upsilon_{\rho}^{(N)})$
  converges to that of $(\Psi^{(1)},\ldots,\Psi^{(N)})$
  (defined in \zcref[range]{eq:upsiloniintroeqn,eq:upsiloniintroic})
  as $\rho\searrow0$.

  Finally, we note that
  \begin{align*}
    \lim_{\rho\searrow0}\sfU_{\rho}^{(i)}(t_{\rho}^{(i)}) & \overset{\substack{\cref{eq:sfUdef},                                                     \\
        \cref{eq:Qidef-1}
      }
    }{=}\lim_{\rho\searrow0}\bigl[\hat{Q}_{\rho}^{(i)}-\sfS_{\rho}(R_{\rho}^{(i)})\bigr]\overset{\substack{\cref{eq:dijqilimits},\cref{eq:Srhodef}, \\
        \cref{eq:Qhatito1}
      }
    }{=}q^{(i)}.
  \end{align*}
  Given this, and again using \cref{prop:Jsigmatimereg-1}, we see that
  the estimate \cref{eq:thebigthm-conclusion} yields the asymptotic statement
  \cref{eq:mainthm-multipoint-conclusion}.
\end{proof}
The remainder of this section will be devoted to the proof of \cref{thm:thebigtheorem}.
We let $\overline{T},\beta,M,N,\sigma,\omega$ be as in the statement.
By \cref{thm:QSPDEgreater}, $\Qbar_{\SPDE}(\sigma) > 1$.
We can therefore take %
$\overline{q}\in\bigl(1,\Qbar_{\SPDE}(\sigma)\bigr)$.
Throughout, $C$ will denote a finite positive
constant depending on $\overline{T},\beta,M,N,$ and $\sigma$ that may change from line to line. %
The next three sections, \zcref[range]{sec:parameters,sec:agreement}, are concerned with the interactions of multiple points. The reader who is only interested in the case $N=1$ may proceed directly to the proof of \cref{thm:thebigtheorem} in \cref{sec:bigtheoremproof}.

\subsection{Multipoint parameters\label{sec:parameters}}
To begin, we introduce additional notation and study the relationships between the various space-time points. %
Recalling the definition \cref{eq:nurhodef} of $\nu_{\rho}$, define
the parabolic cone
\begin{equation}
  \ttE^{(i)}= \big\{(r,y)\in(-\infty,T_{1}^{(i)})\times\RR^{2}\suchthat|x^{(i)}-y|^{2}\le\nu_{\rho}(T_1^{(i)}-r)\big\}.\label{eq:Eidef}
\end{equation}
We will treat pairs of indices $i,j$ differently depending on whether the space-time points
$(T_1^{(i)},x^{(i)})$ and $(T_1^{(j)},x^{(j)})$
are ``time-separated'' or ``space-separated.'' For each $i$,
define
\begin{equation}
  P^{\mathrm{time}}(i)=\left\{ j\in\{1,\ldots,n\}\suchthat 32\nu_{\rho}|T_1^{(i)}-T_1^{(j)}|>|x^{(i)}-x^{(j)}|^{2}\right\} \label{eq:Ptimdef}
\end{equation}
and
\begin{equation}
  P^{\mathrm{space}}(i)=\{1,\ldots,n\}\setminus P^{\mathrm{time}}(i)=\left\{ j\in\{1,\ldots,n\}\suchthat 32\nu_{\rho}|T_1^{(i)}-T_1^{(j)}|\le|x^{(i)}-x^{(j)}|^{2}\right\} .\label{eq:Pspacei}
\end{equation}
This means that \begin{equation}D^{(i,j)} = \begin{cases}|T_1^{(i)}-T_1^{(j)}|&\text{if }j\in P^{\mathrm{time}}(i);\\ %
\frac1{32\nu_\rho}|x^{(i)}-x^{(j)}|^2&\text{if }j\in P^{\mathrm{space}}(i).\label{eq:PcasesDij}\end{cases}\end{equation}
We also define
\begin{equation}
  P_{-}^{\mathrm{time}}(i)=\{j\in P^{\mathrm{time}}(i)\suchthat T_1^{(j)}\le T_1^{(i)}\}.\label{eq:Ptimeminus}
\end{equation}
and
\begin{equation}
  \ttS^{(i)}=\ttE^{(i)}\setminus\bigcup_{j\in P^{\mathrm{time}}_-(i)}\bigl([T_{*}^{(i,j)},T_{1}^{(j)}]\times\RR^{2}\bigr).\label{eq:Sidef}
\end{equation}
(See \cref{fig:Sis} for an illustration.)
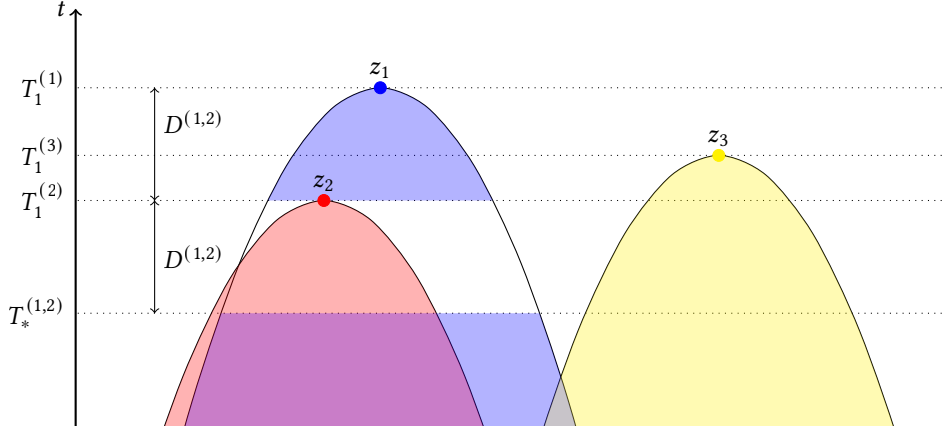
\begin{figure}
	\centering
	\tikzset{ %
    clip even odd rule/.code={\pgfseteorule}, %
    invclip/.style={
        clip,insert path=
            [clip even odd rule]{
                [reset cm](-\maxdimen,-\maxdimen)rectangle(\maxdimen,\maxdimen)
            }
    }
}
	\begin{tikzpicture}[scale=1.5]
		\begin{scope}[blend group=lighten]
		\clip (-3.5,1)--(5,1)--(5,5)--(-3.5,5)--cycle;
		\draw[->,thick] (-2.7,-1)--(-2.7,4.7) node[left] {$t$};
	\draw[black] plot[smooth,domain=-5:5] (\x-0.5,{3-(\x)^2});
	\fill[red,opacity=0.3] plot[smooth,domain=-5:5] (\x-0.5,{3-(\x)^2})--cycle;
	\draw[black] plot[smooth,domain=-5:5] (\x,{4-(\x)^2});
	\begin{scope}
		\path[invclip] (-3.5,2) -- (7,2) -- (7,3) -- (-3.5,3)--cycle;
	\fill[blue,opacity=0.3] plot[smooth,domain=-5:5] (\x,{4-(\x)^2})--cycle;
\end{scope}
	\draw[black] plot[smooth,domain=-5:5] (\x+3,{3.4-(\x)^2});
	\fill[yellow,opacity=0.3] plot[smooth,domain=-5:5] (\x+3,{3.4-(\x)^2})--cycle;
	\draw[dotted,thin] (7,3)--(-2.7,3) node[left] {$T^{(2)}_1$};
	\draw[dotted,thin] (7,4)--(-2.7,4) node[left] {$T^{(1)}_1$};
	\draw[dotted,thin] (7,3.4)--(-2.7,3.4) node[left] {$T^{(3)}_1$};
	\draw[dotted,thin] (7,2)--(-2.7,2) node[left] {$T_*^{(1,2)}$};
	\draw[<->,thin] (-2,3)--(-2,3.7) node[right] {$D^{(1,2)}$} -- (-2,4);
	\draw[<->,thin] (-2,2)--(-2,2.5) node[right] {$D^{(1,2)}$} -- (-2,3);
\end{scope}
		\filldraw[color=red] (-0.5,3) circle (1.5pt) node[anchor=south,color=black] {$z_2$};
		\filldraw[color=blue] (0,4) circle (1.5pt) node[anchor=south,color=black] {$z_1$};
		\filldraw[color=yellow] (3,3.4) circle (1.5pt) node[anchor=south,color=black] {$z_3$};

	\end{tikzpicture}
	\caption{\label{fig:Sis}Schematic illustration of the sets $\ttS^{(i)}$. The points $z^{(i)} = (T^{(i)}_1,x^{(i)})$ are drawn with colored dots, and the corresponding parabolic cones $\ttE^{(i)}$ are the regions under the parabolas. In this example, we have $P^{\mathrm{time}}(1) = \{2\}$, $P^{\mathrm{time}}(2) = \{1\}$, and $P^{\mathrm{time}}(3)=\emptyset$, so $P^{\mathrm{time}}_-(1) =\{2\}$ and $P^{\mathrm{time}}_-(2) = P^{\mathrm{time}}_-(3)=\emptyset$. Therefore, to form the set $\ttS^{(1)}$ (shaded in blue), we remove the time interval $[T_*^{(1,2)},T_1^{2}]$ from $\ttE^{(1)}$. The sets $\ttS^{(2)}$ and $\ttS^{(3)}$ (shaded in red and yellow, respectively) are equal to all of $\ttE^{(2)}$ and $\ttE^{(3)}$, respectively.} 
\end{figure}
Finally, let
\begin{equation*}
  \ttE_{t}^{(i)}=\{y\in\RR^{2}\suchthat (t,y)\in \ttE^{(i)}\}\qquad\text{and}\qquad \ttS_{t}^{(i)}=\{y\in\RR^{2}\suchthat (t,y)\in \ttS^{(i)}\}
\end{equation*}
be the projections onto the time slice $t$, and let \begin{equation}\ttXi_{t}^{(i)}=\RR^{2}\setminus \ttS_{t}^{(i)}.\label{eq:ttXidef}\end{equation}
\begin{lem}
  \label{lem:Sisseparated}Let $i\ne j\in\{1,\ldots,n\}$. Whenever
  $t\in[T_{*}^{(i,j)},T_{1}^{(i)}]$, we have
  \begin{equation}
    \dist(\ttS_{t}^{(i)},\ttS_{t}^{(j)})\ge\nu_{\rho}^{1/2}\left((T_{1}^{(i)}-t)^{1/2}+(T_{1}^{(j)}-t)^{1/2}\right).\label{eq:slicesseparated}
  \end{equation}
\end{lem}
In interpreting \cref{eq:slicesseparated}, we use the usual convention
that
\begin{equation}
  \dist(\emptyset,\ttA)=\dist(\ttA,\emptyset)=\infty\qquad\text{for any }\ttA\subseteq\RR^{2}.\label{eq:emptysetdist}
\end{equation}

\begin{proof}
  Since $i\ne j$, we have by the definition \cref{eq:T0ij-1} that
  \begin{equation}\label{eq:Tstarijdefifinej}
	  T_{*}^{(i,j)}=\max\{ T_{1}^{(i)}\wedge T_{1}^{(j)}-D^{(i,j)},T_0^{(i)}\}.%
  \end{equation}
  We consider several cases.
 If $t\ge T_1^{(j)}$, then $\ttS_t^{(j)}\subseteq \ttE_t^{(j)}=\emptyset$ by the definitions \cref{eq:Sidef,eq:Eidef}, %
	  so $\dist(\ttS_{t}^{(i)},\ttS_{t}^{(j)})=\infty$
    by the convention \cref{eq:emptysetdist}, and thus \cref{eq:slicesseparated}
    holds vacuously. So we can assume that 
$t\le T_1^{(j)}$. %
Under this assumption, if  $j\in P^{\mathrm{time}}(i)$, %
	then $\ttS_{t}^{(i)}=\emptyset$ by \cref{eq:Sidef} since $t\in [T_*^{(i,j)},T_1^{(j)}]$, so again we have $\dist(\ttS_{t}^{(i)},\ttS_{t}^{(j)})=\infty$
    and \cref{eq:slicesseparated} holds.
So for the rest of the proof, we assume that $t\le T_1^{(j)}$ and $j\in P^{\mathrm{space}}(i)$.

Under these assumptions, we see that %
$D^{(i,j)} = \frac1{32\nu_\rho}|x^{(i)}-x^{(j)}|^2$ by \cref{eq:PcasesDij},  so
\[T_{1}^{(j)}-\frac1{32\nu_\rho}|x^{(i)}-x^{(j)}|^{2}\overset{\cref{eq:Tstarijdefifinej}}\le T_{*}^{(i,j)}\le t\le T_{1}^{(i)}\wedge T_1^{(j)}.\] Rearranging, we see that
    \begin{align*}
	    |x^{(i)}-x^{(j)}|^{2}        \ge32\nu_{\rho}[T_{1}^{(i)}\wedge T_1^{(j)}-t]                                                                      
					 &=16\nu_{\rho}\sum_{n\in\{i,j\}}[T_{1}^{(n)}-t]-16\nu_{\rho}|T_{1}^{(i)}-T_{1}^{(j)}|            \\
	    \overset{\cref{eq:Pspacei}} & {\ge}16\nu_{\rho}\sum_{n\in\{i,j\}}[T_{1}^{(n)}-t]-\frac{1}{2}|x^{(i)}-x^{(j)}|^{2}.
    \end{align*}
    This means that
    \begin{equation}
    |x^{(i)}-x^{(j)}|      >2\nu_\rho^{1/2}\sqrt{\sum_{n\in{i,j}}[T^{(n)}-t]}
			  \ge2\nu_{\rho}^{1/2}\sum_{n\in\{i,j\}}[T^{(i)}-t]^{1/2},\label{eq:ourbound}
    \end{equation}
    so
    \begin{align*}
      \dist(\ttS_{t}^{(i)},\ttS_{t}^{(j)})  \overset{\cref{eq:Sidef}}  {\ge}\dist(\ttE_{t}^{(i)},\ttE_{t}^{(j)})                                                      %
      \overset{\cref{eq:Eidef}}                                 & {\ge}|x^{(i)}-x^{(j)}|-[\nu_{\rho}(T^{(i)}-t)]^{1/2}-[\nu_{\rho}(T^{(j)}-t)]^{1/2} \\
      \overset{\cref{eq:ourbound}}                              & {\ge}\nu_{\rho}^{1/2}\left([T^{(i)}-t]^{1/2}+[T^{(j)}-t]^{1/2}\right),
    \end{align*}
    which is \cref{eq:slicesseparated}.
\end{proof}

\begin{lem}
  \label{lem:ifthereisspacethenthereisenoughspace}There is a constant
  $C<\infty$ such that for all $i,j\in\{1,\ldots,n\}$, if $T_{0}^{(i)}\vee T_{0}^{(j)}<T_{*}^{(i,j)}$,
  then
  \begin{equation}
    |T_{1}^{(i)}-T_{1}^{(j)}|\vee|x^{(i)}-x^{(j)}|^{2}/2\le \frac{C\sfL(\sfL(1/\rho))(T_{0}^{(i)}\vee T_{0}^{(j)})}{\sfL(1/\rho)}.\label{eq:ifthereisspacethenthereisenoughspace}
  \end{equation}
\end{lem}

\begin{proof}
  The assumption that $T_{0}^{(i)}\vee T_{0}^{(j)}<T_{*}^{(i,j)}$ implies
  that in particular $T_{0}^{(i)}\ne T_{*}^{(i,j)}$, so by \cref{eq:T0ij-1}
  we must have
  \[
    T_{*}^{(i,j)}=T_{1}^{(i)}\wedge T_{1}^{(j)}-D^{(i,j)}.
  \]
  The assumption then becomes
  \begin{equation}
    T_{0}^{(i)}\vee T_{0}^{(j)}\le T_{1}^{(i)}\wedge T_{1}^{(j)}-D^{(i,j)}.\label{eq:newassumption}
  \end{equation}

  For the remainder of the proof, we will only rely on the assumption
  \cref{eq:newassumption}, which is symmetric in $i$ and $j$. So we
  can assume without loss of generality that $T_{1}^{(i)}\le T_{1}^{(j)}$;
  this implies in particular that $T_{0}^{(i)}\le T_{0}^{(j)}$ as well
  (recalling the definition \cref{eq:T1idef}). The inequality \cref{eq:newassumption}
  then becomes
  \[
    T_{0}^{(j)}\le T_{*}^{(i,j)}\overset{\cref{eq:T0ij-1}}{=}T_{1}^{(i)}-D^{(i,j)}.
  \]
  Recalling the definition \cref{eq:Dijdef} of $D^{(i,j)}$ and rearranging, we obtain
  \begin{align*}
    |T_{1}^{(i)}-T_{1}^{(j)}|\vee\frac{|x^{(i)}-x^{(j)}|^{2}}{32\nu_{\rho}}\le T_{1}^{(i)}-T_{0}^{(j)}\le T_{1}^{(j)}-T_{0}^{(j)} & \overset{\cref{eq:T1idef}}{=}\frac{T_{0}^{(j)}\sfL(1/\rho)}{\omega(\rho)}\overset{\cref{eq:omegaisbigenough}}{\le}\frac{CT_{0}^{(j)}}{\sfL(1/\rho)}.
  \end{align*}
  and then \cref{eq:ifthereisspacethenthereisenoughspace} follows when we recall the definition \cref{eq:nurhodef} of $\nu_\rho$.
\end{proof}
\subsection{Independence after the split time\label{sec:independence}}
For each $i\in\{1,\ldots,n\}$, let $\bigl(\hat{B}^{(i)}(t)\bigr)_{t\in[T_{0}^{(i)},T_{1}^{(i)}]}$
be the process $\bigl(\hat{B}(t)\bigr)_{t\in[T_{0}^{(i)},T_{1}^{(i)}]}$ defined
in \cref{eq:Bhatdef} (in the statement of \cref{prop:laydowntheSDE}),
with the choices $x\setto x^{(i)}$, $T_{1}\setto T_{1}^{(i)}$,
and $T_{0}\setto T_{0}^{(i)}$. We extend this definition
to all $t$ by constants outside of $[T_{0}^{(i)},T_{1}^{(i)}]$:
let $\hat{B}^{(i)}(t)=0$ for $t\le T_{0}^{(i)}$ and let $\hat{B}^{(i)}(t)=\hat{B}^{(i)}(T_{1}^{(i)})$
for all $t\ge T_{1}^{(i)}$. Let $\dif \widetilde W^{(i)}$ and $\zeta_r^{(i)}=[\sfL(1/\rho)^{-1}(T_1^{(i)}-r)]^{1/2}$ also be as in the statement of \cref{prop:laydowntheSDE}, with the $\dif \widetilde W^{(i)}$s coupled as in the hypothesis of \cref{prop:dBQVub} below.

Our goal in this section is to find another family of Brownian motions $\widetilde{B}^{(i)}$ such that each $\widetilde{B}^{(i)}$ individually has the same quadratic variation process as $\hat{B}^{(i)}$ (\cref{eq:QVsagree} below), the increments of $\widetilde{B}^{(i)}$ and $\widetilde{B}^{(i)}$ are independent after time $T^{(i,j)}_*$ (\cref{eq:noQCV} below), and $\widetilde{B}^{(i)}$ and $\hat{B}^{(i)}$ are close to each other (\cref{lem:QVbd} below). In other words, we seek to replace the ``approximate independence'' of increments by true independence after time $T_*^{(i,j)}$.

We define
\[
  g_{r}^{(i)}(y)=G_{T_{1}^{(i)}+\rho-r}(x^{(i)}-y)\mathbf{1}_{\ttS^{(i)}}(r,y)
\]
and, for $t\in[T_{0}^{(i)},T_{1}^{(i)}]$,
\begin{align*}
  \widetilde{B}^{(i)}(t) & =\gamma_{\rho}\int_{T_{0}^{(i)}}^{t}\!\!\int g_{r}^{(i)}(y)\,\dif\widetilde{W}_{r}^{(i)}(y)+\gamma_{\rho}\int_{T_{0}^{(i)}}^{t}\!\!\int\left(G_{T_{1}^{(i)}+\rho-r}(x^{(i)}-y)-g_{r}^{(i)}(y)\right)\,\dif\overline{W}_{r}^{(i)}(y)                      \\
                         & =\gamma_{\rho}\int_{T_{0}^{(i)}}^{t}\!\!\int_{\ttS_{r}^{(i)}}G_{T_{1}^{(i)}+\rho-r}(x^{(i)}-y)\,\dif\widetilde{W}_{r}^{(i)}(y)+\gamma_{\rho}\int_{T_{0}^{(i)}}^{t}\!\!\int_{\ttXi_{r}^{(i)}}G_{T_{1}^{(i)}+\rho-r}(x^{(i)}-y)\,\dif\overline{W}_{r}^{(i)}(y),
\end{align*}
where $\dif\overline{W}^{(1)},\ldots,\dif\overline{W}^{(n)}$
are new independent $\RR^{m}$-valued space-time white noises,
also independent of everything else.

We note the agreement of the
quadratic variations
\begin{equation}
    [\widetilde{B}^{(i)}](t)=[\hat{B}^{(i)}](t)=\Id_m\sfU^{(i)}_\rho(t)\overset{\cref{eq:sfUdef}}= \frac{\Id_m}{\sfL(1/\rho)}\log\frac{T_1^{(i)}-T_0^{(i)}+\rho}{T_1^{(i)}-t+\rho}\label{eq:QVsagree}
\end{equation}
for each $t\in[T_{0}^{(i)},T_{1}^{(i)}]$, by the same computation as in \cref{eq:BhatQV}.
Also, \cref{lem:Sisseparated} implies that, whenever $i\ne j$ and $t\in [T_*^{(i,j)},T_1^{(i)}]$, we have
\[\dist(\supp g_r^{(i)},\supp g_r^{(j)})=\dist(\ttS^{(i)}_r,\ttS^{(j)}_r)\ge \nu_\rho^{1/2}\sum_{n\in\{i,j\}} (T_1^{(n)}-t)^{1/2}\ge \sqrt{2}(\zeta_r^{(i)}+\zeta_r^{(j)})\]
(at least for $\rho$ less than an absolute constant),
so the hypothesis \cref{eq:suppsfarenough} is satisfied and \cref{prop:dBQVub}
tells us that
\begin{equation}\label{eq:noQCV}
  \dif[\widetilde{B}^{(i)},\widetilde{B}^{(j)}](t)=0\qquad\text{whenever }t\in[T_{*}^{(i,j)},T_{1}^{(i)}]. %
\end{equation}
This implies that (recalling the definitions \cref{eq:jstaridef})
\begin{equation}
	\left\{ \left(\widetilde{B}^{(i)}(t)-\widetilde{B}^{(i)}(T_*^{(i)})\right)_{t\in[T_{*}^{(i)},T_{1}^{(i)}]}\right\} _{i\in\{1,\ldots,n\}}\text{ is an independent family of random variables}.\label{eq:independent-1}
\end{equation}
The independence property \cref{eq:independent-1} is the reason we wish to work with $\widetilde{B}^{(i)}$ instead of $\hat{B}^{(i)}$.
The next lemma shows that $\widetilde{B}^{(i)}$ can replace $\hat{B}^{(i)}$ without incurring too much error.

\begin{lem}\label{lem:QVbd}
  We have
  \begin{equation}
    \left\lvert  \int_{T_{0}^{(i)}}^{T_{1}^{(i)}}\dif[\widetilde{B}^{(i)}-\hat{B}^{(i)}](t)\right\rvert  _{\op}\le\frac{C}{\sfL(1/\rho)}.\label{eq:QVbd}
  \end{equation}
\end{lem}

\begin{proof}
    First we deal with the interval $[T_1^{(i)}-\rho,T_1^{(i)}]$. For this interval, we have
    \begin{equation}
        \left\lvert\int_{T_1^{(i)}-\rho}^{T_1^{(i)}} \dif[\widetilde{B}^{(i)}](t)\right\rvert_{\op} \overset{\cref{eq:QVsagree}}=  \left\lvert\int_{T_1^{(i)}-\rho}^{T_1^{(i)}} \dif[\widetilde{B}^{(i)}](t)\right\rvert_{\op} \overset{\cref{eq:QVsagree}}= \frac{\log 2}{\sfL(1/\rho)}.\label{eq:finalinterval}
    \end{equation}

    Now suppose that \begin{equation}T_{0}^{(i)}\le t\le T_1^{(i)}-\rho.\label{eq:tnotatend}\end{equation} We have
  \begin{align*}
    \dif & [\widetilde{B}^{(i)}-\hat{B}^{(i)}](t)/\dif t=2\gamma_{\rho}^{2}\Id_{m}\int_{\ttXi_{t}^{(i)}}G_{T_{1}^{(i)}+\rho-t}^{2}(x^{(i)}-y)\,\dif y,
  \end{align*}
  so
  \begin{align}
    |\dif & [\widetilde{B}^{(i)}-\hat{B}^{(i)}](t)/\dif t|_{\op}\le2\gamma_{\rho}^{2}\int_{\ttXi_{t}^{(i)}}G_{T_{1}^{(i)}+\rho-t}^{2}(x^{(i)}-y)\,\dif y\nonumber                                                                                                                    \\
          & \le2\gamma_{\rho}^{2}\int_{\RR^{2}\setminus \ttE_{t}^{(i)}}G_{T_{1}^{(i)}+\rho-t}^{2}(x^{(i)}-y)\,\dif y+\frac{8\pi}{\sfL(1/\rho)(T_{1}^{(i)}+\rho-t)}\sum_{j\in P_{-}^{\mathrm{time}}(i)}\mathbf{1}_{[T_{*}^{(i,j)},T_{1}^{(j)}]}(t),\label{eq:Btildebhatdiff}
  \end{align}
where in the last inequality we used \zcref[range]{eq:Sidef,eq:ttXidef} and~\cref{eq:GtauL2norm}.
  We can estimate the integral as
  \begin{align*}
      &\int_{\RR^{2}\setminus \ttE_{t}^{(i)}}G_{T_{1}^{(i)}+\rho-t}^{2}(x^{(i)}-y)\,\dif y  \overset{\cref{eq:GT2}}{=}\frac{1}{4\pi(T_{1}^{(i)}+\rho-t)}\int_{\RR^{2}\setminus \ttE_{t}^{(i)}}G_{\frac{1}{2}[T_{1}^{(i)}+\rho-t]}(x^{(i)}-y)\,\dif y         \\
                                                                                        &\qquad \overset{\cref{eq:Eidef}}{=}\frac{1}{4\pi\exp\left(\nu_{\rho}\frac{T^{(i)}_1-t}{T^{(i)}_1+\rho-t}\right)(T_{1}^{(i)}+\rho-t)}\overset{\cref{eq:tnotatend}}\le\frac1{4\pi\e^{\nu_\rho/2}(T_1^{(i)}+\rho-t)}
                                                                                    \overset{\cref{eq:nurhodef}}{\le}\frac{1}{4\pi\sfL(1/\rho)(T_{1}^{(i)}+\rho-t)}.
  \end{align*}
  Using this in \cref{eq:Btildebhatdiff}, we see that, for $t\ge T_{*}^{(i,j)}$,
  we have
  \[
    |\dif[\widetilde{B}^{(i)}-\hat{B}^{(i)}](t)/\dif t|_{\op}\le\frac{2}{\sfL(1/\rho)(T_{1}^{(i)}+\rho-t)}\left(\frac{1}{\sfL(1/\rho)}+4\pi\sum_{j\in P_{-}^{\mathrm{time}}(i)}\mathbf{1}_{[T_{*}^{(i,j)},T_{1}^{(j)}]}(t)\right),
  \]
  and hence
  \begin{align}
    \left\lvert  \int_{T_{0}^{(i)}}^{T_{1}^{(i)}-\rho}\dif[\widetilde{B}^{(i)}-\hat{B}^{(i)}](t)\right\rvert  _{\op} & \le\frac{2}{\sfL(1/\rho)}\int_{T_{0}^{(i)}}^{T_{1}^{(i)}}\frac{1}{T_{1}^{(i)}+\rho-t}\left(\frac{1}{\sfL(1/\rho)}+4\pi\sum_{j\in P^{\mathrm{time}}(i)}\mathbf{1}_{[T_{*}^{(i,j)},T_{1}^{(j)}]}(t)\right)\,\dif t\nonumber \\
                                                                                                  & \le\frac{C}{\sfL(1/\rho)}+\frac{8\pi}{\sfL(1/\rho)}\sum_{j\in P_{-}^{\mathrm{time}}(i)}\int_{T_{*}^{(i,j)}}^{T_{1}^{(j)}}\frac{\dif t}{T_{1}^{(i)}+\rho-t}\nonumber                                                       \\
                                                                                                  & =\frac{C}{\sfL(1/\rho)}+\frac{8\pi}{\sfL(1/\rho)}\sum_{j\in P_{-}^{\mathrm{time}}(i)}\log\frac{T_{1}^{(i)}-T_{*}^{(i,j)}+\rho}{T_{1}^{(i)}-T_{1}^{(j)}+\rho}.\label{eq:boundtheapproxerror}
  \end{align}
  Now if $j\in P_{-}^{\mathrm{time}}(i)$, then we have (recalling the
  definitions \zcref[range]{eq:Ptimdef,eq:Ptimeminus}) $T^{(i)}\ge T^{(j)}$
  and $T_{*}^{(i,j)}=T_{1}^{(j)}-|T_{1}^{(i)}-T_{1}^{(j)}|$, which
  means that
  \[
    \frac{T_{1}^{(i)}-T_{*}^{(i,j)}+\rho}{T_{1}^{(i)}-T_{1}^{(j)}+\rho}=\frac{T_{1}^{(i)}-(T_{1}^{(j)}-|T_{1}^{(i)}-T_{1}^{(j)}|)+\rho}{T_{1}^{(i)}-T_{1}^{(j)}+\rho}=\frac{2|T_{1}^{(i)}-T_{1}^{(j)}|+\rho}{T_{1}^{(i)}-T_{1}^{(j)}+\rho}\le2.
  \]
  Using this in \cref{eq:boundtheapproxerror}, and then combining the result with \cref{eq:finalinterval}, we get \cref{eq:QVbd}.
\end{proof}
\subsection{Agreement before the split time\label{sec:agreement}}
Now we show that, up to the time $T_*^{(i,j)}$, the $i$th and $j$th martingales stay close together.
\begin{lem}
  \label{lem:approxtheVijs}There exists $C<\infty$ such that
  for all $i,j\in\{1,\ldots,n\}$ and $t\in[T_{0}^{(i)}\vee T_{0}^{(j)},T_{*}^{(i,j)})$,
  we have
  \begin{equation}
    \EE \bigl\lvert V_{t}^{\rho,T_{1}^{(i)}}(x^{(i)})-V_{t}^{\rho,T_{1}^{(j)}}(x^{(j)})\bigr\rvert ^{2}\le C\langle\vvvert v_{0}\vvvert_{\ell}\rangle^{2}\frac{\sfL(\sfL(1/\rho))}{\sfL(1/\rho)}.\label{eq:approxtheVijs}
  \end{equation}
\end{lem}

\begin{proof}
  If the interval $[T_{0}^{(i)}\vee T_{0}^{(j)},T_{*}^{(i,j)})$ is
  empty (and in particular if $i=j$) then there is nothing to prove.
  Thus, we assume that $T_{0}^{(i)}\vee T_{0}^{(j)}<T_{*}^{(i,j)}$,
  so \cref{eq:T0ij-1} becomes
  \[
    T_{*}^{(i,j)}=T_{1}^{(i)}\wedge T_{1}^{(j)}-D^{i,j}%
  \]
  and \cref{lem:ifthereisspacethenthereisenoughspace} applies. We have
  by \cref{prop:continuitybd} (and \cref{prop:momentbd} to control the
  moment) that, for sufficiently small $\rho$,
  \begin{align}
    \EE & \bigl\lvert V_{t}^{\rho,T_{1}^{(i)}}(x^{(i)})-V_{t}^{\rho,T_{1}^{(j)}}(x^{(j)})\bigr\rvert ^{2}\overset{\cref{eq:Vdef}}{=}\EE |\clG_{T_{1}^{(i)}-t}v_{t}(x^{(i)})-\clG_{T_{1}^{(j)}-t}v_{t}(x^{(j)})|^{2}\nonumber \\
        & \le\frac{C\langle\vvvert v_{0}\vvvert_{\ell}\rangle^{2}}{\sfL(1/\rho)}\left[\sfL\left(\frac{|T_{1}^{(i)}-T_{1}^{(j)}|+|x^{(i)}-x^{(j)}|^{2}/2}{T_{1}^{(i)}\wedge T_{1}^{(j)}-t+\rho}\right)+\sfL(\sfL(1/\rho))\right]\nonumber         \\
        & \qquad+C\langle\vvvert v_{0}\vvvert_{\ell}\rangle^{2}\left(\frac{|T_{1}^{(i)}-T_{1}^{(j)}|+|x^{(i)}-x^{(j)}|^{2}/2}{t}\right).\label{eq:applycontinuityandmomentbounds}
  \end{align}
  We note that
  \begin{equation}
    \sfL\left(\frac{|T_{1}^{(i)}-T_{1}^{(j)}|+|x^{(i)}-x^{(j)}|^{2}/2}{T_{1}^{(i)}\wedge T_{1}^{(j)}-t+\rho}\right)\overset{\text{\zcref[range]{eq:Dijdef,eq:T0ij-1}}}{\le}\sfL\left(\frac{|T_{1}^{(i)}-T_{1}^{(j)}|+|x^{(i)}-x^{(j)}|^{2}/2}{|T_{1}^{(i)}-T_{1}^{(j)}|\vee\frac{|x^{(i)}-x^{(j)}|^{2}}{32\nu_{p}}+\rho}\right)\overset{\cref{eq:nurhodef}}{\le}\sfL(\sfL(1/\rho))\label{eq:boundfirstthing}
  \end{equation}
  and that
  \begin{equation}
    \frac{|T_{1}^{(i)}-T_{1}^{(j)}|+|x^{(i)}-x^{(j)}|^{2}/2}{t}\le\frac{|T_{1}^{(i)}-T_{1}^{(j)}|+|x^{(i)}-x^{(j)}|^{2}/2}{T_{0}^{(i)}\vee T_{0}^{(j)}}\overset{\cref{eq:ifthereisspacethenthereisenoughspace}}{\le}\frac{C\sfL(\sfL(1/\rho))}{\sfL(1/\rho)}.\label{eq:boundsecondthing}
  \end{equation}
  Using \cref{eq:boundfirstthing} and \cref{eq:boundsecondthing} in \cref{eq:applycontinuityandmomentbounds},
  we obtain \cref{eq:approxtheVijs}.
\end{proof}
\subsection{Proof of \texorpdfstring{\cref{thm:thebigtheorem}}{Theorem~\ref{thm:thebigtheorem}}\label{sec:bigtheoremproof}}
Now we can complete the proof of \cref{thm:thebigtheorem}.
Taking $T_{0}\setto T_{*}^{(i)}$, $T_{1}\setto T_{1}^{(i)}$,
and $x\setto x^{(i)}$ in \cref{prop:laydowntheSDE}, we obtain
\begin{align*}
  \EE & \left\lvert  V_{t}^{\rho,T_{1}^{(i)}}(x^{(i)})-V_{T_{*}^{(i)}}^{\rho,T_{1}^{(i)}}(x^{(i)})-\int_{T_{*}^{(i)}}^{t}J_{\sigma}\bigl(\sfS_{\rho}(T_{1}^{(i)}-r),V_{r}^{\rho,T_{1}^{(i)}}(x^{(i)})\bigr)\,\dif\hat{B}^{(i)}(r)\right\rvert  ^{2} \\
      &\hspace{9cm}\le C\langle\vvvert v_{0}\vvvert_{\ell}\rangle^{2}\frac{\sfL(\sfL(1/\rho))}{\sfL(1/\rho)}.
\end{align*}
Using the triangle inequality, we make two modifications to the left side.
First, we apply \cref{lem:approxtheVijs} with $j\setto j_{*}(i)$ to adjust the index in the middle term.
Second, we use \cref{eq:Juniflipschitz}, \cref{prop:momentbd}, and \cref{eq:QVbd} to replace $\h{B}$ by $\widetilde{B}$.
We find:
\begin{equation}
  \begin{aligned}\EE & \left\lvert  V_{t}^{\rho,T_{1}^{(i)}}(x^{(i)})-V_{T_{*}^{(i)}}^{\rho,T_{1}^{(j_{*}(i))}}(x^{(j_{*}(i))})-\int_{T_{*}^{(i)}}^{t}J_{\sigma}\bigl(\sfS_{\rho}(T_{1}^{(i)}-r),V_{r}^{\rho,T_{1}^{(i)}}(x^{(i)})\bigr)\,\dif\widetilde{B}_{r}^{(i)}\right\rvert  ^{2} \\
                     &\hspace{8.5cm}\le C\langle\vvvert v_{0}\vvvert_{\ell}\rangle^{2}\frac{\sfL(\sfL(1/\rho))}{\sfL(1/\rho)}.
  \end{aligned}
  \label{eq:usetriangle}
\end{equation}
Using \cref{lem:VTcontinuity} (and \cref{prop:momentbd}
to bound the moment), we also have
\begin{equation}
  \begin{aligned}\EE \bigl\lvert  & V_{T_{0}^{(i)}}^{\rho,T_{1}^{(i)}}(x^{(i)})-\clG_{T_{1}^{(i)}}v_{0}(x^{(i)})\bigr\rvert ^{2}=\EE \bigl\lvert V_{T_{0}^{(i)}}^{\rho,T_{1}^{(i)}}(x^{(i)})-V_{0}^{\rho,T_{1}^{(i)}}(x^{(i)})\bigr\rvert ^{2}                                                                                                                                                                                                   \\
                                  & \le\frac{C\langle\vvvert v_{0}\vvvert_{\ell}\rangle^{2}}{\sfL(1/\rho)}\sfL\left(\frac{T_{0}^{(i)}}{T_{1}^{(i)}+\rho-T_{0}^{(i)}}\right)\le\frac{C\langle\vvvert v_{0}\vvvert_{\ell}\rangle^{2}}{\sfL(1/\rho)}\sfL\left(\frac{\omega(\rho)}{\sfL(1/\rho)}\right)\overset{\cref{eq:omegaissmallenough}}{\le}C\langle\vvvert v_{0}\vvvert_{\ell}\rangle^{2}\frac{\sfL(\sfL(1/\rho))}{\sfL(1/\rho)}.
  \end{aligned}
  \label{eq:dealwiththezero}
\end{equation}

We next follow the change of variables procedure used in \cref{step:chgvar}
of the proof of \cref{prop:Japprox}. For each $i=1,\ldots,n$ and $q\in[0,\hat{Q}^{(i)}]$, we define
\begin{equation}
  B^{(i)}(q)=\widetilde{B}^{(i)}\bigl(\sfR_{\rho}^{(i)}(q)\bigr).\label{eq:Aidef-2}
\end{equation}
(Note by \cref{eq:RiUidefs,eq:Qidef-1,eq:Rendpoints} that $\sfR_{\rho}^{(i)}(0)=T_{0}^{(i)}$
and $\sfR_{\rho}^{(i)}(\hat{Q}^{(i)})=T_{1}^{(i)}$.)
Mimicking \cref{eq:BhatQV}, we compute
\begin{equation*}
  \dif[B^{(i)}](q)/\dif q=\Id_{m},
\end{equation*}
so each $B^{(i)}$ is a standard $\RR^{m}$-valued Brownian motion. We can then
make the change of variables $t=\sfR_{\rho}^{(i)}(q)$, $r=\sfR_{\rho}^{(i)}(p)$
in \cref{eq:usetriangle}. We define
\begin{equation}
  \hat{\Upsilon}^{(i)}(q)\coloneqq V_{\sfR_{\rho}^{(i)}(q)}^{\rho,T_{1}^{(i)}}(x^{(i)}),\qquad q_{*}^{(i)}\coloneqq\sfU_{\rho}^{(i)}(T_{*}^{(i)}),\qquad\overline{q}_{*}^{(i)}=\sfU_{\rho}^{(j_{*}(i))}(T_{*}^{(i)}),\label{eq:upsilonhatdef}
\end{equation}
and note that $\sfS_{\rho}(T_{1}^{(i)}-\sfR_{\rho}^{(i)}(p))=\hat{Q}^{(i)}-p$
by \cref{eq:sfRdef} and \cref{eq:Qidef-1} to get
\[
  \sup_{q\in[q_{*}^{(i)},\hat{Q}^{(i)}]}\EE \left\lvert  \hat{\Upsilon}^{(i)}(q)-\hat{\Upsilon}^{(j_*(i))}(\overline{q}_{*}^{(i)})-\int_{q_{*}^{(i)}}^{q}J_{\sigma}\bigl(\hat{Q}^{(i)}-p,\hat{\Upsilon}^{(i)}(p)\bigr)\,\dif B^{(i)}(p)\right\rvert  ^{2}\le C\langle\vvvert v_{0}\vvvert_{\ell}\rangle^{2}\frac{\sfL(\sfL(1/\rho))}{\sfL(1/\rho)}.
\]
By \cref{prop:SDEwellposed}, \cref{eq:dealwiththezero}, and induction,
this implies that, if $\Upsilon^{(i)}$ solves the SDE
\begin{align}
  \dif\Upsilon^{(i)}(q)       & =J_{\sigma}\bigl(\hat{Q}^{(i)}-q,\Upsilon^{(i)}(q)\bigr)\dif B^{(i)}(q)\qquad\text{for }q\in[q_{*}^{(i)},\hat{Q}^{(i)}];\label{eq:upsilonSDE-proof} \\
  \Upsilon^{(i)}(q_{*}^{(i)}) & =\begin{cases}
                                   \Upsilon^{(j_*(i))}(\overline{q}_{*}^{(i)}) & \text{if }j_{*}(i)<i; \\
                                   \clG_{T_{1}^{(i)}}v_{0}(x^{(i)})       & \text{if }j_{*}(i)=i,
                                 \end{cases}\label{eq:upsilonic-proof}
\end{align}
then
\begin{equation}
  \sup_{q \in [0, \h{Q}^{(i)}]} \EE |\hat{\Upsilon}^{(i)}(q)-\Upsilon^{(i)}(q)|^{2}\le C\langle\vvvert v_{0}\vvvert_{\ell}\rangle^{2}\frac{\sfL(\sfL(1/\rho))}{\sfL(1/\rho)}.\label{eq:multipointSDEclose}
\end{equation}

Recalling the definition \cref{eq:jstaridef} and the independence
result \cref{eq:independent-1}, and applying the change of variables
\cref{eq:Aidef-2}, we notice that $\bigl\{\bigl(B^{(i)}(q) - B^{(i)}(q_*^{(i)})\bigr)_{q\in[q_{*}^{(i)},\hat{Q}^{(i)}]}\bigr\}_{i\in\{1,\ldots,n\}}$
is an independent family of random variables.
Since the SDE \cref{eq:upsilonSDE-proof}
only depends on $\dif B^{(i)}(q)$ for $q\ge q_{*}^{(i)}$, the problem
\zcref[range]{eq:upsilonSDE-proof,eq:upsilonic-proof} is in fact the
same as the problem \zcref[range]{eq:dUpsiloni,eq:Upsiloniic} in the
theorem statement (in which the driving Brownian motions were assumed
to be independent by fiat).

Finally, we recall that we are ultimately interested in
\[
  \clG_{R^{(i)}}v_{t^{(i)}}(x^{(i)})\overset{\substack{\cref{eq:Vdef},\\
      \cref{eq:T1idef}
    }
  }{=}V_{t^{(i)}}^{\rho,T_{1}^{(i)}}(x^{(i)})\overset{\cref{eq:upsilonhatdef}}{=}\hat{\Upsilon}^{(i)}(\sfU_{\rho}^{(i)}(t^{(i)})).
\]
Using this in \cref{eq:multipointSDEclose}, we see that
\[
  \EE |\clG_{R^{(i)}}v_{t^{(i)}}(x^{(i)})-\Upsilon^{(i)}(\sfU_{\rho}^{(i)}(t^{(i)}))|^{2}\le C\langle\vvvert v_{0}\vvvert_{\ell}\rangle^{2}\frac{\sfL(\sfL(1/\rho))}{\sfL(1/\rho)}.
\]
Since this holds for each $i$, we get \cref{eq:thebigthm-conclusion}.\qed

\section{Universality}
\label{sec:universal}

We now show that \cref{thm:mainthm-multipoint} is stable in the sense that approximate mild solutions of \cref{eq:SPDE} are themselves close to $v_t^\rho$.
This leads to a form of universality: variants of \cref{eq:SPDE} with similar noise structure exhibit the same large-scale statistics.
The main result of this section is \zcref{thm:stable}, a precise version of \zcref{thm:universal-informal}.

\subsection{Stability}

Similar to previous arguments in the paper, we show stability through renormalization/multiscale analysis.
To begin, we introduce notation for the maximal scale on which \cref{eq:SPDE} is stable.
Recall $\beta_*(\sigma)$ from \cref{eq:minimal-beta}.
\begin{defn}
  \label{defn:stable}
  Given $\sigma \in \Lip(\R^m;\m{H}_+^m)$, let $\bar{Q}_\stab(\sigma)$ denote the supremum of all $Q \leq \bar{Q}_\SPDE(\sigma)$ such that the following holds.
  For each $\beta \in (\beta_*(\sigma), Q^{-1/2})$ and $\ell \in (2, 4]$, there exists $C(\sigma,Q,\beta,\ell) < \infty$ such that, if $\rho \in (0, C^{-1}]$, $\bigl(u_t(x)\bigr)_{t \geq 0, \, x \in \R^2}$ is a random field adapted to $(\s{F}_t)$, $0 \leq s \leq t \leq \sfT_\rho(\beta^{-2})$, and $t - s \leq \sfT_\rho(Q)$, we have
  \begin{equation}
    \label{eq:stable}
    \m{M}_{s, t}(u - \m{V}_{s,\anon}^{\sigma,\rho} u_s) \leq C \m{M}_{0, \sfT_\rho(\beta^{-2})}(\m{T}_0^{\sigma,\rho} u - u) + C \hspace*{-10pt}\sup_{s \in [0, \sfT_\rho(\beta^{-2})]} \vvvert u_s \vvvert_\ell \left[\frac{\sfL\bigl(\sfL(1/\rho)\bigr)}{\sfL(1/\rho)}\right]^{1/2}.
  \end{equation}
\end{defn}
Thus up to scale $\bar{Q}_\stab$, if $u$ approximately solves \cref{eq:SPDE} in the sense that it is an approximate fixed point of $\m{T}$, then $u$ is close to the solution $v$ of \cref{eq:SPDE}.
We have already shown this stability up to the scale determined by $\Lip(\sigma)$:
\begin{prop}
  \label{prop:stab-Lip}
  For all $\sigma \in \Lip(\R^m;\m{H}_+^m)$, $\bar{Q}_\stab(\sigma) \geq \Lip(\sigma)^{-2}.$
\end{prop}
\begin{proof}
  Since $\bar Q_\SPDE(\sigma) \geq \Lip(\sigma)^{-2}$ by \cref{prop:QSPDEbasecase}, fix $Q \in \bigl(0, \Lip(\sigma)^{-2}\bigr)$ and $\beta \in (\beta_*(\sigma), Q^{-1/2})$.
  Take $0 \leq s \leq t \leq \sfT_\rho(\beta^{-2})$ such that $t - s \leq \sfT_\rho(Q)$.
  Then by \cref{prop:solnapprox}, we have
  \begin{equation*}
    \m{M}_{s,t}(u - \m{V}_{0,\anon}^{\sigma,\rho}u_s) \leq \frac{\m{M}_{s, t}(\m{T}_s^{\sigma,\rho} u - u)}{1 - \Lip(\sigma)Q^{1/2}}.\qedhere
  \end{equation*}
\end{proof}
To extend stability to longer timescales, we reduce it to closeness of the coarse-grained field.
\begin{prop}
  \label{prop:stab-coarse}
  Let $\sigma \in \Lip(\R^m; \m{H}_+^m)$, $Q_0 \in (0, \Qbar_\SPDE(\sigma) \wedge 1)$, $0 \leq T \leq \sfT_\rho(Q_0)$, and $\ell \in (2, 4]$.
  Suppose $v_t^i$ solve \cref{eq:SPDE} for $i \in \{1,2\}$.
  Then there exists $C = C(\sigma, Q_0, \ell) \geq 1$ such that
  \begin{equation}
    \label{eq:stab-coarse}
    \vvvert v_T^1 - v_T^2 \vvvert_2 \leq C \vvvert \m{G}_T v_0^1 - \m{G}_T v_0^2 \vvvert_2 + C \max_i \langle\vvvert v_0^i\vvvert_\ell\rangle \left[\frac{\sfL\bigl(\sfL(1/\rho)\bigr)}{\sfL(1/\rho)}\right]^{1/2}.
  \end{equation}
\end{prop}
\begin{proof}
  Let $\h v \coloneqq (v^1,v^2)$, which (given that the $\RR^m$ and $\RR^{2m}$-valued noises are coupled in the obvious way, with the first $m$ components of the latter noise equal to the former noise) solves \cref{eq:SPDE} with block-matrix nonlinearity
  \begin{equation*}
    \h\sigma_0(b_1,b_2) \coloneqq
    \begin{pmatrix}
      \sigma(b_1) & 0\\
      \sigma(b_2) & 0
    \end{pmatrix}.
  \end{equation*}
  Symmetrizing, let $\h \sigma \coloneqq (\h \sigma_0 \h \sigma_0^{\mathrm{T}})^{1/2}.$
  Then the solution of \cref{eq:SPDE} has the same law with nonlinearity $\h\sigma_0$ as it has with nonlinearity $\h\sigma$, so it suffices to study $\h v$ driven by the nonlinearity $\h \sigma$.
  
  \cref{prop:multigrowth,prop:multivarwellposedforsametime} ensure that $\beta_*(\h \sigma) = \beta_*(\sigma)$ and $\Qbar_\FBSDE(\h \sigma) = \Qbar_{\FBSDE}(\sigma)$.
  Now if we have ${\beta_*(\sigma)^{-2} \wedge \Qbar_\FBSDE(\sigma) > 1}$, \cref{thm:QSPDEgreater} implies $\Qbar_\SPDE(\h \sigma) > 1$.
  And if $\beta_*(\sigma)^{-2} \wedge \Qbar_\FBSDE(\sigma) \leq 1$, then \cref{prop:extendQSPDE-1} yields $\Qbar_\SPDE(\h \sigma) = \beta_*(\sigma)^{-2} \wedge \Qbar_\FBSDE(\sigma) = \Qbar_\SPDE(\sigma)$.
  In either case, ${Q_0 < \Qbar_\SPDE(\h \sigma)}$.

  For the sake of brevity, let $K \coloneqq \max_i \langle\vvvert v_0^i\vvvert_\ell\rangle$.
  Throughout, we allow constants $C < \infty$ to depend on $\sigma,Q_0$, and $\ell$ and vary from line to line.
  Fix $\beta \in (\beta_*(\sigma),Q_0^{-1/2})$ and let $\omega_{\h\sigma,Q_0,\beta,\ell}$ be the function provided by \cref{def:QSPDE}.
  We then define
  \begin{equation}
    \label{eq:T0-choice}
    T_0 \coloneqq \frac{T}{\sfL(1/\rho) \omega(\rho) + 1}.
  \end{equation}
  For all $x \in \R^2$ and $t \in [T_0, T]$, \cref{prop:laydowntheSDE} yields
  \begin{equation}
    \label{eq:pair-SDE}
    \E \abs{\h v_t(x) - \m{G}_{T - T_0} \h v_{T_0}(x) - \int_{T_0}^t J_{\h \sigma}\big(\sfS_\rho(T - s), \m{G}_{T - s} \h v_s(x)\big) \d \h B(s)}^2 \leq C K^2 \frac{\sfL\big(\sfL(1/\rho)\big)}{\sfL(1/\rho)},
  \end{equation}
  where the martingale $\h B$ has differential quadratic variation
  \begin{equation}
    \label{eq:BQV}
    \dn [\h B](t) = \frac{\mathrm{Id}_{2m}}{\sfL(1/\rho)}\cdot \frac{1}{T - t + \rho}.
  \end{equation}
  We adjust parameters so that \cref{eq:pair-SDE} runs from $s = 0$.
  As usual, we can assume $\ell$ is sufficiently close to $2$ that $\eta(\beta,\ell)Q_0 < 1$.
  Then by \cref{prop:momentbd}, $\Xi_{\ell,0,T}^{\sigma,\rho}(v_0^i) \leq C K$.
  Hence \cref{lem:VTcontinuity} yields
  \begin{equation*}
    \E \abs{\m{G}_{T - T_0} \h v_{T_0}(x) - \m{G}_{T} \h v_0(x)}^2 \leq C K^2 \sfL_\rho\left(\frac{T_0}{T + \rho - T_0}\right) \overset{\cref{eq:T0-choice}}\leq C K^2 \sfL_\rho\big(C \omega(\rho)\big) \overset{\cref{eq:omegaissmallenough}}\leq C K^2 \frac{\sfL\big(\sfL(1/\rho)\big)}{\sfL(1/\rho)}.
  \end{equation*}
  Now extend $\h B$ independently to the interval $[0, T_0]$ with the same differential quadratic variation formula \cref{eq:BQV}.
  Recalling that $Q_0 < \Qbar_\SPDE(\h \sigma) \leq \Qbar_\FBSDE(\h \sigma)$, $J_{\h \sigma}$ is uniformly Lipschitz in $b$.
  Thus \cref{prop:momentbd} and \cref{eq:BQV} yield
  \begin{equation*}
    \E \abs{\int_0^{T_0} J_{\h \sigma}\big(\sfS_\rho(T - s), \m{G}_{T - s} \h v_s(x)\big) \d \h B(s)}^2 \leq \frac{C K^2}{\sfL(1/\rho)} \log \frac{T + \rho}{T - T_0 + \rho} = C K^2 \sfL_\rho\left(\frac{T_0}{T + \rho - T_0}\right).
  \end{equation*}
  By the same reasoning, we obtain the same error term.
  Therefore \cref{eq:pair-SDE} yields
  \begin{equation}
    \label{eq:pair-clean}
    \E \abs{\h v_t(x) - \m{G}_T \h v_0(x) - \int_0^t J_{\h \sigma}\big(\sfS_\rho(T - s), \m{G}_{T - s} \h v_s(x)\big) \d \h B(s)}^2 \leq C K^2 \frac{\sfL\big(\sfL(1/\rho)\big)}{\sfL(1/\rho)}.
  \end{equation}
  
  Given $q \in [0, Q_0]$, let $t(q) \coloneqq \sfR_{0,T,\rho}(q) = T - \sfT_\rho(Q_0 - q)$ from \cref{eq:sfRdef} and $\Theta_q \coloneqq \m{G}_{T - t(q)} \h v_{t(q)}(x)$.
  Then we can alternatively write \cref{eq:pair-clean} as
  \begin{equation*}
    \Theta_q = \Theta_0 + \int_0^q J_{\h \sigma}(Q_0 - p, \Theta_p) \d B(p) + g_q
  \end{equation*}
  where $B$ is a standard Brownian motion in $\R^{2m}$ and
  \begin{equation*}
    \sup_{q \in [0, Q_0]} \E \abss{g_q}^2 \leq C K^2 \frac{\sfL\big(\sfL(1/\rho)\big)}{\sfL(1/\rho)}.
  \end{equation*}
  Now let $\Theta^\Delta$ solve the SDE
  \begin{equation*}
    \dn \Theta_q^\Delta = J_{\h \sigma}(Q_0 - q, \Theta_q^\Delta) \ds B(q), \quad \Theta_0^\Delta = (\Theta_0^1, \Theta_0^1).
  \end{equation*}
  On the diagonal $b_1 = b_2$, \cref{lem:diagonal} ensures that $J_{\h \sigma}(p, b_1, b_2)$ projects to the diagonal for all $p$.
  It follows that $\Theta_q^\Delta$ remains on the diagonal, so $\Theta_q^{\Delta,1} = \Theta_q^{\Delta,2}$ for all $q$.
  Taking the difference with $\Theta_q$, we can write
  \begin{equation*}
    \Theta_q - \Theta_q^\Delta = \Theta_0 - \Theta_0^\Delta + \int_0^q [J_{\h \sigma}(Q_0 - p, \Theta_p) - J_{\h \sigma}(Q_0 - p, \Theta_p^\Delta)] \d B(p) + g_p.
  \end{equation*}
  Now $J_{\h \sigma}$ is uniformly Lipschitz in $b$, so
  \begin{equation*}
    \E \absb{\Theta_q - \Theta_q^\Delta}^2 \leq C \E \absb{\Theta_0 - \Theta_0^\Delta}^2 + C \int_0^q \E \absb{\Theta_p - \Theta_p^\Delta}^2 \d p + C K^2 \frac{\sfL\big(\sfL(1/\rho)\big)}{\sfL(1/\rho)}.
  \end{equation*}
  By Gr\"onwall,
  \begin{equation*}
    \E \absb{\Theta_{Q_0} - \Theta_{Q_0}^\Delta}^2 \leq C \E \absb{\Theta_0 - \Theta_0^\Delta}^2 + C K^2 \frac{\sfL\big(\sfL(1/\rho)\big)}{\sfL(1/\rho)}.
  \end{equation*}
  In particular, the distance from $\Theta_{Q_0}$ to the diagonal is bounded as above, so
  \begin{equation*}
    \E \absb{\Theta_{Q_0}^1 - \Theta_{Q_0}^2}^2 \leq C \E \absb{\Theta_0 - \Theta_0^\Delta}^2 + C K^2 \frac{\sfL\big(\sfL(1/\rho)\big)}{\sfL(1/\rho)}.
  \end{equation*}
  In terms of $\h v$, this becomes
  \begin{equation*}
    \E \abs{v_T^1(x) - v_T^2(x)}^2 \leq C \E \abs{\m{G}_T v_0^1(x) - \m{G}_T v_0^2(x)}^2 + C K^2 \frac{\sfL\big(\sfL(1/\rho)\big)}{\sfL(1/\rho)}.
  \end{equation*}
  Taking the supremum over $x \in \R^2$, the proposition follows.
\end{proof}
We now show that we can extend stability so long as the decoupling function remains Lipschitz.
\begin{prop}
  \label{prop:stab-extend}
  Let $\sigma \in \Lip(\R^m;\m{H}_+^m)$ and $Q_0 \in \big[0, \bar{Q}_\stab(\sigma) \wedge 1\bigr)$.
  Then
  \begin{equation*}
    \bar{Q}_\stab(\sigma) \geq \big[Q_0 + \Lip\bigl(J_\sigma(Q_0, \anon)\bigr)^{-2}\big] \wedge \bar{Q}_\SPDE(\sigma).
  \end{equation*}
\end{prop}
As a consequence, we obtain the following theorem.
\begin{thm}
  \label{thm:stable}
  Let $M\in(0,\infty)$, $\beta\in(0,1)$, and $\sigma\in\quadset(M,\beta)$.
  If $\Qbar_{\FBSDE}(\sigma)>1$, then $\Qbar_{\stab}(\sigma)>1$.
\end{thm}
The proof mimics that of \cref{thm:QSPDEgreater} above.
\begin{proof}[Proof of \cref{prop:stab-extend}]
  Fix
  \begin{equation*}
    Q_1 \in \bigl(Q_0, [Q_0 + \Lip\bigl(J_\sigma(Q_0, \anon)\bigr)^{-2}] \wedge \bar{Q}_\SPDE(\sigma)\bigr).
  \end{equation*}
  Because $Q_1 < \Qbar_\SPDE(\sigma) \leq \beta_*(\sigma)^{-2}$, we can take
  \begin{equation}
    \label{eq:beta-cond}
    \beta \in \big(\beta_*(\sigma) \vee [Q_0 + \Lip\bigl(J_\sigma(Q_0, \anon)\bigr)^{-2}]^{-1/2}, Q_1^{-1/2}\big).
  \end{equation}
  Then fix $\ell \in (2, 4]$.
  Let $u$ be a random field adapted to $(\s{F}_t)$ and define
  \begin{equation*}
    K \coloneqq \sup_{s \in [0, \sfT_\rho(\beta^{-2})]} \vvvert u_s \vvvert_\ell \And \delta \coloneqq \m{M}_{0,\sfT_\rho(\beta^{-2})}(\m{T}^{\sigma,\rho} u - u).
  \end{equation*}
  Then for all $t \in [0, \sfT_\rho(\beta^{-2})]$, we have
  \begin{equation}
    \label{eq:residue}
    u_t(x) = \m{G}_t u_0(x) + \gamma \int_0^t \m{G}_{t - r + \rho}[\sigma(u_r) \ds W_r](x) + g_{t}(x)
  \end{equation}
  for an adapted random field $g_{t}$ satisfying
  \begin{equation*}
    \sup_{x \in \R^2} \E \abss{g_{t}(x)}^2 \leq \delta^2.
  \end{equation*}
  We now renormalize.
  Let $\tau \coloneqq \sfT_\rho(Q_0)$.
  Applying $\m{G}_\tau$ to \cref{eq:residue}, we find
  \begin{equation}
    \label{eq:residue-renorm}
    \m{G}_\tau u_t(x) = \m{G}_t \m{G}_\tau u_0(x) + \gamma \int_0^t \m{G}_{t + \tau - r + \rho}[\sigma(u_r) \ds W_r](x) + \m{G}_\tau g_{t}(x),
  \end{equation}
  and Jensen yields
  \begin{equation*}
    \sup_{x \in \R^2} \E \abss{\m{G}_\tau g_{t}(x)}^2 \leq \delta^2.
  \end{equation*}
  Now let $\omega = \omega_{\sigma,Q_0,\beta,\ell}$ be the function provided by \cref{def:QSPDE}.
  We apply \cref{prop:coupling} with $\eta_r \mapsfrom t + \tau - r + \rho$ and $\zeta_r \mapsfrom [\tau/\omega(\rho)]^{1/2}$ to obtain a new space-time white noise $(\dn \widetilde{W}_r)$ such that
  \begin{equation}
    \label{eq:new-noise}
    \hspace*{-2mm}\E \abs{\int_0^t \m{G}_{t + \tau - r + \rho}\big[\sigma(u_r) \ds W_r - \m{S}_{(\tau/\omega)^{1/2}}[\sigma^2 \circ u_r]^{1/2} \ds \widetilde{W}_r](x)}^2
    \leq CK^2 \int_0^t \frac{\tau \d r}{(t + \tau - r + \rho)^2} \leq CK^2,
  \end{equation}
  for an absolute constant $C = C(\sigma,Q_0,Q_1,\beta,\ell) < \infty$ that may change from instance to instance.

  We now turn to the quantity $\m{S}_{(\tau/\omega)^{1/2}}[\sigma^2 \circ u_r]$.
  Let
  \begin{equation*}
    v_r^{(\tau)} \coloneqq \m{V}_{r - \tau,r}^{\sigma,\rho} u_{r - \tau}.
  \end{equation*}
  Because $Q_0 < \bar{Q}_\stab$ and $\tau < \sfT_\rho(Q_0)$, for all $r \in [\tau, \sfT_\rho(\beta^{-2})]$, \cref{eq:stable} yields
  \begin{equation}
    \label{eq:stable-small}
    \vvvert v_r^{(\tau)} - u_r\vvvert_2^2 \leq C \delta^2 + C K^2 \frac{\sfL\big(\sfL(1/\rho)\big)}{\sfL(1/\rho)}.
  \end{equation}
  In particular,
  \begin{align*}
    \sup_{x \in \R^2} \E \Big|\m{S}_{(\tau/\omega)^{1/2}}[\sigma^2 \circ v_r^{(\tau)}]^{1/2}(x) - &\m{S}_{(\tau/\omega)^{1/2}}[\sigma^2 \circ u_r]^{1/2}(x)\Big|_\Frob^2\\
    &\!\overset{\cref{eq:matrixvaluedreversetriangleinequality}}{\leq}\! \sup_{x \in \R^2} \E \big|\sigma \circ v_r^{(\tau)}(x) - \sigma \circ u_r(x)\big|_\Frob^2
    \overset{\cref{eq:stable-small}}{\leq} C \delta^2 + C K^2 \frac{\sfL\big(\sfL(1/\rho)\big)}{\sfL(1/\rho)}.
  \end{align*}
  Next, define $q \coloneqq \sfS_\rho\bigl(\tau/\omega(\rho)\bigr) < Q_0$.
  Because $q < Q_0 < \bar{Q}_\stab \leq \bar{Q}_\SPDE$, the key approximation \cref{eq:QSPDEapprox-3} ensures that
  \begin{equation*}
    \sup_{x \in \R^2} \E \Big|\m{S}_{(\tau/\omega)^{1/2}}[\sigma^2 \circ v_r^{(\tau)}]^{1/2}(x) - J_\sigma\bigl(q, \m{G}_{\tau/\omega} v_r^{(\tau)}(x)\bigr)\Big|_\Frob^2
    \leq C K^2 \frac{\sfL\bigl(\sfL(1/\rho)\bigr)}{\sfL(1/\rho)}.
  \end{equation*}
  Recalling that $\tau = \sfT_\rho(Q_0)$, one can readily check that $\abs{q - Q_0} \leq C \sfL\bigl(\sfL(1/\rho)\bigr)/\sfL(1/\rho)$.
  Following the proof of \cref{lem:applyIH}, we can thus adjust the arguments of $J_\sigma$ to obtain
  \begin{equation*}
    \sup_{x \in \R^2} \E \Big|\m{S}_{\tau^{1/2}}[\sigma^2 \circ v_r^{(\tau)}]^{1/2}(x) - J_\sigma\bigl(Q_0, \m{G}_\tau v_r^{(\tau)}(x)\bigr)\Big|_\Frob^2
    \leq C K^2\frac{\sfL\bigl(\sfL(1/\rho)\bigr)}{\sfL(1/\rho)}.
  \end{equation*}
  This adjustment rests on \cref{prop:Jsigmatimereg-1} and \cref{prop:continuitybd}; we omit the repeated details.

  Next, $Q_0 < \Qbar_\SPDE(\sigma) \leq \Qbar_\FBSDE(\sigma)$ and \cref{eq:stable-small} yield
  \begin{equation*}
    \sup_{x \in \R^2} \E \Big|J_\sigma\bigl(Q_0, \m{G}_\tau v_r^{(\tau)}(x)\bigr) - J_\sigma\bigl(Q_0, \m{G}_\tau u_r(x)\bigr)\Big|_\Frob^2 \leq C \delta^2 + C K^2 \frac{\sfL\big(\sfL(1/\rho)\big)}{\sfL(1/\rho)}.
  \end{equation*}
  Combining the above observations, the triangle inequality yields
  \begin{equation*}
    \sup_{x \in \R^2} \E \Big|\m{S}_{\tau^{1/2}}[\sigma^2 \circ u_r]^{1/2}(x) - J_\sigma\bigl(Q_0, \m{G}_\tau u_r(x)\bigr)\Big|_\Frob^2 \leq C \delta^2 + C K^2 \frac{\sfL\bigl(\sfL(1/\rho)\bigr)}{\sfL(1/\rho)}.
  \end{equation*}
  It follows that
  \begin{align}
    \E \bigg|\gamma \int_{\tau}^t &\m{G}_{t + \tau - r + \rho}\bigl(\big[\m{S}_{\tau^{1/2}}[\sigma^2 \circ u_r]^{1/2} - J_\sigma\bigl(Q_0, \m{G}_\tau u_r(x)\bigr)\big] \ds \widetilde{W}_r\bigr)(x)\bigg|^2\nonumber\\
    &= \E \bigg| \gamma \int_{\tau}^t \!\!\!\int_{\R^2} G_{t + \tau - r + \rho}(x - y)\m{G}_{t + \tau - r + \rho}\bigl(\big[\m{S}_{\tau^{1/2}}[\sigma^2 \circ u_r]^{1/2} - J_\sigma\bigl(Q_0, \m{G}_\tau u_r(x)\bigr)\big] \ds \widetilde{W}_r\bigr)(y)\bigg|^2\nonumber\\
    &= \gamma^2 \int_{\tau}^t \!\!\!\int_{\R^2} G_{t + \tau - r + \rho}^2(x - y) \E \Big|\m{S}_{\tau^{1/2}}[\sigma^2 \circ u_r]^{1/2}(y) - J_\sigma\bigl(Q_0, \m{G}_\tau u_r(y)\bigr)\Big|_\Frob^2 \d y \ds r\label{eq:stable-Ito}\\
    &\leq C \delta^2 + C K^2 \frac{\sfL\bigl(\sfL(1/\rho)\bigr)}{\sfL(1/\rho)}.\nonumber
  \end{align}
  To treat the early times $r \in [0, \tau]$, we recall that $\abs{J_\sigma(Q_0, b)} \leq C \langle b \rangle$ by \cref{prop:Jsigmaub}.
  Hence a variation on \cref{eq:stable-Ito} yields
  \begin{equation*}
    \E \bigg|\gamma \int_0^{\tau} \m{G}_{t + \tau - r + \rho}\bigl(\big[\m{S}_{\tau^{1/2}}[\sigma^2 \circ u_r]^{1/2} - J_\sigma\bigl(Q_0, \m{G}_\tau u_r(x)\bigr)\big] \ds \widetilde{W}_r\bigr)(x)\bigg|^2\leq C K^2 \frac{\sfL\bigl(\omega(\rho)\bigr)}{\sfL(1/\rho)} \overset{\cref{eq:omegaissmallenough}}{\leq} C K^2 \frac{\sfL\bigl(\sfL(1/\rho)\bigr)}{\sfL(1/\rho)}.
  \end{equation*}
  Combining these integral estimates, we find
  \begin{equation*}
    \E \bigg|\gamma \int_0^t \m{G}_{t + \tau - r + \rho}\bigl(\big[\m{S}_{\tau^{1/2}}[\sigma^2 \circ u_r]^{1/2} - J_\sigma\bigl(Q_0, \m{G}_\tau u_r(x)\bigr)\big] \ds \widetilde{W}_r\bigr)(x)\bigg|^2
    \leq C \delta^2 + C K^2 \frac{\sfL\bigl(\sfL(1/\rho)\bigr)}{\sfL(1/\rho)}.
  \end{equation*}
  Recalling \cref{eq:new-noise}, this implies that
  \begin{equation*}
    \E \bigg|\gamma \int_0^t \m{G}_{t + \tau - r + \rho}\big[\sigma(u_r) \ds W_r - J_\sigma\bigl(Q_0, \m{G}_\tau u_r(x)\bigr) \ds \widetilde{W}_r\big](x)\bigg|^2 \leq C \delta^2 + C K^2 \frac{\sfL\bigl(\sfL(1/\rho)\bigr)}{\sfL(1/\rho)}.
  \end{equation*}
  Using this in \cref{eq:residue-renorm}, we thus obtain
  \begin{equation*}
    \E \bigg|\m{G}_\tau u_t(x) - \m{G}_t \m{G}_\tau u_0(x) - \gamma \int_0^t \m{G}_{t + \tau - r + \rho}[J_\sigma\bigl(Q_0, \m{G}_\tau u_r(x)\bigr) \ds \widetilde{W}_r](x)\bigg|^2 \leq C \delta^2 + C K^2 \frac{\sfL\bigl(\sfL(1/\rho)\bigr)}{\sfL(1/\rho)}.
  \end{equation*}
  Let $\ti{\m{T}}$ denote the analogue of the mild operator $\m{T}_0^{\sigma,\rho}$ with $\rho \mapsfrom \ti\rho \coloneqq \tau + \rho = \sfT_\rho(Q_0) + \rho$, $W \mapsfrom \widetilde{W}$, and $\sigma \mapsfrom (1 - Q_0)^{1/2}J_\sigma(Q_0, \anon)$ as in \cref{eq:sigmatildedef-1}.
  This is the mild operator corresponding to the $Q_0$-renormalized equation.
  Then we have shown that
  \begin{equation}
    \label{eq:renorm-error}
    \E \absb{\tilde{\m{T}} \m{G}_\tau u - u}^2 \leq C \delta^2 + C K^2 \frac{\sfL\bigl(\sfL(1/\rho)\bigr)}{\sfL(1/\rho)}.
  \end{equation}
  We intend to apply \cref{defn:stable} to the renormalized equation.
  To this end, choose
  \begin{equation*}
    \ti\beta \in \Bigl(\Lip(\ti \sigma),\Bigr(\frac{\beta^{-2} - Q_0}{1 - Q_0}\Bigl)^{-1/2}\Bigr),
  \end{equation*}
  which is possible due to \cref{eq:beta-cond}.
  In particular, $\ti \beta > \Lip(\ti \sigma) \geq \beta_*(\ti\sigma)$.
  Next choose
  \begin{equation*}
    \ti{Q} \in \Bigl(\frac{Q_1 - Q_0}{1 - Q_0}, \ti \beta^{-2}\Bigr),
  \end{equation*}
  which is possible because $\beta < Q_1^{-1/2}$ by \cref{eq:beta-cond}.
  Now $\ti Q < \ti \beta^{-2} < \Lip(\ti \sigma)^{-2}$, so by \cref{prop:stab-Lip}, $\ti Q < \Qbar_\stab(\ti \sigma)$.
  Since $\ti \beta \in (\beta_*(\ti \sigma), \ti{Q}^{-1/2})$, we can apply \cref{defn:stable} with data $(\ti \sigma,\ti Q,\ti \beta,\ell)$.
  By \cref{eq:renorm-error},
  \begin{equation}
    \label{eq:renorm-stable}
    \m{M}_{s,t}(\m{G}_\tau u - \widetilde{\m{V}}_{s,\anon}^{\ti \sigma,\ti \rho} \m{G}_\tau u_s)^2 \leq C \delta^2 + C K^2 \frac{\sfL\bigl(\sfL(1/\rho)\bigr)}{\sfL(1/\rho)}
  \end{equation}
  whenever $0 \leq s \leq t \leq \sfT_{\ti \rho}(\ti \beta^{-2})$ and $t - s \leq \sfT_{\ti \rho}(\ti Q)$.
  Here we let $\widetilde{\m{V}}$ denote the evolution operator corresponding to \cref{eq:SPDE} with noise $\dn \widetilde{W}$ in place of $\dn W$.
  Strictly speaking, the constant $C$ now additionally depends on $(\ti \sigma,\ti Q,\ti \beta)$.
  We can reduce this to the earlier dependence by observing that $\ti \sigma$ is a function of $(\sigma, Q_0)$ and we can choose $\ti \beta$ and $\ti Q$ to be functions of $(\beta,Q_0,Q_1)$.

  In the following, we assume that $0 \leq s \leq t \leq \sfT_\rho(\beta^{-2})$ and $t - s \leq \sfT_\rho(Q_1)$.
  By \cref{eq:Sapprox} and \cref{eq:Tapprox},
  \begin{equation*}
    \sfS_{\ti{\rho}}\bigl(\sfT_\rho(Q_1)\bigr) \to \frac{Q_1 - Q_0}{1 - Q_0} < \ti Q \And \sfS_{\ti{\rho}}\bigl(\sfT_\rho(\beta^{-2})\bigr) \to \frac{\beta^{-2} - Q_0}{1 - Q_0} < \ti \beta^{-2} \quad \text{as } \rho \to 0.
  \end{equation*}
  (Here we use $Q_0 < 1$ to ensure that $\ti \rho \ll 1$, so \cref{eq:Sapprox} is valid.)
  It follows that $\sfT_\rho(Q_1) < \sfT_{\ti \rho}(\ti Q)$ and $\sfT_\rho(\beta^{-2}) < \sfT_{\ti \rho}(\ti \beta^{-2})$ once $\rho$ is sufficiently small, so \cref{eq:renorm-stable} holds for the times $(s,t)$ we consider.

  Now let $v_r \coloneqq \m{V}_{s,r}^{\sigma,\rho}u_s$ for $r \in [s, t]$.
  We can repeat the above analysis with $v$ in place of $u$ to find
  \begin{equation*}
    \m{M}_{s,t}(\m{G}_\tau v - \widetilde{\m{V}}_{s,\anon}^{\ti \sigma,\ti \rho} \m{G}_\tau u_s)^2 \leq C K^2 \frac{\sfL\bigl(\sfL(1/\rho)\bigr)}{\sfL(1/\rho)}.
  \end{equation*}
  (This is simply a restatement of \cref{eq:swapthings} from the proof of the key approximation.)
  Together with \cref{eq:renorm-stable}, we see that
  \begin{equation}
    \label{eq:large-scale}
    \m{M}_{s,t}(\m{G}_\tau u - \m{G}_\tau v)^2 \leq C \delta^2 + C K^2 \frac{\sfL\bigl(\sfL(1/\rho)\bigr)}{\sfL(1/\rho)}.
  \end{equation}
  That is, $u$ and $v$ are close at scale $\tau$.

  Given $r \in [s + \tau, t]$, recall $v_r^{(\tau)} = \m{V}_{r - \tau, r}^{\sigma,\rho}u_{r - \tau}$.
  For $h \in [0, \tau]$, take $v_h^1 \coloneqq \m{V}_{r - \tau, r - \tau + h}^{\sigma,\rho}u_{r - \tau}$ and $v_h^2 \coloneqq \m{V}_{r - \tau, r - \tau + h}^{\sigma,\rho}v_{r - \tau}$ in \cref{prop:stab-coarse}.
  Then \cref{eq:stab-coarse,eq:large-scale} imply that
   \begin{equation*}
    \vvvert v_r^{(\tau)} - v_r\vvvert^2 \leq C \delta^2 + C K^2 \frac{\sfL\bigl(\sfL(1/\rho)\bigr)}{\sfL(1/\rho)}.
  \end{equation*}
  Using \cref{eq:stable-small}, we further have
  \begin{equation*}
    \vvvert u_r - v_r \vvvert_2^2 \leq C \delta^2 + C K^2 \frac{\sfL\bigl(\sfL(1/\rho)\bigr)}{\sfL(1/\rho)}
  \end{equation*}
  for all $r \in [s + \tau, t]$.
  When $r \in [s, s + \tau]$, we can use \cref{eq:stable} directly to see that $u_r$ and $v_r$ are similarly close.
  Therefore
  \begin{equation*}
    \m{M}_{s,t}(u - \m{V}_{s,\anon}^{\sigma,\rho} u_s)^2 \leq C \delta^2 + C K^2 \frac{\sfL\bigl(\sfL(1/\rho)\bigr)}{\sfL(1/\rho)}.
  \end{equation*}
  Then we have satisfied \cref{eq:stable} with a constant $C(\sigma,Q_0,Q_1,\beta,\ell)$.
  Finally, we can choose $Q_0$ as a function of $\sigma$ and $Q_1$ much as in \cref{item:thegoal} of the proof of \cref{prop:extendQSPDE-1}.
  Omitting the repeated details, we see that $C$ has the desired dependence on $(\sigma,Q_1,\beta,\ell)$.
  Therefore $\Qbar_\stab(\sigma) \geq Q_1$.
\end{proof}

\subsection{Pre-smoothed noise}

Recall the solution $u_t(x)$ of the problem \cref{eq:u-SPDE} with pre-smoothed noise.
We now show that $u$ approximately solves the problem \cref{eq:SPDE} with post-smoothed noise.

To do so, we first show that $u$ satisfies the same moment bounds as $v$.
Let $\clU$ denote the random propagator for the pre-smoothed problem \cref{eq:u-SPDE}, in analogy with $\m{V}$ (defined after \cref{eq:triplenormdef}) for \cref{eq:SPDE}.
\begin{lem}
  \label{lem:u-moment}
  Fix $\ell \geq 2$, $\beta \in (0, \infty)$, $s\in \R$, $u_s \in \s{X}_s^\ell$, and $\sigma \in \Sigma(M,\beta)$.
  For all $T \in \big[s, s + \sfT_\rho(\eta(\beta,\ell)^{-2})\bigr)$,
  \begin{equation*}
    \vvvert\clU_{s,T}^{\sigma,\rho} u_s\vvvert_\ell^\ell \leq \frac{\vvvert u_{s}\vvvert_{\ell}^{\ell}+2^{1-2/\ell}(\ell-1)\sfS_{\rho}(T-s)M^{\ell/2}}{1-\eta(\beta,\ell)^{2}\sfS_{\rho}(T-s)}.
  \end{equation*}
\end{lem}
\begin{proof}
  This recapitulates \cref{prop:momentbd}, and the proof is very similar.
  Here, we demonstrate how to treat the quadratic variation of $\m{G}_{T - t} u_t$, which is the only novelty.
  Let $U_t \coloneqq \m{G}_{T - t} u_t(x)$, which is a martingale.
  Using \cref{eq:u-SPDE} and properties of white noise, we can compute
  \begin{equation}
    \label{eq:u-QV}
    \der{[U]_r}{t} = \gamma^2 \int_{(\R^2)^2} G_{T - r}(x - y_1)G_{T - r}(x - y_2) (\sigma \circ u_r)(y_1) (\sigma \circ u_r)(y_2) G_{2 \rho}(y_1 - y_2) \ds y_1 \ds y_2.
  \end{equation}
  Following the proof of \cref{prop:momentbd}, we must control the nuclear (or trace) norm of this quadratic variation.
  Using the Schatten--Cauchy--Schwarz inequality and Young's inequality, we can write
  \begin{equation*}
    \tr (\sigma \circ u_r)(y_1) (\sigma \circ u_r)(y_2) \leq (\abs{\sigma}_{\Frob} \circ u_r)(y_1) (\abs{\sigma}_{\Frob} \circ u_r)(y_2) \leq \frac{1}{2}\abs{\sigma}_{\Frob}^2 \circ u_r(y_1) + \frac{1}{2}\abs{\sigma}_{\Frob}^2 \circ u_r(y_2).
  \end{equation*}
  The integrand in \cref{eq:u-QV} is symmetric in $y_1$ and $y_2$, so we find
  \begin{align}
    \tr \der{[U]_r}{t} &\leq \gamma^2 \int_{\R^2} G_{T - r}(x - y_1) \abs{\sigma}_{\Frob}^2 \circ u_s(y_1) \int_{\R^2} G_{T - r}(x - y_2) G_{2\rho}(y_2 - y_1) \d y_2 \ds y_1\nonumber\\
                       &\leq \gamma^2 \int_{\R^2} G_{T - r}(x - y) G_{T - r + 2\rho}(x - y) \abs{\sigma}_{\Frob}^2 \circ u_s(y) \ds y\nonumber\\
                       &\leq \frac{1}{\sfL(1/\rho)(T - r + \rho)} \m{G}_{\frac{(T - r)(T - r + 2 \rho)}{2(T - r + \rho)}} \abs{\sigma}_\Frob^2 \circ u_r(x).\label{eq:u-QV-bd}
  \end{align}
  This allows one to apply Jensen as in \cref{eq:Jensenparty-pre}, and the remainder of the proof of \cref{prop:momentbd} is unchanged.
\end{proof}
Let $\h\Xi$ denote the analogue of $\Xi$ from \cref{eq:momentbdnotation} with the evolution operator $\clU$ in place of $\clV$.
\begin{lem}
  \label{lem:u-continuity}
  \cref{prop:continuitybd} holds for $u$ as well.
\end{lem}
\begin{proof}
  The proof is almost identical.
  To compare at different times, we integrate the quadratic variation of $U_t = \m{G}_{T - t}u_t$ and use \cref{eq:u-QV-bd}.
  We omit the details.
\end{proof}
We can now show that the solution $u$ of the pre-smoothed problem \cref{eq:u-SPDE} approximately solves the post-smoothed problem \cref{eq:SPDE} in a mild sense.
Recall the operator $\clT$ from \cref{eq:clTdef} and the norm $\clM$ from \cref{eq:scrMstdef}.
\begin{prop}
  \label{prop:approx-soln}
  There is an absolute constant $C < \infty$ such that for all $\sigma \in \Lip(\R^m; \clH_+^m)$, $t > s$, and $u_s \in \scrX_s^2$, we have
  \begin{equation}
    \label{eq:approx-soln}
    \clM_{s,t}(\clT_{s}^{\sigma,\rho}u_s - u) \leq C \Lip(\sigma) \h\Xi_{2;s,t}^{\sigma,\rho}(u_s) \sfS_\rho(t - s + \rho)^{1/2} \left[\frac{\sfL\bigl(\sfL(1/\rho)\bigr)}{\sfL(1/\rho)}\right]^{1/2}.
  \end{equation}
\end{prop}
\begin{proof}
  We recall that
  \begin{equation}
    \label{eq:mild-post}
    (\m{T}_s^{\sigma,\rho} u_s)_r = \m{G}_r u_s + \gamma \int_s^r \m{G}_{r - \tau + \rho}(\sigma \circ u_\tau \d W_\tau)
  \end{equation}
  while the mild formulation of \cref{eq:u-SPDE} reads
  \begin{equation}
    \label{eq:mild-pre}
    u_r = \m{G}_r u_s + \gamma \int_s^r \m{G}_{r - \tau}(\sigma \circ u_\tau \, \m{G}_\rho \dn W_\tau).
  \end{equation}
  These differ only in the placement of the operator $\m{G}_\rho$, and we wish to show that the effect is minor.
  Taking the difference and applying the It\^o isometry, we find
  \begin{equation}
    \label{eq:u-diff-L2-prelim}
    \begin{aligned}
      \E |(\m{T}_s^{\sigma,\rho} u_s)_r(x) - u_r(x)&|^2 = \gamma^2 \int_s^r \!\!\!\int_{(\R^2)^2} G_{r-\tau}(x - y_1) G_{r - \tau}(x - y_2) \int_{\R^2} G_\rho(y_1 - z) G_\rho(y_2 - z)\\
                                                             &\cdot \E \tr \bigl([\sigma \circ u_\tau(z) - \sigma \circ u_\tau(y_1)] [\sigma \circ u_\tau(z) - \sigma \circ u_\tau(y_2)]\bigr) \d z \ds y_1 \ds y_2 \ds \tau.      
    \end{aligned}
  \end{equation}
  In the following, we abbreviate $\h\Xi_{2;s,t}^{\sigma,\rho}(u_s)$ as $\h \Xi$.
  Combining Cauchy--Schwarz, Young, and \cref{lem:u-continuity}, \cref{eq:continuitybd} yields
  \begin{equation}
    \label{eq:three-parts}
    \E \tr \prod_{i=1,2} [\sigma \circ u_\tau(z) - \sigma \circ u_\tau(y_i)] \leq C [\Lip(\sigma)]^2 \h \Xi^2 \sum_i \left[\sfL_\rho\left(\frac{\abs{z - y_i}^2}{2\rho}\right) + \frac{\abs{z - y_i}^2}{2(\tau - s)} + \frac{\sfL\bigl(\sfL(1/\rho)\bigr)}{\sfL(1/\rho)}\right].
  \end{equation}
  The terms in the sum contribute three terms to the right of \cref{eq:u-diff-L2-prelim}, which we label $E_1$, $E_2$, and $E_3$.

  To treat $E_1$, we observe that
  \begin{equation*}
    \sfL_\rho\left(\frac{\abs{z - y_i}^2}{2\rho}\right) \leq C \frac{\sfL\bigl(\sfL(1/\rho)\bigr)}{\sfL(1/\rho)} + C\sfL(1 + \abs{z - y_i}^2) \tbf{1}_{\abs{z - y_i}^2 \geq \rho \sfL(1/\rho)^2}.
  \end{equation*}
  Multiplying by $G_\rho(z - y_i)$, we can write
  \begin{equation*}
    G_\rho(z - y_i) \sfL_\rho\left(\frac{\abs{z - y_i}^2}{2\rho}\right) \leq C \frac{\sfL\bigl(\sfL(1/\rho)\bigr)}{\sfL(1/\rho)} G_\rho(z - y_i) + C \rho^{100} G_{\rho/2}(z - y_i).
  \end{equation*}
  It is straightforward to check that the second term on the right contributes negligibly.
  The first term has the same form as $E_3$, so we absorb it there and treat $E_3$ below.

  For $E_2$, we apply \cref{eq:three-parts} when $\tau \geq s + \rho^{3/2}$ and control \cref{eq:u-diff-L2-prelim} differently when $\tau \in [s, s + \rho^{3/2}]$.
  Let $E_2 = E_{2,1} + E_{2,2}$ denote these quantities.
  To treat $E_{21}$, we observe that
  \begin{equation*}
    \abs{z - y_i}^2 G_\rho(z - y_i) \leq C \rho^2 G_{\rho/2}(z - y_i).
  \end{equation*}
  We can thus write
  \begin{equation*}
    E_{2,1} \leq C [\Lip(\sigma)]^2 \h \Xi^2 \rho^2 \gamma^2 \int_{s + \rho^{3/2}}^r \frac{\dn \tau}{\tau - s} \int_{(\R^2)^2} \prod_i G_{r - \tau}(x - y_i) \int_{\R^2} \prod_i G_{\rho/2}(z - y_i) \d z \ds y_1 \ds y_2.
  \end{equation*}
  Using the identity
  \begin{equation*}
    \prod_i G_{\rho/2}(z - y_i) = G_{\rho}(y_1 - y_2) G_{\rho/4}(z - \bar y)
  \end{equation*}
  with $\bar y \coloneqq \frac{y_1 + y_2}{2}$ and similarly for $\prod_i G_{r - \tau}(x - y_i)$, we integrate and find
  \begin{align*}
    E_{2,1} &\leq C [\Lip(\sigma)]^2 \h \Xi^2 \rho^2 \gamma^2 \int_{s + \rho^{3/2}}^r \frac{\dn \tau}{\tau - s} \int_{\R^2} G_{2(r - \tau)}(y) G_\rho(y) \d y\\
    &\leq C [\Lip(\sigma)]^2 \h \Xi^2 \int_{s + \rho^{3/2}}^r \frac{\rho^2 \gamma^2 \d \tau}{(\tau - s)(r - \tau + \rho)}
    \leq C [\Lip(\sigma)]^2 \h \Xi^2 \sfS_\rho(r - s) \rho^{1/2},
  \end{align*}
  which is sufficiently small.
  Next, we extract the very early portion of \cref{eq:u-diff-L2-prelim} to find
  \begin{align*}
    E_{2,2} &\leq C [\Lip(\sigma)]^2 \h \Xi^2 \gamma^2 \int_s^{(s + \rho^{3/2}) \wedge r} \!\!\! \int_{(\R^2)^2} \prod_i G_{r - \tau}(x - y_i) \int_{\R^2} \prod_i G_\rho(z - y_i) \d z \ds y_1 \ds y_2 \ds \tau\\
            &\leq C [\Lip(\sigma)]^2 \h \Xi^2 \gamma^2 \int_s^{(s+\rho^{3/2}) \wedge r} \frac{\dn \tau}{r - \tau + \rho} \leq C [\Lip(\sigma)]^2 \h \Xi^2 \min\{\sfS_\rho(r - s), \rho^{1/2}\}.
  \end{align*}
  We crudely control this by $\sfS_\rho(r - s + \rho) \sfL\bigl(\sfL(1/\rho)\bigr)/\sfL(1/\rho)$ in \cref{eq:approx-soln}.
  So \cref{eq:approx-soln} is lax when $t \leq s + \rho$, but this will not matter in our applications.
  
  Finally, we come to $E_3$.
  Mimicking the above estimates, we find
  \begin{equation*}
    E_3 \leq  C [\Lip(\sigma)]^2 \h \Xi^2 \frac{\sfL\bigl(\sfL(1/\rho)\bigr)}{\sfL(1/\rho)} \gamma^2 \int_s^r \frac{\dn \tau}{r - \tau + \rho} = C [\Lip(\sigma)]^2 \h \Xi^2 \sfS_\rho(r - s) \frac{\sfL\bigl(\sfL(1/\rho)\bigr)}{\sfL(1/\rho)}.
  \end{equation*}
  Taking the supremum over $x \in \R^2$ and $r \in [s,t]$, \cref{eq:approx-soln} follows.
\end{proof}
Modifying this proof, one can show that an approximate mild solution of \cref{eq:u-SPDE} is an approximate mild solution of \cref{eq:SPDE} and vice versa, with suitable quantitative bounds.
So there is a certain symmetry between the formulations \cref{eq:u-SPDE} and \cref{eq:SPDE}.

Since $u$ is an approximate mild solution of \cref{eq:SPDE}, \cref{thm:stable} implies that $u$ remains close to $v$.
\begin{thm}
  Fix $\bar{T} > 0$, $\beta \in (0, 1)$, $M \in (0, \infty)$, $\ell \in (2, 4]$, and $\sigma \in \Sigma(M,\beta)$.
  Suppose $\Qbar_\FBSDE(\sigma) > 1$.
  Then for all $u_0 \in \s{X}_0^\ell$, there exists a constant $C(\sigma,M,\beta,\ell,\bar{T}) < \infty$ such that
  \begin{equation}
    \label{eq:pre-smoothed}
    \m{M}_{0,\bar{T}}(\m{U}_{0,\anon}^{\sigma,\rho}u_0 - \m{V}_{0,\anon}^{\sigma,\rho}u_0) \leq C \langle\vvvert u_0\vvvert_\ell\rangle \frac{\sfL\bigl(\sfL(1/\rho)\bigr)}{\sfL(1/\rho)}.
  \end{equation}
\end{thm}
\begin{proof}
  Combining \cref{lem:u-moment} and \cref{prop:approx-soln}, we find
  \begin{equation*}
    \m{M}_{0,\sfT_\rho(\beta^{-2})}(\m{T}_0^{\sigma,\rho}u_0 - \m{U}_{0,\anon}^{\sigma,\rho}u_0) \leq C \langle\vvvert u_0\vvvert_\ell\rangle \frac{\sfL\bigl(\sfL(1/\rho)\bigr)}{\sfL(1/\rho)}
  \end{equation*}
  for a constant $C(\sigma,M,\beta,\ell) < \infty$.
  \cref{thm:stable} ensures that $\Qbar_\stab(\sigma) > 1$.
  Hence \cref{eq:pre-smoothed} follows from \cref{defn:stable} and \cref{lem:u-moment}.
\end{proof}
\begin{cor}
  \cref{thm:thebigtheorem} holds for $u$ as well as $v$, i.e.\ for solutions of \cref{eq:u-SPDE} as well as \cref{eq:SPDE}.
\end{cor}
\noindent
These results immediately imply \cref{cor:pre-post}.

\appendix

\section{Auxiliary results}\label{appendix:auxiliary}

In this first appendix we establish some basic auxiliary results that
are used throughout the paper.

\subsection{The two-dimensional Gaussian density}

We collect some basic results on the two-dimensional Gaussian density
\[
  G_{\tau}(x)=\frac{1}{2\pi\tau}\exp\left(-\frac{|x|^{2}}{2\tau}\right) ,\qquad\tau>0,\,x\in\RR^{2}.
\]
We note that
\begin{equation}
  G_{\tau}(x)^{2}=\frac{1}{4\pi\tau}G_{\tau/2}(x);\label{eq:GT2}
\end{equation}
in particular, this means that
\begin{equation}
  \|G_{\tau}\|_{L^{2}(\RR)^{2}}^{2}=\frac{1}{4\pi\tau}.\label{eq:GtauL2norm}
\end{equation}
We can also estimate the $L^{1}$ distance between two Gaussians
using Pinsker's inequality; this was done in \cite[(5.6)\textendash (5.7)]{DG22}.
Let $D_{\mathrm{KL}}$ denote the KL-divergence (relative entropy).
Then for $t_{2}>t_{1}$, we have
\begin{equation}
  \begin{aligned}\|G_{t_{2}}(x_{2}-\anon)-G_{t_{1}}(x_{1}-\anon)\|_{L^{1}}^{2} & \le2D_{\mathrm{KL}}\bigl(G_{t_{2}}(x_{2}-\anon)\mathbin{\|}G_{t_{1}}(x_{1}-\anon)\bigr)                                                                          \\
                                                                               & =\log\left(\frac{t_{1}}{t_{2}}\right)+\frac{t_{2}}{t_{1}}-1+\frac{|x_{2}-x_{1}|^{2}}{2t_{1}}\le\frac{t_{2}-t_{1}+\frac{1}{2}|x_{2}-x_{1}|^{2}}{t_{1}}.
  \end{aligned}
  \label{eq:L1ofgaussians}
\end{equation}
The first inequality is by Pinsker's inequality (see, e.g., \cite[Theorem 1.5.4 and Lemma 1.5.3]{Iha93}),
the second is by a computation done in \cite[Theorem 1.8.2]{Iha93},
and the third is since $t_{2}>t_{1}$ so the logarithm is negative.

\subsection{SDE solution theory\label{appendix:SDE-solution-theory}}

Here we provide the proof of \cref{prop:SDEwellposed}.
\begin{proof}[Proof of \cref{prop:SDEwellposed}.]
  Let $Q'\in[0,Q]$. For adapted $\RR^{m}$-valued processes
  $\Gamma,\widetilde{\Gamma}$ on $[0,Q']$, we see by the Itô isometry
  and \cref{eq:gLip} that, for any $q\in[0,Q']$,
  \begin{align}
    \EE |\clR_{a,Q}^{g}\Gamma(q)-\clR_{a,Q}^{g}\widetilde{\Gamma}(q)|^{2} & =\int_{0}^{q}\EE |g\bigl(Q-p,\Gamma(p)\bigr)-g\bigl(Q-p,\widetilde{\Gamma}(p)\bigr)|_{\Frob}^{2}\,\dif p\nonumber   \\
                                                                          & \le L^{2}\int_{0}^{q}\EE |\Gamma(p)-\widetilde{\Gamma}(p)|^{2}\,\dif p.\label{eq:useitoeasiest}
  \end{align}
  If we define the norm $\clK_{L,Q'}$ on adapted processes on
  $[0,Q']$ by
  \[
    \clK_{L,Q'}(\Gamma)\coloneqq\sup_{q\in[0,Q']}\e^{-L^{2}q}\bigl(\EE |\Gamma(q)|^{2}\bigr)^{1/2},
  \]
  then we get from \cref{eq:useitoeasiest} that
  \begin{align*}
    \clK_{L,Q'}^{2}(\clR_{a,Q}^{g}\Gamma-\clR_{a,Q}^{g}\widetilde{\Gamma}) & \le L^{2}\clK_{L,Q'}^{2}(\Gamma-\widetilde{\Gamma})\sup_{q\in[0,Q']}\e^{-2L^{2}q}\int_{0}^{q}\e^{2L^{2}p}\,\dif p                                             \\
                                                                           & =\frac{1}{2}\clK_{L,Q'}^{2}(\Gamma-\widetilde{\Gamma})\sup_{q\in[0,Q']}\e^{-2L^{2}q}[\e^{2L^{2}q}-1]\le\frac{1}{2}\clK_{L,Q'}^{2}(\Gamma-\widetilde{\Gamma}).
  \end{align*}
  Therefore, $\clR_{a,Q}^{g}$ is a contraction with respect
  to the norm $\clK_{L,Q'}$. It is straightforward to check
  that the corresponding metric space is complete, so $\clR_{a,Q}^{g}$
  has a unique fixed point $\Theta_{a,Q}^{g}$, and it is also not difficult
  to check that $\Theta_{a,Q}^{g}$ can be chosen to be adapted and
  have continuous sample paths. But a fixed point of $\clR_{a,Q}^{g}$
  is exactly a solution to \zcref[range]{eq:thetaSDE,eq:thetaIC}. Since
  the norm $\clK_{L,Q'}$ is equivalent to the norm
  \[
    \Gamma\mapsto\sup_{q\in[0,Q']}\bigl(\EE |\Gamma(q)|^{2}\bigr)^{1/2},
  \]
  with the constants in the equivalence depending only on $Q$ and $L$,
  we can conclude that \cref{eq:SDEclose} holds as well.

  We now check that $\Theta_{a,Q}^{g}$ has moments of all orders to
  prove \cref{eq:finite-moment}. By Hölder's inequality, it suffices
  to treat the case $\ell\ge2$. For $N\ge|a|$, we define the stopping
  time
  \[
    \tau_{N}\coloneqq\inf\bigl\{q\in[0,Q]\suchthat |\Theta_{a,Q}^{g}(q)|\ge N\bigr\},
  \]
  under the temporary convention that $\inf\emptyset=Q$. Let $\Theta_{a,Q,N}^{g}(q)\coloneqq\Theta_{a,Q}^{g}(q\wedge\tau_{N})$
  We note that
  the gradient of the map $|\anon|^{2}\colon\RR^{m}\to\RR$
  at $X\in\RR^m$ is given by $\nabla|\anon|^2(X)=2X$, and the Hessian by $\nabla^2|\anon|^2(X)=2\Id_m$.
  By the chain rule, the gradient of $|\anon|^{\ell}=(\anon)^{\ell/2}\circ|\anon|^2\colon\RR^{m}\to\RR$
  is given by
  \begin{equation}
    \nabla|\anon|^{\ell}(X)=2(\ell/2)|X|^{2(\ell/2-1)}X=\ell|X|^{\ell-2}X\label{eq:lpowergradient}
  \end{equation}
  and the Hessian by
  \begin{align}
    \nabla^{2}|\cdot|^{\ell}(X) & =4(\ell/2)(\ell/2-1)|X|^{2(\ell/2-2)}X^{\otimes 2}+2(\ell/2)|X|^{2(\ell/2-1)}\Id_m \notag\\
                                & =\ell(\ell-2)|X|^{\ell-4}X^{\otimes2}+\ell|X|^{\ell-2}\Id_m.\label{eq:lpowerHessian}
  \end{align}
  Thus by Itô's formula and \zcref[range]{eq:thetaSDE,eq:thetaIC} we obtain
  \begin{align*}
    \dif &|\Theta^g_{a,Q,N}(r)|^\ell \\&= \ell |\Theta^g_{a,Q,N}(r)|^{\ell-2}\Theta^g_{a,Q,N}(r)\mathbf{1}_{r\le \tau_N}\cdot \dif\Theta^g_{a,Q,N}(r) \\&\quad+ \frac12\tr\left\{\left(\ell(\ell-2)|\Theta^g_{a,Q,N}(r)|^{\ell-4}\Theta^g_{a,Q,N}(r)^{\otimes 2}  +\ell|\Theta^g_{a,Q,N}|^{\ell-2}\Id_m  \right)\mathbf{1}_{r\le \tau_N}g^2\bigl(Q-r,\Theta^g_{a,Q,N}(r)\bigr) \right\}. %
  \end{align*}
  Because $\Theta_{a,Q,N}^{g}$ is bounded, we can take expectations
  to conclude that
  \begin{align*}
    \EE& |\Theta_{a,Q,N}^{g}(q)|^{\ell}-|a|^\ell\\&=
    \frac\ell 2\EE\left[\int_0^{q\wedge\tau_N}\tr\left\{\left((\ell-2)|\Theta^g_{a,Q,N}(r)|^{\ell-4}\Theta^g_{a,Q,N}(r)^{\otimes 2}  +|\Theta^g_{a,Q,N}(r)|^{\ell-2}\Id_m  \right)g^2\bigl(Q-r,\Theta^g_{a,Q,N}(r)\bigr) \right\}\,\dif r\right]%
  \end{align*}
  for all $q\in[0,Q]$.
  Because $g\in\clA_{Q}$, there exists a constant
  $C<\infty$ such that
  \[
    |g(q,b)|_\Frob\le C\langle b\rangle\qquad\text{for all }(q,b)\in[0,Q]\times\RR^m.
  \]
  Together, the last two displays imply that
  \[
    \EE |\Theta^g_{a,Q,N}(q)|^\ell \le C + C\int_0^q \EE|\Theta^g_{a,Q,N}(r)|^\ell \,\dif r,
  \]
  with a new constant $C=C(g,\ell,a,Q)<\infty$ independent of $N$.
  By Grönwall's inequality,
  there exists yet another constant $C$, still independent of $N$, such that %
  \[
    \sup_{q\in[0,Q]}\EE |\Theta_{a,Q,N}^{g}(q)|^{\ell}\le C
  \]
  for all $N\ge|a|$. Taking $N\to\infty$, \cref{eq:finite-moment} follows
  from the monotone convergence theorem.

  Finally, we prove \cref{eq:SDEctsininitialconditions}. We note that,
  for all $q\in[0,Q]$, we have
  \begin{align*}
    \EE |\Theta_{a,Q}^{g}(q)-\Theta_{\widetilde{a},Q}^{g}(q)|^{2} & =|a-\widetilde{a}|^{2}+\int_{0}^{q}\bigl\lvert g\bigl(Q-p,\Theta_{a,Q}^{g}(p)\bigr)-g\bigl(Q-p,\Theta_{\widetilde{a},Q}^{g}(p)\bigr)\bigr\rvert _{\Frob}^{2}\,\dif p \\
                                                                  & \le|a-\widetilde{a}|^{2}+L^{2}\int_{0}^{q}|\Theta_{a,Q}^{g}(p)-\Theta_{\widetilde{a},Q}^{g}(p)|^{2}\,\dif p.
  \end{align*}
  By Grönwall's inequality, this implies that
  \[
    \EE |\Theta_{a,Q}^{g}(q)-\Theta_{\widetilde{a},Q}^{g}(q)|^{2}\le\e^{L^{2}q}|a-\widetilde{a}|^{2},
  \]
  which is \cref{eq:SDEctsininitialconditions} with $C=\e^{L^{2}Q}$.
\end{proof}

\subsection{Compositions of averages of random fields over boxes }

We define the squares $\ttB_{\zeta,k}$, the functions $k_{\zeta, z}$,
and the operators $\clS_{\zeta, z}$ as in \zcref[range]{eq:boxdef,eq:Scaldef}.
It is of course a standard fact of Gaussian convolution that $\clG_{\tau_{1}}\clG_{\tau_{2}}=\clG_{\tau_{1}+\tau_{2}}$.
The following lemma is an analogous estimate for the operators $\clS_{\zeta,z}$.
\begin{lem}
  \label{lem:composeaverages}Let $w$ be a random field on $\RR^{2}$
  satisfying $\vvvert w\vvvert_{2}<\infty$. Let $\zeta_{1}>\zeta_{2}>0$
  and let $\zeta_{3}=(\zeta_{1}^{2}+\zeta_{2}^{2})^{1/2}$. For all
  deterministic $x,z_{2},z_{3}\in\RR^{2}$, there is a deterministic
  $z_{1}\in\RR^{2}$ such that
  \begin{equation*}
    \EE |\clS_{\zeta_{1},z_{1}}[\clS_{\zeta_{2},z_{2}}w](x)-\clS_{\zeta_{3},z_{3}}w(x)|^{2}\le64(\zeta_{2}/\zeta_{1})^{2}\vvvert w\vvvert_{2}.
  \end{equation*}
\end{lem}

\begin{proof}
  Fix a deterministic point $z_{1}\in\RR^{2}$ to be chosen later.
  Given $i \in \{1,2,3\}$ and $y \in \R^d$, we use the shorthand $\ttB_i(y) = \ttB_{\zeta_i, k_{\zeta_i, z_i}(y) + z_i}$.
  Then we have
  \begin{align*}
    \clS_{\zeta_{1},z_{1}}[\clS_{\zeta_{2},z_{2}}w](x) & =\zeta_{1}^{-2}\int_{\ttB_1(x)}\clS_{\zeta_{2},z_{2}}w(y')\,\dif y'=\zeta_{1}^{-2}\zeta_{2}^{-2}\int_{\ttB_1(x)}\!\int_{\ttB_2(y')}w(y)\,\dif y\,\dif y' \\
                                                       & =\zeta_{1}^{-2}\zeta_{2}^{-2}\int w(y)A(y)\,\dif y,
  \end{align*}
  where
  \[
    A(y)\coloneqq|\ttB_1(x)\cap \ttB_2(y)|
  \]
  is the area (Lebesgue measure) of $\ttB_1(x) \cap \ttB_2(y)$.
  This implies that
  \begin{equation}
    \clS_{\zeta_{1},z_{1}}[\clS_{\zeta_{2},z_{2}}w](x)-\clS_{\zeta_{3},z_{3}}w(x)=\zeta_{1}^{-2}\int w(y)[\zeta_{2}^{-2}A(y)-\mathbf{1}_{\ttB_3(x)}(y)]\,\dif y.\label{eq:SSvsScompare}
  \end{equation}

  We note that $0\le A(y)\le\zeta_{2}^{2}$,
  so the term in parentheses inside the integrand on the right side
  of \cref{eq:SSvsScompare} is bounded by $1$ in absolute value. Now
  choose $z_{1}$ such that the squares $\ttB_1(x)$
  and $\ttB_3(x)$ have the same center,
  say $\overline{z}$. Let $\ttB_{\mathrm{inner}}$ and $\ttB_{\mathrm{outer}}$
  be squares centered at $\overline{z}$, with side lengths $\zeta_{1}-2\zeta_{2}$
  and $\zeta_{1}+2\zeta_{2}$, respectively. This means that
  \[
    \ttB_{\mathrm{inner}}\subseteq \ttB_1(x) \subseteq \ttB_3(x)\subseteq \ttB_{\mathrm{outer}}.
  \]
  We note that if $y\in \ttB_{\mathrm{inner}}$, then $\ttB_2(y) \subseteq \ttB_1(x)$,
  so $\zeta_{2}^{-2}A(y)=1=\mathbf{1}_{\ttB_3(x)}(y)$.
  On the other hand, if $y\in\RR^{2}\setminus \ttB_{\mathrm{outer}}$,
  then $\ttB_2(y)\cap \ttB_1(x) = \emptyset$,
  and also $y\not\in \ttB_3(x)$, so $\zeta_{2}^{-2}A(y)=0=\mathbf{1}_{\ttB_3(x)}(y)$.
  Therefore, the support of the integrand on the right side of \cref{eq:SSvsScompare}
  is contained in $\ttB_{\mathrm{outer}}\setminus \ttB_{\mathrm{inner}}$.

  Using these observations in \cref{eq:SSvsScompare}, we see that
  \[
    \left\lvert  \clS_{\zeta_{1},z_{1}}[\clS_{\zeta_{2},z_{2}}w](x)-\clS_{\zeta_{3},z_{3}}w(x)\right\rvert  \le\zeta_{1}^{-2}\int_{\ttB_{\mathrm{outer}}\setminus \ttB_{\mathrm{inner}}}|w(z)|\,\dif y,
  \]
  so
  \begin{align*}
    \EE & \left\lvert  \clS_{\zeta_{1},z_{1}}[\clS_{\zeta_{2},z_{2}}w](x)-\clS_{\zeta_{3},z_{3}}w(x)\right\rvert  ^{2}\le\zeta_{1}^{-4}\EE \left[\int_{\ttB_{\mathrm{outer}}\setminus \ttB_{\mathrm{inner}}}|w(z)|\,\dif y\right]^{2}                                                                                          \\
        & \le\zeta_{1}^{-4}|\ttB_{\mathrm{outer}}\setminus \ttB_{\mathrm{inner}}|\int_{\ttB_{\mathrm{outer}}\setminus \ttB_{\mathrm{inner}}}\EE |w(z)|^{2}\,\dif y\le\zeta_{1}^{-4}|\ttB_{\mathrm{outer}}\setminus \ttB_{\mathrm{inner}}|^{2}\vvvert w\vvvert_{2}^{2}=64\zeta_{1}^{-2}\zeta_{2}^{2}\vvvert w\vvvert_{2}^{2},
  \end{align*}
  and the proof is complete.
\end{proof}

\subsection{Matrix norms\label{subsec:matrix-norms}}

In this section we state several inequalities involving the matrix
norms $|\anon|_{\Frob}$ (Frobenius norm), $|\anon|_{*}$ (nuclear
or trace norm), and $|\anon|_{\op}$ (operator norm) that
we use throughout the paper. We recall that, if $\sigma$ is a matrix
with vector of singular values $\nu_{1},\ldots,\nu_{m}$, then
\[
  |\sigma|_{\Frob}=\sqrt{\nu_{1}^{2}+\cdots+\nu_{m}^{2}},\qquad|\sigma|_{*}=\nu_{1}+\cdots+\nu_{m},\qquad\text{and}\qquad|\sigma|_{\op}=\max\{\nu_{1},\ldots,\nu_{m}\}.
\]
(That is, these norms are the Schatten $2$-, $1$-, and $\infty$-norms,
respectively). We note that
\[
  |\sigma|_{*}=\tr\sigma\qquad\text{for all }\sigma\in\clH_{+}^{m}.
\]
Also, by the Hölder inequality for Schatten norms we have the inequalities
\begin{equation}
  \tr[\sigma_{1}\sigma_{2}^{\mathrm{T}}]\le|\sigma_{1}\sigma_{2}^{\mathrm{T}}|_{*}\le|\sigma_{1}|_{\op}|\sigma_{2}|_{*};\label{eq:simpleholderschatten}
\end{equation}
and
\begin{equation*}
  |\sigma_1\sigma_2|_{\Frob}\le |\sigma_1|_{\op}|\sigma_2|_{\Frob}\le |\sigma_1|_{\Frob}|\sigma_2|_{\Frob}.
\end{equation*}
We also recall the Powers--Størmer
inequality \cite[Lemma 4.1]{PS70}:
\begin{equation}
  |\sigma_{1}-\sigma_{2}|_{\Frob}^{2}\le|\sigma_{1}^{2}-\sigma_{2}^{2}|_{*}\qquad\text{for all }\sigma_{1},\sigma_{2}\in\clH_{+}^{m}.\label{eq:powerstormer}
\end{equation}

The final result of this section is a ``matrix-valued reverse triangle
inequality.'' For $m>1$, the statement is far from trivial: it is
a consequence of Lieb's concavity theorem \cite[Corollary 1.1]{Lie73}.
\begin{prop}
  \label{prop:matrixvaluedreversetriangleinequality}For all random
  matrices $\sigma_{1},\sigma_{2}\in\clH_{+}^{m}$, we have
  \begin{equation}
    |(\EE \sigma_{1}^{2})^{1/2}-(\EE \sigma_{2}^{2})^{1/2}|_{\Frob}\le(\EE |\sigma_{1}-\sigma_{2}|_{\Frob}^{2})^{1/2}.\label{eq:matrixvaluedreversetriangleinequality}
  \end{equation}
\end{prop}

\begin{proof}
  We have
  \begin{align}
    |(\EE \sigma_1^2)^{1/2}-(\EE \sigma_{2}^{2})^{1/2}|_{\Frob}^{2} & =\tr\left[\left((\EE \sigma_1^2)^{1/2}-(\EE \sigma_{2}^{2})^{1/2}\right)^{2}\right]\nonumber                                \\
                                                                    & =\tr\EE \sigma_1^2+\tr\EE \sigma_{2}^{2}-2\tr[(\EE \sigma_1^2)^{1/2}(\EE \sigma_{2}^{2})^{1/2}].\label{eq:bustitup}
  \end{align}
  Now by \cite[Corollary 1.1]{Lie73}, the map $G\colon\clH_{+}^{m}\times\clH_{+}^{m}\to\RR$
  given by
  \[
    G(X,Y)=\tr[X^{1/2}Y^{1/2}]
  \]
  is concave in both $X$ and $Y$. So using Jensen's inequality twice,
  we have
  \begin{equation}
    \tr[(\EE \sigma_1^2)^{1/2}(\EE \sigma_{2}^{2})^{1/2}]=G(\EE \sigma_1^2,\EE \sigma_{2}^{2})\ge\EE [G(\sigma_{1}^{2},\EE \sigma_{2}^{2})]\ge\EE [G(\sigma_{1}^{2},\sigma_{2}^{2})]=\EE \tr[\sigma_{1}\sigma_{2}].\label{eq:applyJensen}
  \end{equation}
  Using \cref{eq:applyJensen} in \cref{eq:bustitup}, we obtain
  \begin{align*}
    |(\EE \sigma_1^2)^{1/2}-(\EE \sigma_{2}^{2})^{1/2}|_{\Frob}^{2} & \le\tr\EE \sigma_1^2+\tr\EE \sigma_{2}^{2}-2\tr \EE(\sigma_{1}\sigma_{2})=\tr\EE (\sigma_{1}-\sigma_{2})^{2}=\EE |\sigma_{1}-\sigma_{2}|_{\Frob}^{2}.
  \end{align*}
  Taking square roots completes the proof.
\end{proof}

We equip $\clH_{+}^{m}$
with the metric induced by the Frobenius norm because if $\sigma_{1},\sigma_{2}\in\clH_{+}^{m}$
and $Z\sim N(0,\Id_{m})$, then
\[
  (\EE |\sigma_{1}Z-\sigma_{2}Z|^{2})^{1/2}=|\sigma_{1}-\sigma_{2}|_{\Frob}.
\]
We frequently use the ``differential'' version of this identity
in the form of Itô's isometry. We note in passing that $(\sigma_{1}Z,\sigma_{2}Z)$
is \emph{not} the $L^{2}$-best coupling of two Gaussian random variables
with covariance matrices $\sigma_{1}^{2}$ and $\sigma_{2}^{2}$.
In fact, according to the results of \cite{DL82,GS84},
\begin{equation}
  \clW_{2}(\sigma_{1}Z,\sigma_{2}Z)= \bigl(\tr[\sigma_{1}^{2}+\sigma_{2}^{2}-2(\sigma_{1}^{2}\sigma_{2}^{2})^{1/2})]\bigr)^{1/2},\label{eq:W2ofgaussianmetriconmatrices}
\end{equation}
which can (if $m>1$) be strictly less than
\[
  |\sigma_{1}-\sigma_{2}|_{\Frob}=\bigl(\tr[\sigma_{1}^{2}+\sigma_{2}^{2}-2\sigma_{1}\sigma_{2}]\bigr)^{1/2}.
\]
Since this paper is ultimately concerned with Wasserstein distances
rather than with particular couplings, it is plausible that similar
results could be obtained if the metric induced by the Frobenius norm
on $\clH_{+}^{m}$ were replaced by the metric given by the
right side of \cref{eq:W2ofgaussianmetriconmatrices}, at the cost of
an explosion in the number of couplings needing to be constructed
throughout the argument. %

If, instead of using the distance induced by the Frobenius norm, we use the optimal coupling distance
\[
  \tilde{d}(\sigma_{1},\sigma_{2})=\clW_{2}(\sigma_{1}Z,\sigma_{2}Z)=\bigl(\tr[\sigma_{1}^{2}+\sigma_{2}^{2}-2(\sigma_{1}^{2}\sigma_{2}^{2})^{1/2}]\bigr)^{1/2},
\]
where $Z\sim N(0,\Id_{m})$, then the analogous statement to \cref{eq:matrixvaluedreversetriangleinequality}
is probabilistically very natural. Indeed, let $\sigma_{1}$ and $\sigma_{2}$
be random elements of $\clH_{+}^{m}$. Assume without loss
of generality that the probability space is given by $[0,1]$ equipped
with the Lebesgue measure.
There is a measurable $\R^m\otimes\R^m$-valued random variable
$U$ on $[0,1]$ such that $UU^{\mathrm{T}}=\Id_m$ almost surely and, if $Z\sim N(0,\Id_{m})$ is independent
of everything else (now extending the probability space to $[0, 1] \times \Omega'$ for some space $\Omega'$), then
\[
  \EE [|U\sigma_{1}Z-\sigma_{2}Z|^{2}\mid\sigma_{1},\sigma_{2}]=\clW_{2}(\sigma_{1}Z,\sigma_{2}Z\mid\sigma_{1},\sigma_{2})
\]
almost surely.
Let $B$ be a standard $\RR^{m}$-valued Brownian motion on $[0,1]$ with probability space $\Omega'$.
Then we have
\begin{align*}
  \EE \tilde{d}(\sigma_{1},\sigma_{2}) & =\int_{0}^{1}\tilde{d}\bigl(\sigma_{1}(t),\sigma_{2}(t)\bigr)\,\dif t=\int_{0}^{1}\clW_{2}(\sigma_{1}(t)Z,\sigma_{2}(t)Z)\,\dif t                                             \\
                                       &= \int_{0}^{1 }\EE^{\Omega'} |U(t)\sigma_{1}(t)Z-\sigma_{2}(t)Z|^{2}\,\dif t= \int_{0}^{1}|U(t)\sigma_{1}(t)-\sigma_{2}(t)|_{\Frob}^{2}\,\dif t\\
                                       &= \EE^{\Omega'} \left\lvert  \int_{0}^{1}U(t)\sigma_{1}(t)\,\dif B(t)-\int_{0}^{1}\sigma_{2}(t)\,\dif B(t)\right\rvert  ^{2} \\
                                       & \ge\clW_{2}^{\Omega'}\bigl((\EE \sigma_{1}^{2})^{1/2}Z,(\EE \sigma_{2}^{2})^{1/2}Z\bigr)=\tilde{d}\bigl((\EE \sigma_{1}^{2})^{1/2},(\EE \sigma_{2}^{2})^{1/2}\bigr),
\end{align*}
with the inequality since $\int_0^1 U(t)\sigma_1(t)\,\dif B(t)\overset{\mathrm{law}}= (\E\sigma_1^2)^{1/2}Z$ and $\int_0^1\sigma_2(t)\,\dif B(t)\overset{\mathrm{law}}= (\EE\sigma_2^2)^{1/2}Z$.
Probabilistically, therefore, we interpret the subtlety in \cref{prop:matrixvaluedreversetriangleinequality}
as coming from the fact that the left side of \cref{eq:matrixvaluedreversetriangleinequality}
is $|(\EE \sigma_{1}^{2})^{1/2}-(\EE \sigma_{2}^{2})^{1/2}|_{\Frob}^{2}=\EE^{\Omega'} |(\EE \sigma_{1}^{2})^{1/2}Z-(\EE \sigma_{2}^{2})^{1/2}Z|^{2}$,
corresponding to a non-optimal coupling. The proposition says that
this coupling is nonetheless superior to the one corresponding to
the right side of \cref{eq:matrixvaluedreversetriangleinequality}.

\section{Coupling of space-time white noises\label{appendix:coupling}}

In this appendix we prove a coupling result for space-time white noises that is used in several essential places throughout the paper. To understand the following proposition, the reader may wish to consider a simpler situation in which we have i.i.d.\ Gaussian random variables $Z_1,\ldots,Z_n\sim N(0,1)$ and coefficients $a_1,\ldots,a_n$. Then the random variables $\sum a_i Z_i$ and $\left(\frac1n\sum a_i^2\right)^{1/2}\sum Z_i$ are not typically close to one another, but they do have the same distribution and hence can be coupled in such a way that they are identical. The following proposition considers a similar situation with a space-time white noise, in which we show that given a white noise $\dif W_r$ and a field of ``coefficients'' $(w_r)$, we can build a new white noise $\dif\widetilde W_r$ such that, in a weak sense, $w_r\dif W_r$ is close to $\dif \tilde W_r$ modulated by a quadratic average of $w_r$ on an intermediate scale.
\begin{prop}
  \label{prop:coupling}Fix $t_{0}\in\RR$. Let $w=w_{t}(x)$
  be a spatially homogeneous $\clH_{+}^{m}$-valued space-time
  random field, adapted to the filtration $(\scrF_{t})$, with
  $\vvvert|w_{t}|_{\Frob}\vvvert_{2}<\infty$ for each $t$, and
  let $(\zeta_{r})_{r\ge t_{0}}$ be a deterministic measurable positive
  real-valued path. There is a space-time $\RR^{m}$-valued white
  noise $(\dif\widetilde{W}_{t})$, adapted to the filtration $(\scrF_{t})$,
  such that, for any deterministic measurable positive real-valued path
  $(\eta_{r})_{r\ge t_{0}}$ and any $t\ge t_0$, we have
  \begin{equation}
    \left\vvvert \int_{t_{0}}^{t}\clG_{\eta_{r}}[w_{r}\,\dif W_{r}-(\clS_{\zeta_{r}}w_{r}^{2})^{1/2}\,\dif\widetilde{W}_{r}]\right\vvvert _{2}^{2}\le\frac{1}{2\pi^{3}}\int_{t_{0}}^{t}\zeta_{r}^{2}\eta_{r}^{-2}\vvvert w_{r}\vvvert_{2}^{2}\,\dif r.\label{eq:propcouplingconclusion}
  \end{equation}
\end{prop}

To prove \cref{prop:coupling}, first we will construct the space-time
white noise $\dif\widetilde{W}_{t}$. Then we will prove that it has
the desired properties.
Let $\NN_{1}=\{1,2,3,\ldots\}$ and $\NN_{0}=\{0,1,2,\ldots\}$.
Let $\{\psi_{\ell}\}_{\ell\in\NN_{0}}$ be an orthonormal basis
for $L^{2}([0,1]^{2})$ such that
\begin{equation}
  \psi_{0}=\mathbf{1}_{[0,1]^{2}}.\label{eq:zeroethorderisflat}
\end{equation}
We consider each $\psi_{\ell}$ as a function on $\RR^{2}$ be extending
it by zero outside of $[0,1]^{2}$. Now we define, for $\zeta\in(0,\infty)$,
$k\in\ZZ^{2}$, and $\ell\in\NN_{0}$, the function
\begin{equation}
  \varphi_{\zeta,k,\ell}(x)\coloneqq\zeta^{-1}\psi_\ell(\zeta^{-1}x-k).\label{eq:rescalephi}
\end{equation}
This means that, for each $\zeta\in(0,\infty)$ and $k\in\ZZ^{2}$,
the family $\{\varphi_{\zeta,k,\ell}\}_{\ell\in\NN_{0}}$ is
an orthonormal basis for $L^{2}(\ttB_{\zeta,k})$.

Given $k\in\ZZ^{2}$ and $\ell\in\NN_{0}$, define the $\RR^{m}$-valued process (for $t\ge t_{0}$)
\begin{equation}
  Z_{k,\ell;t}=\int_{t_{0}}^{t}\!\!\!\int\varphi_{\zeta_{r},k,\ell}(x)w_{r}(x)\,\dif W_{r}(x).\label{eq:Zdef}
\end{equation}
Given an $\RR^{m}$-valued test function $F$
on $[t_{0},\infty)\times\RR^{2}$, this means that
\begin{align}
  \sum_{\substack{k\in\ZZ^{2}                                                                                                       \\
  \ell\in\NN_{0}
  }
  } & \int_{t_{0}}^{t}\left(\int F(r,y)\varphi_{\zeta_{r},k,\ell}(y)\,\dif y\right)\cdot\dif Z_{k,\ell;r} \nonumber                 \\& =\int_{t_{0}}^{t}\!\!\!\int\sum_{\substack{k\in\ZZ^{2}                                \\
  \ell\in\NN_{0}
  }
  }\varphi_{\zeta_{r},k,\ell}(x)\left(\int F(r,y)\varphi_{\zeta_{r},k,\ell}(y)\,\dif y\right)\cdot w_{r}(x)\,\dif W_{r}(x)\nonumber \\
    & =\int_{t_{0}}^{t}\!\!\!\int F(r,x)\cdot w_{r}(x)\,\dif W_{r}(x).\label{eq:implicationofZdef}
\end{align}

We note that $(Z_{k,\ell;t})_{t\ge0}$ is an $\RR^{m}$-valued
$(\scrF_{t})$-martingale.
Given $k,k'\in\ZZ^{2}$ and $\ell,\ell'\in\NN_{0}$, it has quadratic covariation
\begin{equation}
  [Z_{k,\ell},Z_{k',\ell'}]_{t}=\delta_{k,k'}\int_{t_0}^{t}\chi_{k,\ell,\ell';r}\,\dif r,\label{eq:ZQV}
\end{equation}
where we have used the Kronecker $\delta$ symbol and also defined
\begin{equation}
  \chi_{k,\ell,\ell';r}\coloneqq\int\varphi_{\zeta_{r},k,\ell}(x)\varphi_{\zeta_{r},k,\ell'}(x)w_{r}^{2}(x)\,\dif x.\label{eq:chidef}
\end{equation}
Note in particular that, by \cref{eq:zeroethorderisflat} and \cref{eq:rescalephi},
\begin{equation}
  \chi_{k,0,0;r}=\zeta_{r}^{-2}\int_{\ttB_{\zeta_{r},k}}w_{r}^{2}(x)\,\dif x.\label{eq:chikzerozero}
\end{equation}

We will now define a new set of $\RR^{m}$-valued $(\scrF_{t})$-martingales
$\{(\widetilde{Z}_{k,\ell;t})_{t\ge t_{0}}\}_{k\in\ZZ^{2},\ell\in\NN_{0}}$.
For each $k\in\ZZ^{2}$ and $t\ge t_{0}$, we put
\begin{equation}
  \widetilde{Z}_{k,0;t}=Z_{k,0;t}.\label{eq:Ztildek0}
\end{equation}
To define $\widetilde{Z}_{k,\ell;t}$ for $\ell\in\NN_{1}$,
we first let $(X_{k,\ell})_{k\in\ZZ^{2},\ell\in\NN_{1}}$be
a family of iid standard $\RR^{m}$-valued Brownian motions,
each adapted to the filtration $(\scrF_{t})_{t}$ but independent
of all other random variables. For each $\ell\in\NN_{1}$,
define
\begin{equation}
  \widetilde{Z}_{k,\ell;t}=\int_{t_{0}}^{t}\chi_{k,0,0;r}^{1/2}\,\dif X_{k,\ell;r}.\label{eq:Ztildelgt1def}
\end{equation}
Let us now compute the quadratic covariation matrices of the martingales
we have defined. First we note that, for $k,k'\in\ZZ^{2}$
and $\ell,\ell\in\NN_{0}$, we have
\begin{equation}
  [\widetilde{Z}_{k,\ell},\widetilde{Z}_{k',\ell'}]_{t}=\delta_{k,k'}\delta_{\ell,\ell'}\int_{t_{0}}^{t}\chi_{k,0,0;r}\,\dif r,\label{eq:ZtildeQCV}
\end{equation}
which follows from \cref{eq:ZQV}, \cref{eq:Ztildelgt1def}, and the independence
of the $X_{k,\ell}$s. Moreover, for $k,k'\in\ZZ^{2}$ and
$\ell,\ell'\in\NN_{0}$, we have
\begin{equation}
  [Z_{k,\ell},\widetilde{Z}_{k',\ell'}]_{t}=\delta_{k,k'}\delta_{\ell',0}\int_{t_{0}}^{t}\chi_{k,\ell,0;r}\,\dif r.\label{eq:ZZtildeQCV}
\end{equation}

We now define a field $\widetilde{W}_{t}$ by, formally, %
\begin{equation}
  \widetilde{W}_{t}(x)=\sum_{\substack{k\in\ZZ^{2}\\
      \ell\in\NN_{0}
    }
  }\int_{t_{0}}^{t}\varphi_{\zeta_{r},k,\ell}(x)[\clS_{\zeta_{r}}w_{r}^{2}(x)]^{-1/2}\,\dif\widetilde{Z}_{k,\ell;r}.\label{eq:Wttilde}
\end{equation}
Precisely, this means that for any $\RR^{m}$-valued test function $f$ on $\RR^{2}$, we have
\begin{equation}
  \langle\widetilde{W}_{t},f\rangle=\sum_{\substack{k\in\ZZ^{2}\\
      \ell\in\NN_{0}
    }
  }\int_{t_{0}}^{t}\left(\int\varphi_{\zeta_{r},k,\ell}(x)f(x)\cdot[\clS_{\zeta_{r}}w_{r}^{2}(x)]^{-1/2}\,\dif x\right)\dif\widetilde{Z}_{k,\ell;r},\label{eq:testagainstWtilde}
\end{equation}
where here and henceforth we use $\langle\cdot,\cdot\rangle$ to denote
the usual inner product on $L^{2}(\RR^2 ; \RR^{m})$. We note that
\begin{align*}
  \int & \varphi_{\zeta_{r},k,\ell}(x)f(x)\cdot[\clS_{\zeta_{r}}w_{r}^{2}(x)]^{-1/2}\,\dif x=\int_{\ttB_{\zeta_{r},k}}\varphi_{\zeta_{r},k,\ell}(x)f(x)\cdot[\clS_{\zeta_{r}}w_{r}^{2}(x)]^{-1/2}\,\dif x                                                                       \\
       & =\left(\int\varphi_{\zeta_{r},k,\ell}(x)f(x)\,\dif x\right)\cdot\left(\zeta_{r}^{-2}\int_{\ttB_{\zeta_{r},k}}w_{r}^{2}(x)\,\dif x\right)^{-1/2}\overset{\cref{eq:chikzerozero}}{=}\left(\int\varphi_{\zeta_{r},k,\ell}(x)f(x)\,\dif x\right)\cdot\chi_{k,0,0;r}^{-1/2}
\end{align*}
since $\varphi_{\zeta_{r},k,\ell}$ is supported on $\ttB_{\zeta_{r},k}$.
Hence \cref{eq:testagainstWtilde} can be rewritten as
\begin{equation}
  \langle\widetilde{W}_{t},f\rangle=\sum_{\substack{k\in\ZZ^{2}\\
      \ell\in\NN_{0}
    }
  }\int_{t_{0}}^{t}\left(\int\varphi_{\zeta_{r},k,\ell}(x)f(x)\,\dif x\right)\cdot\chi_{k,0,0;r}^{-1/2}\,\dif\widetilde{Z}_{k,\ell;r}.\label{eq:testagainstWtilde-1}
\end{equation}
We note also that for any test function $F\colon[t_{0},\infty)\times\RR^{2}\to\RR^{m}$
, we have
\begin{equation}
  \int_{t_{0}}^{t}\!\!\!\int F(r,x)\cdot[\clS_{\zeta_{r}}w_{r}^{2}(x)]^{1/2}\,\dif\widetilde{W}_{r}(x)=\sum_{\substack{k\in\ZZ^{2}\\
      \ell\in\NN_{0}
    }
  }\int_{t_{0}}^{t}\!\!\!\int\varphi_{\zeta_{r},k,\ell}(x)F(r,x)\cdot\dif\widetilde{Z}_{k,\ell;r}.\label{eq:testagainstSdWtilde}
\end{equation}

\begin{lem}
  The field $\dif\widetilde{W}$ is a standard $\RR^{m}$-valued
  space-time white noise.
\end{lem}

\begin{proof}
  Let $f\in L^{2}(\RR^2;\RR^{m})$ and define $\widetilde{Y}_{t}=\langle\widetilde{W}_{t},f\rangle$.
  Then for any $t\ge t_{0}$, we have
  \begin{align*}
    [\widetilde{Y}]_{t} & \overset{\cref{eq:testagainstWtilde-1}}{=}\int_{t_{0}}^{t}\sum_{\substack{k,k'\in\ZZ^{2}                                                                                                                   \\
    \ell,\ell'\in\NN_{0}
    }
    }\left(\int\varphi_{\zeta_{r},k,\ell}(x)f(x)\,\dif x\right)\cdot\chi_{k,0,0;r}^{-1/2}\dif[\widetilde{Z}_{k,\ell},\widetilde{Z}_{k',\ell'}]_{r}\chi_{k',0,0;r}^{-1/2}\left(\int\varphi_{\zeta_{r},k',\ell'}(x)f(x)\,\dif x\right) \\
                        & \overset{\cref{eq:ZtildeQCV}}{=}\int_{t_{0}}^{t}\sum_{\substack{k\in\ZZ^{2}                                                                                                                                \\
    \ell\in\NN_{0}
    }
    }\left\lvert  \int\varphi_{\zeta_{r},k,\ell}(x)f(x)\,\dif x\right\rvert  ^{2}\,\dif r=(t-t_{0})\|f\|_{L^{2}(\RR^{2};\RR^m)}^{2},
  \end{align*}
  with the last identity by the Plancherel theorem since the family
  $\{\varphi_{\zeta_{r},k,\ell}\}_{k\in\ZZ^{2},\ell\in\NN_{0}}$
  is an orthonormal basis for $L^{2}(\RR^{2})$. This implies that $\dif\widetilde{W}$
  is an $\RR^{m}$-valued space-time white noise.
\end{proof}
To complete the proof of \cref{prop:coupling}, it remains to prove
the estimate \cref{eq:propcouplingconclusion}.
\begin{proof}[Proof of \cref{eq:propcouplingconclusion}]
  Let $(\eta_{r})_{r\ge t_{0}}$ be a measurable positive real-valued
  path. Fix $x\in\RR^{2}$ and, for $t\ge t_{0}$, let
  \begin{equation*}
    I_{t}\coloneqq\EE \left\lvert  \int_{t_{0}}^{t}\clG_{\eta_{r}}[w_{r}\,\dif W_{r}-(\clS_{\zeta_{r}}w_{r}^{2})^{1/2}\,\dif\widetilde{W}_{r}](x)\right\rvert  ^{2}.
  \end{equation*}
  To prove \cref{eq:propcouplingconclusion} it suffices to prove the
  same upper bound on $I_{t}$. Define
  \begin{equation}
    Y_{k,\ell;t}\coloneqq Z_{k,\ell;t}-\widetilde{Z}_{k,\ell;t}\label{eq:Ykltdef}
  \end{equation}
  and
  \begin{equation}
    g_{k,\ell;t}\coloneqq\clG_{\eta_{t}}\varphi_{\zeta_{t},k,\ell}(x)=\langle G_{\eta_{t}}(x-\anon),\varphi_{\zeta_{t},k,\ell}\rangle.\label{eq:gkltdef}
  \end{equation}
  By the bilinearity and symmetry of quadratic variation
  as well as \cref{eq:ZQV}, \cref{eq:ZtildeQCV}, and \cref{eq:ZZtildeQCV}, we have
  \begin{align}
    [Y_{k_{1},\ell_{1}},Y_{k_{2},\ell_{2}}]_{t} & =[Z_{k_{1},\ell_{1}},Z_{k_{2},\ell_{2}}]_{t}-[Z_{k_{1},\ell_{1}},\widetilde{Z}_{k_{2},\ell_{2}}]_{t}-[\widetilde{Z}_{k_{1},\ell_{1}},Z_{k_{2},\ell_{2}}]_{t}+[\widetilde{Z}_{k_{1},\ell_{1}},\widetilde{Z}_{k_{2},\ell_{2}}]_{t}.\nonumber \\
                                                & =\delta_{k_{1},k_{2}}\int_{0}^{t}\left(\chi_{k_{1},\ell_{1},\ell_{2};r}-\delta_{\ell_{2},0}\chi_{k_{1},\ell_{1},0;r}-\delta_{\ell_{1},0}\chi_{k_{2},\ell_{2},0;r}+\delta_{\ell_{1},\ell_{2}}\chi_{k_{1},0,0;r}\right)\,\dif r\nonumber     \\
                                                & =\delta_{k_{1},k_{2}}\int_{0}^{t}\left((1-\delta_{\ell_{1},0}-\delta_{\ell_{2},0})\chi_{k_{1},\ell_{1},\ell_{2};r}+\delta_{\ell_{1},\ell_{2}}\chi_{k_{1},0,0;r}\right)\,\dif r\nonumber                                                    \\
                                                & =\begin{cases}
                                                     0                                                                                                              & \text{if }k_{1}\ne k_{2}\text{ or }0\in\{\ell_{1},\ell_{2}\};  \\
                                                     \int_{0}^{t}\left[\chi_{k_{1},\ell_{1},\ell_{2};r}+\delta_{\ell_{1},\ell_{2}}\chi_{k_{1},0,0;r}\right]\,\dif r & \text{if }k_{1}=k_{2}\text{ and }0\not\in\{\ell_{1},\ell_{2}\}.
                                                   \end{cases}\label{eq:YQCV}
  \end{align}
  Now, using \cref{eq:implicationofZdef}, \cref{eq:testagainstSdWtilde},
  \cref{eq:Ykltdef}, and \cref{eq:gkltdef}, we can write
  \begin{align}
    I_{t} & =\EE \Biggl|\sum_{\substack{k\in\ZZ^{2}                                                                                       \\
    \ell\in\NN_{0}
    }
    }\int_{t_{0}}^{t}g_{k,\ell;r}\,\dif Y_{k,\ell;r}\Biggr|^{2}=\sum_{\substack{k_{1},k_{2}\in\ZZ^{2}                             \\
    \ell_{1},\ell_{2}\in\NN_{0}
    }
    }\EE \left(\int_{t_{0}}^{t}g_{k_{1},\ell_{1};r}g_{k_{2},\ell_{2};r}\tr\dif[Y_{k_{1},\ell_{1}},Y_{k_{2},\ell_{2}}]_{r}\right)\nonumber \\
          & =\sum_{\substack{k\in\ZZ^{2}                                                                                                  \\
    \ell_{1},\ell_{2}\in\NN_{1}
    }
    }\int_{t_{0}}^{t}g_{k,\ell_{1};r}g_{k,\ell_{2};r}\EE (\tr\chi_{k,\ell_{1},\ell_{2};r}+\delta_{\ell_{1},\ell_{2}}\tr\chi_{k,0,0;r})\,\dif r,\label{eq:Iexpand}
  \end{align}
  with the last identity by \cref{eq:YQCV}. Taking expectations in \cref{eq:chidef}
  and using the spatial homogeneity of $w$, we have
  \[
    \EE [\tr\chi_{k,\ell_{1},\ell_{2};r}]=\langle\varphi_{\zeta_{r},k,\ell_{1}},\varphi_{\zeta_{r},k,\ell_{2}}\rangle\EE [\tr w_{r}^{2}(x)]=\langle\varphi_{\zeta_{r},k,\ell_{1}},\varphi_{\zeta_{r},k,\ell_{2}}\rangle \EE |w_{r}(x)|_{\Frob}^{2} = \delta_{\ell_1,\ell_2} \EE |w_{r}(x)|_{\Frob}^{2}
  \]
  and
  \[
    \delta_{\ell_{1},\ell_{2}}\EE [\tr\chi_{k,0,0;r}]=\delta_{\ell_{1},\ell_{2}}\langle\varphi_{\zeta_{r},k,0},\varphi_{\zeta_{r},k,0}\rangle\EE [\tr w_{r}^{2}(x)]=\delta_{\ell_{1},\ell_{2}}\EE |w_{r}(x)|_{\Frob}^{2}.
  \]
  Using the last two displays in \cref{eq:Iexpand}, we get
  \begin{equation}
    I_{t}  =2\int_{t_{0}}^{t}\Biggl(\sum_{\substack{k\in\ZZ^{2}                                                                                                                                                    \\
        \ell \in\NN_{1}
      }
    }g_{k,\ell;r}^2\Biggr)\EE |w_{r}(x)|_{\Frob}^{2}\,\dif r =2\int_{t_{0}}^{t}\EE |w_{r}(x)|_{\Frob}^{2}\sum_{k\in\ZZ^{2}}\Biggl\|\sum_{\ell\in\NN_{1}}g_{k,\ell;r}\varphi_{\zeta_{r},k,\ell}\Biggr\|_{L^{2}(\ttB_{\zeta_{r},k})}^{2}\,\dif r,\label{eq:Itdevelopmore}%
  \end{equation}
  with the second identity by the Plancherel theorem since $\{\varphi_{\zeta_{r},k,\ell}\}_{\ell\in\NN_{1}}$
  is an orthonormal family in $L^{2}(\ttB_{\zeta_{r},k})$.
  Now we recall
  that $\{\varphi_{\zeta_{r},k,\ell}\}_{\ell\in\NN_{0}}=\{\varphi_{\zeta_{r},k,0}\}\cup\{\varphi_{\zeta_{r},k,\ell}\}_{\ell\in\NN_{1}}$
  is in fact an orthonormal basis for $L^{2}(\ttB_{\zeta_{r},k})$, so by \cref{eq:gkltdef}, we have
  \begin{align}
    \Biggl\| & \sum_{\ell\in\NN_{1}}g_{k,\ell;r}\varphi_{\zeta_{r},k,\ell}\Biggr\|_{L^{2}(\ttB_{\zeta_{r},k})}^{2}=\Biggl\|\sum_{\ell\in\NN_{0}}g_{k,\ell;r}\varphi_{\zeta_{r},k,\ell}-g_{k,0;r}\varphi_{\zeta_{r},0,\ell}\Biggr\|_{L^{2}(\ttB_{\zeta_{r},k})}^{2}\nonumber \\
             & =\Bigl\| G_{\eta_{r}}(x-\cdot)-\fint_{\ttB_{\zeta_{r},k}}G_{\eta_{r}}(x-y)\,\dif y\Bigr\|_{L^{2}(\ttB_{\zeta_{r},k})}^{2}\le\frac{\zeta_{r}^{2}}{\pi^{2}}\|\nabla G_{\eta_{r}}(x-\cdot)\|_{L^{2}(\ttB_{\zeta_{r},k};\RR^2)}^{2},\label{eq:poincare}
  \end{align}
  with the last inequality by the Poincaré inequality and an elementary
  scaling argument (analogous to the one used in the proof of \cite[Theorem 5.8.2]{Eva10}).
  Using \cref{eq:poincare} in \cref{eq:Itdevelopmore}, we obtain
  \begin{align*}
    I_{t} & \le\frac{2}{\pi^{2}}\int_{t_{0}}^{t}\zeta_{r}^{2}\EE |w_{r}(x)|_{\Frob}^{2}\sum_{k\in\ZZ^{2}}\|\nabla G_{\eta_{r}}(x-\cdot)\|_{L^{2}(\ttB_{\zeta_{r},k};\RR^2)}^{2}\,\dif r=\frac{2}{\pi^{2}}\int_{t_{0}}^{t}\zeta_{r}^{2}\|\nabla G_{\eta_{r}}\|_{L^{2}(\RR^{2};\RR^2)}^{2}\EE |w_{r}(x)|_{\Frob}^{2}\,\dif r \\
          & \le\frac{2}{\pi^{2}}||\nabla G_{1}\|_{L^{2}(\RR^{2};\RR^2)}^{2}\int_{t_{0}}^{t}\zeta_{r}^{2}\eta_{r}^{-2}\vvvert w_{r}\vvvert_{2}^{2}\dif r=\frac{1}{2\pi^{3}}\int_{t_{0}}^{t}\zeta_{r}^{2}\eta_{r}^{-2}\vvvert w_{r}\vvvert_{2}^{2}\dif r.
  \end{align*}
  This yields \cref{eq:propcouplingconclusion}.
\end{proof}
When we consider multipoint statistics, it will also be important
to know something about the dependence structure of multiple instances
of the coupling described in \cref{prop:coupling}. In particular, we
need to know that if we sample far-away points of different space-time
white noises, then we get independent random variables. That is the
content of the following proposition.
\begin{prop}
  \label{prop:dBQVub}Let $\bigl((\zeta_{r}^{(i)})_{r\in[T^{(i)}_{0},T^{(i)}_{1}]}\bigr)_{i=1}^{N}$
  be a family of $N$ deterministic positive measurable $\RR$-valued
  paths. %
  For each $i$, let $\dif\widetilde{W}^{(i)}$
  be the space-time white noise constructed in \cref{prop:coupling} with $t_0 \setto T_0^{(i)}$ and $\zeta\setto\zeta^{(i)}$, extended by $0$ outside of $[T_0^{(i)},T_1^{(i)}]$.
  Let $(\chi_{k,\ell,\ell';r}^{(i)})_{r}$,
  $(X_{k,\ell;r}^{(i)})_{r}$, and $(\widetilde{Z}{}_{k,\ell;r}^{(i)})_{r}$
  be the corresponding objects constructed in \textup{\zcref[range]{eq:chidef,eq:Ztildelgt1def}}.
  We assume that the $X_{k,\ell;r}^{(i)}$s are independent across different
  $i$ as well as across different $k$ and $\ell$. Let $g_{r}^{(1)},\ldots,g_{r}^{(N)}\colon\RR^{2}\to\RR$
  for each $r$. Suppose that
  \begin{equation}
    \dist(\supp g_{r}^{(i)},\supp g_{r}^{(j)})\ge\sqrt{2}(\zeta_{r}^{(i)}+\zeta_{r}^{(j)})\label{eq:suppsfarenough}
  \end{equation}
  for each $i,j,r$, where $\dist$ represents the minimum Euclidean
  distance between two sets. Then if we define
  \[
    B^{(i)}_{t}=\gamma_{\rho}\int_{T_{0}^{(i)}}^{t}\!\!\int g^{(i)}(r,y)\,\dif\widetilde{W}_{r}^{(i)}(y),
  \]
  the $B^{(i)},\ldots,B^{(N)}$ are independent.
\end{prop}
\begin{proof}
  Write $I_i \coloneqq [T_0^{(i)}, T_1^{(i)}]$.
  Recalling \cref{eq:Wttilde}, we have
  \[
    \dif\widetilde{W}_{r}^{(i)}(x)=\sum_{\substack{k\in\ZZ^{2}\\
        \ell\in\NN_{0}
      }
    }\varphi_{\zeta_{r}^{(i)},k,\ell}^{(i)}(x)[\clS_{\zeta_{r}^{(i)}}w_{r}^{2}(x)]^{-1/2}\,\dif\widetilde{Z}_{k,\ell;r}^{(i)},
  \]
  where we also set $\dif \widetilde Z^{(i)} = 0$ outside $I_i$.
  From this we see that, for all $y^{(i)},y^{(j)}\in\RR^{2}$,
  we have
  \begin{equation}
    \begin{aligned}\dif & [\widetilde{W}^{(i)}(y^{(i)}),\widetilde{W}^{(j)}(y^{(j)})]_{r}                                                                                                                                                              \\
                        & =\sum_{\substack{k^{(i)},k^{(j)}\in\ZZ^{2}                                                                                                                                                                                   \\
      \ell^{(i)},\ell^{(j)}\in\NN_{0}
      }
      }\varphi_{\zeta_{r}^{(i)},k^{(i)},\ell^{(i)}}(y^{(i)})\varphi_{\zeta_{r}^{(j)},k^{(j)},\ell^{(j)}}(y^{(j)})                                                                                                                         \\
                        & \qquad\qquad\qquad\qquad\times[\clS_{\zeta_{r}^{(i)}}w_{r}^{2}(y^{(i)})]^{-1/2}\dif[\widetilde{Z}_{k^{(i)},\ell^{(i)}}^{(i)},\widetilde{Z}_{k^{(j)},\ell^{(j)}}^{(j)}]_{r}[\clS_{\zeta_{r}^{(j)}}w_{r}^{2}(y^{(j)})]^{-1/2}.
    \end{aligned}
    \label{eq:dWtildeiWtildej}
  \end{equation}
  Recalling \cref{eq:Ztildek0} and \cref{eq:Zdef}, we have
  \[
    \widetilde{Z}_{k,0;t}^{(i)}=\int_{T_0^{(i)}}^{t}\!\!\!\int\varphi_{\zeta_{r}^{(i)},k,0}(x)w_{r}(x)\,\dif W_{r}(x)
  \]
  for $t \in I_i$, so
  \begin{align}
    \dif[\widetilde{Z}_{k^{(i)},0}^{(i)},\widetilde{Z}_{k^{(j)},0}^{(j)}]_{r} & = \mathbf{1}_{I_i \cap I_j}(r) \left(\int\varphi_{\zeta_{r}^{(i)},k^{(i)},0}(x)\varphi_{\zeta_{r}^{(j)},k^{(j)},0}(x)w_{r}^{2}(x)\,\dif x\right)\dif r.\label{eq:dtildeZij0}
  \end{align}
  On the other hand, then
  \begin{equation}
    \dif[\widetilde{Z}_{k^{(i)},\ell^{(i)}}^{(i)},\widetilde{Z}_{k^{(j)},\ell^{(j)}}^{(j)}]_{r}=0\qquad\text{ if }i\ne j\text{ and }(\ell^{(i)},\ell^{(j)})\ne(0,0)\label{eq:differentls}
  \end{equation}
  by \cref{eq:Ztildelgt1def} and the assumption that the $X_{k,\ell;r}^{(i)}$
  and $X_{k,\ell;r}^{(j)}$ are independent. Using \cref{eq:dtildeZij0}
  and \cref{eq:differentls} in \cref{eq:dWtildeiWtildej}, we see that,
  if $i\ne j$, then we have the series of implications
  \begin{align}   & \dif[\widetilde{W}^{(i)}(y^{(i)}),\widetilde{W}^{(j)}(y^{(j)})]_{r}/\dif r\ne0                                                                                                                                                           \nonumber\\
                  & \quad\implies(\exists k^{(i)},k^{(j)})\ \varphi_{\zeta_{r}^{(i)},k^{(i)},0}(y^{(i)})\varphi_{\zeta_{r}^{(j)},k^{(j)},0}(y^{(j)})\int\varphi_{\zeta_{r}^{(i)},k^{(i)},0}(x)\varphi_{\zeta_{r}^{(j)},k^{(j)},0}(x)w_{r}^{2}(x)\,\dif x\ne0 \nonumber\\
                  & \quad\implies \ttB_{\zeta_{r}^{(i)},k_{\zeta_{r}^{(i)}}(y^{(i)})}\cap \ttB_{\zeta_{r}^{(j)},k_{\zeta_{r}^{(j)}}(y^{(j)})}\ne\emptyset                                                                                                          \nonumber\\
                  & \quad\implies|y^{(i)}-y^{(j)}|\le\sqrt{2}(\zeta_{r}^{(i)}+\zeta_{r}^{(j)}).
                    \label{eq:implications}
  \end{align}
  (By this we mean really that the implications hold, not that the statements hold unconditionally.)
  This means that
  \[
    g^{(i)}(r,y^{(i)})g^{(j)}(r,y^{(j)})\dif[\widetilde{W}^{(i)}(y^{(i)}),\widetilde{W}^{(j)}(y^{(j)})]_{r}/\dif r=0,
  \]
  since in order to have $g^{(i)}(r,y^{(i)})g^{(j)}(r,y^{(j)})\ne0$,
  we must have $|y^{(i)}-y^{(j)}|\ge\sqrt{2}(\zeta_{r}^{(i)}+\zeta_{r}^{(j)})$
  by the assumption \cref{eq:suppsfarenough}, which means that $\dif[\widetilde{W}^{(i)}(y^{(i)}),\widetilde{W}^{(j)}(y^{(j)})]_{r}/\dif r=0$
  by \cref{eq:implications}. So for all $r \in \R$, we have
  \[
    \frac{\dif[B^{(i)},B^{(j)}]_{r}}{\dif r}=\gamma_{\rho}^{2}\iint g^{(i)}(r,y^{(i)})g^{(j)}(r,y^{(j)})\frac{\dif[\widetilde{W}^{(i)}(y^{(i)}),\widetilde{W}^{(j)}(y^{(j)})]_{r}}{\dif r}\,\dif y^{(i)}\,\dif y^{(j)}=0,
  \]
  and from this we conclude that that $B^{(1)},\ldots,B^{(N)}$ are
  independent.
\end{proof}

\printnomenclature[0.55in]

\bibliographystyle{hplain-ajd}
\bibliography{fbsdenl}

\end{document}